%% file: drgrevision.arxiv.tex
\documentclass[12pt,a4paper]{article}

\usepackage[colorlinks=true,citecolor=black,linkcolor=black,urlcolor=blue,backref=page]{hyperref}
\newcommand{\arxiv}[1]{\href{http://arxiv.org/abs/#1}{\texttt{arXiv:\linebreak[2]#1}}}

\usepackage{amssymb,amsmath,amsthm,enumerate,mathrsfs,url,graphicx,ifthen}

\renewcommand*{\backref}[1]{}
\renewcommand*{\backrefalt}[4]{%
     \ifcase #1 (Not cited.)%
     \or        (Cited on p.~#2.)%
     \else      (Cited on pp.~#2.)%
     \fi}

\newcommand{\G}{\Gamma}
\newcommand{\gauss}[2]{\genfrac[]{0pt}{}{#1}{#2}}
\renewcommand{\binom}[2]{\genfrac(){0pt}{}{#1}{#2}}

\theoremstyle{plain}
\newtheorem{theorem}{Theorem}[section]
\newtheorem{corollary}[theorem]{Corollary}
\newtheorem{conj}[theorem]{Conjecture}
\newtheorem{lemma}[theorem]{Lemma}
\newtheorem{prop}[theorem]{Proposition}

\theoremstyle{remark}
\newtheorem{problem}{Problem}

\newtheorem*{biconj}{Bannai-Ito conjecture}

\newcommand{\R}{{\mathbb R}}

\newcommand{\A}{\mathcal{A}}
\newcommand{\B}{\mathcal{B}}
\newcommand{\C}{\mathcal{C}}
\newcommand{\D}{\mathcal{D}}

\newcommand{\AL}{\mathbb{A}}
\newcommand{\TT}{\mathbb{T}}

\def\tr{\mathop{\rm tr }\nolimits}

\renewcommand{\cal}{\mathcal}

\newcommand{\ster}{\star}

\newlength{\BiblioSpacing}
\renewenvironment{thebibliography}[1]{%
\begin{oldthebibliography}{#1}%
\setlength{\parskip}{\BiblioSpacing}
\setlength{\itemsep}{\BiblioSpacing}
}%
{%
\end{oldthebibliography}%
}

\title{\bf Distance-regular graphs\footnote{This version is published in the Electronic Journal of Combinatorics (2016), \#DS22.}}
\author{Edwin R. van Dam\\
\small Department of Econometrics and O.R.\\[-0.8ex]
\small Tilburg University\\[-0.8ex]
\small The Netherlands\\
\small\tt Edwin.vanDam@uvt.nl\\
\and
Jack H. Koolen\\
\small School of Mathematical Sciences\\[-0.8ex]
\small University of Science and Technology of China\\[-0.8ex]
\small and\\[-0.8ex]
\small Wu Wen-Tsun Key Laboratory of\\[-0.8ex]
\small Mathematics of CAS\\[-0.8ex]
\small Hefei, Anhui, 230026, China\\
\small\tt koolen@ustc.edu.cn\\
\and
Hajime Tanaka\\
\small Research Center for\\[-0.8ex]
\small Pure and Applied Mathematics\\[-0.8ex]
\small Graduate School of Information Sciences\\[-0.8ex]
\small Tohoku University\\[-0.8ex]
\small Sendai 980-8579, Japan\\
\small\tt htanaka@tohoku.ac.jp\\
}
\date{
}

\begin{document}
\maketitle


\begin{abstract}
\noindent This is a survey of distance-regular graphs. We present an
introduction to distance-regular graphs for the reader who is unfamiliar with
the subject, and then give an overview of some developments in the area of
distance-regular graphs since the monograph `BCN' [Brouwer, A.E., Cohen, A.M.,
Neumaier, A.,
    {\sl Distance-Regular Graphs}, Springer-Verlag, Berlin, 1989] was written.

\end{abstract}


\newpage

\begin{tableofcontents}
\end{tableofcontents}


\section{Introduction}\label{sec1:introduction}

\input{1_introduction.txt}


\section{An introduction to distance-regular graphs}\label{sec2:basics}

\input{2_drgintroduction.txt}


\section{Examples}\label{sec3:classical}

\input{3_examples.txt}


\section{More background}\label{sec:morebackground}

\input{4_morebackground.txt}


\section{\texorpdfstring{$Q$-polynomial}{Q-polynomial} distance-regular graphs}\label{sec:Qpol}

\input{5_Qpolynomial.txt}


\section{The Terwilliger algebra and combinatorics}\label{sec:Talg6}

\input{6_Talgebra.txt}


\section{Growth of intersection numbers and bounds on the diameter}\label{sec:diameterbounds}

\input{7_diameterbounds.txt}


\section{The Bannai-Ito conjecture}\label{sec:BIconjecture}

\input{8_BIconjecture.txt}


\section{Geometric distance-regular graphs}\label{sec:Geometricdrg}

\input{9_geometric.txt}


\section{Spectral characterizations}\label{sec:spectralchar}

\input{10_spectralchar.txt}


\section{Subgraphs}\label{sec: subgraphs}

\input{11_subgraphs.txt}


\section{Completely regular codes}\label{sec:crc}

\input{12_crc.txt}


\section{More combinatorial properties}\label{sec:morecombinatorial}

\input{13_morecomb.txt}


\section{Multiplicities}\label{sec:mult}

\input{14_multiplicities.txt}


\section{Applications}\label{sec:applications}

\input{15_applications.txt}


\section{Miscellaneous}\label{sec:misc}

\input{16_miscellaneous.txt}


\section{Tables}\label{sec:tables}

\input{17_tables.txt}


\section{Open problems and research directions}\label{sec12:Openproblems}

\input{18_openproblems.txt}


\bigskip
\noindent {\bf Acknowledgements.}
The authors thank Eiichi Bannai, Andries Brouwer, Bart De
Bruyn, Miquel \`{A}ngel Fiol, Willem Haemers,
Akira Hiraki,
Tatsuro Ito, Sasha Juri\v{s}i\'{c}, Greg
Markowsky, Bill Martin, Akihiro Munemasa, Jongyook Park, Renata Sotirov,
Hiroshi Suzuki,
and Paul Terwilliger for inspiring
conversations on distance-regular graphs and this survey paper.

The authors also thank Sasha Gavrilyuk, Neil Gillespie, Akihide Hanaki, Arnold Neumaier, Nobuaki Obata, L\'{a}szl\'{o} Pyber, Achill Sch\"{u}rmann, Fr\'{e}d\'{e}ric
Vanhove, and a referee for comments on an earlier version of this paper and making some
valuable suggestions.

A significant part of the work was done while Jack Koolen was working at POSTECH.
Jack Koolen gratefully acknowledges the support of the `100 talents' program of the Chinese Academy of Sciences.
Research of Jack Koolen is partially supported by the National Natural Science Foundation of China (No. 11471009).
The work of Hajime Tanaka was supported by JSPS Grant-in-Aid
for Scientific Research No.~23740002 and No.~25400034.


\input{bibliography.txt}


\end{document}

%% file: 1_introduction.txt

Distance-regular graphs are graphs with a lot of combinatorial
symmetry, in the sense that given an arbitrary ordered pair of
vertices at distance $h$, the number of vertices that are at
distance $i$ from the first vertex and distance $j$ from the second
is a constant (i.e., does not depend on the chosen pair) that only
depends on $h,i$, and $j$. Biggs introduced distance-regular
graphs, by observing that several combinatorial and linear
algebraic properties of distance-transitive graphs were holding for
this wider class of graphs, see Biggs' monograph \cite{biggs} from
1974. Well-known examples are the Hamming graphs and the Johnson
graphs, as these graphs link the subject of distance-regular graphs
to coding theory and design theory, respectively. But there are
many more interesting links to other subjects, such as finite group
theory (and distance-transitive graphs), representation theory,
finite geometry, association schemes, and orthogonal polynomials.
Moreover, distance-regular graphs are frequently used as test
instances for problems on general graphs and other combinatorial
structures, such as problems related to random walks and from
combinatorial optimization. An example is Hoffman's (unpublished;
see \cite[Thm.~3.5.2]{BrHa}) coclique bound, which was first proved
by Delsarte \cite[p.~31]{del} for distance-regular graphs with
diameter two (also known as strongly regular graphs), as an example
of his linear programming method. Distance-regular graphs have
applications in several fields besides the already mentioned
classical coding and design theory, such as (quantum) information
theory, diffusion models, (parallel) networks, and even finance.

In this survey of distance-regular graphs, we give an overview of some
developments in the area of distance-regular graphs since the monograph `BCN'
by Brouwer, Cohen, and Neumaier \cite{bcn} from 1989 was written. This
influential monograph, which is almost like an encyclopedia of distance-regular
graphs, inspired many researchers to work on distance-regular graphs, such as
the authors of this survey. Since then, many papers have been written, many
more than the ones we will discuss in this overview. We intend to discuss the
most relevant developments of the past twenty-seven years, realizing that `most
relevant' is quite subjective. Perhaps we should say that we give our personal
view on the past twenty-seven years. The same is true when we discuss the major
open problems in the area. A recent major breakthrough is the proof of one of
the Bannai-Ito conjectures made in the influential monograph by Bannai and Ito \cite{bi}
from 1984, i.e., the one that states that there are finitely many
distance-regular graphs with given valency (at least three). Just as important
is the theorem stating that there are finitely many non-geometric
distance-regular graphs with both valency and diameter at least three and
smallest eigenvalue at least a given number; a generalization of a well-known
result about strongly regular graphs. The classification of tridiagonal pairs
is an example of an important recent breakthrough in algebraic combinatorics
that is completely inspired by the major (still) open problem of classifying
the $Q$-polynomial distance-regular graphs. The construction of the twisted
Grassmann graphs, that is, of this family of strange examples that were not
expected to be in the picture, gave a better perspective on how difficult this
classification problem really is. It seems to suggest that the problem cannot
be solved just by algebraic methods. In addition, we need to better understand
geometric distance-regular graphs.

This survey is organized as follows. After this brief introduction, we present
an introduction to distance-regular graphs for the reader that is unfamiliar
with the subject. We then present the classical examples of distance-regular
graphs, and an overview of the most important constructions since `BCN' \cite{bcn}. In
Section \ref{sec:morebackground}, we give more necessary and advanced
background for the remaining part of the paper. We then treat several subjects
in Sections \ref{sec:Qpol}-\ref{sec:mult}, for example $Q$-polynomial
distance-regular graphs, the Terwilliger algebra, the Bannai-Ito conjecture,
geometric distance-regular graphs, and spectral characterizations. In Section
\ref{sec:applications}, we discuss important applications of distance-regular
graphs, namely in combinatorial optimization and in the area of
random (classical and quantum) walks (which model diffusion models, dynamic
stock portfolios, and the abelian sandpile, for example). In Section
\ref{sec:misc}, we then discuss some miscellaneous topics, and in Section
\ref{sec:tables} we report progress on the `feasibility' and `uniqueness' of
the intersection arrays that were listed in the tables of parameter sets of
distance-regular graphs in `BCN' \cite{bcn}. We conclude with a section on open problems
and some directions for future research.

Note that we will focus our attention on distance-regular graphs with diameter at least three. We do not completely
exclude strongly regular graphs (the diameter two case), but we are of the opinion that they form a subject of their
own. A separate survey of strongly regular graphs would therefore be warmly welcomed. For some information we refer to
the recent book by Brouwer and Haemers on spectra of graphs \cite[Ch.~9]{BrHa} and the paper by Cohen and Pasechnik \cite{CPsrg}.
Also bipartite distance-regular graphs
with diameter three form a separate subject. These graphs are equivalent to symmetric designs, for which we refer to
the monograph by Ionin and Shrikhande \cite{symmetricdesigns}.

%% file: 2_drgintroduction.txt

In this section we intend to introduce some basics about
distance-regular graphs to the reader that is unfamiliar with
the topic. This includes some basic proofs and questions to
give some (first) flavors of the area of distance-regular
graphs.

\subsection{Definition}\label{sec2:definition}

Let $\G$ denote a simple, undirected, connected graph, with vertex set $V=V_{\G}$ of size $v=|V|$. Whenever there is an
edge between two vertices $x$ and $y$, we say that $x$ is {\em adjacent} to $y$, or that $x$ and $y$ are {\em
neighbors}, use the notation $x \sim y$, and denote the edge by $xy$. The {\em distance} in the graph between two
vertices $x$ and $y$ is denoted by $d(x,y)=d_{\G}(x,y)$, and is given by the length of the shortest path between
$x$ and $y$. The {\em diameter} of the graph is $D=D_{\G}=\max_{x,y\in V} d(x,y)$. The set of vertices at distance
$i$ from a given vertex $z\in V$ is denoted by $\G_i(z)$, for $i=0,1,\dots,D$. The {\em distance-$i$ graph} $\G_i$
is the graph with vertex set $V$, where two vertices $x$ and $y$ are adjacent if and only if $d_{\G}(x,y)=i$. A
graph is called {\em bipartite} if the vertex set can be partitioned into two parts such that every edge has one end
(vertex) in each part.

A connected graph $\G$ with diameter $D$ is called {\em
distance-regular} if there are constants $c_i, a_i, b_i$ ---
the so-called {\em intersection numbers} --- such that for all
$i=0,1,\dots, D$, and all vertices $x$ and $y$ at distance
$i=d(x,y)$, among the neighbors of $y$, there are $c_i$ at
distance $i-1$ from $x$, $a_i$ at distance $i$, and $b_i$ at
distance $i+1$. It follows that $\G$ is a regular graph with
valency $k=b_0$, and that $c_i+a_i+b_i=k$ for all $i=0,1,\dots,
D$. By these equations, the intersection numbers $a_i$ can be
expressed in terms of the others, and it is standard to put
these others in the so-called {\em intersection array}
$$\{b_0,b_1,\dots,b_{D-1};c_1,c_2,\dots,c_D\}.$$
Note that $b_D=0$ and $c_0=0$ are not included in this array, whereas $c_1=1$
is included (note that all numbers in the intersection array are positive
integers). Also the number of vertices can be obtained from the intersection
array. In fact, every vertex has a constant number of vertices $k_i$ at given
distance $i$, that is, $k_i=|\G_i(z)|$ for all $z \in V$. Indeed, this follows
by induction and counting the number of edges between $\G_i(z)$ and
$\G_{i+1}(z)$ in two ways. In particular, it follows that $k_0=1$ and
$k_{i+1}=b_ik_i/c_{i+1}$ for all $i=0,1,\dots,D-1$. The number of vertices now
follows as $v=k_0+k_1+ \cdots + k_D$. In combinatorial arguments such as the
above, it helps to draw pictures; in particular, of the so-called {\em
distance-distribution diagram}, as depicted in Figure \ref{pic:dddiagram}.
\begin{figure}[h!]
\centering
\includegraphics[viewport=0 20 1280 170,width=120mm]{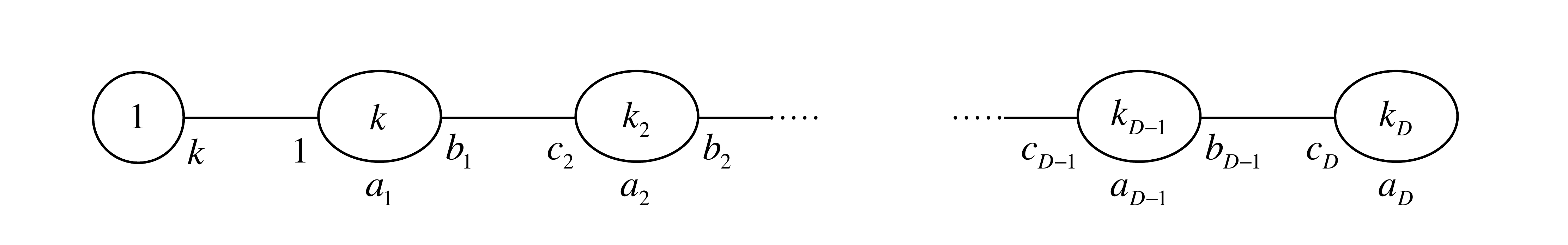}
\caption{Distance-distribution diagram} \label{pic:dddiagram}
\end{figure}

\subsection{A few examples}

\subsubsection{The complete graph} The complete graphs $K_v$
(i.e., the graphs where all vertices are adjacent to each other)
are the distance-regular graphs with diameter $1$, and have
intersection array $\{v-1;1\}$ if $v>1$.

\subsubsection{The polygons}\label{sec:polygons} The polygons (cycles) $C_v$ are the
distance-regular graphs with valency $2$, and have intersection
array $\{2,1,\dots,1;1,1,\dots,1\}$ if $v$ is odd, and
$\{2,1,\dots,1;1,\dots,1,2\}$ if $v$ is even.

\subsubsection{The Petersen graph and other Odd
graphs}\label{sec:oddgraphs} The well-known Petersen graph is a
distance-regular graph with diameter $2$, and has intersection
array $\{3,2;1,1\}$. The distance-regular graphs with diameter $2$
are very special, and form a subject of their own. They are exactly
the connected strongly regular graphs (for more on such graphs, see
\cite[Ch.~9]{BrHa}).

The Petersen graph is the same as the Odd graph $O_3$. For an integer $k \ge
2$, the vertices of the Odd graph $O_k$ are the $(k-1)$-subsets of a set of
size $2k-1$, and two vertices are adjacent if the corresponding subsets are
disjoint. The Odd graph $O_k$ is distance-regular with diameter $k-1$. For odd
$k=2l-1$, its intersection array is
$\{k,k-1,k-1,\dots,l+1,l+1,l;1,1,2,2,\dots,l-1,l-1\}$. For even $k=2l$, the
intersection array is $\{k,k-1,k-1,\dots,l+1,l+1;1,1,2,2,\dots,l-1,l-1,l\}$.
Consequently, the numbers $a_i$ are zero for all $i=0,1,\ldots,D-1$, but
$a_D=l>0$.

\subsection{Which graphs are determined by their intersection
array?}\label{sec:determinedbyarray}

All graphs in the above examples have the property that they are the only ones
that are distance-regular with the given intersection array. In other words,
given the particular intersection array, it is possible to reconstruct the
graph uniquely (up to isomorphism). A typical combinatorial argument can be
used to show this for the Petersen graph.

\begin{prop} The Petersen graph is determined as distance-regular
graph by its intersection array.
\end{prop}
\begin{proof}
Consider a distance-regular graph with intersection array $\{3,2;1,1\}$.
Take a vertex $z$; it has $b_0=3$ neighbors, each of which has $b_1=2$
neighbors at distance 2 from $z$ (and hence there are no triangles in the
graph; $a_1=0$). Each of the vertices at distance 2 from $z$ has precisely
$c_2=1$ common neighbors with $z$ (hence there are no 4-cycles in the graph
either). This already determines the $10=1+3+6$ vertices and all edges except
those having both ends in $\G_2(z)$. The graph induced on $\G_2(z)$ is regular
with valency $a_2=2$, and because the graph has no triangles, this must be a
$6$-cycle. Now there is (up to isomorphism) only one way to make this $6$-cycle
if one takes into account that the entire graph has no triangles and
$4$-cycles; we obtain the Petersen graph as the only graph with intersection
array $\{3,2;1,1\}$; see Figure \ref{pic:petersen}.
\begin{figure}[ht!]
\centering
\includegraphics[angle=90,width=30mm]{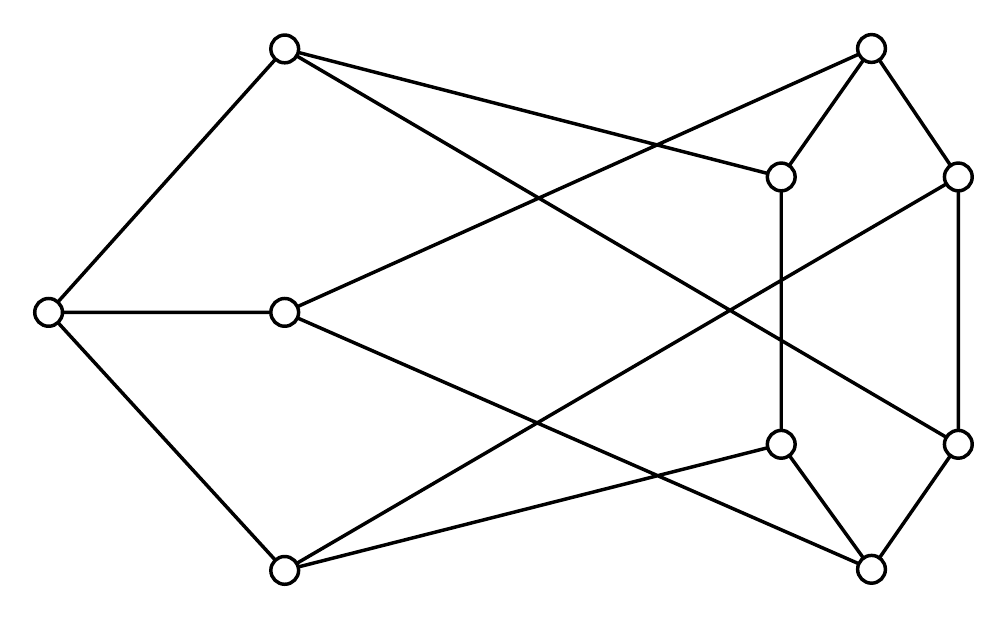}
\caption{The Petersen graph} \label{pic:petersen}
\end{figure}
\end{proof}

\noindent This is clearly a very interesting property; however
it does not hold for all intersection arrays. The smallest
intersection array (smallest in terms of the number of
vertices) that corresponds to more than one graph is
$\{6,3;1,2\}$; it corresponds to the Hamming graph $H(2,4)$
(also known as the lattice graph $L_2(4)$) and the Shrikhande
graph.

One of the problems in the field of distance-regular graphs is
therefore to determine which graphs are determined by their
intersection array, and more generally, to determine all graphs
that have the same intersection array as a given graph. While
for many graphs this problem is still open, for the Odd graphs
the problem was settled already long ago by Moon \cite{Moon}. Her result
was later generalized by Koolen \cite{Ko93} as follows.

\begin{prop}\label{prop:odd} Let $\G$ be a non-bipartite distance-regular graph with diameter
$D \ge 4$, and intersection numbers $a_1=a_2=a_3=0$,
$c_2=1$, and $c_3=c_4=2$. Then $\G$ is an Odd graph.
\end{prop}

\noindent This result shows that we do not always need all
intersection numbers to determine a graph. This is very typical in
the characterization results that we know. We will see more examples
of this later on, for example in the characterizations in Section
\ref{sec:Metsch}. Note that the condition that the graph is
non-bipartite is also a condition on the intersection numbers; it is
not hard to see that a distance-regular graph is bipartite (i.e.,
has no odd cycles) if and only if $a_i=0$ for all $i=0,1,\dots,D$.
To obtain their results, both Moon and Koolen used the
correspondence to a certain Johnson graph; and Moon characterized
this Johnson graph by just a few intersection numbers. Hiraki
\cite{Hi07} also strengthened the result by Moon; he showed ---
among others --- that $a_1=a_2$, $a_4=0$, $c_2=1$, and $c_3=c_4=2$
suffices to determine the Odd graphs among the non-bipartite
distance-regular graphs with diameter $D \ge 5$. These results also
`eliminate' intersection arrays that match the intersection arrays
of the Odd graphs partially.

This brings us to the following question: which intersection
arrays should we look at? Do we need a distance-regular graph
first, before we consider its intersection array? Perhaps
there are beautiful distance-regular graphs that we do not know
of yet. How can we find these? One way is to first try to
classify possible intersection arrays.

In order to find putative intersection arrays of distance-regular
graphs, we should in principle find as many conditions on such
arrays as possible. In this introduction, we will however only
mention some elementary standard conditions. There are many other
--- more technical --- conditions known (some of which we will see
in later sections) that eliminate certain intersection arrays, but
these are beyond the scope of this introduction. We begin with some
combinatorial conditions, and then bring linear algebra into the
game to obtain algebraic conditions.

\subsection{Some combinatorial conditions for the intersection array}

The first trivial conditions that should hold for the
intersection array
$\{b_0,b_1,\dots,b_{D-1};\break c_1,c_2,\dots,c_D\}$ of a
distance-regular graph is that the intersection numbers listed
are positive integers. Moreover, the intersection number
$a_i=b_0-b_i-c_i$ is a nonnegative integer. But we also have
some divisibility conditions as follows.

\begin{prop}
\label{order in b_i, c_i}
With notation as above, the following conditions hold:
\begin{enumerate}[{\em (i)}]
\item $k_{i+1}=\frac{b_0b_1 \cdots b_{i}}{c_1c_2 \cdots c_{i+1}}$
is an integer for $i=0,1,\dots,D-1$,
\item $vk_i$ is even for $i=1,2,\dots,D$,
\item $k_ia_i$ is even for $i=1,2,\dots,D$,
\item $vka_1$ is divisible by $6$.
\end{enumerate}
\end{prop}

\begin{proof}
(i) Earlier on, in Section \ref{sec2:definition},  we
obtained the recurrence $k_{i+1}=b_ik_i/c_{i+1}$ for all
$i=0,1,\dots,D-1$, and this implies that
$$k_{i+1}=\frac{b_0b_1 \cdots b_{i}}{c_1c_2 \cdots c_{i+1}}$$
for $i=0,1,\dots,D-1$. These numbers are clearly positive integers.

(ii) By doubly counting all pairs $(z,e)$, where $z$ is an end
vertex of edge $e$ in $\G_i$, it follows that the number of
edges in $\G_i$ equals $vk_i/2$, which should be an integer.

(iii) Similarly, there are $k_ia_i/2$ edges of $\G$ within
$\G_i(z)$ for a fixed vertex $z$, and this should be an
integer.

(iv) Finally, the number of triangles in $\G$ equals $vka_1/6$.
\end{proof}

\noindent There is also a nice order in the intersection
numbers, and consequently the $k_i$ are unimodal, as we shall see next.

\begin{prop}\label{prop:unimodal} With notation as above, the following conditions hold:
\begin{enumerate}[{\em (i)}]
\item $1=c_1 \le c_2 \le \cdots \le c_D$,
\item $k=b_0 \ge b_1 \ge \cdots \ge b_{D-1}$,
\item If $i+j \le D$, then $c_i \le b_j$,
\item There is an $i$ such that $k_0\le k_1\le\dots\le k_i$ and $k_{i+1}\ge
k_{i+2}\ge\dots\ge k_D$.
\end{enumerate}
\end{prop}

\begin{proof}
(i) and (ii) Let $i=1,2,\dots,D$. Consider two vertices $x$ and $y$ at
distance $i$, and a vertex $z$ that is adjacent to $x$ and at distance $i-1$
from $y$. Now the $c_{i-1}$ neighbors of $y$ that are at distance $i-2$ from
$z$ are all at distance $i-1$ from $x$. Therefore $c_i \ge c_{i-1}$. Similarly,
the $b_i$ neighbors of $y$ that are at distance $i+1$ from $x$ are at distance
$i$ from $z$, hence $b_{i-1} \ge b_i$.

(iii) Consider two vertices $x$ and $y$ at distance $i+j$, and
a vertex $z$ at distance $i$ from $x$ and $j$ from $y$. Then
the $c_i$ neighbors of $z$ that are at distance $i-1$ from $x$
are at distance $j+1$ from $y$. Hence $c_i \le b_j$.

(iv) It follows from (i), (ii), and Proposition \ref{order in b_i, c_i} that
$k_i^2\ge k_{i-1}k_{i+1}$ for $i=1,2,\dots,D-1$. This implies that the $k_i$
are unimodal: there is an $i$ such that $k_0\le k_1\le\dots\le k_i$ and
$k_{i+1}\ge k_{i+2}\ge\dots\ge k_D$.
\end{proof}

\noindent Even though these and other combinatorial conditions
are important, they are insufficient to obtain most of the
advanced results. We need linear algebra.

\subsection{The spectrum of eigenvalues and multiplicities}\label{sec2:evmult}

The adjacency matrix $A$ of a (simple, undirected) graph $\G$
is the $v \times v$ symmetric matrix with entries $0$ and $1$
whose rows and columns are indexed by the vertices of $\G$, and
where $A_{xy}=1$ if and only if $x \sim y$. Because $A$ is real
and symmetric, its eigenvalues are real numbers. The spectrum
of eigenvalues of a graph (that is, of its adjacency matrix)
contains quite some (but in general not all) information about
the graph. Spectra of graphs is a very fruitful subject on its
own, and it has many more applications to distance-regular
graphs than the ones that we shall see here. Good references
for spectra of graphs are the classic monograph by Cvetkovi\'c,
Doob, and Sachs \cite{CDS} and the more recent one by
Brouwer and Haemers \cite{BrHa}.

The {\em adjacency algebra} of $\G$, denoted by
$\AL=\AL(\G)$, is the matrix subalgebra of $M_{v \times
v}(\R)$ of polynomials in $A$, that is, $\AL = \R[A]$. This
algebra plays an important role for distance-regular graphs, as
we shall see later on. We note that the powers of $A$ count
walks in the graph, that is, $(A^{\ell})_{xy}$ equals the
number of walks of length $\ell$ in the graph from $x$ to $y$.
Using this, we can relate the number of distinct eigenvalues to
the diameter of the graph. To do this, assume that $\G$ is
an arbitrary graph with distinct eigenvalues $\theta_0,
\theta_1, \ldots, \theta_d$. Because the minimal polynomial of
$A$ now has degree $d+1$, it is clear that $\{I, A, A^2,
\ldots, A^d\}$ is a basis of $\AL$, and hence that $\dim \AL
= d+1$.

\begin{prop} \label{D<=d}
Let $\G$ be a connected graph with diameter $D$ and
distinct eigenvalues $\theta_0,\theta_1, \dots , \theta_d$.
Then $D \le d$.
\end{prop}
\begin{proof}
Consider two vertices $x$ and $y$ at distance $i \le D$.
Then $(A^{\ell})_{xy} =0$ if $\ell < i$ and $(A^i)_{xy} \neq
0$. This implies that the set of matrices $\{I= A^0, A, \ldots,
A^D\}$ is linearly independent in $\AL$, and hence that $D+1
\le \dim \AL = d+1$.
\end{proof}

\noindent For $i = 0,1, \ldots, d$, we define the matrix $E_i =
\prod_{j=0, j \neq i}^d \frac{A - \theta_j I}{\theta_i -
\theta_j}$. The matrix $E_i$ is the orthogonal projection onto
the eigenspace $V_i$ of $A$ corresponding to $\theta_i$. The set
$\{ E_0, E_1, \ldots, E_d\}$ forms another basis
of $\AL$. Indeed, let ${\bf v}$ be an eigenvector of $A$ with
respect to $\theta_j$. Then $E_i {\bf v} = \delta_{ij} {\bf
v}$. This implies that $\{ E_0, E_1, \ldots, E_d\}$ forms a
linearly independent set of matrices in $\AL$, and hence that
it is a basis of $\AL$. We shall see more of this basis in the
next section.

The adjacency matrix $A_i$ of $\G_i$ is called the {\em
distance-$i$ matrix} of $\G$, for $i=0,1,\dots,D$. Let us now
consider the case that $\G$ is distance-regular. In this case,
we shall see that also $\{I=A_0, A=A_1, A_2, \ldots, A_D\}$ is
a basis of $\AL$, and hence that $D=d$. Translating the
combinatorial definition of distance-regularity into matrix
language, we obtain the equation
\begin{equation}\label{a_i}
A A_i =b_{i-1} A_{i-1} + a_i A_i + c_{i+1} A_{i+1}
\end{equation} for
$i=0,1,\dots,D$. Note that for $i=0$ and $i=D$, the indices in this equation
attain undefined values. Here --- and in similar equations that will follow
later --- we will have the sensible convention that the corresponding summands
are zero (so $b_{-1}A_{-1}=c_{D+1}A_{D+1}=0$). From this recurrence (note that
the coefficients $c_{i+1}$ are nonzero for $i=0,1,\ldots,D-1$), it follows that
there exist polynomials $v_i$ of degree $i$ such that
\begin{equation}\label{distancepolynomials}
A_i = v_i(A)
\end{equation} for $i
=0,1,\ldots, D$. These polynomials also satisfy a three-term recurrence relation
 like (\ref{a_i}), and hence they form a system of orthogonal polynomials.
 Because $\sum_{i=0}^D A_i = J$ (the all-ones matrix) and $AJ=kJ$
(because $\G$ is regular with valency $k$), it follows that
$(\sum_{i=0}^D v_i(A))(A- kI) = 0$. This shows that $\dim \AL
\leq D+1$. We may conclude the following.
\begin{prop}\label{dimleqD+1}
Let $\G$ be a distance-regular graph with diameter $D$.
Then $\dim \AL = D+1$. In particular, $\G$ has exactly
$D+1$ distinct eigenvalues.
\end{prop}

\noindent Remarkably, these $D+1$ distinct eigenvalues of the distance-regular
graph $\G$ can be computed from the intersection numbers only. To see this,
consider the tridiagonal $(D +1) \times (D+1)$ matrix {\em intersection matrix}
\begin{equation}\label{matrixl}L =\left[
\begin{array}{cccccc}
 0 & b_0 & & & & \\
 c_1 & a_1 & b_1 & & 0 &  \\
 & c_2 & \cdot & \cdot & & \\
 & & \cdot & \cdot & \cdot & \\
 & 0 &  & \cdot & \cdot & b_{D-1} \\
 &&&& c_{D} & a_D
\end{array} \right].
\end{equation}
This matrix is diagonalizable because it is similar to a
symmetric matrix. In fact, if $\Delta$ is the diagonal matrix with
diagonal entries $\Delta_{ii}=k_i$ for $i=0,1,\dots,D$, then by
using that $k_{i+1}/k_i=b_i/c_{i+1}$, it can be verified that
$\Delta^{1/2}L\Delta^{-1/2}$ is indeed a symmetric tridiagonal matrix.

Let $\theta$ be an eigenvalue of $L$, and let ${\mathbf u} = (u_0, u_1, \ldots, u_D)^{\top}$ be a corresponding (right) eigenvector,
that is, $L{\mathbf u} = \theta {\mathbf u}$, with $u_0 =1$.
Then $u_1 = \theta/k$ and
\begin{equation}\label{standard_sequence}
	c_i u_{i-1} + a_i u_i + b_i u_{i+1} = \theta u_i
\end{equation}
for $i = 1,2, \ldots, D$.
The sequence $(u_i)_{i=0}^D$ is called the {\em standard}
(or \emph{cosine})
\emph{sequence} of $\G$ with respect to $\theta$.

A consequence of the above symmetrization of $L$ is that the
row vector ${\bf v}={\bf u}^{\top}\Delta$ is a left eigenvector of
$L$. Thus, the components of ${\bf v}$ also satisfy a
recurrence involving the intersection numbers. It can be
verified that these components can be obtained from the
polynomials $v_i$ in (\ref{distancepolynomials}), that
is, ${\bf v}= (v_0(\theta),v_1(\theta),\dots,v_D(\theta))$.
This gives an alternative way to obtain the standard sequence.

\begin{prop}\label{distincteig}
Let $\G$ be a distance-regular graph with diameter $D$.
Then the $D+1$ distinct eigenvalues of $\G$ are precisely the
eigenvalues of $L$.
\end{prop}
\begin{proof}
Let ${\bf u}$ be as above, i.e., an eigenvector of $L$ with
respect to eigenvalue $\theta$, and fix a vertex $x$ of
$\G$. Define the vector ${\bf w}$ by $w_y = u_{d(x,y)}$ for
$y \in V$.
It is not hard (but a bit technical) to check that $A {\bf w} =
\theta {\bf w}$. Indeed, if ${\bf a_i}$ denotes column $x$ of
$A_i$, then ${\bf w}=\sum_{i=0}^D u_i {\bf a_i}$. By
(\ref{a_i}) and the above equations for the standard sequence,
we obtain that
\begin{align*}
A {\bf w}&=\sum_{i=0}^D u_i (b_{i-1}{\bf
a_{i-1}}+a_{i}{\bf a_{i}}+c_{i+1}{\bf a_{i+1}})\\
&= \sum_{i=0}^D
(c_{i}u_{i-1}+a_{i}u_{i}+b_{i}u_{i+1}){\bf a_{i}}=\sum_{i=0}^D
\theta u_i{\bf a_{i}}=\theta {\bf w}.
\end{align*} This shows that all
eigenvalues of $L$ are eigenvalues of $\G$.

What remains is to show that $L$ has $D+1$ distinct
eigenvalues. We already observed that $L$ is diagonalizable, or
in other words, that it has $D+1$ eigenvalues. Because the
intersection numbers $c_1,c_2,\dots,c_D$ are all nonzero, it
follows that the rank of $L - \theta I$ is at least $D$ for all
$\theta \in \R$. This shows that all eigenvalues of $L$ are
distinct, which finishes the proof.
\end{proof}

\noindent Finally, also the multiplicities of the eigenvalues
of $\G$ follow from the intersection numbers, via the standard
sequence. This is known as {\em Biggs' formula}.

\begin{theorem}\label{Biggsformula}{\em (Biggs' formula)}
Let $\G$ be a distance-regular graph with diameter $D$ and
$v$ vertices. Let $\theta$ be an eigenvalue of $\G$ and
$(u_i)_{i=0}^D$ be the standard sequence with respect to
$\theta$. Then the multiplicity $m(\theta)$ of $\theta$ as an
eigenvalue of $\G$ satisfies
$$m(\theta) = \frac{v}{\sum_{i=0}^D k_iu_i^2}.$$
\end{theorem}
\begin{proof}
Let $E$ be the matrix corresponding to the orthogonal
projection onto the eigenspace of $\G$ with respect to $\theta$
(i.e., it is one of the matrices $E_i$ defined before). The
idempotent matrix $E$ only has eigenvalues $0$ and $1$, and the
multiplicity $m(\theta)$ of $\theta$ as an eigenvalue of $\G$
is the same as the multiplicity of eigenvalue $1$ of $E$, which
implies that $m(\theta)=\tr E$. Because $E \in {\AL}$ and
$\{A_0,A_1,\dots,A_D\}$ is a basis of $\AL$, there are real
numbers $\nu_i, i=0,1, \ldots, D$ such that $E = \sum_{i=0}^D
\nu_i A_i$.  Note that $A E = \theta E$, which implies that
$c_i \nu_{i-1} + a_i \nu_i + b_i \nu_{i+1} = \theta \nu_i$ for $
i=0,1, \ldots, D$. From this it follows that $\nu_i = \nu_0 u_i$
for $i=0, 1,
\ldots, D$. By considering the diagonal of the equation $E^2 =
E$, we find that $\sum_{i=0}^D k_i \nu_i^2 = \nu_0$, which implies
that $\sum_{i=0}^D k_i u_i^2 = 1/\nu_0$. Now it follows that
\begin{equation*}
m(\theta) = \tr E = \sum_{x \in V} E_{xx} = v \nu_0 =
\frac{v}{\sum_{i=0}^D k_i u_i^2}.\qedhere
\end{equation*}
\end{proof}

\noindent Thus, it is relatively easy to compute the spectrum of a
distance-regular graph from its intersection array. Remarkably,
the fact that multiplicities of eigenvalues are positive integers is a
condition that many intersection arrays (that satisfy all earlier conditions)
do not satisfy. Note also that algebraically conjugate eigenvalues must have
the same multiplicities. The latter plays an important role in the proof of the
Bannai-Ito conjecture, see Section \ref{sec:proofBIconjecture}.

Related to the vectors {\bf w} in the proof of Proposition
\ref{distincteig} and the standard sequence is the representation
associated to an eigenvalue $\theta$. Let $U$ be a matrix having as
columns an orthonormal basis of the eigenspace of eigenvalue
$\theta$. Then $UU^{\top}$ is the corresponding idempotent matrix
$E$. For every vertex $x \in V$, we denote by $\hat{x}$ the $x$-th
row of $U$. The map $x\mapsto \hat{x}$ is called a {\em
representation} (associated to $\theta$) of $\Gamma$. Given two
vertices $x,y\in V$, we have that $\langle
\hat{x},\hat{y}\rangle=E_{xy}=\nu_0u_{d(x,y)}$, which is why the
standard sequence is also called the cosine sequence. The vectors
$\hat{x}$ ($x\in V)$ all have the same length, $\sqrt{\nu_0}$,
hence we call the representation spherical.

\subsection{Association schemes}

In the previous section we described three different bases for the adjacency algebra $\AL$ of a distance-regular graph:
$\{I, A, A^2,\ldots, A^D\}$, $\{ E_0, E_1, \ldots, E_D\}$, and $\{A_0,A_1,\dots,A_D\}$. The last one was obtained by
explicit use of the property of distance-regularity. A consequence of this is that there are real numbers $p^h_{ij}$
$(h,i,j=0,1,\dots,D)$ such that
\begin{equation}\label{phij}
A_iA_j=\sum_{h=0}^D p^h_{ij} A_h.
\end{equation}
for all $i,j=0,1,\dots,D$. This expression has a combinatorial interpretation:
for each two vertices $x$ and $y$ at distance $h$, there are $p^h_{ij}$
vertices $z$ that are at distance $i$ to $x$ and distance $j$ to $y$. So the
{\em intersection numbers} $p^h_{ij}$ are nonnegative integers. Note that
$p^i_{1,i-1}=c_i, p^i_{1,i}=a_i$, and $p^i_{1,i+1}=b_i$. Also the other
intersection numbers $p^h_{ij}$ can be expressed in terms of the intersection
array. This gives further conditions on the intersection numbers.

What we have here is a special case of an {\em association scheme}: an edge
decomposition of the complete graph into spanning subgraphs $\G_i$
$(i=1,2,\dots,D)$ whose adjacency matrices $A_i$ $(i=1,2,\dots,D)$, together
with $A_0=I$ satisfy (\ref{phij}) for all $i,j=0,1,\dots,D$. Let us look a bit
closer at such an association scheme. Clearly also here $\{A_0,A_1,\dots,A_D\}$
is a basis of an algebra: the {\em Bose-Mesner algebra}. Because the matrices
in this algebra are symmetric, they also commute by (\ref{phij}) (and hence
$p^h_{ij}=p^h_{ji}$). This implies that they share a basis of eigenvectors and
there is also a basis of primitive idempotents $E_i$ $(i=0,1,\dots,D)$ for
$\AL$ (so $ME_i$ is a multiple of $E_i$ for all $M \in \AL$). These $E_i$ are
the projections onto the common eigenspaces, and are the same as before in case
of a distance-regular graph. They satisfy the equations $E_iE_j=\delta_{ij}E_i$
for all $i,j=0,1,\dots,D$ and $\sum_{i=0}^D E_i=I$.

The coefficients to change from one of the two bases to the
other are collected in the so-called {\em eigenmatix} $P$ and
{\em dual eigenmatrix} $Q$. That is,
$$A_i=\sum_{i=0}^D P_{ji}E_j \ \ \mbox{and} \ \ E_i=\frac{1}{v}\sum_{j=0}^D
Q_{ji}A_j$$ for $i=0,1,\dots,D$. Note that so far we did not order the
eigenvalues (or the $E_i$s), so there is some ambiguity in the definition of
$P$ and $Q$. This is not really a problem (as long as we keep some ordering
fixed), except that it has become habit that the first row of $P$ contains the
valencies of the graphs $\G_i$. For this reason, we order the eigenspace of
constant vectors first, so that $E_0=\frac{1}{v}J$, the \emph{trivial}
primitive idempotent of $\AL$. This is also justified by the fact that dually
we could reshuffle the $A_i$ (and $\G_i$), except the trivial $A_0$, and not
really get a `different' association scheme (for a distance-regular graph there
is of course an order given!). Note also that column $i$ of $P$ gives the
eigenvalues of the corresponding graph $\G_i$. The normalization factor
$\frac{1}{v}$ for $Q$ is there to make sure that the entries of $Q$ can be seen
as `dual eigenvalues'; for example the multiplicities $m_i=\tr E_i$ of the
eigenvalues are in the first row of $Q$. Just like in the case of
distance-regular graphs, the eigenvalues and multiplicities, and more
generally, all entries of $P$ and $Q$ can be derived from the intersection
numbers $p^h_{ij}$. In the case of distance-regular graphs, we see now in the
proof of Biggs' formula (Theorem \ref{Biggsformula}) that a column of $Q$ is a
multiple of the corresponding standard sequence.

Observe that the Bose-Mesner (or adjacency) algebra $\AL$ is
not just closed under ordinary matrix multiplication but also
under entrywise (\emph{Hadamard} or \emph{Schur}) matrix
multiplication, denoted by $\circ$. The matrices
$A_0,A_1,\dots,A_D$ are the primitive idempotents of $\AL$ with
respect to $\circ$, i.e., $A_i\circ A_j=\delta_{ij}A_i$,
$\sum_{i=0}^DA_i=J$. This implies that we may write
\begin{equation}\label{Krein}
E_i\circ E_j=\frac{1}{v}\sum_{h=0}^Dq_{ij}^hE_h
\end{equation}
for some real numbers $q_{ij}^h$ $(h,i,j=0,1,\dots,D)$, known as the
\emph{Krein parameters} (or {\em dual intersection numbers}) of $\G$.
Because $\frac{1}{v}q_{ij}^h$ is an eigenvalue of $E_i\circ E_j$, which is a
principal submatrix of the positive semidefinite matrix $E_i \otimes E_j$, we
get the following so-called {\em Krein conditions}.

\begin{prop}\label{Krein condisions}
The Krein parameters $q^h_{ij}$ of an association scheme are
nonnegative numbers.
\end{prop}

\noindent The Krein parameters can be calculated using the dual eigenmatrix as $$q^h_{ij}=\frac{1}{vm_h} \sum_{l=0}^D
k_lQ_{lh}Q_{li}Q_{lj}.$$ This follows from working out the sum of
entries of the matrix $E_i \circ E_j \circ E_h \circ J$ in
different ways. The Krein conditions thus put further constraints
on the intersection array of a distance-regular graph. Moreover, if
a Krein parameter equals zero, then this has certain consequences.
This is perhaps best illustrated in the case of the $Q$-polynomial
distance-regular graphs of Section \ref{sec:Qpol} (see also Section
\ref{sec:2Qpol}), where many Krein parameters vanish. See also
Section \ref{sec:vanishingKrein} for consequences of vanishing
Krein parameters.

From the definition of $P$ and $Q$, it is clear that $PQ=QP=vI$. A different
relation between $P$ and $Q$ can be obtained by working out the trace of
$A_iE_j$ (which equals the sum of entries of $A_i \circ E_j$) in different
ways. This gives the relation $m_jP_{ji}=k_iQ_{ij}$ for all $i,j=0,1,\dots, D$.
Together with $PQ=vI$, this gives certain {\em orthogonality relations} between
the columns (and rows) of $P$. For a distance-regular graph $\G$, this relation
also follows from the fact that the polynomials $v_i$ $(i=0,1,\dots,D)$ form a
system of orthogonal polynomials. Here $P_{ji}=v_i(P_{j1})$, which follows from
(\ref{distancepolynomials}), where we remind the reader that $P_{j1}$
$(j=0,1,\dots,D)$ are the distinct eigenvalues of $\G$.

A final condition that we would like to mention is the
{\em absolute bound}.

\begin{prop}\label{absolute bound}
The multiplicities $m_i$ of an association scheme satisfy the
following bound:
$$\sum_{q^h_{ij} \neq 0} m_h \le \left \{
\begin{array}{lll}
m_im_j& \mbox{if} & i \neq j\\
m_i(m_i+1)/2 & \mbox{if} & i = j.
\end{array}\right. $$

\end{prop}
\begin{proof}
The left hand side equals the rank of $E_i \circ E_j$, because of \eqref{Krein} and the fact that the idempotents are mutually orthogonal (and can be diagonalized simultaneously) so that their ranks are additive.
Let $\mathbf{u}_1,\mathbf{u}_2,\dots,\mathbf{u}_{m_i}$ be a basis of $E_i\R^v$, and let $\mathbf{v}_1,\mathbf{v}_2,\dots,\mathbf{v}_{m_j}$ be a basis of $E_j\R^v$.
Then the column space of $E_i \circ E_j$ is contained in the subspace spanned by the vectors $\mathbf{u}_s\circ\mathbf{v}_t$ $(s=1,2,\dots,m_i,\ t=1,2,\dots,m_j)$, thus proving the inequality for $i \neq j$.
For $i=j$, note that the latter subspace is spanned by the vectors $\mathbf{u}_s\circ\mathbf{u}_t$ with $s\leq t$.
\end{proof}

\noindent For more information on association schemes, we refer to the handbook chapter by Brouwer and Haemers \cite{bh95} and the
recent survey by Martin and Tanaka \cite{martintanaka}.

\subsection{The \texorpdfstring{$Q$-polynomial}{Q-polynomial} property}\label{sec:2Qpol}

We already noted that the ordering of graphs and idempotents in an association scheme is not really important. However,
in an association scheme that comes from a distance-regular graph, the graphs $\G_i$ are ordered naturally according to
distance in the graph. This ordering is called a $P$-polynomial ordering. This term comes from the fact that there are
polynomials $v_i$ of degree $i$ such that $A_i=v_i(A_1)$, as we have seen. The association scheme is therefore also
called {\em $P$-polynomial}. An equivalent property of this ordering is that the intersection numbers are such that
$p_{ij}^h = 0$ whenever $0 \le h < |i-j|$ or $i + j < h $, and $p_{ij}^{i+j} > 0$ (for $i+j \le D$). Because of this
property, we call a $P$-polynomial association scheme also {\em metric}. An association scheme can have at most two
$P$-polynomial orderings (that is, there can be at most two distance-regular graphs in it), except for the association
schemes coming from the polygons. For more on association schemes with two $P$-polynomial orderings, see
Section \ref{sec:2Porder}.

It turns out that many important families of distance-regular graphs, that is,
their corresponding association schemes, satisfy the following dual property.
We say that an association scheme (and in particular, a distance-regular graph)
is $Q$-\emph{polynomial} if there is an ordering $E_0,E_1,\dots,E_D$ and there
are polynomials $q_i$ of degree $i$ such that $E_i=q_i(E_1)$, where the matrix
multiplication is entrywise (so that $(E_i)_{xy}=q_i((E_1)_{xy})$ for all
vertices $x$ and $y$). We also say that the corresponding ordering and the
idempotent $E_1$ are $Q$-\emph{polynomial}. Also here there is an equivalent
property in terms of --- in this case --- the Krein parameters: an association
scheme is called {\em cometric} (with ordering $E_0,E_1,\dots,E_D$) if
$q_{ij}^h = 0$ whenever $0 \le h < |i-j|$ or $i + j < h $, and $q_{ij}^{i+j} >
0$ (for $i+j \le D$). It is well known though that to check the cometric
property, it suffices to check the above conditions for $i=1$ (just like in the
metric case). Dual to the intersection numbers of a distance-regular graph, we
here define  $c_i^{\ster}=q^i_{1,i-1}, a_i^{\ster}=q^i_{1,i}$,
$b_i^{\ster}=q^i_{1,i+1}$, and the Krein array
$\{b_0^{\ster},b_1^{\ster},\dots,b_{D-1}^{\ster};c_1^{\ster},c_2^{\ster},\dots,c_D^{\ster}\}.$

It was conjectured by Bannai and Ito \cite[p.~312]{bi} that for large enough $D$, a primitive $D$-class association
scheme is $P$-polynomial if and only if it is $Q$-polynomial.

\subsection{Delsarte cliques and geometric
graphs}\label{sec:geometricgraphs}

Delsarte \cite[p.~31]{del} obtained a linear programming bound for
cliques in strongly regular graphs. It was observed by Godsil
\cite[p.~276]{Godsilac} that the same {\em Delsarte bound} holds for
distance-regular graphs, as follows.

\begin{prop}\label{delbound} Let $\G$ be a distance-regular graph with valency $k$ and smallest
eigenvalue $\theta_{\min}$. Let $C$ be a clique in $\G$ with $c$
vertices. Then $c \leq 1 - \frac{k}{\theta_{\min}}.$
\end{prop}

\begin{proof}
Let $\chi$ be the characteristic vector of $C$, and let $E$ be the
primitive idempotent corresponding to $\theta_{\min}$. The result
follows from working out $\chi^{\top}E\chi \geq 0$.
\end{proof}

\noindent A clique $C$ in a distance-regular graph $\G$ that attains
this Delsarte bound is called a {\em Delsarte clique}. In Section
\ref{sec:crcdelsarte} we will characterize such cliques as certain completely regular codes.

A distance-regular graph $\G$ is called {\em geometric} (with
respect to $\cal C$) if it contains a collection ${\cal C}$ of
Delsarte cliques such that each edge is contained in a unique $C
\in {\cal C}$. The concept of a geometric distance-regular graph was introduced by Godsil
\cite{Godsil93} and generalizes the concept of a geometric
strongly regular graph as introduced by Bose \cite{Bose} (and
indeed the concepts are the same for diameter two).

Many classical examples of distance-regular graphs (see Section
\ref{sec:clasfamilies}), such as Johnson graphs, Grassmann graphs, and Hamming graphs are
geometric. Bipartite distance-regular graphs are trivially
geometric because in this case every edge is a Delsarte clique.

Even though the class of non-bipartite geometric distance-regular
graphs is clearly much more restricted than the class of arbitrary
distance-regular graphs, Koolen and Bang
\cite{KoBa10} showed that for fixed smallest eigenvalue there are
only finitely many non-geometric distance-regular graphs with both
valency and diameter at least three (see Theorem
\ref{thm:nongeometric}). They in fact conjectured that for fixed
smallest eigenvalue there are finitely many distance-regular graphs
with diameter at least three that are not a cycle, Hamming graph,
Johnson graph, Grassmann graph, or bilinear forms graph. This would
generalize a result by Neumaier
\cite{Neu80} on strongly regular graphs. On the other hand, because
geometric distance-regular graphs have more structure than
arbitrary distance-regular graphs, it may be possible to classify
them, or at least the $Q$-polynomial ones with
large diameter.

\subsection{Imprimitivity}

A connected graph $\G$ with diameter $D$ is called imprimitive if not all graphs $\G_i$ $(i=1,2,\dots,D)$ are
connected. Bipartite graphs are examples of imprimitive graphs ($\G_2$ is disconnected). Among the distance-regular
graphs, there are also the antipodal graphs that are imprimitive. These are the graphs for which $\G_D$ is a disjoint
union of complete graphs. In fact, Smith's theorem states that these are all possibilities (see
\cite[Thm.~4.2.1]{bcn}), except for the polygons (indeed, for example $C_9$ has $D=4$ and only $\G_3$ is disconnected
in this case).

\begin{theorem}\label{prop:imprimitive}{\em (Smith's theorem)}
An imprimitive distance-regular graph with valency $k>2$ is
bipartite and/or antipodal.
\end{theorem}

\noindent There is much more to say than this seemingly clear and simple statement. For this we refer to Alfuraidan and
Hall \cite[Thm.~2.9]{AlHaSmith}, who revisited Smith's theorem by working out more precisely all the cases that can
occur.

If $\G$ is a bipartite distance-regular graph, then
$\G_2$ is a graph with two components. The induced graphs on
these components are called the {\em halved graphs} of $\G$.

\begin{prop}\label{prop:halved}
The halved graphs of a bipartite distance-regular graph are
distance-regular.
\end{prop}

\noindent We already noted before that bipartiteness of a distance-regular
graph can be seen from its intersection numbers. Clearly
this is the case whenever $a_i=0$ for all $i=1,2,\dots,D$.

A distance-regular graph is antipodal whenever $b_i=c_{D-i}$ for all
$i=0,1,\dots,D$, except possibly $i=\lfloor D/2 \rfloor$. If $\G$ is an
antipodal distance-regular graph, then by definition, $\G_D$ is a disjoint
union of cliques. These cliques are called the {\em fibres} of $\G$. We can
also construct a smaller distance-regular graph from an antipodal
distance-regular graph: its {\em folded graph} $\overline{\G}$. Its vertices
are the fibres of $\G$, and two such fibres are adjacent whenever there is an
edge (in $\G$) between them. We also say the $\G$ is an {\em antipodal
$r$-cover} of $\overline{\G}$, where $r$ is the size of the cliques of $\G_D$.

\begin{prop}\label{prop:folded}
The folded graph of an antipodal distance-regular graph is
distance-regular.
\end{prop}

\noindent Typically, but certainly not always (see \cite{AlHaSmith}), the
halved graphs or folded graphs of an imprimitive distance-regular graph are
primitive (that is, not imprimitive). This suggests that the theory of
distance-regular graphs can be boiled down to that of primitive
distance-regular graphs. This is not the case however. There is no unique
recipe to construct imprimitive distance-regular graphs from the primitive
ones, for example. The halving and folding constructions mentioned above cannot
be reversed in a generic way, at least not in general. This is best illustrated
by the imprimitive distance-regular graphs with diameter three. All of these
have as halved or folded graph a complete graph. Sometimes, however, there is
an easy way to construct an imprimitive distance-regular graph from a primitive
one as follows. The {\em bipartite double} of a graph $\G$ with vertex set $V$
is the graph with vertex set $V \times \{0,1\}$, where two vertices $(x,i)$ and
$(y,j)$ are adjacent whenever $x$ is adjacent to $y$ in $\G$ and $i \neq j$.
The {\em extended bipartite double} of $\G$ is a variation on this: it has the
same vertex set, and besides the edges of the bipartite double, it has
additional edges between $(x,0)$ and $(x,1)$, $x \in V$.

If $\G$ is a distance-regular {\em generalized odd graph} (also
called {\em almost bipartite graph}) with diameter $D$, that is, if
it has intersection numbers $a_i=0$ for $i<D$ and $a_D>0$ (like the
Odd graphs), then the bipartite double of $\G$ is distance-regular
with diameter $2D+1$. This situation is interesting for several
reasons, one of them being that this bipartite double is not just
bipartite, but it is also an antipodal 2-cover of $\G$. The {\em
Doubled Odd graphs} (for example) are thus showing that
bipartiteness and antipodality can occur in the same graph. Note by
the way that the folded graph of this Doubled Odd graph is again
the Odd graph, but the halved graphs are not (these are isomorphic
to $\G_2$, a Johnson graph).

More generally, one can see from the intersection array of a
distance-regular graph whether the bipartite double or extended
bipartite double is distance-regular (see \cite[\S 1.11]{bcn}).

\subsection{Distance-transitive graphs}

Distance-regular graphs were `invented' by Biggs (for an early
account, see his monograph \cite{biggs}) while he was studying so-called
distance-transitive graphs. An automorphism of a graph is a bijection from the
vertex set to itself that respects adjacencies, i.e., that maps edges to edges.
A graph is called {\em distance-transitive} if it has a group of automorphisms
that acts transitively on each of the sets of pairs of vertices at distance
$i$, for $i=0,1,\dots,D$. In other words, for each $i$ and all pairs of
vertices $(x_1,y_1)$ and $(x_2,y_2)$ with $d(x_1,y_1)=i=d(x_2,y_2)$, there is
an automorphism that maps $x_1$ to $x_2$ and $y_1$ to $y_2$. This property is
easily seen to imply the property of distance-regularity. Many --- but not all
--- classical families of distance-regular graphs, for example the Hamming
graphs, are also distance-transitive. The earlier mentioned Shrikhande graph is
the smallest distance-regular graph that is not distance-transitive. In fact,
it is part of an infinite family of graphs that are distance-regular but not
distance-transitive: the so-called Doob graphs. It also indicates that
distance-transitivity of a distance-regular graph is not a property that can be
recognized from the intersection array.

A distance-transitive graph is clearly also {\em vertex-transitive},
that is, it has a group of automorphisms such that for all $x_1$ and
$x_2$, there is an automorphism that maps $x_1$ to $x_2$. Although
there is no apparent relation between vertex-transitivity and
distance-regularity, it was long believed that distance-regular
graphs with large enough diameter would have to be
vertex-transitive. This belief was proven wrong by the construction
of the twisted Grassmann graphs; see Section \ref{twistedsection}.

%% file: 3_examples.txt

\subsection{The classical families with unbounded diameter}\label{sec:clasfamilies}

The {\em Johnson graph} $J(n,D)$ has as vertices the subsets of size $D$ of a set of size $n$. Two subsets are adjacent
if and only if they differ in precisely one element; cf.~\cite[\S 9.1]{bcn}. Note that $J(n,D)$ is isomorphic to
$J(n,n-D)$; in the following we therefore restrict to $n \geq 2D$. The Johnson graph $J(n,D)$ is characterized as
distance-regular graph by its intersection array unless $n=8$ and $D=2$, in which case there are also three so-called
Chang graphs.

The {\em Grassmann graph} $J_q(n,D)$ has as vertices the $D$-dimensional subspaces of a vector space of dimension $n$
over $GF(q)$. Two subspaces are adjacent if and only if they intersect in a $(D-1)$-dimensional subspace; cf.~\cite[\S
9.3]{bcn}. Note that $J_q(n,D)$ is isomorphic to $J_q(n,n-D)$; again we therefore restrict to $n \geq 2D$. Metsch
\cite{Me95} showed that the Grassmann graphs are determined by the intersection array if $D \neq 2, \frac{n}{2}$, or
$\frac{n-1}{2}$ (for all $q$) and $(D,q) \neq (\frac{n-2}{2},2),
(\frac{n-2}{2},3)$, or $(\frac{n-3}{2},2)$; see also Section
\ref{sec:Metsch}. For $D=2$, the Grassmann graphs are in general
not determined by the intersection array, as the line graph of a
$2$-$((q^n-1)/(q-1),q+1,1)$ design has the same array. Van Dam and
Koolen \cite{DK05} constructed the {\em twisted Grassmann graphs};
these are distance-regular graphs with the same array as the
Grassmann graphs for $n=2D+1, D\geq 2$, see Section
\ref{twistedsection}.

The {\em Hamming graph} $H(D,e)$ is defined on vertex set $X^D$ of
words of length $D$ from an alphabet $X$ of size $e$. Two words are
adjacent if and only if they differ in precisely one position;
cf.~\cite[\S 9.2]{bcn}. The Hamming graph $H(D,e)$ is characterized
by its intersection array unless $e=4$ and $D>1$, in which case
there are also so-called Doob graphs. A {\em Doob graph} is a
cartesian product of cliques of size 4 and Shrikhande graphs. The
Hamming graph $H(D,2)$ is also called a {\em (hyper)cube} or the
{\em $D$-cube}. Its halved graph is called a {\em halved cube}
$\frac12 H(D,2)$ and is characterized by its intersection array
(see \cite[\S 9.2.D]{bcn}).

The {\em bilinear forms graph} $Bil(D \times e,q)$ has as vertices all $D \times e$ matrices with entries from the
field $GF(q)$, where two matrices are adjacent if and only if their difference has rank $1$; cf.~\cite[\S 9.5.A]{bcn}
or \cite{D}. We shall assume $D \leq e$ in the following, so that $D$ is the diameter. The bilinear forms graph can be
considered as the $q$-analogue of the Hamming graph (view the vertices of the latter as the maps from a set of size $D$
to a set of size $e$), hence also the notation $H_q(D,e)$ is used in the literature. The bilinear forms graph has an
alternative description on the $D$-dimensional subspaces of a $(D+e)$-dimensional vector space that intersect a fixed
$e$-dimensional subspace trivially, where two such subspaces are adjacent if they intersect in a $(D-1)$-dimensional
subspace; this shows that it is isomorphic to a subgraph of the Grassmann graph $J_q(D+e,D)$. Rifa and Zinoviev
\cite{RiZi11} showed that the bilinear forms graph is also a quotient (as defined in Section \ref{sec:CRC}) of the
Hamming graph. Metsch \cite{Me99} showed that the bilinear forms
graph $Bil(D \times e,q)$ is characterized by its intersection
array if $q=2$ and $e \geq D + 4$ or $q \geq 3$ and $e \geq D + 3$;
see also Section \ref{sec:Metsch}. Gavrilyuk and Koolen
\cite{GKbilinear15} extended this characterization with the case
$q=2$ and $e=D$.

The {\em alternating forms graph} $Alt(n,q)$ has as vertices all $n \times n$ skew-symmetric matrices with zero
diagonal and entries from $GF(q)$. Two matrices are adjacent if and only if their difference has rank $2$. Note that a
skew-symmetric matrix has even rank; cf.~\cite[\S 9.5.B]{bcn} or \cite{DG}.

The {\em Hermitian forms graph} $Her(D,q^2)$ has as vertices the $D \times D$ Hermitian matrices with entries in
$GF(q^2)$, i.e., matrices $H$ such that $H_{ij}=(H_{ji})^q$ for all $i$ and $j$. Two matrices are adjacent if and only
if their difference has rank $1$; cf.~\cite[\S 9.5.C]{bcn}. The Hermitian forms graphs are determined by their
intersection arrays for $D \geq 3$, see Section \ref{sec:recentclassical}.

The {\em quadratic forms graph} $Qua(n,q)$ has as vertices the quadratic forms in $n$ variables over $GF(q)$. In the
quadratic forms graph two forms are adjacent if and only if the rank of their difference equals $1$ or $2$;
cf.~\cite[\S 9.6]{bcn} or \cite{E}. Under the group of invertible linear transformations of variables, the quadratic
forms fall into $2n+1$ ($q$ odd) or $\lceil \frac{3n+1}2 \rceil$ ($q$ even) orbits: each form of rank $k \neq 0$ is of
one of two types. For even rank there is the well-known distinction between hyperbolic and elliptic forms; in the case
of odd rank, a (parabolic) form is equivalent to $x_1x_2+\cdots+x_{k-2}x_{k-1}+cx_k^2$, for some $c$, and the type
depends on whether $c$ is a square or not (cf.~\cite[Ch.~IV]{N}). If $q$ is even then each field element is a square,
hence there is no distinction for odd rank.

The {\em dual polar graphs}\footnote{These graphs already appear as
distance-transitive graphs in disguise in a paper by Hua
\cite{Hua1945} from 1945.} have as vertices the maximal isotropic
($D$-dimensional) subspaces of one of the below vector spaces $V$
endowed with a non-degenerate quadratic form. Two subspaces are
adjacent if and only if they intersect in a $(D-1)$-dimensional
space; cf.~\cite[\S 9.4]{bcn}. The following dual polar graphs can
be distinguished:

\bigskip

$\C_D(q)$ for $V=GF(q)^{2D}$ with a symplectic form; $e=1$;

$\B_D(q)$ for $V=GF(q)^{2D+1}$ with a quadratic form; $e=1$;

$\D_D(q)$ for $V=GF(q)^{2D}$ with a quadratic form of Witt
index $D$; $e=0$;

$^2\D_{D+1}(q)$ for $V=GF(q)^{2D+2}$ with a quadratic form of
Witt index $D$; $e=2$;

$^2\A_{2D}(\sqrt{q})$ for $V=GF(q)^{2D+1}$ with a Hermitian
form; $e=\frac{3}{2}$;

$^2\A_{2D-1}(\sqrt{q})$ for $V=GF(q)^{2D}$ with a Hermitian
form;
$e=\frac{1}{2}$.

\bigskip

\noindent Here the mentioned parameter $e$ is related to the classical
parameter $\beta$ of the next section (see Table \ref{tableclassical}).

The dual polar graphs $^2\A_{2D-1}(\sqrt{q})$ are determined by
their intersection arrays for $D \ge 4$, see Section
\ref{sec:recentclassical}. The dual polar graphs $\B_D(q)$ and
$\C_D(q)$ have the same intersection array but are non-isomorphic
unless $q$ is even. The dual polar graph $\D_D(q)$ is the extended
bipartite double of $\B_{D-1}(q)$, and its halved graph, called a
{\em half dual polar graph} $\D_{D,D}(q)$, is the distance
$1$-or-$2$ graph of $\B_{D-1}(q)$. The extended bipartite double of
$\C_{D-1}(q)$ is also distance-regular and is called a {\em
Hemmeter graph}
\cite{BH1992EJC}; its halved graph is the distance $1$-or-$2$ graph
of $\C_{D-1}(q)$ and is called an {\em Ustimenko graph}
\cite{IMU1989EJC}.

\subsubsection{Classical parameters}\label{sec:classicalparameters}

The `classical' distance-regular graphs from the previous section
have intersection numbers that can be expressed in terms of four
parameters, that is, diameter $D$ and numbers $b$, $\alpha$,
$\beta$, in the following way:
\begin{align} \label{bb}
b_i & = \bigg(\gauss{D}{1} - \gauss{i}{1}\bigg)
          \bigg(\beta - \alpha \gauss{i}{1}\bigg)
             ~~(i =0,1,\dots, D-1),\notag \\[-5mm]  \\
c_i & = \ \gauss{i}{1}\bigg(1 + \alpha \gauss{i - 1}{1}\bigg)
            ~~(i =1,2,\dots, D),  \notag
\end{align}
where $ \gauss{j}{1} = 1+b+b^2 +\cdots +b^{j-1} $ is a Gaussian binomial
coefficient. Therefore, a distance-regular graph is said to have classical
parameters $(D,b,\alpha,\beta)$ if its intersection numbers can be expressed as
in (\ref{bb}). We note that the parameter $b$ must be an integer not equal to 0
or $-1$. The classical examples of distance-regular graphs from the previous
section have classical parameters as in Table \ref{tableclassical} (note that
one family of dual polar graphs has intersection numbers that can be expressed
in two ways). More basic information on distance-regular graphs with classical
parameters can be found in \cite[Ch.~6,~9]{bcn}. Important to mention is that
distance-regular graphs with classical parameters must be $Q$-polynomial. In
Section \ref{sec:Qpol}, we will therefore include also some results on
distance-regular graphs with classical parameters.

\begin{table}
\begin{center}
    \begin{tabular}{|l|c|c|c|c|}
    \hline
    \ \ \ \ \    & $D$ & $b$ & $\alpha$& $\beta$ \\
    \hline
    Johnson graph $J(n,D), n \geq 2D$ & $D$  & $1$ & $1$     & $n-D$   \\
    \hline
    Grassmann graph $J_q(n,D)$, $n \geq 2D$;  & $D$  & $q$ & $q$
    &$\frac{q^{n-D+1}-1}{q-1}-1$\\        
    twisted Grassmann graph ($n=2D+1$)&&&&\\
    \hline
    Hamming graph $H(D,e)$; & $D$              & $1$ & $0$     & $e-1$   \\
     Doob graph ($e=4$)&&&&\\
    \hline
    Halved Cube $\frac12 H(n,2)$&$\lfloor \frac n2 \rfloor$ &$1$&$2$&$2\lceil \frac n2 \rceil -1$ \\
        \hline
    Bilinear forms graph $Bil(D \times e,q)$,   & $D$              & $q$ & $q - 1$ & $q^e-1$ \\
    $D\leq e$&&&& \\
    \hline

    Alternating forms graph $Alt(n,q)$,&$\lfloor \frac n2\rfloor$
                                              &$q^2$& $q^2-1$ & $q^m-1$ \\
    $m=2\lceil \frac n2 \rceil -1$&&&& \\
    \hline
    Hermitian forms graph $Her(D,q^2)$ & $D$              & $-q$& $-q-1$  & $-(-q)^D-1$
                                                                        \\
    \hline
    Quadratic forms graph $Qua(n,q)$, & $\lfloor \frac{n+1}2\rfloor$
                                             & $q^2$ &$q^2-1$ & $q^m-1$ \\
    $m=2\lfloor \frac n2 \rfloor +1$&&&& \\
    \hline
    Dual polar graph;       &         $D$      & $q$ & $0$     & $q^e$   \\
    Hemmeter graph ($e=0$);&&&&\\
    $^2\A_{2D-1}(\sqrt{q})$ also: & $D$ & $-\sqrt{q}$ &  $\sqrt{q} \frac{1+\sqrt{q}}{1-\sqrt{q}}$  &  $\sqrt{q} \frac{1+(-\sqrt{q})^D}{1-\sqrt{q}}$  \\
    \hline
    Half dual polar graph $\D_{n,n}(q)$,&$\lfloor \frac n2 \rfloor$&$q^2$&$q^2+q$&$\frac{q^{m+1}-1}{q-1}-1$\\
    $m=2\lceil \frac n2 \rceil -1$;&&&&\\
    Ustimenko graph&&&&\\
    \hline
    \end{tabular}\vspace{3mm}
     \caption{Classical parameters of families of distance-regular graphs with unbounded diameter}\label{tableclassical}
\end{center}
\end{table}

\subsubsection{Other families with unbounded diameter}\label{sec:unboundeddiameter}
One of the ultimate problems in this area is to classify the
families of distance-regular graphs with unbounded diameter.
Besides the above known families of distance-regular graphs with
classical parameters and the {\em polygons} (see Section
\ref{sec:polygons}), also the below six families are known. All of them are related to the classical
ones, but they do not have classical parameters themselves.

The {\em folded cube} is obtained by folding the hypercube
$H(n,2)$. Unless $n=6$, it is determined by its intersection array.
For $n=6$, every graph with the relevant intersection array is the
incidence graph of a symmetric $2$-$(16,6,2)$ design. This gives
two other distance-regular graphs (see \cite[\S 9.2.D]{bcn}).

For $n$ even, the folded cube is still bipartite (and the halved
cube is still antipodal). Its halved graph is the {\em folded
halved cube} and it is determined by its intersection array for $n
\geq 12$ (that is, when its diameter is at least $3$; see Section
\ref{sec: partition graphs}).

The Johnson graph $J(2n,n)$ is antipodal, and its folding is called
a {\em folded Johnson graph}. This folded graph is determined by
its intersection array for $n \geq 6$ (that is, when its diameter
is at least $3$; see Section
\ref{sec: partition graphs}).

The folded cube, folded halved cube, and folded Johnson graph are
so-called {\em partition graphs} and these are known to be
$Q$-polynomial (see \cite[\S 6.3]{bcn}).

In Section \ref{sec:oddgraphs}, we already described the {\em Odd
graphs}, which are determined by their intersection array by
Proposition \ref{prop:odd}. The Odd graph is the distance-$D$ graph
of the Johnson graph $J(2D+1,D)$, and it is $Q$-polynomial.

Also the bipartite double of the Odd graph, the {\em Doubled Odd
graph}, is determined by its intersection array (see \cite[\S
9.1.D]{bcn}), but it is not $Q$-polynomial.

The final known family of distance-regular graphs with unbounded
diameter is the family of {\em Doubled Grassmann graphs}. This
graph is the bipartite double of the distance-$D$ graph of the
Grassmann graph $J_q(2D+1,D)$. Like the Doubled Odd graph, it is
determined by its intersection array (see Section
\ref{sec:otherinfinite}), and it is not $Q$-polynomial.


\subsection{New constructions}\label{sec:newconstructions}

In this section, we mention some relatively new constructions
of distance-regular graphs.

\subsubsection{The twisted Grassmann graphs} \label{twistedsection}

Van Dam and Koolen \cite{DK05} constructed the first family of
non-vertex-transitive distance-regular graph with unbounded diameter. These
graphs have the same intersection array as certain Grassmann graphs, and are
constructed as follows. Let $q$ be a prime power, and let $D \geq 2$ be an
integer. Let $W$ be a $(2D+1)$-dimensional vector space over $GF(q)$, and let
$H$ be a hyperplane in $W$. Vertices are the $(D+1)$-dimensional subspaces of
$W$ that are not contained in $H$, and the $(D-1)$-dimensional subspaces of
$H$. Two vertices of the first kind are adjacent if they intersect in a
$D$-dimensional subspace; a vertex of the first kind is adjacent to a vertex of
the second kind if the first contains the second; and two vertices of the
second kind are adjacent if they intersect in a $(D-2)$-dimensional subspace.
This graph is distance-regular with the same intersection array as the
Grassmann graph $J_q(2D+1,D)$. In fact, this Grassmann graph and the twisted
Grassmann graph are the point graph and line graph, respectively, of a partial
linear space whose points are the $D$-dimensional subspaces of $W$, and where
a $(D+1)$-dimensional subspace of $W$ that is not contained in $H$ is incident
to the $D$-dimensional subspaces that it contains, and a $(D-1)$-dimensional
subspace of $H$ is incident to the $D$-dimensional subspaces of $H$ containing
it.

The twisted Grassmann graph is not vertex-transitive (it has two orbits of vertices), and hence it is not isomorphic to
the Grassmann graph. Fujisaki, Koolen, and Tagami \cite{FKT07} showed that the automorphism group of the twisted
Grassmann graphs is $P \G L(2D+1,q)_{2D}$, the subgroup of $P \G L(2D+1,q)$ that fixes $H$. Bang, Fujisaki, and
Koolen \cite{BFK09} determined the spectra of the local graphs, and studied in some detail its Terwilliger algebras (as
defined in Section \ref{sec:Talgebra}). Remarkably, these algebras with respect to vertices in distinct orbits are not
the same. The twisted Grassmann graphs are also counterexamples to two conjectures by Terwilliger
\cite[p.~207-210]{Talgebra92}, see \cite{BFK09}. Jungnickel and Tonchev \cite{JuTo09} constructed designs that are
counterexamples for Hamada's conjecture. Munemasa and Tonchev \cite{MuTo09} showed that the twisted Grassmann graphs
are isomorphic to the block graphs of these designs.
Munemasa \cite{Munemasa2014SHCC} showed that the twisted Grassmann graphs can also be obtained from the Grassmann graphs by Godsil-McKay switching (cf.~\cite[\S1.8.3]{BrHa}).

\subsubsection{Brouwer-Pasechnik and Kasami graphs}\label{sec:Kasami}

For prime powers $q$, Pasechnik \cite{BP2011} constructed a distance-regular
graph with intersection array
$\{q^3,q^{3}-1,q^{3}-q,q^3-q^2+1;1,q,q^{2}-1,q^3\}$ as a subgraph of the dual
polar graph $\D_4(q)$; in particular, the induced subgraph on the set of
vertices at maximal distance from an edge.

Brouwer \cite{BP2011} constructed related distance-regular graphs with
intersection array $\{q^{3}-1,q^{3}-q,q^3-q^2+1;1,q,q^{2}-1\}$ as follows.
Consider the vector space $GF(q)^3$ equipped with a cross product $\times$. The
vertex set is $(GF(q)^3)^2$, where a pair $(u,v)$ is adjacent to a distinct
pair $(u',v')$ if and only if $u'=u+v\times v'$. The extended bipartite doubles
of these graphs are the above mentioned graphs constructed by Pasechnik. In
fact, Brouwer's graph is a subgraph of the dual polar graph $\B_3(q)$; in
particular, the induced subgraph on the set of vertices at maximal distance
from a vertex, see \cite{BP2011}.

For even $q$, the mentioned graphs have the same intersection arrays as certain Kasami graphs, cf.~\cite[Thm.~11.2.1
(11),(13)]{bcn}. Pasini and Yoshiara \cite{PY01} constructed distance-regular graphs with the same intersection array
as (bipartite, diameter 4) Kasami graphs using dimensional dual hyperovals. Also the symmetric bilinear forms graphs
for $q$ even and $n=3$ are distance-regular with the same intersection array as (diameter 3) Kasami graphs,
cf.~\cite[p.~285-286]{bcn} and \cite{BCNcoradd}.

Van Dam and Fon-Der-Flaass used almost bent functions to generalize the Kasami graphs, cf.~\cite{EF00},
\cite[Con.~3]{EF03}: Let $W$ be an $n$-dimensional vector space over $GF(2)$, and $f$ be an almost bent function on $W$
with $f(0) = 0$. Then the graph with vertex set $W^2$, where two distinct vertices $(x, a)$ and $(y, b)$ are adjacent
if $a + b = f (x + y)$ is distance-regular with intersection array $\{2^n - 1, 2^n - 2, 2^{n-1} + 1; 1, 2, 2^{n-1} -
1\}$. Recently, a lot of new almost bent functions have been discovered in the guise of quadratic almost perfect
nonlinear functions in odd dimensional vector spaces over $GF(2)$, cf.~\cite{BBMM08, BCL09, Edel}.

\subsubsection{De Caen, Mathon, and Moorhouse's Preparata graphs and crooked
graphs}\label{sec:Preparata}

De Caen, Mathon, and Moorhouse \cite{CMM95} constructed distance-regular
antipodal $2^{2t-1}$-covers of the complete graph $K_{2^{2t}}$, i.e., with
intersection array $\{2^{2t}-1,2^{2t}-2,1;1,2,2^{2t}-1\}$. These graphs are
defined as follows. Consider the vertex set $V=GF(2^{2t-1})\times GF(2) \times
GF(2^{2t-1})$, and let two vertices $(x,i,a)$ and $(y,j,b)$ be adjacent if
$$a+b=x^2y +xy^2 +(i+j)(x^3+y^3).$$
The construction is a bit more general, cf.~\cite{CMM95}, and is related to the Preparata codes. The construction also
allows for taking quotients. In this way, distance-regular graphs with intersection arrays
$\{2^{2t}-1,2^{2t}-2^h,1;1,2^h,2^{2t}-1\}$ for $h=1,2,\dots,2t$ arise. Prior to this construction, no distance-regular
graphs with these intersection arrays were known for $h<t$.

It is noteworthy that the Kasami graphs of the previous section are
induced subgraphs of the Preparata graphs. Because of this
relation, it is not surprising that variations of the above
construction are possible. To obtain these, De Caen and
Fon-Der-Flaass \cite{DeCaenFonDerFlaass} used Latin squares,
whereas Bending and Fon-Der-Flaass \cite{BF98} and Van Dam and
Fon-Der-Flaass \cite{EF03} used highly nonlinear functions such as
crooked functions and almost bent functions with accomplices: Let
$W$ be an $n$-dimensional vector space over $GF(2)$, and $f$ be a
crooked function on $W$. Then the (crooked) graph with vertex set
$W \times GF(2) \times W$, where two distinct vertices $(x,i, a)$
and $(y,j, b)$ are adjacent if $a + b = f (x + y)
+(i+j+1)(f(x)+f(y))$ is distance-regular with the same intersection
array as the Preparata graphs. Godsil and Roy
\cite{GodsilRoy08} determined that the above equation defines a
distance-regular graph precisely when $f$ is crooked. The Gold
functions, given by $f(x)=x^{2^e+1}$ on $GF(2^n)$ with
$\gcd(e,n)=1$ and $n=2t-1$, give the Preparata graphs.

It follows from the observations in \cite[p.~92]{EF03} that bijective quadratic almost perfect nonlinear functions
(that map 0 to 0) are crooked. A new family of such functions was thus constructed by Budaghyan, Carlet, and Leander
\cite[Prop.~1]{BCL08}. See also \cite{Bierbrauer}, but beware that Bierbrauer used a less strict definition of
crookedness (compared to the original one) there.

The paper by De Caen and Fon-Der-Flaass
\cite{DeCaenFonDerFlaass} initiated the prolific construction
by Fon-Der-Flaass \cite{FonDerFlaassprolific} of
distance-regular $n$-covers of complete graphs $K_{n^2}$ by
using affine planes of order $n$. Fon-Der-Flaass realized that
in general, his method produces many (potentially)
non-isomorphic such graphs; at least $2^{\frac12n^3 \log
n(1+o(1))}$ to be more precise. Computational results by
Degraer and Coolsaet \cite{DegraerCoolsaet} confirm this; they
verified that at least 80 of the 94 distance-regular antipodal
$4$-covers of $K_{16}$ can be constructed by Fon-Der-Flaass'
prolific construction. Also the (three) distance-regular
antipodal $4$-covers of $K_{10}$ \cite{DegraerCoolsaet}, the
(two) distance-regular antipodal 3-covers of $K_{14}$
\cite{Deg07}, and the (four) distance-regular antipodal
3-covers of $K_{17}$ \cite{Deg07} were classified by computer
by Degraer and Coolsaet. We also remark that Muzychuk
\cite{Muzy07} extended Fon-Der-Flaass' ideas further.

Godsil and Hensel \cite{GodsilHensel92} (see also \cite{CMM95}) described a
relation between regular antipodal covers of complete graphs and generalized
Hadamard matrices. By constructing skew generalized Hadamard matrices, Klin and
Pech \cite{KlinPech} thus constructed new infinite families of distance-regular
antipodal covers of complete graphs. Their paper contains a good overview of
the state of the art concerning such covers, and has many interesting ideas and
connections. For more background on antipodal covers of complete graphs, we
also refer to Godsil and Hensel \cite{GodsilHensel92} and Godsil
\cite{Godsil96}. For the classification of distance-transitive antipodal covers
of complete graphs, we refer to the paper by Godsil, Liebler, and Praeger
\cite{GLP98}.

\subsubsection{Soicher graphs and Meixner graphs}
\label{sec: Soicher and Meixner}

Soicher \cite{So93} obtained three distance-regular graphs of
diameter four, each being a triple cover of a strongly regular
graph. The first has intersection array $\{416,315,64,1;\linebreak
1,32,315,416\}$, and is a triple cover of the Suzuki graph. The
second has intersection array $\{56,45,16,1;1,8,45,56\}$, and is a
triple cover of the second subconstituent of the McLaughlin graph.
In an unpublished manuscript, Brouwer \cite{AEBSoicher} (see also
\cite{BCNcoradd}) showed that this cover is the only cover of
the second subconstituent of the McLaughlin graph, hence it is the
only graph with the given intersection array. The third cover
constructed by Soicher is the second subconstituent of the second
one, it has intersection array $\{32,27,8,1;1,4,27,32\}$, and is a
triple cover of the Goethals-Seidel graph (the second
subconstituent of the second subconstituent of the McLaughlin
graph). Soicher \cite{Soicher15} also showed that this graph is the
only graph with the given intersection array.

Meixner \cite{Meixner91} implicitly constructed two distance-transitive antipodal covers with intersection arrays
$\{176,135,36,1;1,12,135,176\}$ and $\{176,135,24,1; 1,24,135,176\}$, as the collinearity graphs of the
geometries in \cite[Prop.~4.3]{Meixner91}, see also \cite{BCNcoradd}. Juri\v{s}i\'{c} and Koolen \cite{JuKopre} showed
that the antipodal Meixner $4$-cover is uniquely determined by its intersection array.

Munemasa observed that the Meixner $2$-cover is the extended $Q$-bipartite double of the Moscow-Soicher graph of the
next section, cf.~\cite[Ex.~3.4]{MMW07}.

\subsubsection{The Koolen-Riebeek graph and the Moscow-Soicher graph}\label{sec: KoolenRiebeek}

Brouwer, Koolen, and Riebeek \cite{BKR98} gave a construction of a bipartite
distance-regular graph with intersection array $\{45,44,36,5;1,9,40,45\}$ from
the ternary Golay code. Each of its halved graphs is the complement of the
Berlekamp-van Lint-Seidel graph.

Soicher \cite{So95} constructed another distance-regular graph related to one of the Golay codes, in this case the
binary. It has intersection array $\{110,81,12;1,18,90\}$. Farad\v{z}ev, Ivanov, Klin, and Muzychuk
\cite[p.~119]{FKM94} already mentioned the underlying association scheme of this graph without realizing it was metric.

%% file: 4_morebackground.txt

\subsection{Miscellaneous definitions}\label{sec:miscdefs}

A non-complete $k$-regular graph $\G$ on $v$ vertices is called
{\em strongly regular} with parameters $(v,k,\lambda,\mu)$ if each
two adjacent vertices have $\lambda$ common neighbors, and each two
nonadjacent vertices have $\mu$ common neighbors. Thus, the
connected strongly regular graphs are precisely the
distance-regular graphs with diameter two. The definition of an
{\em amply regular graph} with parameters $(v,k,\lambda,\mu)$ is
obtained by replacing the condition on the nonadjacent vertices by
the condition that each two vertices at distance $2$ have $\mu$
common neighbors.

For a graph $\G$ and $x \in V$, the graph induced on the set $\G_i(x)$
is called an $i$-th {\em subconstituent} of $\G$. The first subconstituent
in consideration is also called a {\em local graph} of $\G$, and is denoted
by $\Upsilon(x)$. We say that $\G$ is {\em locally} $\Delta$ if all local
graphs are isomorphic to $\Delta$. More generally, we let $\Upsilon(x,y)$ be
the induced subgraph on the set of common neighbors of $x$ and $y$ (so it is a
local graph of a local graph if $x$ and $y$ are adjacent), etc.. A {\em
Terwilliger graph} is a non-complete graph such that $\Upsilon(x,y)$ is a
clique of size $\mu$ for each two vertices $x$ and $y$ at distance two, for
some $\mu$. Thus, a Terwilliger graph has no induced quadrangles.

Let $\G$ be a distance-regular graph with valency $k$ and diameter $D$. Let
$\ell(c, a, b) = |\{i= 1,2, \dots, D-1 : (c_i, a_i, b_i) = (c,a,b)\}|$. In
particular, let $h = h(\G)$ and $t= t(\G)$ be defined by $h(\G) =
\ell(c_1, a_1, b_1)$ and $t(\G) = \ell(b_1,a_1, c_1) $. The parameter
$h(\G)$ is called the {\em head} of $\G$ and $t(\G)$ is called the
{\em tail} of $\G$.

The {\em girth} of $\G$ is the length of its shortest cycle. The
{\em numerical girth} of $\G$ is $2h+3$ if $c_{h+1} = 1$, else it
is $2h+2$. If $a_1 = 0$, then the girth is equal to the numerical
girth. If $\G$ is locally a disjoint union of cliques, then the
{\em geometric girth} of $\G$ is the minimal length of a cycle for
which the induced subgraph on each triple of its vertices is not a
triangle; this equals the numerical girth. If $\G$ has a local
graph that is not a disjoint union of cliques, then the geometric
girth is defined as $3$. For geometric graphs, the geometric girth
is half the girth of the incidence graph of the corresponding
partial linear space (see Section
\ref{sec:generalizations+geometric}). Note that the girth and the
numerical girth are determined by the intersection array, but in
general the geometric girth is not (for example, the Doob graphs
have geometric girth $3$, whereas the Hamming graphs (with the same
intersection array) have geometric girth $4$).

A quadruple $(x,y,z,u)$ of vertices is called a {\em parallelogram of length
$i$} if $d(x,y) = 1 = d(z,u)$, $d(x,z) = d(y,u) = d(y, z) = i-1$, and $d(x,u) =
i$. The graph $\G$ is called {\em $m$-parallelogram-free} for some
$m=2,3,\ldots, D$ if $\G$ does not contain any parallelogram of length at
most $m$. We say $\G$ is parallelogram-free if it does not contain any
parallelogram. Related conditions called $\mathrm{(CR)}_m$ and $\mathrm{(SS)}_m$ are given by Hiraki
\cite{Hi198, Hi499}.

A quadruple $(x,y,z,u)$ of vertices of $\G$ is called a \emph{kite of
length} $i$ if $d(x,y)=d(x,z)=d(y,z)=1$, $d(x,u)=i$, and $d(y,u)=d(z,u)=i-1$.

A subgraph $\Delta$ of $\G$ is called \emph{geodetically closed}, or \emph{closed} for short,  if $z \in
V_{\Delta}$ for all $x,y \in V_{\Delta} $ and $z$ on a geodetic between $x$ and $y$.  (A closed subgraph is also called
{\em convex} by some authors.) The subgraph $\Delta$ is called {\em strongly closed} if $z \in V_{\Delta}$ for all
vertices $x, y \in V_{\Delta}$ and $z \in V_{\G}$ such that $d_{\G}(x,z) + d_{\G}(z,y) \leq d_{\G}(x,y)
+ 1$. (The term \emph{weak-geodetically closed} is also used for strongly closed.) It is known that if $c_2>1$ then all
strongly closed subgraphs are regular; cf.~\cite[Lemma~5.2]{W98} or \cite{Su295}. A distance-regular graph $\G$
with diameter $D$ is said to be $m$-\emph{bounded} for some $m=1,2, \ldots, D$ if for all $i=1,2, \ldots, m$ and all
vertices $x$ and $y$ at distance $i$ there exists a strongly-closed subgraph $\Delta(x,y)$ with diameter $i$,
containing $x$ and $y$ as vertices. (Note that Weng \cite{W97,We99} also required that $\Delta(x,y)$ is regular.)

\subsection{A few comments on the eigenspaces}

Consider an association scheme with primitive idempotents $E_0,E_1,\dots,E_D$.
By computing the squared norm, it follows that
\begin{equation}\label{3-tensor}
	\sum_{x\in V}E_i\mathbf{e}_x\otimes E_j\mathbf{e}_x\otimes E_h\mathbf{e}_x=0 \quad \text{if and only if}\ q_{ij}^h=0 \quad (h,i,j=0,1,\ldots, D),
\end{equation}
where $\mathbf{e}_x\in\R^v$ denotes the characteristic vector of $\{x\}$.
In fact, this computation also gives an alternative proof of the Krein conditions; cf.~Proposition \ref{Krein condisions}.
Recall that the absolute bound (cf.~Proposition \ref{absolute bound}) was an immediate consequence of the obvious observation that $(E_i \circ E_j) \R^v \subseteq \mathrm{span} (E_i \R^v \circ E_j \R^v)$.
We remark here that these two subspaces indeed coincide:
\begin{equation}\label{colsp Hadamard}
	\mathrm{span} (E_i \R^v \circ E_j \R^v) = (E_i \circ E_j) \R^v = \sum_{q_{ij}^h \ne
0} E_h \R^v  \quad (i,j=0,1,\dots,D).
\end{equation}
To see this, note that $\langle \mathbf{u}\circ\mathbf{v},\mathbf{w}\rangle=\langle \mathbf{u}\otimes\mathbf{v}\otimes\mathbf{w},\sum_{x\in V}E_i\mathbf{e}_x\otimes E_j\mathbf{e}_x\otimes E_h\mathbf{e}_x\rangle$ for all $\mathbf{u}\in E_i\R^v$, $\mathbf{v}\in E_j\R^v$, and $\mathbf{w}\in E_h\R^v$, where $\langle\, ,\, \rangle$ denotes the standard inner product.
Therefore, it follows from \eqref{3-tensor} that $\mathrm{span} (E_i \R^v \circ E_j \R^v)$ is orthogonal to $E_h\R^v$ whenever $q_{ij}^h=0$.
These results are due to Cameron, Goethals, and Seidel \cite{CGS1978NAWIM} (cf.~\cite[\S II.8]{bi}, \cite{Tanaka2009LAAa}), and are quite fundamental in the theory of distance-regular graphs and association schemes; see, e.g., Sections \ref{sec:Talgebra} and \ref{sec:triple intersection numbers}.
We note that, in view of \eqref{colsp Hadamard}, the ordering $E_0,E_1,\dots,E_D$ is $Q$-polynomial if and only if $\sum_{\ell=0}^i (E_1\R^v)^{\circ\ell}=\sum_{\ell=0}^i E_{\ell}\R^v$ for all $i=0,1,\dots,D$, where $(E_1\R^v)^{\circ\ell}=E_1\R^v\circ E_1\R^v\circ\dots\circ E_1\R^v$ ($\ell$ times).

\subsection{The Terwilliger algebra}\label{sec:Talgebra}

The \emph{Terwilliger} (or \emph{subconstituent}) \emph{algebra} of an association scheme was introduced in
\cite{Talgebra92}. Though it should be stressed that this algebra also plays an important role in the theory of general
distance-regular graphs (cf.~Section \ref{sec:Talg6}), it is particularly well-suited for $Q$-polynomial
distance-regular graphs. In fact, this algebra has (part of) its roots in the study of \emph{balanced sets}
(cf.~\eqref{balanced set condition}); see, e.g., \cite[p.~93,~Note 1]{Terwilliger1988GC}.

In the context of the Terwilliger algebra, the Bose-Mesner algebra of an association scheme is always assumed to be
\emph{over} $\mathbb{C}$, that is,
\begin{equation*}
	\AL=\mathrm{span}_{\mathbb{C}}\{A_0,A_1,\dots,A_D\}\subset M_{v\times v}(\mathbb{C}).
\end{equation*}
Fix a `base vertex' $x\in V$. For each $i=0,1,\dots,D$, let $E_i^{\ster}=E_i^{\ster}(x)$, $A_i^{\ster}=A_i^{\ster}(x)$
be the diagonal matrices\footnote{We use $\ster$-notation instead of the usual $\ast$-notation in order to
avoid confusion with the conjugate transpose.} in $M_{v\times v}(\mathbb{C})$ with diagonal entries $(E_i^{\ster})_{yy}=(A_i)_{xy}$,
$(A_i^{\ster})_{yy}=v(E_i)_{xy}$.
Note that $E_i^{\ster}E_j^{\ster}=\delta_{ij}E_i^{\ster}$,
$\sum_{i=0}^DE_i^{\ster}=I$, and moreover $A_i^{\ster}A_j^{\ster}=\sum_{h=0}^Dq_{ij}^hA_h^{\ster}$. These matrices span
the \emph{dual Bose-Mesner algebra} $\AL^{\ster}=\AL^{\ster}(x)$ \emph{with respect to} $x$:
\begin{equation*}
	 \AL^{\ster}=\mathrm{span}_{\mathbb{C}}\{E_0^{\ster},E_1^{\ster},\dots,E_D^{\ster}\}=\mathrm{span}_{\mathbb{C}}\{A_0^{\ster},A_1^{\ster},\dots,A_D^{\ster}\}\subset M_{v\times v}(\mathbb{C}).
\end{equation*}
Note that if the association scheme is $Q$-polynomial with respect to the ordering
$(E_i)_{i=0}^D$ then $A_1^{\ster}$ generates $\AL^{\ster}$. The \emph{Terwilliger
algebra} $\TT=\TT(x)$ \emph{with respect to} $x$ is the
subalgebra of $M_{v\times v}(\mathbb{C})$ generated by $\AL$ and $\AL^{\ster}$
\cite{Talgebra92}. The following are relations in $\TT$:
\begin{equation}\label{relations}
	E_i^{\ster}A_jE_h^{\ster}=0 \ \Leftrightarrow p_{ij}^h=0, \quad E_iA_j^{\ster}E_h=0 \ \Leftrightarrow q_{ij}^h=0 \quad (h,i,j=0,1,\ldots, D).
\end{equation}
We note that the latter is a variation of \eqref{3-tensor}. Because $\TT$ is closed under conjugate-transposition, it
is semisimple and every two non-isomorphic irreducible $\TT$-modules in $\mathbb{C}^v$ are orthogonal.
Let $G$ be the full automorphism group of the association scheme.
Then $\TT$ is a subalgebra of the centralizer algebra\footnote{Dunkl \cite{Dunkl1976IUMJ,Dunkl1978SIAM,Dunkl1978MM,Dunkl1979P} and Stanton \cite{Stanton1981GD} studied this latter algebra in detail in the context of addition theorems for orthogonal polynomials associated with some classical families of distance-regular graphs.} of the action
of the stabilizer $G_x$ of $x$ on $\mathbb{C}^v$.
The two algebras are known to be equal, e.g., for Hamming graphs;
cf.~\cite[Prop.~3]{GST2006JCTA}. We also note that the structure of $\TT$ may depend on the choice of $x$ if $G$ is not
transitive on $V$; cf.~Section \ref{twistedsection}.

Let $W$ be an irreducible $\TT$-module. When the association scheme is $P$-polynomial (resp.~$Q$-polynomial) with
respect to the ordering $(A_i)_{i=0}^D$ (resp.~$(E_i)_{i=0}^D$), we define the \emph{endpoint} (resp.~\emph{dual
endpoint}) of $W$ by $\min\{i\, :\, E_i^{\ster}W\ne 0\}$ (resp.~$\min\{i\, :\, E_iW\ne 0\}$). We call $W$ \emph{thin}
(resp.~\emph{dual thin}) if $\dim E_i^{\ster}W\leq 1$ (resp.~$\dim E_iW\leq 1$) for $i=0,1,\dots,D$. We also define the
\emph{diameter} and the \emph{dual diameter} of $W$ by $|\{i\, :\, E_i^{\ster}W\ne 0\}|-1$ and $|\{i\, :\, E_iW\ne
0\}|-1$, respectively. If the association scheme is $P$-polynomial
(resp.~$Q$-polynomial), then thin (resp.~dual thin) implies dual
thin (resp.~thin) \cite{Talgebra92}. There is a unique irreducible
$\TT$-module with $E_0^{\ster}W\ne 0$ and $E_0W\ne 0$, called the
\emph{primary} (or \emph{trivial}) $\TT$-module; it is thin, dual
thin, and given by
$\mathrm{span}_{\mathbb{C}}\{A_0\mathbf{e}_x,A_1\mathbf{e}_x,\dots,A_D\mathbf{e}_x\}$,
where $\mathbf{e}_x\in\mathbb{C}^v$ denotes the characteristic
vector of $\{x\}$. We say the association scheme is $i$-\emph{thin
with respect to} $x$ if every irreducible $\TT(x)$-module $W$ with
$E_i^{\ster}W\ne 0$ is thin.\footnote{The definition of the
$i$-thin condition here is taken from
\cite{Dickie1995D,DT1998JAC}. This is slightly different from the
standard one for the case when the association scheme is
$P$-polynomial, where it is called $i$-\emph{thin with respect to}
$x$ if every irreducible $\TT(x)$-module with endpoint at most $i$
is thin. On the other hand, the present definition of course has
the advantage that it makes sense for general association schemes.
We shall be careful below not to cause any confusion when we
discuss results involving this concept.} It is said to be
\emph{thin with respect to} $x$ if it is $i$-thin with respect to
$x$ for all $i=0,1,\dots,D$. Finally, we say the association scheme
is \emph{thin} (resp.~$i$-\emph{thin}) if it is thin
(resp.~$i$-thin) with respect to $x$ for all $x\in V$.

In the study of the Terwilliger algebra, it is often quite important to consider the following three matrices:
\begin{equation}\label{quantum decomposition}
	L=\sum_{i=1}^DE_{i-1}^{\ster}AE_i^{\ster}, \quad F=\sum_{i=0}^DE_i^{\ster}AE_i^{\ster}, \quad R=\sum_{i=0}^{D-1}E_{i+1}^{\ster}AE_i^{\ster},
\end{equation}
called the \emph{lowering}, \emph{flat}, and \emph{raising matrices},
respectively. Note that $A=L+F+R$. As an illustrative example, suppose $\G$
is the $D$-cube $H(D,2)$, and let
$A^{\ster}=A_1^{\ster}=\sum_{i=0}^D(D-2i)E_i^{\ster}$ correspond to the
$Q$-polynomial idempotent $E_1$ associated with the second largest eigenvalue
$\theta_1=D-2$. Then $F=0$ because $\G$ is bipartite, and it follows that
$L,R$, and $A^{\ster}$ generate $\TT$. Moreover, we can easily verify that
$LR-RL=A^{\ster}$, $RA^{\ster}-A^{\ster}R=2R$, and $LA^{\ster}-A^{\ster}L=-2L$,
so that the Terwilliger algebra $\TT$ is a homomorphic image of the universal
enveloping algebra of the Lie algebra $\mathfrak{sl}_2(\mathbb{C})$. Therefore,
every irreducible $\TT$-module $W$ has the structure of an irreducible
$\mathfrak{sl}_2(\mathbb{C})$-module, and $\bigoplus_{i=e}^{D-e}E_i^{\ster}W$
gives the weight space decomposition of $W$, where $e$ denotes the endpoint of
$W$. In particular, $H(D,2)$ is thin. We refer the reader to Terwilliger
\cite{Terwilliger1993N} and Go \cite{Go2002EJC} for more details.

\subsection{Equitable partitions and completely regular codes}\label{sec:CRC}

\subsubsection{Interlacing, the quotient matrix, and the quotient
graph}\label{sec:interlacing}

Eigenvalue interlacing is a useful tool in studying distance-regular graphs,
and more generally, in spectral graph theory; see the survey by Haemers
\cite{HaeInterlacing}. A sequence of numbers $\mu_1 \geq \mu_2 \geq \cdots \geq
\mu_m$ is said to {\em interlace} a sequence $\lambda_1 \geq \lambda_2 \geq
\cdots \geq \lambda_n$, with $n>m$, if $\lambda_i \geq \mu_i \geq
\lambda_{n-m+i}$ for all $i=1,2,\dots,m$. The interlacing is called {\em tight}
if for some $k \in \{0,1,\dots,m\}$ the equalities $\lambda_i=\mu_i,
i=1,\dots,k$ and $\lambda_{n-m+i}=\mu_i, i=k+1,\dots,m$ hold.

An elementary interlacing result states that the eigenvalues of a principal
submatrix $B$ of a symmetric matrix $A$ interlace the eigenvalues of $A$
itself. When applied to graphs: the eigenvalues of an induced subgraph of a
graph $\G$ interlace the eigenvalues of $\G$.

A somewhat more complicated --- but very useful --- result concerns the
so-called quotient matrix. Let $\Pi = \{P_1, P_2, \ldots, P_t\}$ be a partition
of the vertex set of a graph $\G$. Let $f_{ij}$ be the average number of
neighbors in $P_j$ of a vertex in $P_i$, for $i,j =1,2, \ldots, t$. The matrix
$F = (f_{ij})$ is called the {\em quotient matrix} of $\Pi$. The partition
$\Pi$ is called {\em equitable} if every vertex in $P_i$ has exactly $f_{ij}$
neighbors in $P_j$. Also the eigenvalues of $F$ interlace the eigenvalues of
$\G$. Moreover, if the interlacing is tight, then the partition is
equitable. In this case, it can easily be seen that an eigenvector ${\bf u}$ of
$F$ can be `blown up' to an eigenvector ${\bf v}$ of $\G$ (with the same
eigenvalue) by setting $v_x = u_i$ if $x \in P_i$. An example of an equitable
partition in a distance-regular graph $\G$ is the distance partition $\Pi
=\{\G_0(z),\G_1(z),\dots,\G_D(z)\}$ of a vertex $z$, and its
quotient matrix is the intersection matrix $L$ as in (\ref{matrixl}).

Given a partition $\Pi= \{ P_1, P_2, \ldots, P_t\}$ of the vertex set of a
graph $\G$, we define the {\em quotient graph} $\G/ \Pi$ with vertex
set $\Pi$, where $P_i \sim P_j$ if $i \neq j$ and there  exist $x \in P_i$ and
$y \in P_j$ such that $x \sim y$ in $\G$.

We call an equitable partition $\Pi$ {\em uniformly regular} if its quotient
matrix $F$ and the adjacency matrix $B$ of $\G/ \Pi$ are related as
$F=fI+\tilde{f}B$, for some numbers $f$ and $\tilde{f} \neq 0$. It is clear
that in this case, the eigenvalues of the quotient $\G/ \Pi$ follow in a
straightforward way from the eigenvalues of $F$, and the latter are eigenvalues
of $\G$, as we just observed. An example of a uniformly regular partition
is given by the partition into fibres of an antipodal distance-regular graph.
In this case, the corresponding quotient graph is the folded graph.

\subsubsection{Completely regular codes}\label{sec:crcdelsarte}

Let $\G$ be a connected graph, say with diameter $D$, and let $C$ be a subset of $V=V_{\G}$. For $i \geq 0$,
let $C_i = \{ x \in V : d(x, C) = i\}$, where $d(x, C) = \min\{ d(x,c) : c \in C\}$. The {\em covering radius} of $C$,
denoted by $\rho=\rho(C)$, is the maximum $i$ such that $C_i \neq \emptyset.$ The subset (or code) $C$ is called {\em
completely regular} if the distance partition $\Pi = \{C_i : i=0,1, \ldots, \rho\}$ is equitable. Note that the
corresponding quotient matrix is tridiagonal; it is therefore common to denote $f_{i, i-1}, f_{i,i}$ and $f_{i,i+1}$ by
$\gamma_i, \alpha_i$, and $\beta_i$, respectively. These numbers are called the {\em intersection numbers} of $C$. This
definition of a completely regular subset (or code) was introduced by Neumaier \cite{Neu92} and he showed that for
distance-regular graphs it is equivalent to Delsarte's definition \cite[p.~67]{del} in terms of the so-called outer
distribution. It is clear that if $C$ is completely regular then so is $C_{\rho}$. Note that for a distance-regular
graph, every singleton $\{z\}$ is a completely regular code with $ \gamma_i= c_i$, $\alpha_i= a_i$, and $\beta_i =
b_i$. In general, the behavior of the intersection numbers of a completely regular code can however be quite different
from that of the intersection numbers of a distance-regular graph. For example, it is not true in general that the
$\gamma_i$ are non-decreasing; see \cite{Ko95}. For more background information on completely regular codes, we refer
to the work of Martin \cite{Martinthesis, MartinCRC}.

A partition $\Pi = \{ P_1, P_2, \ldots, P_t\}$ of $V$ is called a {\em
completely regular partition} if it is equitable and all of the $P_i$ are
completely regular with the same intersection numbers. It is known that a
completely regular partition is uniformly regular. A typical (and motivating)
example of a completely regular partition is the partition into cosets of a
linear code $C$ of length $n$ over $GF(q)$ that is completely regular in the
Hamming graph $H(n,q)$.
(More generally, we can consider a completely regular additive code in a translation distance-regular graph.\footnote{A \emph{translation} distance-regular graph is a distance-regular Cayley graph on an abelian group. An \emph{additive} code in such a graph is just a subgroup of the abelian group ($=$ vertex set).})
In this case we call $\G/ \Pi$ the coset graph of
$C$. This coset graph is distance-regular by the following result.

\begin{theorem}\label{uniformlyregularpartition}
{\em \cite[Thm.~11.1.5,~11.1.8]{bcn}} Let $\G$ be a distance-regular graph and $\Pi$ a uniformly regular partition
of $\G$ with quotient matrix $F$. Then $\G/ \Pi$ is distance-regular if and only if $\Pi$ is completely
regular. Moreover, if so, then the intersection numbers of $\G/ \Pi$ can be explicitly calculated from the
intersection numbers of $\G$ and $F$.
\end{theorem}


%
%
%

Delsarte cliques are examples of completely regular codes. Indeed,
the following result characterizes such cliques.

\begin{prop} {\em \cite[Lemmas~13.7.2,~13.7.4]{Godsilac}}
Let $\G$ be a distance-regular graph with valency $k$, diameter $D$ and
smallest eigenvalue $\theta_{\min}$. Let $C$ be a clique in $\G$ with $c$
vertices. Then $C$ is a Delsarte clique if and only if $C$ is a completely
regular code with covering radius $D-1$. Moreover, if so, then $\phi_i u_i +
(c-\phi_i)u_{i+1} = 0$, where $(u_i)_{i=0}^D$ is the standard sequence for
$\theta_{\min}$ and $\phi_i= |\G_i(x) \cap C| $ for a vertex $x$ at
distance $i$ from $C$.
\end{prop}

\noindent Note that the equation $\phi_i u_i + (c-\phi_i)u_{i+1} = 0$ follows from the fact that $E\chi = 0$ (with $E$
and $\chi$ as in the proof of Proposition \ref{delbound}).
Indeed, if $E=UU^{\top}$, then $U^{\top}\chi=0$, and hence for
the corresponding representation associated to $\theta_{\min}$ (see
Section
\ref{sec2:evmult}) we have that
\begin{equation}\label{eq:repdelsarte}
\sum_{z\in C}\hat{z}=0.
\end{equation}
Taking the inner product with $\hat{x}$, where $x$ is a vertex at
distance $i$ from $C$ gives the required equation.
 This implies (by using
\cite[Thm.~4.1]{Neu92}) that the intersection numbers of a Delsarte
clique can be explicitly calculated from the intersection numbers
of $\G$.

For a subset of the vertex set of an association scheme, with characteristic vector $\chi$, the {\em degree} and {\em
dual degree} are defined by $|\{i \neq 0:\chi^{\mathsf{T}}A_i\chi\ne 0\}|$ and $|\{i \neq 0:\chi^{\textsf{T}}E_i\chi\ne
0\}|$, respectively.

\subsection{Distance-biregular graphs and weakly geometric graphs}\label{sec:generalizations+geometric}

For an arbitrary graph with vertices $x$ and $y$ at distance $i$, we define
$c_i(x,y)$, $a_i(x,y)$, and $b_i(x,y)$ as the numbers of neighbors of $y$ that
are at distance $i-1$, $i$, and $i+1$, respectively. Thus, a connected graph
with diameter $D$ is distance-regular if these numbers do not depend on $x$ and
$y$ (but only on their distance $i$). If in an arbitrary graph the numbers
$c_i(x,y)$, $a_i(x,y)$, or $b_i(x,y)$ do not depend on $x$ and $y$, for some
$i$, then we will write $c_i$, $a_i$, or $b_i$, respectively (as in
distance-regular graphs). For example, in an arbitrary bipartite graph, one has
$a_i=0$ for all $i$.

For ease of notation and formulation, we will call the two biparts of a
bipartite graph its color classes $R$ and $B$, and say that a vertex in $R$ is
red, and a vertex in $B$ is blue.

Now a connected bipartite graph is called {\em distance-biregular} if the
numbers $c_i(x,y)$ and $b_i(x,y)$ depend only on $i$ and the color of $x$. We
denote these numbers by $c_i^R$, $b_i^R$, $c_i^B$, and $b_i^B$. Straightforward
examples are the complete bipartite graphs.

We say that a graph $\G$ is distance-regular around a vertex $x$ if the
singleton $\{x\}$ is a completely regular code in $\G$. A well-known result
by Godsil and Shawe-Taylor \cite{GSbiregular} states that if $\G$ is a
connected graph that is distance-regular around every vertex, then $\G$ is
distance-regular or distance-biregular.

A bipartite graph is called {\em semiregular} (or biregular) if the valency of
a vertex only depends on its color. We denote these valencies by $k_R$ and
$k_B$.

Powers \cite{Powerssemiregular} used the term semiregular for a concept that he
introduced, and what we now call distance-semiregular (following Suzuki
\cite{Su95, Su99}). A connected bipartite graph is called {\em
distance-semiregular} with respect to one of its color classes, $R$ say, if it
is distance-regular around all red vertices, with the same parameters (i.e,
there are $b_i^R$ and $c_i^R$ such that $b_i(x,y) = b_i^R$ and $c_i(x,y) =
c_i^R$ if $x \in R$ and $d(x,y) = i$). Note that every distance-biregular graph
is distance-semiregular, and in turn, each distance-semiregular graph is
semiregular, with valencies $k_R=b_0^R$ and $k_B=1+b_1^R$. The Hoffman graph
\cite{hoffman63} (the unique graph cospectral but not isomorphic to $H(4,2)$)
is an example of a (regular!) distance-semiregular graph that is not
distance-biregular.

Let $\G$ be distance-semiregular with respect to $R$, then its halved graph
$\G^R_2$ (i.e., the distance-2 graph of $\G$, induced on $R$) is
distance-regular. Let $C = \G(x)$ for some blue vertex $x$. Then $C$ is a
clique in $\G^R_2$, that is also a completely regular code in $\G^R_2$.
This leads to the following definition.

A distance-regular graph $\Delta$ is called {\em weakly geometric}
(with respect to ${\cal C}$) if it contains a collection ${\cal C}$
of cliques such that each edge is contained in a unique $C \in
{\cal C}$ and all $C \in {\cal C}$ are completely regular codes
with the same parameters. Thus, a geometric distance-regular graph
(see Section \ref{sec:geometricgraphs}) is weakly geometric.
Because of the property that each edge is contained in a unique
clique, there is a naturally associated partial linear space, whose
points are the vertices of $\Delta$ and whose lines are the cliques
of ${\cal C}$, and incidence is defined by containment. The point
(or collinearity) graph of this partial linear space is $\Delta$.
The bipartite (point-line) incidence graph of the partial linear
space is a distance-semiregular graph with girth at least 6; in
fact, this gives a one-to-one correspondence between the latter
type of graphs and weakly geometric distance-regular graphs. The
partial linear space has also been studied by De Clerck, De Winter,
Kuijken, and Tonesi
\cite{DeDeKuTo06, KuiTon05} under the name {\em distance-regular
geometry}.

Using the same correspondence, certain distance-semiregular graphs
with girth 4 correspond to Delsarte graphs and Delsarte clique
graphs as introduced by Bang, Hiraki, and Koolen \cite{BaHiKo07}
(see also \cite{BaHiKo10}). Delsarte graphs and Delsarte clique
graphs are closely related to the geometric distance-regular graphs
of Section \ref{sec:geometricgraphs}.

 We remark that the Johnson graphs
$J(n,D)$ and Grassmann graphs $J_q(n,D)$ are not just (weakly)
geometric with respect to a set of Delsarte cliques (the
$(D-1)$-sets or $(D-1)$-dimensional subspaces; that is, the set of
vertices containing a fixed $(D-1)$-set or $(D-1)$-space is a
Delsarte clique), but also weakly geometric with respect to another
set of cliques, namely the $(D+1)$-sets or $(D+1)$-dimensional
subspaces (i.e., the sets of vertices contained in these),
respectively. The corresponding incidence graphs are
distance-biregular; for $n=2D+1$, we obtain the distance-regular
Doubled Odd graph and Doubled Grassmann graph, respectively. In
Section
\ref{sec:Geometricdrg} we will discuss geometric distance-regular
graphs in more detail.

Following Suzuki \cite{Su99}, we say a distance-regular graph $\G$
is {\em of order $(s, t)$} (for some integers $s$ and $t$) if it is
locally the disjoint union of $t+1$ cliques of size $s$. This is
equivalent to the property that $\G$ contains no induced complete
tripartite graph $K_{2,1,1}$ (a kite of length $2$).

A distance-regular graph $\G$ of order $(s,t)$ with diameter $D$ is called
a {\em regular near polygon} if $a_i = c_i a_1$ for all $i=1,2,\ldots,D-1$. If
$a_d=c_D a_1$ we call $\G$ a regular near $2D$-gon; otherwise it is called
a regular near $(2D+1)$-gon. A regular near polygon of diameter $D$ is
geometric if and only if it is a regular near $2D$-gon.  We say $\G$ is
{\em thick} if $s \geq 2$ (the regular near polygons with $s=1$ are exactly the
bipartite distance-regular graphs and the generalized odd graphs).

Weng \cite{We99} defined a distance-regular graph to have {\em
geometric parameters $(D, b, \alpha)$} if it has classical
parameters $(D, b, \alpha, \beta)$ with $b \neq 1$ and $\beta =
\alpha\frac{1+b^D}{1-b}$. He used this concept in the partial
classification of distance-regular graphs with classical parameters
with $b<-1$. This does not seem to be related to geometric
distance-regular graphs.

\subsection{Homogeneity}\label{sec:homogeneity4}

Let $\G$ be a connected graph. For two distinct vertices $x$ and $y$,
define $\G_{i,j}(x,y) = \G_i(x) \cap \G_j(y)$. If it is clear (or
irrelevant) which pair $x,y$ is meant we will write $\G_{i,j}$ instead of
$\G_{i,j}(x,y)$. For $u \in \G_{i,j}$, let $p_{i,j;r,s}(u) = |\{ z \in
\G_{r,s} : z \sim u\}|$. We say the parameter $p_{i,j;r,s}$ exists with respect
to the pair $x,y$ if $p_{i,j;r,s}(u) =p_{i,j;r,s}(u')$ for all $u, u' \in
\G_{i,j}(x,y)$.

A connected graph $\G$ with diameter $D$ is called {\em $i$-homogeneous}
(in the sense of Nomura), $i =0,1, \ldots, D$ if for all pairs $x,y$ at
distance $i$ and all $r,s,r', s' \in \{0,1, \ldots, D\}$, the parameter
$p_{r,s; r', s'}$ exists and does not depend on the pair $x,y$, or in other
words, the partition $\{\G_{i,j}(x,y) : \G_{i,j}(x,y) \neq \emptyset, i,j =0,1,
\ldots, D\}$ is equitable for each pair $x,y$ at distance $i$ and the
parameters do not depend on the pair $x,y$.\footnote{In most of the literature
$\G_{i,j}$ is denoted as $D_i^j$. We chose different notation because $D$
stands for the diameter, and the superscript-subscript notation seems useful
only in intersection diagrams.}

Note that a $0$-homogeneous graph is distance-regular, and a $1$-homogeneous
graph is distance-regular. Examples of $1$-homogeneous distance-regular graphs
are the Johnson graphs $J(2D,D)$, the bipartite distance-regular graphs, and
the regular near $2D$-gons. To study $i$-homogeneous graphs, it is sometimes
useful to draw intersection diagrams with respect to two vertices $x$ and $y$.
In Figure \ref{pic:intersectiondiagram} we have an example of such an
intersection diagram for the Johnson graph $J(6,3)$.

\begin{figure}[h!]
\centering
\includegraphics[width=45mm,viewport=140 90 850 718,clip]{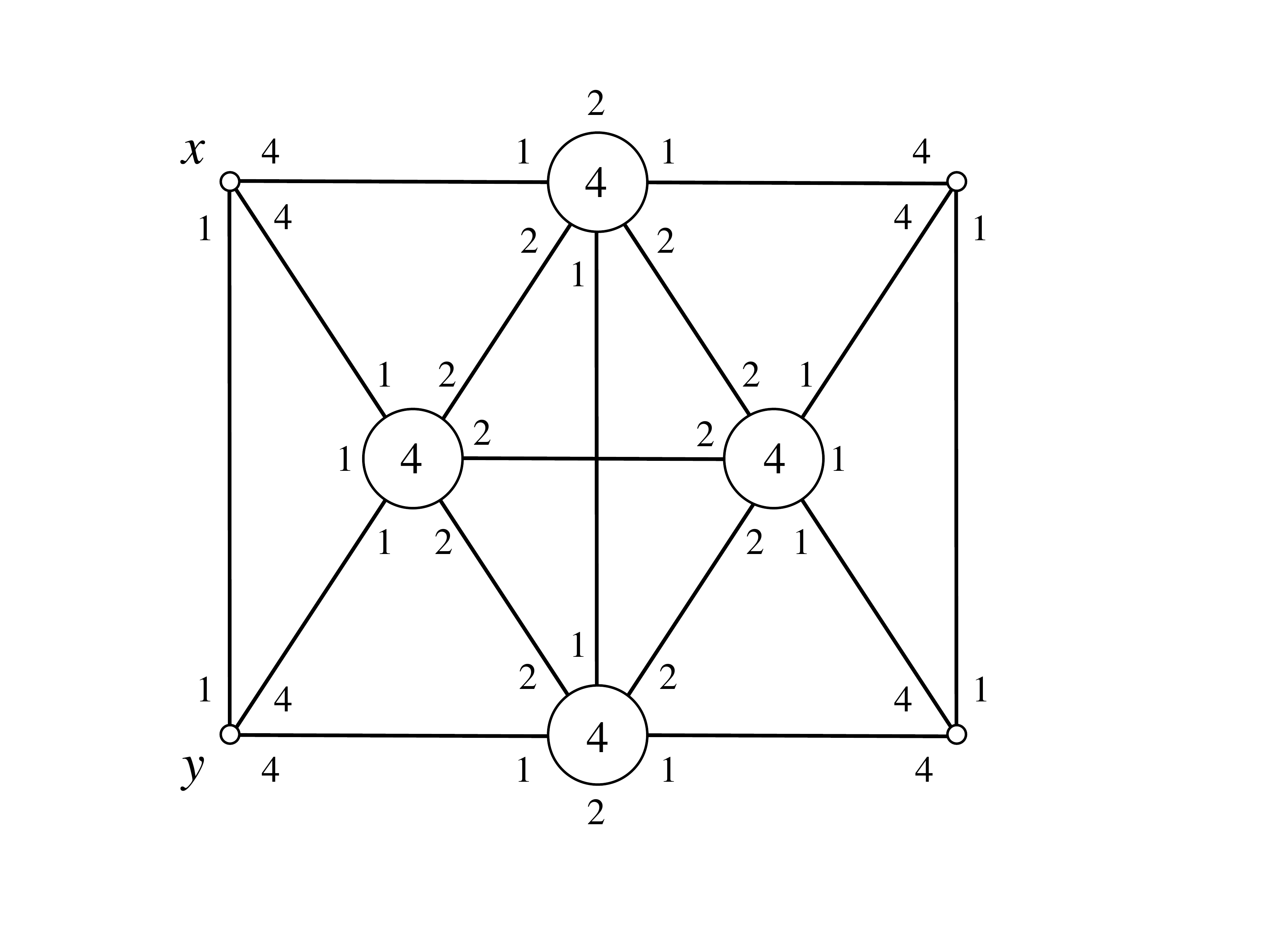}
\caption{Intersection diagram of $J(6,3)$} \label{pic:intersectiondiagram}
\end{figure}

\subsection{Designs}\label{sec:designs}

Consider an association scheme with primitive idempotents $E_i$ ($i=0,1,\dots,D$).
Let $T$ be a subset of $\{1,2,\dots,D\}$.
A set $Y$ of vertices of the association scheme with characteristic vector $\chi$ is called a (\emph{Delsarte}) $T$-\emph{design} if $E_i\chi=0$ for all $i\in T$.
This definition is due to Delsarte \cite{del}.

Suppose that the association scheme is $Q$-polynomial with respect to the ordering $E_0,E_1,\dots,E_D$.
In this case, a $\{1,2,\dots,t\}$-design is simply called a $t$-design.
The \emph{strength} of $Y$ is then defined by $\min\{i\ne 0:E_i\chi\ne 0\}-1$, i.e., it is the maximum integer $t$ for which $Y$ is a $t$-design.
Delsarte \cite{del} showed that the $t$-designs in the Johnson graphs and Hamming graphs are precisely the combinatorial block $t$-designs and the orthogonal arrays of strength $t$, respectively.
A similar interpretation was established for the other classical families of distance-regular graphs by Delsarte \cite{Delsarte1976JCTA}, Munemasa \cite{Munemasa1986GC}, and Stanton \cite{Stanton1986GC}.

For more results on $T$-designs in association schemes, see the recent survey by Martin and Tanaka \cite{martintanaka} and the references therein.

%% file: 5_Qpolynomial.txt


In this section, we collect (relatively new) results on $Q$-polynomial
distance-regular graphs. Throughout this section, we shall use the following
notation unless otherwise stated. Let $\G$ denote a distance-regular graph
with diameter $D\geq 3$ and valency $k\geq 3$. Let $\theta$ be an eigenvalue of
$\G$, $E$ the corresponding primitive idempotent, and $(u_i)_{i=0}^D$ the
standard sequence with respect to $\theta$.

Suppose for the moment that $E$ is $Q$-polynomial, and let $E_0,E_1=E,E_2,\dots,E_D$ be the corresponding $Q$-polynomial ordering.
Then by Leonard's theorem
(cf.~\cite[\S 8.1]{bcn}) there exist $p,r,r^{\ster}\in\mathbb{C}$ such that
\begin{equation}\label{TTR}
	u_{i-1}+u_{i+1}=pu_i+r, \quad \theta_{i-1}+\theta_{i+1}=p\theta_i+r^{\ster} \quad (i=1,2,\dots,D-1).
\end{equation}
It should be remarked that the sequence of polynomials $(v_i)_{i=0}^D$ (see~\eqref{distancepolynomials}) belongs to the
terminating branch of the Askey scheme \cite{KS1998R,KLS2010B} of (basic) hypergeometric orthogonal polynomials. (We also allow the specialization\footnote{The polynomials corresponding to the case $q=-1$ have recently been receiving considerable attention; see, e.g., \cite{GVZ2014SIGMA} and the references therein.} $q\rightarrow -1$.)
See also \cite{Terwilliger2006N}.
We call $E$ \emph{classical} if $(u_{i-1}-u_i)/(u_i-u_{i+1})$
is independent of $i=1,2,\dots,D-1$. It follows that $\G$ has a classical $Q$-polynomial idempotent if and only if it
has classical parameters such that $p=b+b^{-1}$; cf.~\cite[Thm.~8.4.1]{bcn}, \cite[Prop.~6.2]{Tanaka2011EJC}.

We begin with discussions on graphs with classical parameters.

\subsection{The graphs with classical parameters with \texorpdfstring{$b=1$}{b=1}}
\label{sec: the case b=1}

All graphs with classical parameters with $b=1$ have been determined: the Hamming graphs, Doob graphs, halved cubes,
Johnson graphs, and the Gosset graph; cf.~\cite[Thm.~6.1.1]{bcn} or \cite{Neu-b1}. Main contributors to this
classification were Egawa \cite{E-b1}, who characterized the Hamming and Doob graphs, and Terwilliger \cite{Ter-b1,
Ter-root} and Neumaier \cite{Neu-b1}, who used the classification of root lattices and the representation with respect
to the second largest eigenvalue to come to the final classification.
We note that $b=1$ implies $\theta_1=b_1-1$,
and the graphs satisfying the latter have been classified; cf.~\cite[Thm.~4.4.11]{bcn}.
In \cite{KS94,Shp98}, Koolen and Shpectorov used metric theory to
classify the distance-regular graphs whose distance-matrix has exactly one
positive eigenvalue. The distance-regular graphs with classical parameters with
$b=1$ have this property.
Godsil \cite{G98} considered the convex hull of the representation with respect
to a fixed eigenvalue. He classified when the $1$-skeleton of this polytope with
respect to the second largest eigenvalue is isomorphic to the original
distance-regular graph. In his classification he again finds all
distance-regular graphs with classical parameters with $b=1$.

\subsection{Recent results on graphs with classical parameters}\label{sec:recentclassical}

Metsch \cite[Cor. 1.3]{Me99} showed that if $\G$ has classical parameters and is not a Johnson, Grassmann, Hamming, or bilinear forms graph, then the parameter $\beta$ is bounded in terms of $D, b,$ and $\alpha$.

Terwilliger \cite{Terwilliger1995EJC} showed that if $\G$ has classical parameters with $b<-1$ then $\G$ has no kites of any length $i=2,3,\dots,D$. This result,
combined with earlier work of Ivanov and Shpectorov \cite{IS1991JMSJ}, proves that the Hermitian forms graphs
$\mathit{Her}(D,q^2)$ with $D\ge 3$ are uniquely determined by their intersection arrays. See also \cite{Weng1995GC}. A
related result by Weng \cite{W98} is as follows (cf.~Section \ref{sec: strongly closed subgraphs}).

\begin{prop}
Suppose $\G$ is $Q$-polynomial with $D\ge 3$, $c_2>1$, and $a_1\ne 0$. Then
the following are equivalent:
\begin{enumerate}[{\em (i)}]
\item $\G$ has classical parameters, and either $b < -1$, or $\G$ is
    a dual polar graph or a Hamming graph,
\item $\G$ has no parallelogram of length $2$ or $3$,
\item $\G$ is $D$-bounded.
\end{enumerate}
\end{prop}

\noindent Liang and Weng \cite{LW97} showed that if $\G$ is $Q$-polynomial and $D\ge 4$ then $\G$ is parallelogram-free
if and only if either (i) $\G$ is bipartite, (ii) $\G$ is a generalized odd graph, or (iii) $\G$ has classical parameters and either $b < -1$ or
$\G$ is a Hamming graph or a dual polar graph.
Weng \cite{W97} showed that if $\G$ has
classical parameters with $b<-1, a_1\neq 0, c_2>1$, and $D \geq 4$, then $\G$
has geometric parameters
(cf.~Section
\ref{sec:generalizations+geometric}).
Building on this, he showed among other results that there are no distance-regular graphs with
classical parameters with $D \ge 4$, $c_2=1$, and $a_2>a_1>1$,
and that under the assumption
$D \ge 4$ and $c_2>1$, the dual polar graphs $^2\A_{2D-1}(-b)$ are the only
graphs with classical parameters with $b=-a_1-1$.
The latter characterizes the
dual polar graphs $^2\A_{2D-1}(\sqrt{q})$ by their intersection arrays for $D \ge 4$.
Weng \cite{We99} also showed the following result.

\begin{theorem}\label{thm:b<-1}
If $\G$ has classical parameters with $b<-1$, $a_1\neq 0$, $c_2>1$, and $D
\geq 4$, then either $\G$ is a dual polar graph $^2\A_{2D-1}(-b)$ or a
Hermitian forms graph $Her(D,(-b)^2)$, or $\alpha=(b-1)/2$, $\beta=-(1+b^D)/2$,
and $-b$ is a power of an odd prime.
\end{theorem}

\noindent
Vanhove \cite{Vanhove2012JAC} showed that
a $((1-b)/2)$-\emph{ovoid} (i.e., a $(D-1)$-design with index $(1-b)/2$) in the
dual polar graph $^2\A_{2D-1}(-b)$ with $b$ odd would induce a distance-regular
graph having classical parameters of the latter case. For $D=2$, such
$((1-b)/2)$-ovoids are better known as \emph{hemisystems} and these were
constructed by Cossidente and Penttila \cite{CP2005JLMS} for every odd prime
power $-b$; see also~\cite{BGR2010BLMS}. No construction of a $((1-b)/2)$-ovoid
is known for $D\ge 3$.

Triangle-free distance regular graphs with classical parameters have been
studied by Pan, Lu, and Weng \cite{PanLuWeng08,PanWeng08,PanWeng09} and Hiraki
\cite{Hiraki09a}. One of the results is that if $\G$ has classical
parameters and $a_1=0, a_2 \neq 0, D\geq 3$ then either (i)
$(b,\alpha,\beta)=(-2,-2,((-2)^{D+1}-1)/3)$ ($c_2=1$), or (ii)
$(b,\alpha,\beta)=(-2,-3,-1-(-2)^D)$ ($c_2=2$), or
(iii)$(b,\alpha,\beta)=(-3,-2,-(1+(-3)^D)/2)$ ($c_2=2$); cf.~\cite{PanWeng09,Hiraki09a}.
Case (i) with $D=3$ is uniquely realized by the Witt graph $M_{23}$ \cite[\S 11.4B]{bcn}, whereas Huang, Pan, and Weng \cite{HPW2009pre} ruled out case (i) with $D\geq 4$.
Case (ii) is uniquely realized by the Hermitian forms graph $\mathit{Her}(D,4)$.

\subsection{Imprimitive graphs with classical parameters and partition graphs}
\label{sec: partition graphs}

It is known (\cite[Prop.~6.3.1]{bcn}) when a distance-regular graph with classical
parameters $(D,b,\alpha,\beta)$ with $D \geq 3$ is imprimitive:
it is bipartite
if and only if $\alpha=0$ and $\beta=1$, whereas it is antipodal if and only if
$b=1$ and $\beta=1+\alpha(D-1)$, in which case it is an antipodal double cover
of its folded graph. This folded graph has diameter $D'$ and intersection
numbers $b_i=(D-i)(1+\alpha(D-1-i))$ and $c_i=i(1+\alpha(i-1))$ for $i<D'$,
$b_{D'}=0$, and $c_{D'}=\gamma D'(1+\alpha(D'-1))$, where $\gamma=1$ if
$D=2D'+1$ and $\gamma=2$ if $D=2D'$.
The distance-regular graphs with such intersection numbers
are called \emph{pseudo partition graphs}.
Bussemaker and Neumaier \cite[Thm.~3.3]{BN1992MC} showed that pseudo
partition graphs with diameter $D' \geq 3$
must have the same intersection
arrays as in one of the three families of partition graphs: the folded cubes
($\alpha=0$), the folded Johnson graphs ($\alpha=1$), and the folded
halved cubes ($\alpha=2$).

The folded cubes are determined by their intersection arrays, except for the
folded $6$-cube, which has two mates (i.e., non-isomorphic distance-regular
graphs with the same intersection array) in the form of (other) incidence
graphs of $2$-$(16,6,2)$ designs (cf.~\cite[Thm.~9.2.7]{bcn}).
The characterization of the other two families of partition graphs is now complete,
due to work by Metsch, Gavrilyuk and Koolen:
\begin{prop}\label{folded Johnson}
{\em \cite{M297,M197,GK2012pre}} The folded Johnson graphs with
diameter at least three are uniquely determined as distance-regular graphs by their
intersection arrays.
\end{prop}

\begin{prop}\label{folded halved cubes}
{\em \cite{M297,Me03,GK2012pre}}
The folded halved cubes with diameter at least three are uniquely determined as distance-regular graphs by their
intersection arrays.
\end{prop}

\noindent
Thus, all pseudo partition graphs with diameter at least three are known.

\subsection{Characterizations of the \texorpdfstring{$Q$-polynomial}{Q-polynomial} property}\label{sec:Qpolcharacterizations}

Bannai and Ito \cite[p.~312]{bi} conjectured that every
\emph{primitive} distance-regular graph with sufficiently large
diameter is $Q$-polynomial. We note that the Doubled Odd graphs are
not $Q$-polynomial yet have arbitrarily large diameter, so that the
`primitivity' condition in the conjecture is necessary. Currently
we know of no (real) progress towards proving the conjecture;
however there are several new characterizations of the
$Q$-polynomial property (since `BCN'
\cite{bcn}). For completeness and because of its importance, we begin with Terwilliger's \emph{balanced set
condition} \cite{Terwilliger1987JCTA,Terwilliger1995DM}; cf.~\cite[\S 2.11, \S
8.3]{bcn}.
For distinct $x,y\in V$ and for $i,j=0,1,\dots,D$, we let
$\chi_{i,j}(x,y)=\sum_{z\in \G_{i,j}(x,y)}\mathbf{e}_z$ denote the
characteristic vector of $\G_{i,j}(x,y)=\G_i(x)\cap \G_j(y)$; cf.~Section \ref{sec:homogeneity4}.

\begin{theorem}{\em (Balanced set condition \cite{Terwilliger1987JCTA,Terwilliger1995DM})}
The primitive idempotent $E$ is $Q$-polynomial if and only if
$u_i\ne 1$  for all $i=1,2,\dots,D$ and
\begin{equation}\label{balanced set condition}
	E\chi_{i,j}(x,y)-E\chi_{j,i}(x,y)=p_{ij}^h\frac{u_i-u_j}{1-u_h}(E\mathbf{e}_x-E\mathbf{e}_y)
\end{equation}
for all $i,j=0,1,\dots,D$, $h=1,2,\dots,D$, and $x,y\in V$ with $d(x,y)=h$.
\end{theorem}

\noindent Terwilliger \cite{Terwilliger1995DM} obtained an inequality for every
$\ell=3,4,\dots,D$ involving only the intersection numbers, $\theta$, and
$(u_i)_{i=0}^D$, by applying Cauchy-Schwarz to
$E\chi_{i,1}(x,y)-E\chi_{1,i}(x,y)$ and $E\mathbf{e}_x-E\mathbf{e}_y$
with $\{i,h\}=\{\ell,\ell-1\}$, and averaging over $x,y\in V$ with
$d(x,y)=h$.
Equality is attained for all $\ell=3,4,\dots,D$ (or just for $\ell=3$) if
and only if $E$ is $Q$-polynomial; cf.~\cite[\S 8.3]{bcn}.
Instead of the four vectors in \eqref{balanced set condition}, we may also
consider the linear dependency of $E\chi_{i,j}(x,y)$, $E\mathbf{e}_x$, and
$E\mathbf{e}_y$. This was worked out in detail by Terwilliger
\cite{Terwilliger1988GC}.\footnote{The term `balanced set' was introduced in \cite{Terwilliger1988GC} in this context, but many authors now refer to \eqref{balanced set condition} as \emph{the} balanced set condition.}
In particular, he applied Cauchy-Schwarz to
$E\chi_{1,1}(x,y)$ and $E\mathbf{e}_x+E\mathbf{e}_y$, and took the average over
each of the sets $\{(x,y):d(x,y)=h\}$ $(h=1,2)$ to obtain an inequality
involving only $a_1,b_1,c_2,u_1,u_2$; in this case, equality is attained if and
only if $E$ is $Q$-polynomial with
$a_0^{\ster}=a_1^{\ster}=\dots=a_{D-1}^{\ster}=0$. The linear dependency among
$E\chi_{1,1}(x,y)$, $E\mathbf{e}_x$, and $E\mathbf{e}_y$ for adjacent $x$ and
$y$ is also relevant to the property of being tight; cf.~Section
\ref{sec:tightDRG}. There is also a `symmetric' version of \eqref{balanced set
condition} due to Terwilliger \cite[Thm.~2.6]{Terwilliger1995EJC}. This lead, in particular, to
the characterization of the Hermitian forms graphs by their intersection
arrays; cf.~Section \ref{sec:recentclassical}. Tonejc \cite{JurTonpre} recently
presented several inequalities by considering the vectors
$E\chi_{i,1}(x,y)+E\chi_{1,i}(x,y)$ and $E\mathbf{e}_x+E\mathbf{e}_y$.

The following result is due to Pascasio \cite{Pascasio2008DM} and may be viewed
as an extension of \cite[Thm.~8.2.1]{bcn} for the bipartite case.

\begin{prop}
\label{Pascasio} $E$ is $Q$-polynomial if and only if all following properties
hold:
\begin{enumerate}[{\em (i)}]
\item there exist $p,r\in\mathbb{C}$ such that $u_{i-1}+u_{i+1}=pu_i+r$
    $(i=1,2,\dots,D-1)$,
\item there exist $\xi,\omega,\eta^{\ster}\in\mathbb{C}$ such that
    $a_i(u_i-u_{i-1})(u_i-u_{i+1})=\xi u_i^2+\omega u_i+\eta^{\ster}$
    $(i=0,1,\dots,D)$, where $u_{-1}$ and $u_{D+1}$ are defined by (i) with
    $i=0$ and $i=D$, respectively,
\item $u_i\ne 1$ $(i=1,2,\dots,D)$.
\end{enumerate}
\end{prop}

\noindent We call $E$ a \emph{tail} \cite{Lang2002EJC} if $E\circ E$ is a
linear combination of $E_0,E$, and at most one other primitive idempotent of
$\AL$. Juri\v{s}i\'{c}, Terwilliger, and \v{Z}itnik \cite{JTZ2010JCTB}
established a characterization similar to Proposition \ref{Pascasio}, where
property (ii) is replaced by $E$ being a tail. We shall discuss tails in detail
in Section \ref{sec:vanishingKrein}.

The following characterization is due to
Kurihara and Nozaki \cite{KN2012JCTA}; cf.~\cite{NT2011LAA}.

\begin{prop}
Let $F$ be a primitive idempotent other than $E$. Then there is a
$Q$-polynomial ordering $(E_i)_{i=0}^D$ such that $E=E_1$ and $F=E_D$ if and
only if $u_0,u_1,\dots,u_D$ are distinct, and for $i=0,1,\dots,D$, the
eigenvalue of $A_i$ for $F$ is
\begin{equation*}
	\frac{(1-u_1)(1-u_2)\cdots(1-u_D)}{(u_i-u_0)\cdots(u_i-u_{i-1})(u_i-u_{i+1})\cdots(u_i-u_D)}.
\end{equation*}
\end{prop}

\noindent
This result
originated in an investigation of the $D$ distances occurring in the spherical
embedding $\{E\mathbf{e}_x\,:\, x\in V\}\subset\mathbb{R}^{m(\theta)}$,
extending a similar observation by Bannai and Bannai \cite{BB2005EJC} for strongly regular graphs.\footnote{See e.g., \cite{CGS1978NAWIM,Munemasa2004EJC,Suda2011JCD} for some results about spherical designs (cf.~\cite{BB2009EJC}) obtained in this way from $Q$-polynomial distance-regular graphs.}
Nozaki \cite{Nozaki2013pre} recently showed that $E$ is $Q$-polynomial provided that $v>\binom{m(\theta)+D-2}{D-1}+\binom{m(\theta)+D-3}{D-2}$ and $u_i\ne 1$ for $i=1,2,\dots,D$.

There are also many results characterizing $Q$-polynomial graphs within certain subclasses of distance-regular graphs,
such as bipartite graphs and tight graphs (cf.~Section \ref{sec:tightDRG}); see, e.g.,
\cite{Terwilliger1995DM,Pascasio2001GC,Lang2004JCTB,TW05,Suzuki2007EJC}. For example, if $\G$ is a thick regular near
polygon with $D\geq 3$, then $\G$ is $Q$-polynomial if and only if $\G$ has classical parameters;
cf.~\cite[Thm.~8.5.1]{bcn}. It should be remarked that De Bruyn and Vanhove \cite{DV2012pre} recently showed that for
$D\geq 4$ there are no $Q$-polynomial thick regular near polygons, apart from the Hamming graphs and dual polar graphs.
See also \cite[Thm.~C]{W97} and Theorem \ref{thm:RNPthick}.

\subsection{Classification results}
\label{Q-classification}

In this section, suppose that $E$ is a $Q$-polynomial idempotent, and let $p,r,r^{\ster}$ be as in \eqref{TTR}.
Note that these scalars depend on $E$.
We note also that, in the notation of Bannai and Ito \cite[\S III.5]{bi}
(cf.~\cite[\S 2]{Talgebra92}), $E$ is classical if and only if the $Q$-polynomial structure satisfies either type I with
$s^*=0$ or one of types IA, IIA, IIC; see \cite[Prop.~6.2]{Tanaka2011EJC} (cf.~\cite[Thm.~8.4.1]{bcn}). It turns out that most
graphs with $p=\pm 2$ already appeared in Sections \ref{sec: the case b=1} and \ref{sec: partition graphs}.

\subsubsection{Case \texorpdfstring{$p\ne\pm2$}{p ne pm 2}}
The $Q$-polynomial structure is type I or type IA in \cite{bi}.
The following result is due to Terwilliger [unpublished].

\begin{prop}
Type IA does not occur.
\end{prop}

\begin{proof}
If the $Q$-polynomial structure is type IA then $E$ is classical and we have
\begin{equation*}
	\theta_i=\theta_0-sb(1-b^i), \quad b_i=-tb^{i+1}(1-b^{i-D}), \quad c_i=b(1-b^i)(s-tb^{i-D-1})
\end{equation*}
for $i=0,1,\dots,D$, where $b,s,t\in\mathbb{C}\backslash\{0\}$ and $p=b+b^{-1}$; cf.~\cite[\S III.5]{bi}, \cite[\S
2]{Talgebra92}.
The corresponding classical parameters are $(D,b,\alpha,\beta)$, where $\alpha=tb^{1-D}(1-b)^2$ and $\beta=tb^{1-D}(1-b)$. In particular, $b$ is an integer distinct from $0,\pm 1$, and thus $s,t\in\mathbb{R}$.
From $\theta_0>\theta_1,\theta_2$, it follows that $b\ge 2$ and $s<0$.
Moreover, because $c_2>0$ we have $sb^D\leq 2sb^{D-1}<2t$.
But then
$\theta_0+\theta_D=2b_0-sb(1-q^D)=b(1-b^{-D})(sb^D-2t)<0$, so that $\theta_D<-\theta_0$, a contradiction.
\end{proof}

\noindent It follows that all graphs having classical parameters with $b\ne 1$ fall into type I with $s^*=0$ (with
respect to the associated $Q$-polynomial ordering).

\subsubsection{Case \texorpdfstring{$p=2$}{p=2}, \texorpdfstring{$r\ne 0$}{r ne 0}, \texorpdfstring{$r^{\ster}\ne 0$}{r* ne 0}}
The $Q$-polynomial structure is type II in \cite{bi}. Terwilliger \cite{Terwilliger1986JCTB} showed that if $D\geq 14$
then either $\G$ is the halved $(2D+1)$-cube, or $\G$ has the same intersection array as a folded Johnson graph or a
folded halved cube. By Propositions \ref{folded Johnson} and \ref{folded halved cubes}, the classification is now
complete for $D\geq 14$.

\subsubsection{Case \texorpdfstring{$p=2$}{p=2}, \texorpdfstring{$r=0$}{r=0}, \texorpdfstring{$r^{\ster}\ne 0$}{r* ne 0}} $E$ is classical and the $Q$-polynomial structure is type IIA in \cite{bi}.
It follows that $\G$ is either a Johnson graph, a halved cube, or the Gosset graph; cf.~Section \ref{sec: the case
b=1}.

\subsubsection{Case \texorpdfstring{$p=2$}{p=2}, \texorpdfstring{$r\ne 0$}{r ne 0}, \texorpdfstring{$r^{\ster}=0$}{r*=0}}
The $Q$-polynomial structure is type IIB in \cite{bi}.
Terwilliger \cite{Terwilliger1988C} showed that $\G$ is either a folded cube or one of the other two non-isomorphic graphs with the intersection array $\{6,5,4;1,2,6\}$ of the folded $6$-cube; cf.~\cite[\S
9.2D]{bcn}.

\subsubsection{Case \texorpdfstring{$p=2$}{p=2}, \texorpdfstring{$r=r^{\ster}=0$}{r=r*=0}}
$E$ is classical and the $Q$-polynomial structure is type IIC in \cite{bi}.
Egawa \cite{E-b1} showed that $\G$ is either a Hamming
graph or a Doob graph; cf.~\cite[\S 9.2B]{bcn}.

\subsubsection{Case \texorpdfstring{$p=-2$}{p=-2}}
The $Q$-polynomial structure is type III in \cite{bi}.
Terwilliger \cite{Terwilliger1987JCTB} showed that $\G$ is either the $D$-cube ($D$
even), the Odd graph $O_{D+1}$, or the folded $(2D+1)$-cube.

\medskip
\noindent Next, we move on to (almost) imprimitive graphs.

\subsubsection{Bipartite graphs}\label{sec:bipartiteDRG}
Suppose $\G$ is bipartite. Then $r^{\ster}=0$ by \cite[Thm.~8.2.1]{bcn}. If $p=\pm 2$ then it follows from the above
results that $\G$ is either the $D$-cube, the folded $2D$-cube, or one of the other two graphs with intersection array
$\{6,5,4;1,2,6\}$. Caughman \cite{Caughman2004GC} showed that if $p\ne\pm 2$ and $D\geq 12$ then $\G$ has classical
parameters $(D,b,0,1)$ where $b$ is an integer at least $2$.
These parameters are realized by the dual polar graphs $\mathcal{D}_D(b)$ and the Hemmeter graphs.

\subsubsection{Antipodal graphs}\label{sec:antipodalDRG}
Curtin \cite{Curtin1998DM} showed that bipartite $Q$-polynomial antipodal (double) covers are precisely the bipartite
$2$-homogeneous distance-regular graphs, and the latter graphs were classified by Nomura \cite{Nomura1995JCTB};
cf.~Section \ref{sec:homogeneity}. These are the $D$-cube, the regular complete bipartite graphs minus a perfect
matching,
the Hadamard graphs,
and the graphs with intersection arrays satisfying
\begin{equation*}
	(c_1,c_2,\dots,c_5)=(1,\mu,k-\mu,k-1,k), \quad b_i=c_{5-i} \ (i=0,1,\dots,4),
\end{equation*}
where $k=\gamma(\gamma^2+3\gamma+1)$, $\mu=\gamma(\gamma+1)$, and $\gamma\geq 2$ is an integer. The last case is
uniquely realized for $\gamma=2$ by the double cover of the Higman-Sims graph.

Dickie and Terwilliger \cite{DT1996EJC} gave a classification of non-bipartite $Q$-polynomial antipodal
distance-regular graphs as follows: the Johnson graph $J(2D,D)$, the halved $2D$-cube, the
non-bipartite Taylor graphs, and the graphs satisfying
\begin{gather}
	(c_1,c_2,c_3,c_4)=(1,p\eta,(p^2-1)(2\eta-p+1), p(2\eta+2\eta p-p^2)), \label{Q-AT4} \\
	b_i=c_{4-i} \ (i=0,1,2,3), \notag
\end{gather}
where $p\geq 3$, $\eta\geq 3p/4$ are integers and $\eta$ divides $p^2(p^2-1)/2$.
An example of the last case is the Meixner double cover ($p=4,\eta=6$); cf.~Section \ref{sec: Soicher and Meixner}. The
array \eqref{Q-AT4} with $p,\eta$ odd has been ruled out by Juri\v{s}i\'{c} and Koolen \cite[Cor.~3.2]{JuKo00EuJC}.

\subsubsection{Almost bipartite graphs}\label{sec:Qalmostbipartite}
The $Q$-polynomial generalized odd graphs have been classified by Lang and Terwilliger \cite{LT2007EJC}:
the folded $(2D+1)$-cube, the Odd graph $O_{D+1}$, and the graphs with $D=3$ satisfying
\begin{equation*}
	k=1+(p^2-1)(p(p+2)-(p+1)c_2), \quad c_3=-(p+1)(p^2+p-1-(p+1)c_2),
\end{equation*}
where $p<-2$ is an integer. No example is known for the last case.
We recall that the distance-$2$ graph $\G_2$ is again distance-regular, as it is the halved graph of
the bipartite double of $\G$; cf.~Section \ref{sec:2Porder}.

\subsubsection{Almost \texorpdfstring{$Q$-bipartite}{Q-bipartite} graphs}
Suppose $D\geq 4$ and $\G$ is almost $Q$-bipartite, i.e., $a_i^{\ster}=0$ for $i<D$ and $a_D^{\ster}>0$. Dickie
\cite{Dickie1995D} showed that $\G$ is either the halved $(2D+1)$-cube, the folded $(2D+1)$-cube, or a dual polar graph
$^2\mathcal{A}_{2D-1}(\sqrt{q})$. We note that the `$Q$-bipartite double' of $\G$ is a cometric association
scheme, and that $E_2$ is again a $Q$-polynomial idempotent; cf.~Sections \ref{sec:Qmultipleordering} and
\ref{sec:cometricschemes}.

\subsection{The Terwilliger algebras of \texorpdfstring{$Q$-polynomial}{Q-polynomial} distance-regular graphs}
\label{sec:thinness-Q}

Below we collect `handy' sufficient conditions for $\G$ being thin when it is $Q$-polynomial.

\begin{prop}{\em \cite[\S 5]{Talgebra92}}\label{thin criterion}
Suppose $\G$ is $Q$-polynomial with respect to the ordering $(E_i)_{i=0}^D$.
Then the following properties hold.
\begin{enumerate}[{\em (i)}]
\item $\G$ is thin with respect to $x\in V$ if for $i=1,2,\dots,D$ and for every $y,z\in\G_i(x)$, there is an
    automorphism $\pi$ of $\G$ such that $\pi(x)=x$, $\pi(y)=z$, and $\pi(z)=y$,
\item $\G$ is thin if $a_2=a_3=\dots=a_{D-1}=0$,
\item $\G$ is thin if $a_2^{\ster}=a_3^{\ster}=\dots=a_{D-1}^{\ster}=0$.
\end{enumerate}
\end{prop}

\noindent In particular, a $Q$-polynomial distance-regular graph is thin provided that it is bipartite (=$Q$-antipodal), almost bipartite, antipodal (=$Q$-bipartite), or almost $Q$-bipartite. It also follows that
many of the known graphs with classical parameters as well as partition graphs (cf.~Section \ref{sec: partition
graphs}) are thin; see \cite[Ex.~6.1]{Talgebra92} for details. The following graphs are known to be \emph{non}-thin:
Doob graphs, (bilinear, alternating, Hermitian, quadratic) forms graphs, and the twisted Grassmann graphs. The
irreducible $\TT$-modules of the Doob graphs were determined by Tanabe \cite{Tanabe1997JAC}. Concerning the twisted
Grassmann graph (cf.~Section \ref{twistedsection}), Bang, Fujisaki, and Koolen \cite{BFK09} showed that it is thin with
respect to any base vertex $x$ which is an $(e-1)$-dimensional subspace of the fixed hyperplane $H$, by verifying a
different combinatorial criterion for thinness \cite[Thm.~5.1(v)]{Talgebra92}. However, they also showed that if $x$ is
not contained in $H$ then the twisted Grassmann graph is not $1$-thin with respect to $x$.

The irreducible $\TT$-modules of bipartite (resp.~almost bipartite) $Q$-polynomial distance-regular graphs were
described by Caughman \cite{Caughman1999DM} (resp.~Caughman, MacLean, and Terwilliger~\cite{CMT2005DM}). For
these graphs, it turns out that the intersection array completely determines the structure of $\TT$. In particular,
explicit formulas for the multiplicities of the irreducible $\TT$-modules in $\mathbb{C}^v$ with small endpoints were
successfully used in the classification of these graphs; cf.~Section \ref{Q-classification}. Curtin and Nomura
\cite{CN2000JAC} and Curtin \cite{Curtin2001JCTB} studied the Terwilliger algebra of bipartite $Q$-polynomial antipodal
(double) covers which are not the $D$-cube; in this case, it follows that $\TT$ is a homomorphic image of the quantum
enveloping algebra $U_q(\mathfrak{sl}_2)$; cf.~Section \ref{sec: TD systems}.

In general, if $\G$ is $Q$-polynomial then the structure of irreducible $\TT$-modules with endpoint $1$ is determined
by the intersection array and the spectrum of the local graph $\Upsilon(x)$ with respect to the base vertex $x$;
cf.~\cite[Lecture 35]{Terwilliger1993N}. To be more precise, suppose for the moment that $\G$ is $Q$-polynomial, and
let $W$ be an irreducible $\TT$-module with endpoint $1$. Then $\dim E_1^{\ster}W=1$, so that $E_1^{\ster}W$ is an
eigenspace for $E_1^{\ster}AE_1^{\ster}$; let $\eta$ denote the corresponding eigenvalue. Then the isomorphism class of
$W$ is determined by $\eta$. Moreover, $W$ is thin if and only if $\eta$ is a root of a polynomial $T$ of degree $4$,
which we call the \emph{Terwilliger polynomial} of $\G$;
if $\G$ has classical parameters $(D,b,\alpha,\beta)$ then its four roots are $-1,-b-1,\beta-\alpha-1$, and $\alpha b\gauss{D-1}{1}-1$.
See also \cite[Lemma 4.7]{GK2012pre} and \cite[Cor.~4.12(5)]{Talgebra92}.
If $W$ is non-thin then it follows that\footnote{In fact, Terwilliger \cite[Lectures 34--37]{Terwilliger1993N} stated this result as a conjecture, and showed that $\dim E_i^{\ster}W\leq 2$ for $i=2,3,\dots,D-1$, $\dim E_D^{\ster}W\leq 1$, and that $W$ is thin if $\dim E_2^{\ster}W=1$. Now, that $W$ has diameter $D-1$ follows from a result of Go and Terwilliger \cite[Thm.~9.8]{GT2002EJC}, and the values of the $\dim E_i^{\ster}W$ follow from their symmetric and unimodal properties proved by Ito, Tanabe, and Terwilliger \cite{ITT2001P} in the more general context of tridiagonal systems; cf.~Section \ref{sec: TD systems}.}
$W$ has diameter $D-1$ and that $\dim E_1^{\ster}W=\dim E_D^{\ster}W=1$, $\dim E_2^{\ster}W=\dots=\dim E_{D-1}^{\ster}W=2$.
Hobart and Ito \cite{HI1998JAC} studied in detail the structure of such a non-thin irreducible $\TT$-module with
endpoint $1$. Miklavi\v{c} \cite{Miklavic2004EJC,Miklavic2009EJC} showed that $\G$ is $1$-homogeneous if it is
$Q$-polynomial with $a_1=0$, and described the unique irreducible $\TT$-module with endpoint $1$ (with $\eta=0$) when
$a_2\ne 0$, which turns out to be non-thin. Miklavi\v{c} \cite{Mik09} also described the irreducible $\TT$-modules with
endpoint $1$ when $\G$ has classical parameters with $b<-1$, $a_1\ne 0$, and is not a near polygon; there are exactly
two isomorphism classes, and the first one is thin with $\eta=-1$ and the second one is non-thin with $\eta=a_1$.

Suppose again that $\G$ is $Q$-polynomial, and let $\eta$ be a local eigenvalue of $\G$ (with respect to the base
vertex $x$), i.e., an eigenvalue of $\Upsilon(x)$. We call $\eta$ \emph{non-degenerate} if it has an eigenvector
orthogonal to the all-ones vector, and \emph{degenerate} otherwise. We note that $a_1$ is the only possible degenerate local eigenvalue and that it is
non-degenerate precisely when $\Upsilon(x)$ is disconnected. The Terwilliger polynomial $T$ mentioned above depends
only on the intersection array of $\G$ and the $Q$-polynomial ordering, and has the property that $T(\eta) \geq 0$
for every non-degenerate local eigenvalue $\eta$ for every base vertex $x$.
We note that if $\G$ has two $Q$-polynomial orderings then $T$ may be different for the different ordering.
Using the polynomial $T$, Gavrilyuk and Koolen \cite{GK2012pre} recently showed the uniqueness of the folded halved
$2m$-cube for $m \geq 6$; cf.~Section \ref{sec: partition graphs}. With the same approach we can also show the
uniqueness of the folded Johnson graphs. For the Grassmann graphs $J_q(2D, D)$ $(D \geq 3)$, Gavrilyuk and Koolen also
obtained partial results.
See also \cite[\S 4.3]{TTW16+} for more discussions on the Terwilliger polynomial.

See Section \ref{sec:thinness} for more results on the irreducible $\TT$-modules with endpoint $1$ of general
distance-regular graphs.

\subsection{Further results on \texorpdfstring{$Q$-polynomial}{Q-polynomial} distance-regular graphs}
\label{sec: further results}

In this section, we always assume that $\G$ is $Q$-polynomial.

\subsubsection{Antipodal covers}
Van Bon and Brouwer \cite{BB1988P} determined the distance-regular antipodal covers of the classical families of
distance-regular graphs; cf.~\cite[\S 6.12]{bcn}. Suppose $E$ is $Q$-polynomial, and recall (cf.~\eqref{TTR}) that
there exist $p,r\in\mathbb{C}$ such that $u_{i-1}+u_{i+1}=pu_i+r$ for $i=1,2,\dots,D-1$. Terwilliger
\cite{Terwilliger1993EJC} showed that if $\G$ has an antipodal cover of diameter $\tilde{D}\geq 7$, then this
three-term recurrence extends to $i=1,2,\dots,\tilde{D}-1$, where we formally define $u_i=u_{\tilde{D}-i}$
($i=D+1,D+2,\dots,\tilde{D}$). This parametric condition provides simple proofs of some of the (non-existence) results
in \cite{BB1988P}, and may be applied to the twisted Grassmann graphs as well; cf.~\cite{FKT07}. Caughman
\cite{Caughman1999JCTB} used the condition to show that if $\G$ is bipartite with $D\geq 4$ and has an antipodal cover
then $\G$ is the folded $2D$-cube; cf.~\cite[Cor.~12.3]{Lang2004JCTB}.

\subsubsection{Distance-regular graphs with multiple \texorpdfstring{$Q$-polynomial}{Q-polynomial} orderings}\label{sec:Qmultipleordering}
An association scheme can have at most two $P$-polynomial orderings, except for those coming from the polygons; cf.~\cite[\S 4.2D]{bcn}.
Bannai and Ito \cite[pp.~354--360]{bi} showed that if $k\geq 3$ and $D\geq 34$ then $\G$ has at most two $Q$-polynomial idempotents and moreover all eigenvalues are integral.
Brouwer, Cohen, and Neumaier \cite[p.~247]{bcn} conjectured
that the assumption $D\geq 34$ can be replaced by $D\ne 4$. Dickie
\cite[pp.~69--70]{Dickie1995D} established the result under the assumption $D\geq 5$.
Indeed, he showed that if $k\geq 3$ and $D\geq 5$ then $\G$ has more than one $Q$-polynomial idempotent if and only if $\G$ is either the $D$-cube ($D$ even), the halved $(2D+1)$-cube, the folded $(2D+1)$-cube, or a dual polar graph
$^2\!\mathcal{A}_{2D-1}(\sqrt{q})$, and these graphs have precisely two $Q$-polynomial idempotents but no non-integral eigenvalues.
(Note that if $\G$ has non-integral eigenvalues and $E$ is $Q$-polynomial then $E^{\sigma}$ is again $Q$-polynomial for any $\mathbb{Q}$-automorphism $\sigma$ of the splitting field over $\mathbb{Q}$.)
Building on work by
Dickie \cite{Dickie1995D}, Suzuki \cite{Suzuki1998JACb} showed that every association scheme has at most two
$Q$-polynomial idempotents, again except for those coming from the polygons; cf.~Section \ref{sec:cometricschemes}.
For $D\in\{2,3,4\}$, the known $Q$-polynomial distance-regular graphs with $k\geq 3$ and with non-integral eigenvalues belong to the following four families:
the conference graphs $(D=2)$, the incidence graphs of symmetric designs ($D=3$),
the Taylor graphs ($D=3$), and the Hadamard graphs ($D=4$).\footnote{It seems that the above conjecture by Brouwer et al.~was not properly stated, because we already have counterexamples with $D\in\{2,3\}$. We note that the graphs in the last three families are imprimitive.}
Note that the graphs in these families always have two $Q$-polynomial idempotents.
The other candidate intersection arrays $\{\mu(2\mu+1),(\mu-1)(2\mu+1),\mu^2,\mu;1,\mu,\mu(\mu-1),\mu(2\mu+1)\}$ ($\mu\geq 2$) of primitive $Q$-polynomial distance-regular graphs with non-integral eigenvalues given by Brouwer et al.~\cite[pp.~247--248]{bcn} were ruled out by Godsil and Koolen \cite{GoKo95}; cf.~Section \ref{sec:tablesD4prnon}.
Ma and Koolen \cite{KM2014pre} recently classified the distance-regular graphs with $k\geq 3$, $D=4$, and with two $Q$-polynomial idempotents; these are the $4$-cube, the halved $9$-cube, the folded $9$-cube, the dual polar graphs $^2\!\mathcal{A}_7(\sqrt{q})$, and the Hadamard graphs.

\subsubsection{Bounds for the girth}\label{sec:girth}
Brouwer, Cohen, and Neumaier \cite[p.~248]{bcn} conjectured that $\G$ has girth at most $6$, with equality only for the Odd graph $O_{D+1}$, and showed that the numerical girth $g$ of $\G$ is at most $7$.
Lewis \cite{Lewis2000DM} showed $c_3\geq 2$,
proving $g\leq 6$.
We note that if $\G$ has girth $6$, i.e., $a_1=a_2=0$ and $c_2=1$, then it follows from Proposition \ref{Pascasio} (or
\cite[Thm.~6.3]{Miklavic2004EJC}) that $a_1=a_2=\dots=a_{D-1}=0$, so that $\G$ is bipartite or almost
bipartite.
Miklavi\v{c} \cite{Miklavic2007DM} showed that if $\G$ is bipartite and $D=4$ then $c_2\geq 2$, i.e., $g=4$.

\subsubsection{The \texorpdfstring{Erd\H{o}s-Ko-Rado}{Erdos-Ko-Rado} theorem}\label{sec:EKR}
At the end of each of Sections 9.1-9.4 and 9.5A in \cite{bcn} there is a remark about the Erd\H{o}s-Ko-Rado theorem
for the graph in question. See
\cite{Tanaka2006JCTA,PSV2011JCTA,Tanaka2010pre,IM2013DCC,Tanaka2012pre,GM2016B}
for recent results on this topic.

\subsubsection{Unimodality of the multiplicities}\label{sec: unimodality of mi}
We recall from Proposition \ref{prop:unimodal} (iv) that the $k_i$ are unimodal. Concerning the multiplicities $m_i$,
Pascasio \cite{Pascasio2002EJC} showed that if $\G$ is $Q$-polynomial with respect to the ordering $(E_i)_{i=0}^D$ then
$m_{i-1}\le m_i\le m_{D-i}$ for $i=1,2,\dots,\lfloor D/2\rfloor$. This
result was originally conjectured by Dennis Stanton in 1993, and is a simple application of the theory of
tridiagonal systems; cf.~Section \ref{sec: TD systems}. We note that Bannai and Ito \cite[p.~205]{bi} earlier
conjectured that the multiplicities of a cometric association scheme satisfy the unimodal property.

\subsubsection{Posets associated with \texorpdfstring{$Q$-polynomial}{Q-polynomial} distance-regular graphs}
\label{sec: posets}

There are several classes of finite ranked posets that are closely related to $Q$-polynomial distance-regular graphs:
\emph{regular semilattices} \cite{Delsarte1976JCTA,Stanton1985JCTA}, \emph{uniform posets} \cite{Terwilliger1990P},
\emph{quantum matroids} \cite{Terwilliger1996P}. (For definitions, see the references given.)
Many of the known families of $Q$-polynomial distance-regular graphs arise as the top fibers of these posets, where two
vertices are adjacent if and only if they cover a common element.

Concerning quantum matroids, Terwilliger \cite[Thm.~38.2]{Terwilliger1996P} showed that if a quantum matroid is
`non-trivial'
 and `regular', then the graph on the top fiber with the above adjacency is distance-regular.
Moreover, in this case, the graph has classical parameters if its diameter is equal to the rank of the quantum matroid.
The culmination of the study of quantum matroids is the classification (\cite[Thm.~39.6]{Terwilliger1996P}) of
non-trivial regular quantum matroids with rank at least four: they are precisely those posets naturally associated with
Johnson, Hamming, Grassmann, bilinear forms, and dual polar graphs. We may use this classification as follows.

Fix a $Q$-polynomial ordering $(E_i)_{i=0}^D$ of $\G$. Let $Y$ be a non-empty subset of $V$ and let $\chi$ be its
characteristic vector. Brouwer, Godsil, Koolen, and Martin~\cite{BGKM03} defined the \emph{width} and \emph{dual width}
of $Y$ by $w=\max\{i:\chi^{\mathsf{T}}A_i\chi\ne 0\}$ and $w^{\ster}=\max\{i:\chi^{\textsf{T}}E_i\chi\ne 0\}$,
respectively. They showed among other results that $w+w^{\ster}\ge D$, and we call $Y$ a \emph{descendent}
(cf.~\cite{Tanaka2011EJC}) of $\G$ if equality holds. It follows that every descendent is completely regular, and that the
induced subgraph is a $Q$-polynomial distance-regular graph if it is connected;
cf.~\cite[Thm.~1--3]{BGKM03}.\footnote{The results in \cite{BGKM03} are in contrast with Delsarte theory \cite{del}
based on the minimum distance and (maximum) strength of a subset. We may remark that Suda \cite{Suda2012JCTA} recently
developed a theory which unifies and `interpolates' some of the theorems in \cite{del} and \cite{BGKM03} to a certain
extent.} We say that a set $\mathscr{D}$ of descendents of $\G$ \emph{satisfies} $\mathrm{(UD)}_i$ if each two vertices $x,y\in
V$ at distance $i$ are contained in a unique descendent in $\mathscr{D}$ with width $i$.

\begin{prop}{\em \cite{Tanaka2011EJC}}
Let $\mathscr{D}$ be a set of descendents of $\G$. Suppose that the following properties hold.
\begin{enumerate}[{\em (i)}]
\item $\G$ has classical parameters,
\item $\mathscr{D}$ satisfies $\mathrm{(UD)}_i$ for all $i$,
\item $Y_1\cap Y_2\in\mathscr{D}\cup\{\emptyset\}$ for all $Y_1,Y_2\in\mathscr{D}$.
\end{enumerate}
Then $\mathscr{D}$, together with the partial order defined by reverse inclusion, forms a non-trivial regular quantum
matroid. In particular, if $D\geqslant 4$ then $\G$ is either a Johnson, Hamming, Grassmann, bilinear forms, or dual
polar graph.
\end{prop}

\noindent It was also shown that if $\mathscr{D}$ is the set of \emph{all} descendents of $\G$ then condition (iii) in
the above proposition is implied by the other two. See \cite{BGKM03,Tanaka2006JCTA,HS2007EJC,Tanaka2011EJC} for more information
on descendents.

Unlike regular semilattices and quantum matroids, uniform posets are not assumed to be semilattices, but give rise to
at least 13 infinite families of $Q$-polynomial distance-regular graphs with unbounded diameter, rather than just five
as above; cf.~\cite[\S 4]{Terwilliger1990P}.
Suppose $\G$ is ($Q$-polynomial and) bipartite, and fix $x\in V$. Then we may view $\G$ as the Hasse diagram
of a ranked poset with $D+1$ fibers $\G_i(x)$ $(i=0,1,\dots,D)$. Miklavi\v{c} and Terwilliger \cite{MT2011pre} recently
showed that this poset is uniform.\footnote{See \cite{MT2011pre} for the precise statement of the result, noting that
the hypercube $H(D,2)$ with $D$ even has two $Q$-polynomial structures. They also introduced the concept of
\emph{strongly uniform} and investigated when the poset $\G$ has that property.} Caughman \cite{Caughman2003EJC} showed
that the graph on the top fiber $\G_D(x)$ defined in the previous manner (which is in this case the induced subgraph of
the distance-$2$ graph of $\G$) is distance-regular and $Q$-polynomial.
See \cite{TW2013AMC} and the references therein for more results on uniform posets.

The poset $\mathscr{S}$ consisting of all strongly closed subgraphs of $\G$ with partial order defined by reverse
inclusion plays an important role in the study of distance-regular graphs having classical parameters with $b<-1$.
Suppose $\G$ has geometric parameters $(D,b,\alpha)$ (cf.~Section \ref{sec:generalizations+geometric}) with $D\geq 4$ and is
$D$-bounded in the sense of Weng \cite{W97,We99}, i.e., every $\Delta\in\mathscr{S}$ is assumed to be regular. Then
$b<-1$ by \cite[Lemma~5.5]{We99}. (Conversely, if $\G$ has classical parameters with $b<-1,D\geq 4,a_1\ne 0,c_2>1$ then
$\G$ is $D$-bounded and has geometric parameters; cf.~\cite[Thm.~5.7,~5.8]{We99}.) In this case, Weng \cite{W97} showed
that $\mathscr{S}$ is a ranked (meet) semilattice and every interval is a modular atomic lattice which is isomorphic to
a projective geometry over $GF(b^2)$.

\subsection{Tridiagonal systems}
\label{sec: TD systems}

Let $W$ be a finite dimensional vector space over $\mathbb{C}$.
Let $\mathfrak{a}\in\mathrm{End}_{\mathbb{C}}(W)$ be diagonalizable, and let $(\theta_i)_{i=0}^{\delta}$ be an ordering of the distinct eigenvalues of $\mathfrak{a}$.
Then there is a sequence of elements $(\mathfrak{e}_i)_{i=0}^{\delta}$ in $\mathrm{End}_{\mathbb{C}}(W)$ such that
(i) $\mathfrak{ae}_i=\theta_i\mathfrak{e}_i$;
(ii) $\mathfrak{e}_i\mathfrak{e}_j=\delta_{ij}\mathfrak{e}_i$;
(iii) $\sum_{i=0}^{\delta}\mathfrak{e}_i=\mathfrak{1}$, where $\mathfrak{1}$ is the identity element in $\mathrm{End}_{\mathbb{C}}(W)$.
(Specifically, $\mathfrak{e}_i=\prod_{j\ne i}\frac{\mathfrak{a}-\theta_j\mathfrak{1}}{\theta_i-\theta_j}$ ($i=0,1,\dots,\delta$).)
We call $\mathfrak{e}_i$ the \emph{primitive idempotent} of $\mathfrak{a}$ associated with $\theta_i$ ($i=0,1,\dots,\delta$).
Let $\mathfrak{a}^{\ster}$ be another diagonalizable element in $\mathrm{End}_{\mathbb{C}}(W)$.
Let $(\theta_i^{\ster})_{i=0}^{\delta^{\ster}}$ be an ordering of the distinct eigenvalues of $\mathfrak{a}^{\ster}$ and let $(\mathfrak{e}_i^{\ster})_{i=0}^{\delta^{\ster}}$ be the corresponding sequence of the primitive idempotents.
The sequence
$\Phi=(\mathfrak{a}; \mathfrak{a}^{\ster}; (\mathfrak{e}_i)_{i=0}^{\delta};
(\mathfrak{e}_i^{\ster})_{i=0}^{\delta^{\ster}})$ is a \emph{tridiagonal system} (or \emph{TD system}) if
\begin{align*}
	\mathfrak{e}_i^{\ster}\mathfrak{a}\mathfrak{e}_j^{\ster}&=0 \quad \text{if } |i-j|>1 \quad (i,j=0,1,\dots,\delta^{\ster}), \\
	\mathfrak{e}_i\mathfrak{a}^{\ster}\mathfrak{e}_j&=0 \quad \text{if } |i-j|>1 \quad (i,j=0,1,\dots,\delta),
\end{align*}
and $W$ is irreducible as a $\mathbb{C}[\mathfrak{a},\mathfrak{a}^{\ster}]$-module.
This definition is due to Ito, Tanabe, and Terwilliger \cite{ITT2001P}.\footnote{TD systems can be defined on vector spaces over arbitrary fields, and many of the results are valid over wider classes of fields. However, for simplicity and in view of the connections to the theory of $Q$-polynomial distance-regular graphs, we shall only discuss TD systems over $\mathbb{C}$.}
Note that if $\G$ is a $Q$-polynomial distance-regular graph then it follows from \eqref{relations} that every irreducible $\TT$-module naturally has the structure of a TD system.

Suppose that $\Phi$ is a TD system.
Ito et al.~\cite{ITT2001P} showed $\delta=\delta^{\ster}$.
Define $U_i=(\sum_{h=0}^i\mathfrak{e}_h^{\ster}W)\cap(\sum_{\ell=i}^{\delta}\mathfrak{e}_{\ell}W)$ ($i=0,1,\dots,\delta$).
Note that $U_0=\mathfrak{e}_0^{\ster}W$, and that $(\mathfrak{a}-\theta_i\mathfrak{1})U_i\subseteq U_{i+1}$, $(\mathfrak{a}^{\ster}-\theta_i^{\ster}\mathfrak{1})U_i\subseteq U_{i-1}$ ($i=0,1,\dots,\delta$), where $U_{-1}=U_{\delta+1}=0$.
They showed $W=\bigoplus_{i=0}^{\delta}U_i$.
It also turns out that $\dim\mathfrak{e}_iW=\dim\mathfrak{e}_i^{\ster}W=\dim U_i$ $(i=0,1,\dots,\delta)$.
The sum $W=\bigoplus_{i=0}^{\delta}U_i$ is called the \emph{split decomposition}
and plays a crucial role in the theory of TD systems.
Let $\rho_i=\dim\mathfrak{e}_iW$ $(i=0,1,\dots,\delta)$ and call the sequence $(\rho_i)_{i=0}^{\delta}$ the \emph{shape} of $\Phi$.
They showed that the shape is symmetric and unimodal: $\rho_i=\rho_{\delta-i}$
($i=0,1,\dots,\delta$) and $\rho_{i-1}\le \rho_i$ ($i=1,2,\dots,\lfloor \delta/2\rfloor$).
A TD system with $\rho_0=\dots=\rho_{\delta}=1$ is called a \emph{Leonard system} \cite{Terwilliger2001LAA}.
Leonard systems provide a linear algebraic framework for Leonard's theorem and have been extensively studied; see \cite{Terwilliger2006N} and the references therein.
Note that if the TD system $\Phi$ is afforded on an irreducible $\TT$-module of a $Q$-polynomial distance-regular graph, then the $\TT$-module $W$ is thin if and only if $\Phi$ is a Leonard system.
See \cite{Cerzo2010LAA} for a detailed description of thin irreducible $\TT$-modules motivated by the theory of Leonard systems.

Ito et al.~\cite{ITT2001P} showed that there exist scalars $p,\gamma,\gamma^{\ster},\varrho,\varrho^{\ster}\in\mathbb{C}$ such that
\begin{align}
	 0&=[\mathfrak{a},\mathfrak{a}^2\mathfrak{a}^{\ster}-p\mathfrak{a}\mathfrak{a}^{\ster}\mathfrak{a}+\mathfrak{a}^{\ster}\mathfrak{a}^2-\gamma(\mathfrak{a}\mathfrak{a}^{\ster}+\mathfrak{a}^{\ster}\mathfrak{a})-\varrho\mathfrak{a}^{\ster}], \label{TD relation} \\
	 0&=[\mathfrak{a}^{\ster},\mathfrak{a}^{\ster 2}\mathfrak{a}-p\mathfrak{a}^{\ster}\mathfrak{a}\mathfrak{a}^{\ster}+\mathfrak{a}\mathfrak{a}^{\ster 2}-\gamma^{\ster}(\mathfrak{a}^{\ster}\mathfrak{a}+\mathfrak{a}\mathfrak{a}^{\ster})-\varrho^{\ster}\mathfrak{a}], \label{TD relation*}
\end{align}
where $[\mathfrak{b},\mathfrak{c}]:=\mathfrak{bc}-\mathfrak{cb}$, and (cf.~\eqref{TTR})
\begin{equation}\label{AW sequence}
	\frac{\theta_{i-2}-\theta_{i+1}}{\theta_{i-1}-\theta_i}=\frac{\theta_{i-2}^{\ster}-\theta_{i+1}^{\ster}}{\theta_{i-1}^{\ster}-\theta_i^{\ster}}=p+1 \quad (i=2,3,\dots,\delta-1).
\end{equation}
The relations \eqref{TD relation} and \eqref{TD relation*} generalize the $q$-Serre relations (which are among the defining relations of the quantum affine algebra $U_q(\widehat{\mathfrak{sl}}_2))$ and the Dolan-Grady relations (which are the defining relations of the Onsager algebra); cf.~\cite{Terwilliger2001P}.
It is conjectured (\cite[Conj.~13.7]{ITT2001P}) that there exist positive integers $\delta_1,\delta_2,\dots,\delta_n$ such that $\sum_{i=0}^{\delta}\rho_it^i=\prod_{j=1}^n(1+t+\dots+t^{\delta_j})$, where $t$ is an indeterminate.
This conjecture in fact suggests that $\Phi$ would be regarded as a `tensor product' of Leonard systems.
Let $q$ be a nonzero scalar in $\mathbb{C}$ such that $p=q^2+q^{-2}$.
Using the representation theory of $U_q(\widehat{\mathfrak{sl}}_2)$ (cf.~\cite{CP1991CMP}), Ito and Terwilliger \cite{IT2007JAA,IT2010KJM} indeed constructed \emph{all} TD systems (up to isomorphism\footnote{A TD system $\Phi'=(\mathfrak{a}'; \mathfrak{a}^{\ster \prime}; (\mathfrak{e}_i')_{i=0}^{\delta}; (\mathfrak{e}_i^{\ster \prime})_{i=0}^{\delta})$ on a vector space $W'$ is \emph{isomorphic} to $\Phi$ if there is an isomorphism of vector spaces $\sigma:W\rightarrow W'$ such that $\sigma\mathfrak{a}=\mathfrak{a}'\sigma$, $\sigma\mathfrak{a}^{\ster}=\mathfrak{a}^{\ster \prime}\sigma$, and $\sigma\mathfrak{e}_i=\mathfrak{e}_i'\sigma$, $\sigma\mathfrak{e}_i^{\ster}=\mathfrak{e}_i^{\ster \prime}\sigma$ for $i=0,1,\dots,\delta$.}) explicitly as tensor products of Leonard systems (i.e., \emph{evaluation modules}), under the assumption that $q$ is not a root of unity.
We remark that in this case the split decomposition corresponds to the weight space decomposition.
See also \cite{IT2004JPAA,IT2007CA,IT2009JCISS,IT2009JA,Funk-Neubauer2009LAA,HI2012pre}.
Ito [private communication] pointed out that the proofs of most of the results in \cite{IT2010KJM} work under the weaker assumption $q^2\ne\pm 1$, i.e., $p\ne\pm 2$.
It seems that the above conjecture is still open for general TD systems, but Nomura and Terwilliger \cite{NT2008LAAd,NT2010LAAa} showed among other results that $\rho_0=1$, and more generally, $\rho_i\leq \binom{\delta}{i}$ ($i=0,1,\dots,\delta$), a result which would follow directly from the conjecture.
See, e.g., \cite{Hartwig2007LAA,IT2007LAA,IS2014LAA} for some results on TD systems with $p=2$.

Observe now that the $1$-dimensional subspace $\mathfrak{e}_0^{\ster}W$ is invariant under
\begin{equation*}
	(\mathfrak{a}^{\ster}-\theta_1^{\ster}\mathfrak{1})(\mathfrak{a}^{\ster}-\theta_2^{\ster}\mathfrak{1})\dots(\mathfrak{a}^{\ster}-\theta_i^{\ster}\mathfrak{1})(\mathfrak{a}-\theta_{i-1}\mathfrak{1})\dots(\mathfrak{a}-\theta_1\mathfrak{1})(\mathfrak{a}-\theta_0\mathfrak{1})
\end{equation*}
for $i=0,1,\dots,\delta$, and let $\zeta_i$ be the corresponding eigenvalue ($i=0,1,\dots,\delta$).
The sequence $((\theta_i)_{i=0}^{\delta};(\theta_i^{\ster})_{i=0}^{\delta};(\zeta_i)_{i=0}^{\delta})$ is called the \emph{parameter array} of $\Phi$.
Nomura and Terwilliger \cite{NT2008LAAd} showed that the parameter array is a complete invariant for a TD system.
Ito, Nomura, and Terwilliger \cite{INT2011LAA} established the following theorem:

\begin{theorem}{\em \cite[Thm. 3.1]{INT2011LAA}}\label{classification of sharp TD systems}
Let $\pi=((\theta_i)_{i=0}^{\delta};(\theta_i^{\ster})_{i=0}^{\delta};(\zeta_i)_{i=0}^{\delta})$ be a sequence of scalars in $\mathbb{C}$ such that $\theta_i\ne\theta_j$, $\theta_i^{\ster}\ne\theta_j^{\ster}$ if $i\ne j$ $(i,j=0,1,\dots,\delta)$, and suppose that \eqref{AW sequence} holds for some $p\in\mathbb{C}$.
Then there exists a (unique) TD system with parameter array $\pi$ if and only if $\zeta_0=1$, $\zeta_{\delta}\ne 0$, and $\sum_{i=0}^{\delta}\zeta_i\prod_{\ell=i+1}^{\delta}(\theta_0-\theta_{\ell})(\theta_0^{\ster}-\theta_{\ell}^{\ster})\ne 0$.
\end{theorem}

\noindent
We remark that the left-hand side of the last condition on the $\zeta_i$ is a certain value of the \emph{Drinfel'd polynomial} of the corresponding TD system; cf.~\cite{IT2009JCISS,IT2010KJM}.
See, e.g., \cite{NT2009LAAa,NT2010LAAb,Bockting-Conrad2012LAA} for more results on TD systems.

Given the above progress in the theory of TD systems, it is important to `pull back' the results to the study of $Q$-polynomial distance-regular graphs.
For example, Pascasio \cite{Pascasio2002EJC} used the symmetric and unimodal property of the shape of $\Phi$ to study the multiplicities $m_i$ of a $Q$-polynomial distance-regular graph $\G$; cf.~Section \ref{sec: unimodality of mi}.
Terwilliger \cite{Terwilliger2005GC} `extended', so to speak, the split decompositions of the TD systems on the irreducible $\TT$-modules to the entire standard module $\mathbb{C}^v$, and obtained the \emph{split} and \emph{displacement decompositions for} $\G$.
Ito and Terwilliger \cite{IT2009EJC} used these decompositions to show that for the forms graphs there are four natural algebra homomorphisms from $U_q(\widehat{\mathfrak{sl}}_2)$ to $\TT$ via the so-called $q$-\emph{tetrahedron algebra}
$\boxtimes_q$ \cite{IT2007CA}, and that $\TT$ is generated by each of their images together with the center $Z(\TT)$.
Corresponding results for the case $p=2$, i.e., for Hamming and Doob graphs, were recently obtained by Morales and Pascasio \cite{MP2012pre}.
See also \cite{IT2009MMJ,Kim2009EJC,Kim2010DM} for more results on the split and displacement decompositions.
Worawannotai \cite{Worawannotai2012PhD} applied a similar idea to dual polar graphs to show (among other results) that there are two algebra homomorphisms from the quantum algebra $U_q(\mathfrak{sl}_2)$ to $\TT$, and that $\TT$ is again generated by each of their images together with $Z(\TT)$.
The split and displacement decompositions have also been applied to the Assmus-Mattson theorem for codes in $Q$-polynomial distance-regular graphs \cite{Tanaka2009EJC}; cf.~\cite[\S 2.8]{bcn}.

%% file: 6_Talgebra.txt

In this section, let $\G$ be a distance-regular graph with diameter $D\geq 3$, valency $k\geq 3$, and eigenvalues $k =
\theta_0 > \theta_1 > \cdots > \theta_D$. Concerning $1$-homogeneity of distance-regular graphs, we shall occasionally
consider the following weaker concepts. We say $\G$ is $1$-\emph{homogeneous with respect to an edge} $xy$ if the
parameters $p_{i,j;r,s}$ exist with respect to $x,y$ for all $i,j,r,s=0,1,\dots,D$; cf.~Section \ref{sec:homogeneity4}.
We say $\G$ is $1$-\emph{homogeneous with respect to a vertex} $x\in V$ if it is $1$-homogeneous with respect to the
edge $xy$ for every $y\in \G(x)$ and the parameters $p_{i,j;r,s}$ do not depend on the choice of $y$.

\subsection{Homogeneity and tight distance-regular graphs}

\subsubsection{Tight distance-regular graphs}
\label{sec:tightDRG}

Juri\v{s}i\'{c}, Koolen, and Terwilliger \cite{JKT2000JAC} showed the following so-called `fundamental bound':
\begin{equation}\label{FB}
	\left(\theta_1 + \frac{k}{a_1+1}\right) \left(\theta_D + \frac{k}{a_1+1}\right) \geq -\frac{ka_1b_1}{(a_1+1)^2}.
\end{equation}
For $a_1 =0$, equality holds if and only if $\G$ is bipartite.
One way to prove this bound is to use the fact
(\cite[Thm.~4.4.3]{bcn}) that $(\eta_i-\tilde{\theta}_1)(\eta_i-\tilde{\theta}_D)\leq 0$ for $i=2,3,\dots,k$, where
$\tilde{\theta}_1=-1-\frac{b_1}{1+\theta_1}$, $\tilde{\theta}_D=-1-\frac{b_1}{1+\theta_D}$, and
$a_1=\eta_1\geq \eta_2 \geq \dots \geq \eta_k$ are the eigenvalues of a local graph; cf.~\cite{JuKo00EuJC}. This
immediately shows that if $a_1\ne 0$ then equality holds if and only if every (or at least one) local graph is
connected strongly regular with non-trivial eigenvalues $\tilde{\theta}_1,\tilde{\theta}_D$.
We may also prove \eqref{FB} by considering the determinants of the Gram matrices of the three vectors
$E\textbf{e}_x,E\textbf{e}_y,E\chi_{1,1}(x,y)$ for adjacent vertices $x,y\in V$ and $E\in\{E_1,E_D\}$, where
$\chi_{1,1}(x,y)$ is the characteristic vector of $\G_{1,1}(x,y)=\G(x)\cap\G(y)$; cf.~\cite{JKT2000JAC}. See also
\cite{Pascasio2003DM} for another proof. We say $\G$ is {\em tight} if $a_1 \neq 0$ and equality holds in \eqref{FB}.
Juri\v{s}i\'{c} et al.~\cite{JKT2000JAC} showed that $\G$ is tight if and only if $a_1 \neq 0$, $a_D =0$,
and $\G$ is $1$-homogeneous.
To be more precise, call an edge $xy$ \emph{tight with respect to} a non-trivial
eigenvalue $\theta$ if $E\textbf{e}_x,E\textbf{e}_y,E\chi_{1,1}(x,y)$ are
linearly dependent, where $E$ is the primitive idempotent corresponding to
$\theta$. Then the following properties are all equivalent: (i) $\G$ is
tight; (ii) $a_1\ne 0$ and every (or at least one) edge of $\G$ is tight
with respect to both $\theta_1$ and $\theta_D$; (iii) $a_1\ne 0$, $a_D=0$, and
$\G$ is $1$-homogeneous (or $1$-homogeneous with respect to an edge).
Pascasio \cite{Pascasio2001GC} showed that if $\G$ is $Q$-polynomial then the following
properties are equivalent: (i) $\G$ is tight; (ii) $\G$ is non-bipartite and $a_D=0$; (iii) $\G$ is non-bipartite and
$a_D^{\ster}=0$. More characterizations of the tightness property will be given in the next
sections. The fundamental bound inspired quite a bit of the later research by Terwilliger and his students.

It follows from the above result of Pascasio that the non-bipartite antipodal $Q$-polynomial distance-regular graphs are tight; examples are the Johnson graph $J(2D,D)$, the
halved $2D$-cube, the non-bipartite Taylor graphs and the Meixner $2$-cover; cf.~Section \ref{sec:antipodalDRG}.
There are several sporadic examples known, all of which have diameter $4$, and of which only one is primitive, namely
the Patterson graph.
Juri\v{s}i\'{c} and Koolen \cite[Thm.~3.2]{JK2002DM} showed that tight distance-regular graphs with $D=3$ are precisely the non-bipartite Taylor graphs.
Suda \cite{Suda2012EJC} recently gave a simple proof of this result by looking at the
intersection matrix $L$; cf.~\eqref{matrixl}.

Using the fact that the Patterson graph, the Meixner $4$-cover, the $3.O_7(3)$-graph, and the $3.O_6^-(3)$-graph are
tight and hence $1$-homogeneous, one can easily show that the minimal convex subgraph of two vertices at distance two
is a complete multipartite graph $K_{n \times t}$ with $n \geq 2$, $t \geq 2$.
This leads in each of the cases to its
uniqueness as a distance-regular graph; cf.~\cite{JK2003JAC,JuKo07JAC,JuKo11,JuKopre,BrJuKo08}.

The family of tight antipodal distance-regular graphs with $D=4$ is called the $\mathrm{AT}4$-family. That they are
$1$-homogeneous gives rise to several feasibility conditions; cf.~\cite{JuKo00EuJC}.
Juri\v{s}i\'{c} and Koolen
\cite{JuKo11} classified the members of the $\mathrm{AT}4$-family with complete multipartite $\mu$-graphs. Juri\v{s}i\'{c},
Munemasa, and Tagami \cite{JuMuTa10} simplified, generalized, and strengthened some of the results in \cite{JuKo11}.

Vidali and Juri\v{s}i\'{c} \cite{VJ2013pre} recently showed the non-existence of primitive tight distance-regular graphs with classical parameters $(D,b,b-1,b^{D-1})$, where $D\geq 4$ and $b>1$.

\subsubsection{The \texorpdfstring{$\mathrm{CAB}$}{CAB} condition and \texorpdfstring{$1$}{1}-homogeneous distance-regular graphs}

Juri\v{s}i\'{c} and Koolen \cite{JuKo00DCC} introduced the $\mathrm{CAB}_j$ condition.
For vertices $x,y\in V$
at distance $i =0,1, \dots, D$, define the sets
$C_i(x,y) = \G_{i-1}(x)\cap \G(y)$, $A_i(x,y) =
\G_i(x) \cap \G(y)$, and $B_i(x,y) = \G_{i+1}(x) \cap \G(y)$ (with $\G_{-1}(x) = \G_{D+1}(x) =
\emptyset$).
For $j=0,1, \ldots, D$, we say $\G$ \emph{satisfies} $\mathrm{CAB}_j$, if for
all $i=0,1, \ldots, j$ and $x, y\in V$ at distance $i$, the partition $\{C_i(x,y),
A_i(x,y), B_i(x,y)\}$ of the local graph $\Upsilon(y)$ is equitable (where we
assume that empty sets are excluded from the partition).
It is clear that $\G$ satisfies $\mathrm{CAB}_0$, and that $\G$ satisfies $\mathrm{CAB}_1$ if and only if it is locally strongly regular.
Note that if $\G$ satisfies
$\mathrm{CAB}_2$ then the $\mu$-graph $\Upsilon(x,y)$ for vertices $x, y\in V$ at distance $2$ is
regular.
Juri\v{s}i\'{c} and Koolen \cite{JuKo00DCC} showed that if $\G$ satisfies $\mathrm{CAB}_j$ then for all $i=0,1,\dots,j$ and $x,y\in V$ at
distance $i$, the quotient matrix of $\{C_i(x,y),
A_i(x,y), B_i(x,y)\}$ does not depend on the pair $x,y$, but only on $i$.
They also showed that if $a_1 \neq 0$ then $\G$ satisfies $\mathrm{CAB}_D$ if and only if it is $1$-homogeneous.
Note that if $a_1=0$ then
$\G$ always satisfies $\mathrm{CAB}_D$.
Nomura \cite{N294} showed that  the $1$-homogeneous distance-regular graphs of order $(s,t)$ with $s \geq 2$, $t \geq
1$ are exactly the regular near $2D$-gons, a result that can be shown easily using the $\mathrm{CAB}_D$
condition.
Juri\v{s}i\'{c} and Koolen \cite{JuKo00DCC} also determined the $1$-homogenous Terwilliger graphs,
and gave an algorithm to determine all $1$-homogeneous distance-regular graphs
that are locally a given strongly regular graph.
See also \cite{JK2003JAC}.

\subsubsection{More results on homogeneity}\label{sec:homogeneity}
Miklavi\v{c} \cite{Miklavic2004EJC} showed that the triangle-free $Q$-polynomial distance-regular graphs are $1$-homogeneous.
Note that if $a_1=0$ then the multiplicity of an eigenvalue distinct from $\pm k$ is at least $k$ by Terwilliger's tree bound; cf.~Section \ref{sec:tree_bound}.
Coolsaet, Juri\v{s}i\'{c}, and Koolen \cite{CoJuKo08EuJC} showed among other results that $\G$ is $1$-homogeneous if it has an eigenvalue with multiplicity $k$, $a_1=0$, $a_2\ne 0$, and $a_4=0$ (when $D\geq 4$), and then ruled out the infinite family of intersection arrays $\{2\mu^2 + \mu, 2\mu^2 + \mu -1, \mu^2, \mu, 1; 1, \mu, \mu^2, 2\mu^2 + \mu -1, 2\mu^2 + \mu\}$ ($\mu\geq 2$).
For $\mu=1$, this intersection array is uniquely realized by the dodecahedron.
Juri\v{s}i\'{c}, Koolen, and \v{Z}itnik \cite{JKZ2008EJC} showed among other results that if $\G$ is primitive and has an eigenvalue with multiplicity $k$, $a_1=0$, and $D=3$, then the association scheme underlying $\G$ is formally self-dual and thus $\G$ is $Q$-polynomial and $1$-homogeneous.

Nomura \cite{Nomura1995JCTB} classified the $2$-homogeneous bipartite distance-regular graphs; cf.~Section \ref{sec:antipodalDRG}.
Nomura \cite{Nomura1996P} also classified the $2$-homogeneous generalized odd graphs.
Yamazaki \cite{Yamazaki1996JCTB} observed that if $\G$ is bipartite then $\G$ has an eigenvalue with multiplicity $k$ if and only if it is $2$-homogeneous, while Curtin \cite{Curtin1998DM} showed that if $\G$ is bipartite then $\G$ is $2$-homogeneous if and only if it is $Q$-polynomial and antipodal.

\subsection{Thin modules}
\label{sec:thinness}

Thin irreducible $\TT$-modules with endpoint $1$ have been extensively
studied;
see e.g., \cite{Terwilliger1993N,GT2002EJC,Terwilliger2002LAA,Terwilliger2004JAC} and Section \ref{sec:thinness-Q}. For
example, let $\mathbf{v}$ be a nonzero vector in $E_1^{\ster}\mathbb{C}^v$ which is orthogonal to $A_1\mathbf{e}_x$, so
that $E_0\mathbf{v}=0$. Go and Terwilliger \cite{GT2002EJC} showed that if $E_i\mathbf{v}$ vanishes for some
$i=1,2,\dots,D$ then $i\in\{1,D\}$ and $\AL\mathbf{v}$ is a thin irreducible $\TT$-module with endpoint $1$ and
diameter $D-2$. There is also a characterization of thin irreducible $\TT$-modules with endpoint $1$, involving the
pseudo primitive idempotents introduced by Terwilliger and Weng \cite{TW04}. Let $\theta\in\mathbb{C}$ (not necessarily
an eigenvalue of $\G$).
The \emph{pseudo cosine sequence for} $\theta$ is the sequence $(\sigma_i)_{i=0}^D$ defined by
$\sigma_0=1$ and the recursion $c_i\sigma_{i-1}+a_i\sigma_i+b_i\sigma_{i+1}=\theta\sigma_i$ for $i=0,1,\dots,D-1$;
cf.~\eqref{standard_sequence}.
A \emph{pseudo primitive idempotent} $E_{\theta}$ \emph{associated with} $\theta$ is then any
nonzero scalar multiple of $\sum_{i=0}^D\sigma_iA_i$.
We also define $E_{\infty}$ to be any  nonzero scalar multiple of $A_D$. Let $\mathbf{v}$ be as
above, and let $(\AL;\mathbf{v})=\{M\in\AL:M\mathbf{v}\in E_D^{\ster}\mathbb{C}^v\}$.
Note that $J\in (\AL;\mathbf{v})$.
Terwilliger and Weng \cite{TW04} showed that $\TT\mathbf{v}$ is a thin irreducible
$\TT$-module (with endpoint $1$) if and only if $\dim(\AL;\mathbf{v})\geq 2$. Moreover, if this is the case, then
$\dim(\AL;\mathbf{v})=2$ and we have $(\AL;\mathbf{v})=\mathrm{span}_{\mathbb{C}}\{J,E_{\tilde{\eta}}\}$, where $\eta$
is the local eigenvalue corresponding to $\TT\mathbf{v}$ and
\begin{equation*}
	\tilde{\eta}=\begin{cases} \infty & \text{if}\ \eta=-1, \\ -1 & \text{if}\ \eta=\infty, \\ -1-\frac{b_1}{1+\eta} & \text{if}\ \eta\ne -1,\infty. \end{cases}
\end{equation*}
Terwilliger \cite{Terwilliger2004JAC} obtained an
inequality\footnote{\label{similar to SET}We may remark that the
discussions in \cite{Terwilliger2004JAC} and those in the proof of
the spectral excess theorem (Theorem \ref{spectral excess theorem})
given by Fiol and Garriga \cite{FG97} are similar in nature. In
\cite{Terwilliger2004JAC}, Terwilliger is concerned with the
thinness of irreducible $\TT$-modules with endpoint $1$ of a
distance-regular graph, whereas Fiol and Garriga \cite{FG97} take a
``local approach'', which can be understood as being concerned with
the thinness of the primary $\TT$-module of a general (finite,
simple, and connected) graph. (See \cite{Terwilliger1993N} for the
basic theory about the Terwilliger algebra of a general graph.) In
both cases, the characterization of the thinness as equality in a
bound is obtained by focusing on two specific vectors in
$E_D^{\ster}\mathbb{C}^v$.} involving the local eigenvalues of
$\G$, and showed that equality is attained if and only if $\G$ is
$1$-thin with respect to the base vertex $x$. Go and Terwilliger
\cite[Thm.~13.7]{GT2002EJC} showed that the following properties
are all equivalent: (i) $\G$ is tight; (ii) $\G$ is non-bipartite,
$a_D=0$, and $\G$ is $1$-thin; (iii) $\G$ is non-bipartite,
$a_D=0$, and $\G$ is $1$-thin with respect to at least one vertex.

We have some comments. It is well known that $a_1\ne 0$ implies $a_i\ne 0$ for
$i=1,2,\dots,D-1$; cf.~\cite[Prop.~5.5.1]{bcn}.
Dickie and Terwilliger \cite{DT1998JAC} showed that if $\G$ is $1$-thin with respect to at least one vertex
then $a_1=0$ implies $a_i=0$ for $i=1,2,\dots,D-1$. We note that these results
have dual versions for $Q$-polynomial association schemes;
cf.~\cite{Dickie1995D,DT1998JAC}.

Collins \cite{Collins1997GC} showed that $\G$ is thin with $c_3=1$ if and only if it is a generalized octagon of order
$(1,t)$. This shows that if $\G$ is thin then the numerical girth $g$ is at most $8$ (and cannot be $7$).
(Collins \cite{Collins1997GC} only mentioned the implication for the girth of $\G$.)
Suzuki \cite{Suzuki2006EJC} strengthened this result as follows.
Suppose $\G$ is of order $(s,t)$, and recall that $g$ coincides with the geometric girth in this case.
Suzuki showed among other
results that (i) $g\leq 11$ if there is a thin irreducible $\TT$-module with endpoint $3$; (ii) $\G$ is a regular near
polygon\footnote{The referee kindly pointed out an error in \cite[Thm.~1.2(ii)]{Suzuki2006EJC}.} if and only if it is $1$-thin; (iii) if $g\geq 8$ then $\G$ is a generalized $2D$-gon of order $(1,t)$ if and
only if it is $1$- and $2$-thin; (iv) if $g\geq 8$ then $\G$ is a generalized octagon of order $(1,t)$ if and only if
it is $1$-, $2$-, and $3$-thin.

Curtin \cite{Curtin1999GC} studied the Terwilliger algebras of bipartite distance-regular graphs.
Suppose for the moment that $\G$ is bipartite.
Then he showed among other results that $\G$ is always $1$-thin with a unique irreducible $\TT$-module with endpoint $1$ up to isomorphism, and that if $\G$ is $2$-thin with respect to the base vertex $x$ then the intersection array is determined by $D$ and the multiplicity in $\mathbb{C}^v$ of each of the irreducible $\TT$-modules $W$ with endpoint $2$, together with the scalar $\psi(W)=-\frac{b_2b_3}{c_2(\eta(W)+1)}-1$, where $\eta(W)$ is the eigenvalue of $E_2^*A_2E_2^*$ on $E_2^*W$, which is an eigenvalue of the local graph of $x$ in the halved graph of $\G$.
See also \cite{Curtin1999EJC}.
In particular, if $\G$ is $2$-thin with respect to $x$ with (at most) two irreducible $\TT$-modules $W_1,W_2$ with endpoint $2$ up to isomorphism, then it turns out that the intersection array is determined by $D,k,c_2,\psi(W_1)$, and $\psi(W_2)$.

Collins \cite{Collins2000DM} studied in detail the relation between the irreducible $\TT$-modules of an almost bipartite distance-regular graph $\G$ and those of its bipartite double $\tilde{\G}$.
In particular, he showed that $\G$ is thin if and only if $\tilde{\G}$ is thin.

\subsection{Vanishing Krein parameters}\label{sec:vanishingKrein}

Vanishing of Krein parameters often leads to strong (combinatorial) consequences.
A classical example is a result of Cameron, Goethals, and Seidel \cite{CGS1978JA} which states that if a strongly regular graph satisfies either $q_{11}^1=0$ or $q_{22}^2=0$ then for every vertex, the induced subgraphs on both of the subconstituents are strongly regular.\footnote{In passing, by the results of \cite{CGS1978JA} one can quickly find all the irreducible $\TT$-modules of a strongly regular graph.
In particular, it is always thin; cf.~\cite{TY1994KJM}.}
See \cite{Godsil1992AJC,JK2002DM,Jurisic2003DM} for similar results for antipodal distance-regular graphs with diameter $3$ or $4$.
In this section, we discuss more results on this topic.

\subsubsection{Triple intersection numbers}\label{sec:triple intersection numbers}
Coolsaet, Juri\v{s}i\'{c}, and others used
vanishing Krein parameters to obtain information on triple intersection
numbers as follows.
Let $x, y, z\in V$.
For $r, s, t =0,1,\dots, D$, let $p^{x,y,z}_{r, s, t} = |\{ u\in V : d(x,u) = r, d(y, u) =
s, d(z,u) = t\}|$.
Now, if $q^h_ {ij} =0$ then it follows from \eqref{3-tensor} that
\begin{equation}\label{triples}
	\sum_{r, s, t = 0}^D Q_{ri} Q_{sj} Q_{th} p^{x,y,z}_{r, s, t} = 0.
\end{equation}
This equation gives some extra information on the triple intersection numbers. We note that
\eqref{3-tensor} was also used earlier by Terwilliger \cite{Terwilliger1985AGG} to study the number of $4$-vertex
configurations with given mutual distances; he showed that if $\G$ is $Q$-polynomial then such numbers can be computed
from the intersection array and the numbers of $4$-vertex cliques in $\G_1,\G_2,\dots,\G_{\lfloor D/2\rfloor}$.
Using \eqref{triples}, Coolsaet and Juri\v{s}i\'{c} \cite{CoJu08} ruled out the infinite family of intersection arrays
$\{4r^3+8r^2+6r+1, 2r(r+1)(2r+1), 2r^2+2r+1; 1, 2r(r+1), (2r+1)(2r^2+2r+1)\}$ $(r \geq 2)$.
The case $r=1$, i.e., $\{19, 12, 5; 1, 4, 15\}$, was eliminated by Neumaier; cf.~\cite[\S 5.5A]{BCNcoradd}.
Coolsaet and Juri\v{s}i\'{c} also
ruled out the intersection array $\{ 74, 54, 15; 1, 9, 60\}$.
Juri\v{s}i\'{c} and Vidali \cite{JurVidpre}
used the above idea of triple intersection numbers to show that there exists a set of vertices mutually at distance
$3$ of size $p^3_{33}+2$  for distance-regular graphs with intersection arrays $\{(2r^2-1)(2r+1), 4r(r^2-1), 2r^2; 1,
2(r^2-1), r(4r^2-2)\}$ or $\{ 2r^2(2r+1), (2r-1)(2r^2+r+1), 2r^2; 1, 2r^2, r(4r^2-1)\}$ $( r \geq 2)$, and showed that
consequently such graphs do not exist. Urlep \cite{urlep12} used \eqref{triples} to rule out the intersection arrays
$\{(r + 1)(r^3 - 1), r(r - 1)(r^2 + r - 1), r^2 - 1; 1, r(r + 1), (r^2 - 1)(r^2 + r - 1)\}$ ($r \geq 3$). For $r=2$,
this intersection array is uniquely realized by the halved $7$-cube.
Vidali \cite{Vidali} recently used \eqref{triples} again to rule out the intersection array $\{55, 54, 50, 35, 10; 1, 5, 20, 45, 55\}$.

\subsubsection{Hadamard products of two primitive idempotents}
\label{sec: Hadamard products} Another important use of vanishing Krein parameters is the study of pairs of non-trivial
primitive idempotents $E,F$ such that $E\circ F$ is a linear combination of a small number of primitive idempotents;
cf.~\eqref{Krein}.
For convenience, we define
$e(M)=\{E_i:ME_i\ne 0\}$ for $M\in\AL$. Pascasio \cite{Pascasio1999JAC} showed that non-trivial primitive idempotents
$E,F$ satisfy $|e(E\circ F)|=1$ precisely when one of the following holds: (i) $\G$ is tight, $\{E,F\}=\{E_1,E_D\}$,
and $e(E\circ F)=\{E_{D-1}\}$; (ii) $\G$ is bipartite and $E_D\in\{E,F\}$.

Suppose for the moment that $\G$ is bipartite with $D\geq 4$.
Let $\theta,\theta'$ be eigenvalues of $\G$ other than $\pm k$, and let $E,F$ be the primitive idempotents associated with $\theta,\theta'$.
Then $|e(E\circ F)|>1$ by (ii) above.
MacLean \cite{MacLean2000DM} called the pair $\{E,F\}$ \emph{taut} if $|e(E\circ F)|=2$.
He showed that $|e(E\circ F)|=2$ if and only if $\theta,\theta'$ attain equality in what he called the `bipartite fundamental bound'.
We comment on the proof of this result. Let $E=E_i$ and $F=E_j$, and for $t=0,1,2$, let $\mathbf{f}_t$ be the vector in
$\mathbb{R}^{D+1}$ with $h$-coordinate
\begin{equation}\label{MacLean's vector}
	\theta_h^t\sqrt{\frac{q_{ij}^hm_h}{m_im_j}} \quad (h=0,1,\dots,D).
\end{equation}
Then $\mathbf{f}_0,\mathbf{f}_1,\mathbf{f}_2$ are linearly dependent if and only if $|e(E\circ F)|=2$, and computing the determinant of the (positive semidefinite) Gram matrix of $\mathbf{f}_0,\mathbf{f}_1,\mathbf{f}_2$ yields the bipartite fundamental bound.
See \cite{MT2006DM,MacLean2012DM} for more proofs of this bound.
We say $\G$ is \emph{taut} if it has a taut pair of primitive idempotents and is not $2$-homogeneous.
MacLean \cite[Thm.~1.4]{MacLean2000DM} showed that $\{E,F\}$ is taut precisely when one of the following holds:
(i) $\G$ is taut and $\{E,F\}\in\{\{E_h,E_{\ell}\}:h\in\{1,D-1\},\ell\in\{\tau,D-\tau\}\}$ where $\tau=\lfloor D/2\rfloor$; (ii) $\G$ is $2$-homogeneous and $\{E,F\}\cap\{E_1,E_{D-1}\}\ne\emptyset$.
For $D=4,5$, $\G$ is taut or $2$-homogeneous if and only if $\G$ is antipodal \cite[\S\S7--8]{MacLean2000DM}.
MacLean and Terwilliger \cite{MT2006DM} showed among other results that if $D$ is odd then the following are equivalent:
(i) $\G$ is taut or $2$-homogeneous; (ii) $\G$ is antipodal and $2$-thin; (iii) $\G$ is antipodal and $2$-thin with respect to at least one vertex; see also \cite{MT2008DM}.
Examples of taut graphs with odd $D\geq 5$ are the Doubled Odd graphs, the Doubled Hoffman-Singleton
graph, the Doubled Gewirtz graph, and the Doubled $77$-graph; cf.~ \cite[p.~131]{MacLean2003JAC}. For $D$ even and at least
$6$, MacLean \cite[Thm.~5.8]{MacLean2000DM} showed that $\G$ is taut or $2$-homogeneous if and only if its halved
graphs are tight. If $\G$ is taut in this case, then it turns out however that $D\ne 6$ and that no known example of a
tight distance-regular graph with diameter at least $4$ can be a halved graph of $\G$; cf.~\cite{MacLean2004JCTB}.

Retaining the situation of the last paragraph, let $\Delta_E$ be the representation diagram\footnote{The \emph{representation diagram} of $E=E_i$ is
the simple graph with vertex set $\{0,1,\dots,D\}$, where two distinct vertices $h,\ell$ are adjacent whenever
$q_{ih}^{\ell}\ne 0$.} of $E=E_i$, and let $(u_h)_{h=0}^D$ be the standard sequence associated with $E$. Note that $0$ and
$D$ are leaves (i.e., terminal vertices) in $\Delta_E$, and that $j$ is a leaf in $\Delta_E$ precisely when $|e(E\circ
F)|=2$ and $F\in e(E\circ F)$. Lang \cite{Lang2004JCTB} showed that
\begin{equation}\label{Lang's inequality}
	(u_1-u_{h+1})(u_1-u_{h-1})\geq (u_2-u_h)(u_0-u_h) \quad (h=1,2,\dots,D-1),
\end{equation}
with equality for every $h=1,2,\dots,D-1$ (or just for $h=3$) if and only if $u_{h-1}-pu_h+u_{h+1}$ is independent of
$h=1,2,\dots,D-1$ for some $p\in\mathbb{R}$. When $E$ attains equality, Lang \cite{Lang2002EJC} showed that (i) $u_D\ne
1$ if and only if $\Delta_E$ is a path (i.e., $E$ is $Q$-polynomial); (ii) $u_D=1$ if and only if $\Delta_E$ is the
disjoint union of two paths. It follows that if case (ii) occurs then $\G$ is antipodal and the folded graph is
$Q$-polynomial; cf.~\cite[Thm.~10.2, 10.4]{Lang2004JCTB}. Note that in both cases (i) and (ii), $E$ is a \emph{tail},
i.e., $|e(E\circ E)|\leq 3$ and $|e(E\circ E)\backslash\{E_0,E\}|\leq 1$. Conversely, Lang \cite{Lang2002EJC} showed
that if $E$ is a tail and $D\ne 6$ then $E$ attains equality in \eqref{Lang's inequality}.
Lang \cite{Lang2003JAC} also showed that $\Delta_E$ has a leaf other than $0,D$ if and only if $E$ attains
equality in \eqref{Lang's inequality} and case (ii) occurs above. One of the other results in \cite{Lang2004JCTB} is
that if $D\geq 6$ and $\G$ has more than one primitive idempotent that attains equality in \eqref{Lang's inequality},
then $\G$ is the $D$-cube.

Suppose now that $\G$ is arbitrary (i.e., $D\geq 3$ and not necessarily bipartite).
By considering the Gram matrix of $\mathbf{f}_0$ and $\mathbf{f}_1$, Pascasio \cite{Pascasio2003DM} later extended some of the results in \cite{Pascasio1999JAC}, as well as the
fundamental bound, to the level of $P$-polynomial character algebras.
Tomiyama \cite{Tomiyama2001DM} considered the situation where one of $1,D$ is a leaf in $\Delta_E$ and generalized some of the results in \cite{Lang2003JAC,Lang2004JCTB,Pascasio1999JAC,Pascasio2001GC}.

Assume $E=F$ (so $i=j$) and $\theta(=\theta')\neq\pm k$.
Then $|e(E\circ E)|\geq 2$.
We call $E$ a \emph{light tail}\footnote{It is easy to see that if $e ( E \circ E ) = \{ E_0, E \}$ then $\G$ is antipodal with $D=3$; cf.~\cite[Thm.~ 4.1(b)]{JTZ2010EJC}. We view this case as degenerate,
so we propose to assume $E \not\in e ( E \circ E )$ as well in the definition of a light tail.} \cite{JTZ2010EJC} if $|e(E\circ E)|=2$.
Let $\mathbf{f}_0',\mathbf{f}_1'$ be the vectors obtained from $\mathbf{f}_0,\mathbf{f}_1$, respectively, by the removal of
the $0$-coordinate. Note that $\mathbf{f}_0',\mathbf{f}_1'$ are linearly dependent if and only if $E$ is a light tail.
Juri\v{s}i\'{c}, Terwilliger, and \v{Z}itnik \cite{JTZ2010EJC} considered the Gram matrix of
$\mathbf{f}_0',\mathbf{f}_1'$. In this case, the resulting inequality gives a lower bound on the multiplicity of
$\theta$; cf.~Proposition \ref{light tail bound}. They showed among other results that distance-regular graphs with a
light tail are close to being $1$-homogeneous, i.e., the parameters $p_{i,j;r,s}$ exist with respect to, and are
independent of, every pair of adjacent vertices $x,y\in V$ for all $i,j,r,s=0,1,\dots,D$ except possibly $i=j=2,3,\dots,D-1$. In
particular, the local graphs are strongly regular. We note that these results generalize those of Cameron, Goethals,
and Seidel \cite{CGS1978JA} mentioned at the beginning of Section \ref{sec:vanishingKrein}. They indeed showed that
primitive strongly regular graphs with a light tail (and $k\geq 3$) are precisely the Smith graphs.

\subsection{Relaxations of homogeneity}

In the previous sections, we explored connections among homogeneity, thin modules, tightness,
local graphs, Hadamard products of two primitive idempotents, and so on. In fact, many of these results can be
generalized in several directions, as we discuss below.

We say $\G$ is \emph{pseudo $1$-homogeneous with respect to an edge} $xy$ \cite{JT2008JAC} if the parameters $p_{i,j;r,s}$ exist with respect to $x,y$ for all $i,j,r,s=0,1,\dots,D$ except possibly $i=j=D$.
Let $\theta\in\mathbb{R}\backslash \{k\}$, and let $E_{\theta}$ be a pseudo primitive idempotent associated with $\theta$; cf.~Section \ref{sec:thinness}.
We say the edge $xy$ is \emph{tight with respect to} $\theta$ \cite{JT2008JAC} if a non-trivial linear combination of
$E_{\theta}\mathbf{e}_x,E_{\theta}\mathbf{e}_y,E_{\theta}\chi_{1,1}(x,y)$ is contained in the subspace
$\mathrm{span}_{\mathbb{R}}\{\mathbf{e}_z:z\in \G_{D,D}(x,y)\}$. Juri\v{s}i\'{c} and Terwilliger \cite{JT2008JAC}
showed among other results that if $a_1\ne 0$ then the edge $xy$ is tight with respect to two distinct real numbers if
and only if $\G$ is pseudo $1$-homogeneous with respect to $xy$ and the induced subgraph on $\G_{1,1}(x,y)$ is not a
clique. Under the condition $a_1\ne 0$, Curtin and Nomura \cite{CN2005JCTB} characterized the
situation where $\G$ is $1$-thin with respect to $x$ with precisely two non-isomorphic irreducible $\TT(x)$-modules with
endpoint one, in terms of the pseudo $1$-homogeneous property\footnote{It should be
remarked that Curtin and Nomura \cite{CN2005JCTB} do not require the existence of the parameter $p_{D,D-1;r,s}$ with
respect to $x,y$.} of the edges $xy$ ($y\in\G(x)$).
They studied in detail the case where $a_1=0$ as well.
Extending the work of Pascasio \cite{Pascasio1999JAC} on the tightness property, Pascasio and Terwilliger \cite{PT2006LAA} described exactly when $E_{\theta}\circ E_{\theta'}$ with $\theta,\theta'\in\mathbb{R}$ is a scalar multiple of $E_{\tau}$ for some $\tau\in\mathbb{R}$.

Suppose for the moment that $\G$ is bipartite with $D\geq 4$, and let $x,y$ be vertices with $d(x,y)=2$.
Curtin \cite[\S\S 4--5]{Curtin1998DM} showed that $\G$ is $2$-homogeneous if and only if $|\G_{1,1}(x,y)\cap\G_{i-1}(z)|$ depends only on $i=1,2,\dots,D-1$ and is independent of $z\in \G_{i,i}(x,y)$.
We say $\G$ is \emph{almost $2$-homogeneous} \cite{Curtin2000EJC} if the same condition holds for $i=1,2,\dots,D-2$.
Recall that $\G$ is $1$-thin with a unique irreducible $\TT$-module with endpoint $1$ up to
isomorphism.
Curtin \cite{Curtin2000EJC} showed among other results that $\G$ is almost $2$-homogeneous if and only if it
is $2$-thin with a unique irreducible $\TT$-module with endpoint $2$ up to isomorphism.
Curtin \cite{Curtin2000EJC}
and Juri\v{s}i\'{c}, Koolen, and Miklavi\v{c} \cite{JKM2005JCTB} classified the almost $2$-homogeneous bipartite
distance-regular graphs: the $2$-homogeneous graphs (cf.~Section \ref{sec:antipodalDRG}), the generalized
$2D$-gons of order $(1,k-1)$, the folded $2D$-cube, and the coset graph of the extended binary Golay
code.\footnote{See also \cite{Suzuki2008GC} for a generalization of this result (as well as Nomura's
classification \cite{Nomura1995JCTB,Nomura1996P} of bipartite or almost bipartite $2$-homogeneous distance-regular
graphs) to triangle-free distance-regular graphs. That the coset graph of the extended binary Golay code is almost
$2$-homogeneous was pointed out by Lang \cite[Lemma~3.4]{Lang2008EJC}.}
Lang \cite{Lang2008EJC} considered when
$E_{\theta}\circ E_{\theta}$ with $\theta\in\mathbb{C}\backslash\{k,-k\}$ is a linear combination of $J$ and $E_{\tau}$ for some $\tau\in\mathbb{C}$, and showed that this occurs precisely when $\G$ is almost $2$-homogeneous and $c_2\geq 2$.

In some cases, it is possible to get an equitable partition of
$V$ from $\Pi=\{\G_{i,j}(x,y):\G_{i,j}(x,y)\ne\emptyset,\,
i,j=0,1,\dots,D\}$, where $d(x,y)=h\in\{1,2\}$, by refining some of the
$\G_{i,i}(x,y)$ $(i=2,3,\dots,D)$ into two cells, even when $\Pi$ itself is
not equitable. This was worked out in detail by Miklavi\v{c} for
distance-regular graphs having classical parameters with $b<-1$, $a_1\ne 0$
\cite{Miklavic2005JCTB} ($h=1$), for bipartite $Q$-polynomial distance-regular
graphs with $c_2=1$ \cite{Miklavic2007DM} ($h=2$), and for the bipartite dual
polar graphs $\D_D(q)$ \cite{Miklavic2013GC} ($h=2$). See also
\cite{Miklavic2007EJC} for a description of the $\AL$-module spanned by
$\{\chi_{i,j}(x,y):i,j=0,1,\dots,D\}$ with $h=2$ (cf.~Section
\ref{sec:Qpolcharacterizations}) for bipartite $Q$-polynomial distance-regular
graphs. The parameters of the new equitable partition give rise to additional
integrality conditions, and he used these conditions to show that there is no
bipartite $Q$-polynomial distance-regular graph with $D=4$ and girth $6$;
cf.~Section \ref{sec:girth}.

%% file: 7_diameterbounds.txt

In this section we will look at the growth of intersection numbers and its
consequences for bounds on the diameter.

\subsection{The Ivanov bound}
Ivanov \cite{Iv83} obtained the first general diameter bound for
distance-regular graphs.

\begin{theorem} \label{ivanovbound}{\em (The Ivanov bound)}\label{Ivanov}
Let $\G$ be a distance-regular graph with diameter $D \geq 2$, head $h$, and
valency $k$.
Let $2 \leq i \leq i+j \leq D-1$. If $(c_{i-1}, a_{i-1}, b_{i-1}) \neq (c_i, a_i, b_i) = (c_{i+j}, a_{i+j} , b_{i+j})$, then $j \leq i$.
In particular, $D < 2^{k-1}(h +1).$
\end{theorem}

\noindent Suzuki \cite[p.~67]{Su99} gave a proof of this bound using so-called intersection diagrams. Bang, Hiraki, and
Koolen \cite{BHK06, HK98} improved the Ivanov bound, as we shall discuss below. One of the tools that they used is the
following result of Koolen \cite{Ko192}, \cite[Prop.~2.3]{Koothesis}.

\begin{prop}\label{kooprop}
Let $\G$ be a distance-regular graph with diameter $D$.
\begin{enumerate}[{\em (i)}]
\item If $c_i > c_{i-1}$ for some $i =2, \ldots, D$, then $c_{i-j} +c_j \leq c_i$ for all $j=1,\ldots,i-1$,
\item If $b_i > b_{i+1}$ for some $i =0, \ldots, D-2$, then $b_i \geq b_{i+j} +
c_j$ for all $j=1,\ldots, D-i$.
\end{enumerate}
\end{prop}

\noindent Wajima \cite{Waji94} also obtained Proposition \ref{kooprop} (but with a completely different method), and
Hiraki \cite{Hi07} obtained slight improvements of this result. Another tool by Bang et al.~\cite{BHK06} is the
following.

\begin{prop} \label{BHKprop}
Let $\G$ be a distance-regular graph with valency $k$ and diameter $D$. For
$1 \leq c \leq k$, define $\xi_c = \min\{ i : c_i \geq c\}$ and $\eta_c = |\{ i
: c_i = c\}|$. Then $\eta_c \leq \xi_c -1$.
\end{prop}

\noindent Using a combination of  Propositions \ref{kooprop} and \ref{BHKprop}, Bang et al.~\cite{BHK06} improved the
Ivanov bound as follows (using the notation as introduced in Proposition \ref{BHKprop}):

\begin{prop} Let $\G$ be a distance-regular graph with valency $k$ and diameter $D$.
Then $$D < \frac{1}{2}k^{\alpha}\eta_1 +1,$$ where $\alpha = \inf\{ x >0 :
4^{\frac{1}{x}} - 2^{\frac{1}{x}} \leq 1\} \approx 1.441$.
\end{prop}

\noindent Hiraki \cite{Hi01} showed, using earlier work from \cite{CH99, Hi94, Hi98,
HK02}, that if $h \geq 2$, then $c_{2h+3} \geq 2$ or, in other words, $\eta_1
\leq 2h+2$. This  immediately implies that if $h \geq 2$, then $$D \leq
k^{\alpha}(h+1) +1.$$ For $h =1$, it is conjectured by Hiraki \cite{Hi95} that
there exists a constant $C$ such that  $\eta_1 \leq C$. Chen, Hiraki, and Koolen \cite{CHK98}
showed that if $a_1 \neq 2$ and $a_1 \leq 100$, then $c_4 \geq 2$.

In the next sections we present better diameter bounds for certain subclasses
of distance-regular graphs.

\subsection{Distance-regular graphs of order \texorpdfstring{$(s,t)$}{(s,t)}}

The following result was first shown by Terwilliger
\cite{Terwscak} for distance-regular graphs with
$a_1 = 0$ or $c_2 \geq 2$. Later it was generalized by Faradjev, Ivanov, and Ivanov \cite{FII} to
distance-regular graphs with $a_1> 0$. We present their bound for the case that $\G$ is locally a
disjoint union of cliques and $c_{h+1} \geq 2$ holds. In the next section we will
also present the bound of Terwilliger for the case $c_2 \geq 2$.

\begin{prop}{\em (cf.~\cite[Thm.~1.4.3,~Cor.~1.4.4]{Su99})}
Let $\G$ be a distance-regular graph of order $(s,t)$ with
head $h$, valency $k$, and diameter $D \geq 2$. If $c_{h+1} > 1$, then $b_i > b_{i+h}$
and $c_i < c_{i+h}$  for all $i=0,1, \ldots, D-h$, and in particular, $D \leq
t h +1$.
\end{prop}

\noindent For the bipartite case, Koolen \cite{Ko192} and Hiraki
\cite{Hi07} made some improvements. Hiraki \cite{Hi07} showed that
if $\G$ is a bipartite distance-regular graph with head $h \geq
2$ and diameter $D$, then $\G$ is a Doubled Odd graph or $D \leq
\lfloor \frac{k+2}{2} \rfloor h$; see also Section
\ref{sec:otherinfinite}. For the weakly geometric case, Suzuki
obtained the following.

\begin{prop}{\em(cf.~\cite[Prop.~3.1.6]{Su99})}
Let $\G$ be a weakly geometric distance-regular graph of order $(s,t)$ with
head $h$ and diameter $D$. Then $b_i > b_{i+h+1}$ and $c_i < c_{i+h+1}$
for all $i=0,1, \ldots, D-h-1$. In particular, $D \leq t (h+1) +1$.
\end{prop}

\begin{corollary}\label{s>t implies geometric}
Let $\G$ be a distance-regular graph of order $(s,t)$
with head $h$ and diameter $D$. If $s > t$, then $\G$ is
geometric and hence $b_i > b_{i+h+1}$ and $c_i < c_{i+h+1}$ for all
$i=0,1, \ldots, D-h-1$ and in particular,  $D \leq t (h+1) +1$.
\end{corollary}

\begin{corollary} \label{cordiameter} For all integer $t \geq 1$ there exists
a constant $C_t$ such that for all distance-regular graphs $\G$
of order $(s,t)$, the diameter of $\G$ is bounded by $C_t
h(\G)$, where $h(\G)$ is the head of $\G$.
\end{corollary}

\subsection{A bound for distance-regular graphs with \texorpdfstring{$c_2 \geq2$}{c2 ge 2}}\label{sec:boundc2}

Terwilliger \cite{Terwscak} obtained the following bound for
distance-regular graphs with $c_2 \geq 2$.

\begin{prop}{\em (cf.~\cite[Thm.~5.2.5,~Prop.~1.9.1]{bcn})}
Let $\G$ be a distance-regular with diameter $D \geq 2$ and $c_2 \geq
2$. If $c_2 \geq 2(a_1 +1)$, then  $c_i -b_i \geq c_{i-2} - b_{i-2} + 2 \ \ (i
=2, 3, \ldots, D)$, and in particular $D \leq k$. Moreover, if $\max\{a_1, 2\}
\leq c_2$, then $c_i \geq c_{i-1} + 1$ $(i=2,3, \ldots, D)$.
\end{prop}

\noindent Caughman \cite{Cau97} improved this result for bipartite
distance-regular graphs as follows.

\begin{prop}
Let $\G$ be a bipartite distance-regular graph with valency $k$, diameter
$D \geq 3$, and $c_2 \geq 2$. Let $i =1,2,\ldots, D-1$. If $k  > c_i
((c_2-1)(c_2-2)(c_i-c_{i-1}-1)/2 +1)$, then $c_{i+1} \geq c_i(c_2 -1) + 1$.
\end{prop}

\noindent Moreover, Terwilliger \cite{Ter-quad} obtained the following diameter bound.

\begin{prop}\label{thm:terwilintersectionnos} {\em (cf.~\cite[Thm.~5.2.1,~Cor.~5.2.2]{bcn})}
Let $\G$ be a distance-regular graph with diameter $D$. If $\G$
contains an induced quadrangle, then $ c_i - b_i \geq c_{i-1} - b_{i-1} + a_1
+2$ and, in particular, $D \leq \frac{k+c_D}{a_1 +2}$.
\end{prop}

\noindent The distance-regular graphs with diameter $\frac{k+c_D}{a_1 +2}$ and containing a quadrangle have second
largest eigenvalue $b_1-1$ and have been classified: besides the strongly regular graphs with smallest eigenvalue $-2$,
these are the Hamming graphs, Doob graphs, halved cubes, Johnson graphs, locally Petersen graphs, and the Gosset graph,
see \cite[Thm.~5.2.3]{bcn}; also cf.~Section \ref{sec: the case b=1}. Note that if a distance-regular graph contains a
quadrangle then the second largest eigenvalue is at most $b_1-1$.

Neumaier \cite{N04} showed among other results that if there are infinitely
many distance-regular graphs with fixed $a_1, c_2, a_i, c_i$ containing an
induced quadrangle then necessarily $c_{i+1} \geq  1+(c_2-1)c_i$. For dual
polar graphs, equality holds.

\subsection{The Pyber Bound}\label{sec:Pyber}

Using a slightly weaker result than Proposition \ref{kooprop}, Pyber \cite{Py99} showed that $ D \leq 5 \log_2v$ for a
distance-regular graph with $v$ vertices and diameter $D$. This essentially settles a problem in `BCN'
\cite[p.~189]{bcn}. Pyber's bound was improved by Bang, Hiraki, and Koolen \cite{BHK06} to $D < \frac{8}{3} \log_2v$.

%% file: 8_BIconjecture.txt

In 1984, Bannai and Ito \cite[p.~237]{bi} made the following conjecture.

\begin{biconj} There are finitely many distance-regular graphs with fixed
valency at least three.
\end{biconj}

This Bannai-Ito conjecture has recently been proved by Bang, Dubickas, Koolen,
and Moulton \cite{BaDuKoMo09}. In the next section, we will give an outline of
this proof.

\subsection{Proof of the Bannai-Ito conjecture}\label{sec:proofBIconjecture}

Let $\G$ be a distance-regular graph with valency $k \geq 3$, head $h$, and
diameter $D$. The Ivanov bound (Theorem \ref{ivanovbound}) tells us that $D
\leq 4^k h$. So in order to prove the conjecture it suffices to bound $h$ as a
function of $k$.

Bannai and Ito \cite{BI89} obtained the following result, using head and tail.
Recall that the latter is defined by $t = \ell(b_1, a_1, c_1)$.

\begin{theorem}\label{banito}
Let $M \geq 1$ and $k \geq 3$. Then there are finitely many triangle-free
distance-regular graphs with valency $k$, and diameter $D \leq h + t + M$.
\end{theorem}

\noindent The key idea of the proof of this theorem is as follows. By
interlacing, $\G$ has an eigenvalue $\theta$ in the interval $(2\sqrt{k-1} \cos
\frac{3\pi}{h+1}, 2\sqrt{k-1} \cos \frac{\pi}{h+1})$. Let $\Theta$ be the set
of algebraic conjugates of $\theta$. Then
$$\prod_{\theta' \in \Theta} ((\theta')^{ 2} - k+1)$$
is a nonzero integer. Let $S = \{ x \in [-k, k] : |x^2 - k+1| > 1\}$. If $h$ is
large enough, then $\theta\not\in S$ and there is an algebraic conjugate
$\theta'$ of $\theta$ which is in $S$.  Now it can be shown, using Biggs'
formula (Theorem \ref{Biggsformula}), that the multiplicity of $\theta$ is of
order $\frac{v}{h^3}$. If $\theta' \in (-2\sqrt{k-1}, +2\sqrt{k-1})$, then the
multiplicity of $\theta'$ is $\Omega(\frac{v}{h})$, and else the multiplicity
of $\theta'$ is $O(\frac{v}{a^h})$ for some fixed real $a > 1$. This shows that
the multiplicities of $\theta$ and $\theta'$ are not the same if $h$ large,
which is a contradiction to the fact that they are algebraic conjugates.
Therefore $h$ is bounded.

Suzuki \cite{S594} generalized this result by replacing the triangle-free
condition by the condition $(a_1+1)(a_1+2) \leq k$. Bang, Koolen, and Moulton
\cite{BaKoMo07} extended the result as follows.

\begin{prop} \label{bakomo-epsilon}
Let $k \geq 3$. Then there exists a positive $\epsilon = \epsilon_k$ such that
there are finitely many distance-regular graphs $\G$ with valency $k$, and
diameter $D \leq h + t + \epsilon h$.
\end{prop}

\noindent The proof of Proposition \ref{bakomo-epsilon} closely follows the
proof of Theorem \ref{banito}. Instead of considering an eigenvalue in the
above mentioned interval close to $2\sqrt{k-1}$, so-called indicator intervals
are used.

Let $G = \{ (c_i, a_i, b_i) : i=1,2,\ldots, D-1\}$ and $g = |G|$; note that $g
\leq 2k-3$. We will assume that $G=\{ (\gamma_i, \alpha_i, \beta_i) :
i=1,2,\ldots, g\}$ is ordered by $\gamma_{i+1} \geq \gamma_i$ and $\beta_{i+1}
\leq \beta_i$. Let $\ell_i = \ell(\gamma_i, \alpha_i, \beta_i)$ for $i=1,2,
\ldots, g$, whence $h = \ell_1$. Let ${\cal L}_i = \alpha_i - 2\sqrt{\gamma_i
\beta_i}$ and  ${\cal R}_i = \alpha_i + 2\sqrt{\gamma_i \beta_i}$  be the {\em
left} and {\em right indicator points}, respectively. The {\em indicator
interval}  is defined as the open interval ${\cal I}_i=({\cal L}_i, {\cal
R}_i)$, $i=1,2, \ldots, g$. Using the fact that the $c_i$s are non-decreasing
and the $b_i$s are non-increasing, it is fairly easy to see that $({\cal
R}_i)_i$ is a unimodal sequence.  The fact that $b_i + a_i \geq a_1 +2$
$(i=1,2,\ldots, D-1)$ implies that ${\cal R}_g \geq {\cal R}_1$.

By removing from the (tridiagonal) intersection matrix $L$ rows and columns $0,
\ell_1, \ell_1 +\ell_2, \ldots, \ell_1 + \cdots +\ell_{g-1}, D$ and using
interlacing, it follows that at most $2g+2$ eigenvalues of $\G$ (in general
this is a relatively small number compared to the total) do not lie in any of
the indicator intervals. Instead of taking $\theta$ close to ${\cal R}_1$, one
can show that there must exist an eigenvalue $\theta$ close to a right
indicator point different from ${\cal R}_1$ if $D-h-t$ is large enough. Then it
is shown in a similar way as in Theorem \ref{banito} that there exists an
algebraic conjugate of $\theta$ whose multiplicity is different from the
multiplicity of $\theta$; again a contradiction, so $h$ is bounded.

Until now the approach was to calculate the multiplicity of a specific
eigenvalue precisely and then show that this eigenvalue has an algebraic
conjugate with a different multiplicity. For a proof of the Bannai-Ito
conjecture one needs to use another tactic, especially in the case that $D-h-t$
is large. The idea here is to find an interval ${\cal I}$ in which there are at
least $\delta h$ eigenvalues (where $\delta$ is a positive real number only
depending on $k$) and in which every two algebraic conjugate eigenvalues
$\theta$ and $\theta'$ satisfy $|\theta- \theta'| \leq f(h)$, where $f(h)
\rightarrow 0 \ \ (h \rightarrow \infty)$. The main reason that one can find
such an interval ${\cal I}$ is that the right indicator points form a unimodal
sequence. Although to calculate the multiplicities of the eigenvalues only
involves three-term recurrence relations, to show that ${\cal I}$ really exists
and that we can approximate the multiplicities in ${\cal I}$ well enough is
extremely technical and subtle. Using some elementary number theory, it then
follows that the number of algebraic conjugates of eigenvalues in ${\cal I}$
(which must all be eigenvalues of $\G$) is at least $z(h) h $, where $z(h)
\rightarrow \infty \ \ (h \rightarrow \infty)$. But as the the number of
eigenvalues besides the valency is exactly $D$, which --- by the Ivanov bound
--- is at most $4^k h$, we see that this is a contradiction if $h$ is large.
Again, this means that $h$ is bounded, which proves the Bannai-Ito conjecture.

\subsection{Extensions of the Bannai-Ito conjecture}
Bannai and Ito \cite{BI87} showed that the length $\ell(c, k-2c, c)$ is bounded by $10k 2^k$ for every $c$. This
inspired Hiraki, Suzuki, and others to obtain bounds for $\ell(1, k-2, 1)$. The current best bound is by Hiraki
\cite{Hi103}, who obtained $\ell(1, k-2, 1) \leq 14$ if $k \geq 3$ and $\ell(1, k-2, 1) \leq 1$ if $k \geq 58$.
Inspired by this, Bang, Koolen, and Moulton \cite{BKM03} showed that if $b$ and $c$ are positive integers, then there
exists a constant $k_{\min} \geq \max \{b+c, 3\}$ such that if $\G$ is a distance-regular graph with valency $k
\geq k_{\min}$ and $h \geq 2$, then $\ell(c, k-b-c,b)\leq 1$. This implies, by using the validity of the Bannai-Ito
conjecture, that if $b$ and $c$ are positive integers, then there exists a constant $\ell_{\max}$ such that for every
distance-regular graph $\G$ with valency $k \geq \max \{b+c, 3\}$ and $h \geq 2$, we have that $\ell(c,
k-b-c,b)\leq \ell_{\max}$. It is still an open problem whether this is true for $h =1$ and $c_2 = 1$. This has been
conjectured by Bang et al.~\cite{BKM03}. Park, Koolen, and Markowsky \cite{PaKoMar} extended the Bannai-Ito conjecture
as follows.

\begin{prop}
Let $M$ be a positive integer. Then there are finitely many distance-regular
graphs with valency $k \geq 3$, diameter $D \geq 6$, and $\frac{k_2}{k} \leq
M$.
\end{prop}

\noindent For diameter at most four, the analogous result is not true. For
diameter two this is clear. For diameter three, the Taylor graphs have $k_2 =k$
and the incidence graphs of the complements of projective planes of order $t$
have $k = t^2$ and $k_2 = t^2 +t$. For diameter four, the Hadamard graphs  have
$k_2 = 2(k-1)$.

Koolen and Park  \cite{KoPa11} showed that the only primitive distance-regular
graphs with $\frac{k_2}{k} \leq 1.5$, and diameter at least three are  the
Johnson graph $J(7,3)$ and the halved $7$-cube.

\subsection{The distance-regular graphs with small valency}
The edge is the only distance-regular graph with valency one, and the polygons are the distance-regular graphs with
valency two. The distance-regular graphs with valency three have been classified by Biggs, Boshier, and Shawe-Taylor
\cite{BBS} (see also \cite[Thm.~7.5.1]{bcn}): There are exactly 13 of them and all have diameter at most 8.

The intersection arrays of the distance-regular graphs with valency four have been
classified by Brouwer and Koolen \cite{BK99}. There are exactly 17 such
intersection arrays and all have diameter at most 7 (all graphs are known, except perhaps
for point-line incidence graphs of a generalized hexagon of order three).

The distance-regular graphs with valency 6 and $a_1 = 1$ (i.e., of order
$(2,2)$) have been classified by Hiraki, Nomura, and Suzuki \cite{HNS}. There are
exactly five of them and they are all geometric. This last result also
completes the classification of all  distance-regular graphs with valency at
most 7 and $a_1 \geq 1$ (see \cite{HNS} for a complete list).

The larger $t$ is, the more difficult it is to classify the distance-regular
graphs of order $(s,t)$. For example, it is much harder to classify the
distance-regular graphs with valency $5$ than the distance-regular with valency
$6$ and $a_1 = 1$. For $t=1$ we have the line graphs, and Yamazaki \cite{Y95}
developed some theory for the case $t=2$. It is not known whether for a
distance-regular graph with order $(s,t)$, one can bound the diameter in terms
of $t$ only, if $t \geq 2$ (see also Corollary \ref{cordiameter}).

%% file: 9_geometric.txt

\subsection{Metsch's characterizations}\label{sec:Metsch}

As mentioned before, Metsch characterized most of the Grassmann graphs and
bilinear forms graphs by their intersection arrays. From these intersection
arrays, he recovers the geometric properties of these graphs. An important
ingredient for this is the following proposition, which is used to construct
lines --- large cliques --- that partition the edge set.

\begin{prop}\label{Metschresult-2.2}{\em \cite[Result~2.2]{Me95}}
Let $\mu \geq 1$, $\lambda_1,~\lambda_2$, and $m$ be integers. Assume that
$\G$ is a connected graph with the following properties:
\begin{enumerate}[{\em (i)}]
\item Every two adjacent vertices have at least $\lambda_1$ and at most
    $\lambda_2$ common neighbors,
\item Every two nonadjacent vertices have at most $\mu$ common neighbors,
\item $2\lambda_1-\lambda_2>(2m-1)(\mu-1)-1$,
\item Every vertex has fewer than $(m+1)(\lambda_1+1)-\frac{1}{2}m(m+1)(\mu-1)$ neighbors.
\end{enumerate}
Define a line to be a maximal clique $C$ satisfying $|C|\geq
\lambda_1+2-(m-1)(\mu-1)$. Then every vertex is on at most $m$ lines, and every
two adjacent vertices lie in a unique line.
\end{prop}

\noindent In \cite{Me95}, Metsch used Proposition \ref{Metschresult-2.2} and a characterization of projective incidence
structures by Ray-Chaudhuri and Sprague \cite{RCSprague} (see also \cite[Thm.~9.3.9]{bcn}, and \cite{Cuypers92} for a
generalization by Cuypers) to characterize the Grassmann graphs. The interesting thing is that he hardly required any
of the regularity conditions that follow from the intersection array. The only conditions that Metsch used were the
intersection number $c_2$ and upper and lower bounds on the number $b_0(x)$ of neighbors of a vertex $x$, and on the
number of common neighbors $a_1(x,y)$ of two adjacent vertices $x$ and $y$, and a lower bound on the number $b_2(x,y)$
of vertices $x$ and $y$ at distance two. In weaker form, the characterization is as follows.

\begin{prop} {\em \cite[Thm.~1.1]{Me95}}
Let $q \ge 2$ be an integer, and let $D$ and $n$ be integers satisfying $2D \le
n$. Let $s+1=(q^{n-D+1}-1)/(q-1)$ and $m=(q^{D}-1)/(q-1)$. Let $\G$ be a
connected $ms$-regular graph with the property that every two adjacent vertices
have $a_1=s-1+(m-1)q$ common neighbors and every two vertices at distance two
have $c_2=(q+1)^2$ common neighbors, and such that every two vertices $x$ and
$y$ at distance two have $b_2(x,y)>(m-q-1)(s-q^2-q)$. If $D \neq 2,
\frac{n}{2}$, or $\frac{n-1}{2}$ (for all $q$) and $(D,q) \neq
(\frac{n-2}{2},2), (\frac{n-2}{2},3)$, or $(\frac{n-3}{2},2)$, then $q$ is a
prime power and $\G$ is the Grassmann graph $J_q(n,D)$.
\end{prop}

\noindent Building on work by Huang \cite{Huang87} and a characterization of
attenuated spaces by Sprague \cite{Sprague}, Metsch \cite{Me99} also used
Proposition \ref{Metschresult-2.2} to characterize the bilinear forms graphs.

\begin{prop} Let $\G$ be a distance-regular graph with
classical parameters $(D,q,\alpha,\beta)$, where $\alpha=q-1$ and $D \ge 3$.
Suppose that either $q=2$ and $\beta \ge q^{D+4}-1$ or $q \ge 3$ and $\beta \ge
q^{D+3}-1$. Then $q$ is a prime power, $\beta=q^e-1$ for some integer $e$, and
$\G$ is the bilinear forms graph $Bil(D \times e,q)$.
\end{prop}

\noindent Proposition \ref{Metschresult-2.2} can be used further in
characterizing other geometric distance-regular graphs. We will get back to
this in Section \ref{sec:geometric}.

\subsection{Characterization of Doubled Odd and Doubled Grassmann
graphs}\label{sec:otherinfinite} In Section \ref{sec:generalizations+geometric}
we mentioned the distance-biregular graphs that arise as incidence graphs
between the vertices and the cliques coming from the $(D+1)$-subspaces in the
Grassmann graph $J_q(n,D)$, and the similar one (with $(D+1)$-subsets) from the
Johnson graph $J(n,D)$. Cuypers \cite{Cuypers92} classified the
distance-biregular graphs with diameter at least 5 and $c_2^R=1<c_3^R=c_4^R$
($R$ being one of the color classes): the only ones are the above mentioned
graphs and the Doubled Moore graphs. This implies that the Doubled Grassmann
graphs, the Doubled Odd graphs (the case $n=2D+1$), and also the Doubled
Hoffman-Singleton graph are determined as distance-regular graphs by their
intersection arrays.

Hiraki \cite{Hi07} (also) characterized the Odd graphs and the Doubled Odd
graphs among the distance-regular graphs by a few of their intersection
numbers, as we already mentioned for the Odd graphs in Section
\ref{sec:determinedbyarray}. Moreover, Hiraki \cite{Hi03D} characterized the
Doubled Grassmann graphs, the Doubled Odd graphs, and the Odd graphs by their
strongly closed subgraphs.

\subsection{Bounds on claws}\label{sec:claws}
An $m$-claw in a graph $\G$ is an induced $m$-star $K_{1,m}$, or in other
words, a coclique of size $m$ in one of the local graphs $\Upsilon(x), x \in
V$.

If $\G$ is a geometric distance-regular graph (with respect to a set of
Delsarte cliques $\cal{C}$) with smallest eigenvalue $-m$, then it follows
easily that each vertex is in $m$ cliques of $\cal{C}$, and hence $\G$ has
no $(m+1)$-claws. Under some conditions, a reverse statement can be made, as we
shall see at the end of this section. Besides this, the existence of claws of
certain size gives rise to new parameter conditions. But as we shall see,
sometimes the intersection numbers force the existence of claws, thus giving
some nonexistence results.

A special case of a result of Metsch's work \cite[Lemma~1.1.b]{metsch91} on the existence of large cliques in graphs is
the following (see also work by Godsil \cite[Lemma~2.3]{Godsil93} or Koolen and Park \cite[Lemma~2]{KoPa10}).

\begin{lemma}
Let $\G$ be a distance-regular graph. Let $x$ be a vertex and let $\mu$ be
the maximum size of $\G(x) \cap \G(y) \cap \G(z)$ where $y \sim x
\sim z$ and $y \not\sim z$. If the local graph $\Upsilon(x)$ contains a
coclique of size $m$, then $\mu \geq \frac{ m (a_1 +1) - k}{\binom{m}{2}}$.
\end{lemma}

\noindent The following consequence of this lemma was observed by Koolen and Park \cite[Thm.~4]{KoPa10}, using that a
`greedy' coclique in $\Upsilon(x)$ has at least $k/(a_1+1)$ vertices.

\begin{prop}\label{shillaterw}
Let $\G$ be a distance-regular graph with valency $k$ and diameter $D \geq
2$, and let $m' = \lceil \frac{k}{a_1+1} \rceil.$ Then
\begin{equation}\label{KoPa}
	c_2 -1 \geq \frac{m'(a_1+1)-k}{\binom{m'}{2}}
\end{equation}
with equality implying that $\G$ is a Terwilliger graph.
\end{prop}

\noindent This result shows that there are no distance-regular graphs with
intersection arrays $\{ 44, 30, 5; 1, 3, 40\}$, $\{65, 44, 11; 1, 4, 55\}$, $\{
81, 56, 24, 1; 1, 3, 56, 81\}$, $\{117, 80, 30, 1; 1, 6, 80, \linebreak 117\}$, $\{ 117,
80, 32, 1; 1, 4, 80, 117\}$ and $\{ 189, 128, 45, 1; 1, 9, 128, 189\}$ (the
last four were also ruled out by Juri\v{s}i\'{c} and Koolen \cite{JuKo00EuJC}).

Gavrilyuk \cite{GavElec10} showed that the only distance-regular graphs with $c_2>1$ for which equality holds in
\eqref{KoPa} are the Icosahedron, the Conway-Smith graph, and the Doro graph. Gavrilyuk \cite{Gav11} also extended the
above by using Brooks' theorem to eliminate the existence of a distance-regular graph with intersection array $\{ 55,
36, 11; 1, 4, 45\}$. Brooks' theorem (see \cite[Thm.~14.4]{BondyMurty}) states that the chromatic number of a connected
graph is at most its maximum valency, except for the odd cycles and complete graphs. For a distance-regular graph
$\G$, this implies that $\Upsilon(x)$ has a coclique of size at least $k/a_{1}$, unless possibly when $\Upsilon(x)$
contains an odd cycle (if $a_1=2$) or an $(a_1+1)$-clique as one of its components. This means that Proposition
\ref{shillaterw} can be sharpened a bit.

\begin{prop} Let $\G$ be a distance-regular graph with valency $k$ and
diameter $D \geq 2$, and let $m'' =\lceil \frac{k}{a_1} \rceil$. If for some
vertex $x$, the local graph $\Upsilon(x)$ does not contain an odd cycle or an
$(a_1+1)$-clique as one of its components, then $$c_2 -1 \geq
\frac{m''(a_1+1)-k}{\binom{m''}{2}}.$$
\end{prop}

\noindent In particular, if $\G$ has smallest eigenvalue $\theta_{\min}$,
then this proposition can be applied when $-1- \frac{b_1}{\theta_{\min}+1} <
a_1$ (in which case the local graph is connected) and $a_1 \neq 2$.

In her work on distance-regular graphs without $4$-claws, Bang \cite{Bapre} also ruled out the intersection array $\{
55, 36, 11; 1, 4, 45\}$. Moreover, she related such graphs to geometric distance-regular graphs with smallest
eigenvalue $-3$. This led to the following more general result by Bang and Koolen \cite{BaKo??}.

\begin{prop}
Let $m \geq 3$ be an integer, and let $\G$ be a distance-regular graph with
diameter $D \geq 2$ and valency larger than $\max\{m^2-m, \frac{m^2-1}m(a_1
+1)\}.$ Then $\G$ has no $(m+1)$-claws if and only if $\G$ is a
geometric distance-regular graph with smallest eigenvalue $-m$.
\end{prop}

\noindent For $m=3$, a slightly stronger result was obtained by Bang \cite{Bapre}, in the sense that the
corresponding result holds for valency larger than $\max\{3, \frac83(a_1 +1)\}.$ We finally note that the
distance-regular graphs without $3$-claws have been determined by Blokhuis and Brouwer \cite{BlBr97},
and that some more work on distance-regular graphs without $4$-claws has been done by Guo and
Makhnev \cite{GM2013} and Bang, Gavrilyuk, and Koolen \cite{BGK4claws}.

\subsection{Sufficient conditions}\label{sec:geometric}

\begin{prop}\label{suffgeom1}
Let $\G$ be a distance-regular graph with diameter $D$ with the property
that there exists a positive integer $m$ and a set $\cal C$ of cliques in
$\G$ such that every edge is contained in exactly one clique of $\cal C$
and every vertex $x$ is contained in exactly $m$ cliques of $\cal C$. If
$|{\cal C}| < |V|$, then $\G$ is geometric with smallest eigenvalue $-m$.
In particular, this is the case if $\min\{ |C| : C \in {\cal C}\} > m$.
\end{prop}
\begin{proof}
Consider the $|V| \times |{\cal C}|$ incidence matrix $N$, where $N_{xC} =
1$ if $x \in C$ and 0 otherwise. Then $NN^{\top} = m I + A$, so the smallest
eigenvalue  $\theta_{\min}$ of $\G$ satisfies $\theta_{\min} \geq -m$.
Suppose now that $| V| > |{\cal C}|$. Then $NN^{\top}$ is singular and hence
$\theta_{\min} = -m$. By the Delsarte bound, every clique $C$ has size at most
$1 -\frac{k}{\theta_{\min}}= 1 + \frac{k}{m} $. On the other hand, by
considering the cliques $C  \in {\cal C}$ containing a fixed vertex, we see
that they have $1 + \frac{k}{m}$ vertices on average. This means that all
cliques in $C$ contain exactly $1 + \frac{k}{m}$ vertices, and hence $\G$
is a geometric distance-regular graph. In particular, if $\min\{ |C| : C \in
{\cal C}\} > m$, then it follows by counting the number of incident pairs
$(x,C)$ in two different ways that $|{\cal C}| < |V|$, so $\G$ is
geometric.
\end{proof}

\noindent This proposition implies that distance-regular graphs
of order $(s,t)$ are geometric with smallest eigenvalue $-t-1$ if
$s>t$; see also Corollary \ref{s>t implies geometric}. Using
Proposition \ref{Metschresult-2.2} we now obtain the following
result.

\begin{prop}\label{prop@} Let $m \geq 2$ be an integer, and let $\G$ be a
distance-regular graph with $(m-1)(a_1+1) < k < m(a_1 +m)$ and
diameter $D \geq 2$. If $a_1 \geq \frac12m(m+1)(c_2+1) $, then
$\G$ is geometric with smallest eigenvalue $-m$.
\end{prop}
\begin{proof}
The two given lower bounds on $a_1$ assure that Proposition
\ref{Metschresult-2.2} can be applied, i.e., that $ a_1 \geq (2m-1)(c_2-1)$ and
$k<(m+1)(a_1+1)-\frac12m(m+1)(c_2-1)$. Thus the set ${\cal C}$ of maximal
cliques of size at least $a_1 +2 - (m-1)(c_2-1)$ forms a set of lines such that
each vertex is in at most $m$ lines, and each edge is in exactly one line. The
given lower bound on $k$ assures that each vertex is in exactly $m$ lines.
Because the minimal line size is at least $a_1 +2 - (m-1)(c_2-1)$, which is at
least $m+1$ by one of the given inequalities, it follows from Proposition
\ref{suffgeom1} that $\G$ is geometric with smallest eigenvalue $-m$.
\end{proof}

\noindent We note that the assumption $k<m(a_1+m)$ holds for all
distance-regular graphs with smallest eigenvalue $-m$ (see \cite{KoBa10}).

\subsection{Distance-regular graphs with a fixed smallest eigenvalue}\label{sec:fixedsmallestev}

Generalizing results by Neumaier \cite{Neu80} and Godsil \cite{Godsil93}, Koolen
and Bang \cite{KoBa10} showed the following.
\begin{theorem}\label{thm:nongeometric}
For given $m \geq 2$, there are only finitely many non-geometric
distance-regular graphs with both valency and diameter at least $3$ and smallest
eigenvalue at least $-m$.
\end{theorem}

\noindent Note that valency $2$ is excluded because of the odd polygons; and diameter $2$ because of the complete
multipartite graphs. Koolen and Bang \cite{KoBa10} did not quite prove this result, as they restricted themselves to
graphs with $c_2 \geq 2$. The graphs with $c_2 =1$ are of order $(s,t)$, with $s=a_1+1$ and $t=k/s-1$. If such a graph
has smallest eigenvalue at least $-m$, then by interlacing (see Section \ref{sec:interlacing}), the existence of a
$(t+1)$-claw implies that $t+1 \leq m^2$. Because a distance-regular graph of order $(s,t)$ that is not geometric has
$s\leq t$, it follows that $k=s(t+1) \leq (m^2-1)m^2$. Therefore the result for the case $c_2=1$ follows from the
Bannai-Ito conjecture. Note that the (general) result was known for $m=2$, see \cite[Thm.~3.12.4,~4.2.16]{bcn}.

Because the smallest eigenvalue of a geometric graph is always integral, this theorem also gives a partial answer to
the question from \cite[p.~130]{bcn} whether every distance-regular graph with valency at least three and diameter at
least three has an integral eigenvalue besides the valency.

One may also wonder whether it is true that for a given integer $m \geq 2$,
there are only finitely many geometric distance-regular graphs with
$D \geq 3$, $c_2 \geq 2$, and smallest eigenvalue $-m$, besides the Grassmann graphs, Johnson graphs, bilinear forms
graphs, and Hamming graphs.
For $D=2$ this is not true, but Neumaier \cite{Neu80} showed that
in essence the geometric strongly regular graphs fall into two infinite classes.

Concerning large smallest eigenvalue, we know that the distance-regular graphs with smallest eigenvalue $-1$ are
exactly the complete graphs. The ones with smallest eigenvalue $-2$ are either strongly regular (and classified by
Seidel \cite{Seidel-2}) or line graphs (and classified by Mohar and Shawe-Taylor \cite{Mdrline})
\cite[Thm.~3.12.14]{bcn}. Among these, the only geometric distance-regular graphs with $D \geq 3$ and $k \geq 3$ are
the generalized $2D$-gons of order $(s, 1)$, $s \geq 2$ and $D=3,4,6$, and the line graphs of the Petersen graph, the
Hoffman-Singleton graph, and putative Moore graphs on 3250 vertices (see \cite[Thm.~4.2.16]{bcn}).
Bang and Koolen \cite{BaKo14} finished the classification of the geometric distance-regular graphs with diameter at least $3$, $c_2 \geq 2$, and smallest eigenvalue $-3$, by showing that such a graph is a Hamming graph, Johnson graph, or a generalized quadrangle of order $(s,3)$ minus a spread (with $s=3$ or $5$).
Yamazaki \cite{Y95} obtained strong restrictions on distance-regular graphs of
order $(s,2)$, $s \geq 3$. The (five) distance-regular graphs of order $(2,2)$ were classified by Hiraki, Nomura, and
Suzuki \cite{HNS}. All of these graphs are geometric with smallest eigenvalue $-3$.

\subsection{Regular near polygons}\label{sec:rnp}

Recall from Section \ref{sec:generalizations+geometric} that a distance-regular
graph $\G$ of order $(s,t)$ with diameter $D$ is called a regular near
$2D$-gon if $a_i = c_i a_1$ for all $i=1,2, \ldots, D$, and that such a graph is
geometric. We call $\G$ thick if $s \geq 2$.

\begin{theorem}\label{thm:RNPthick}{\em (cf.~\cite[Thm.~6.6.1,~9.4.4]{bcn})}
Let $\G$ be a thick regular near $2D$-gon with $D \geq 4$. If $c_2 \geq 3$
or $c_i = i \ (i=2,3)$, then $\G$ is either a dual polar graph or a Hamming
graph.
\end{theorem}

\begin{proof}
(sketch) For $c_2 \geq 3$, the proof is implicitly given in \cite{bcn}. Brouwer and Wilbrink \cite{BWilbrink}
showed that a thick regular near $2D$-gon with $D \geq 4$ and $c_2 \geq 3$ satisfies $c_3 = c^2_2 - c_2 +1$ (the gap as
mentioned in \cite[p.~206]{bcn} is repaired by De Bruyn \cite{DBr06a}). From \cite[p.~277,~Rem.~ii]{bcn} (a remark on a
a result by Brouwer and Cohen \cite{bcohen}), it follows that a thick regular near $2D$-gon with $c_3 = c^2_2 - c_2 +1$
and $c_2 \geq 3$ is a dual polar graph.

For the case $c_i = i \ (i=2,3)$, we will give a sketch of the proof, as it is not in the literature. Let $\G$ be a
thick regular near $2D$-gon of order $(s, t)$ with $D \geq 4$ and $c_i = i \ (i=2,3)$. First, by a result of Brouwer
and Wilbrink \cite{BWilbrink}, one may assume that $c_i = i$ for $i \leq D-1$. Second, it can be shown that if $\G$
is of order $(s,t)$ with $c_i = i \ (i=2,3)$ and $a_2=c_2a_1$, then there exists a map $\phi: H(t+1,s-1) \rightarrow
\G$, such that the partition $\{\phi^{-1}(x):x\in V\}$ is completely regular (cf.~\cite[Thm.~3]{Nomura90}). Using
Theorem \ref{uniformlyregularpartition}, one can show that $\phi^{-1}(x)$ is a completely regular code with minimum
distance $2D$. Now its truncated code is a perfect $(D-1)$-error-correcting code and by the perfect code theorem (see
for example \cite{Hong}), the only such codes with $D \geq 4$ (that are relevant to us; $s \geq 2$) are the codes
consisting of exactly one code word. This shows that $c_D = D$ and that $\G$ is the Hamming graph $H(D, s+1)$. This
finishes the proof of the theorem.
\end{proof}

\noindent In some cases, the intersection numbers of a distance-regular graph imply that it must be a regular near
$2D$-gon; if $c_2=1$, $a_1 \leq 1$, or the graph has classical parameters $(D,-a_1-1,\alpha,\beta)$, see
\cite{Terwilliger1995EJC}. This for example implies that there can be no distance-regular graphs with intersection
array $\{147,144,135;1,4,49\}$ (and classical parameters $(3,-3,-3,21)$), because it would yield a regular near hexagon
with $(s,c_2,c_3)=(3,4,49)$ and this was ruled out by Shult according to Brouwer \cite{Brouwer1981}.

Let $\G$ be a thick regular near polygon with diameter $D$ and head $h$. Hiraki \cite{H97} showed that if $D \geq
2h+1$, then $h \in \{1,2, 3\}$ (he mentions also the possibility $h=5$, but this would lead to a thick generalized
12-gon, a contradiction). This result also follows from Proposition \ref{mbounded} (ii) ($m=h-1$), as we obtain a thick
generalized $2(h+1)$-gon as strongly closed subgraph, and by the Feit-Higman theorem (cf.~\cite[Thm.~6.5.1]{bcn}), it
follows that $h+1 \in \{2, 3, 4\}$. Hiraki \cite{Hi499} conjectured that if $D> 2h+1$, then $h =1$. For a thick regular
near $2D$-gon with $D \leq 2h$, one can bound the valency in terms of $a_1$, see \cite{HiKo04c}.

De Bruyn and Vanhove \cite{DeBVhpre} (see also \cite{Vanhove2012JAC}) obtained
that for a regular near $2D$-gon with $a_1>0$, the intersection numbers satisfy
$c_2 \leq (a_1+1)^2+ 1$ and
$$\frac{((a_1+1)^i -1)(c_{i-1} - (a_1 +1)^{i-2})}{(a_1+1)^{i-2} -1} \leq c_{i} \leq
\frac{((a_1+1)^i +1)(c_{i-1}+ (a_1 +1)^{i-2})}{(a_1+1)^{i-2} + 1}$$ for $i=3,4, \ldots, D$. Neumaier \cite{Neu1990JCTA}
obtained the upper bound for odd $i$ and the lower bound for even $i$ as a specialization of the balanced set condition
of Section \ref{sec:Qpolcharacterizations} for the smallest eigenvalue of a regular near $2D$-gon. For $D=i=3$, the
upper bound is the Haemers-Mathon bound \cite[p.~60]{Haemersthesis}\footnote{This bound is also called the Mathon
bound. It was obtained jointly by Haemers and Mathon. Yanushka recognized that the bound can be obtained from a Krein
condition. Besides the remark in \cite{Haemersthesis}, this is all unpublished.} for regular near hexagons. The upper
bound for even $i$ and the lower bound for odd $i$ can be seen as a specialization of Tonejc's \cite{JurTonpre}
modification of the balanced set condition.

The following is a slight extension of a result due to Brouwer, Godsil, Koolen, and Martin \cite[Thm.~10]{BGKM03}:
\begin{prop}\label{thm:BGKM}
Let $\G$ be a thick regular near $2D$-gon with quads (i.e., geodetically closed subgraphs with diameter two).
Then the second smallest eigenvalue $\theta_{D-1}$ of $\G$ satisfies
\begin{equation*}
	\theta_{D-1} \geq a_1+1-\frac{b_1}{(a_1+1)(c_2-1)},
\end{equation*}
with equality if and only if every quad has width and dual degree summing to $D$. Equality occurs only for the dual
polar graphs and Hamming graphs.
\end{prop}

\noindent The last sentence of the above proposition follows from the following. Let $H$ be a subhexagon of $\G$ and
$Q$ be a subquadrangle in $H$. Brouwer and Wilbrink \cite{BWilbrink} showed that $c_3 \geq c_2(c_2-1) +1$ with equality
if and only if there is no vertex at distance 2 from $Q$ in $H$; see also \cite[p.~26]{DBr06a}. Suppose there is a
vertex $x$ at distance 2 from $Q$ in $H$. If $D \geq 4$, this means that $Q$ cannot be a completely regular code in
$\G$, as this vertex has distance at most 3 to all vertices in $Q$, while there also exists a vertex $y$ at distance 2
from $Q$ with distance 4 to some vertex in $Q$. If $D=3$, then the dual degree is at least the covering radius of $Q$
in $H$, which is at least two, and therefore the sum of the width and dual degree is at least 4. Therefore $c_3 =
c_2(c_2-1) +1$, and hence $\G$ is a dual polar graph or a Hamming graph (for $D \geq 4$, this follows from Theorem
\ref{thm:RNPthick}, whereas for $D=3$, it follows from \cite[Thm.~9.4.4]{bcn}).

For more results on regular near polygons, we refer to \cite{HiKo04a, HiKo04b, HiKo06, TW05}.

%% file: 10_spectralchar.txt

It is known that distance-regularity of a graph is in general not determined by
the spectrum of the graph; see below and the overview by Van Dam, Haemers,
Koolen, and Spence \cite{DHKS06}. See also the survey by Fiol \cite{Fi02} on
algebraic characterizations of distance-regular graphs, and the surveys by Van
Dam and Haemers \cite{DH03, DH09} on spectral characterizations of graphs.

\subsection{Distance-regularity from the
spectrum}\label{sec:drgfromspectrum}

The following proposition surveys the cases for which it is known that
distance-regularity follows from the spectrum.

\begin{prop}\label{drs1}
If $\G$ is a distance-regular graph with diameter $D$, valency $k$, girth
$g$, and distinct eigenvalues $k=\theta_0,\theta_1, \dots, \theta_D$,
satisfying one of the following properties, then every graph cospectral with
$\G$ is also distance-regular, with the same intersection array as
$\G$:
\begin{enumerate}[{\em (i)}]
\item $g \geq 2D-1$ {\em \cite{BH93}},
\item $g \geq 2D-2$ and $\G$ is bipartite {\em \cite{DH02}},
\item $g \geq 2D-2$ and $c_{D-1}c_D<-(c_{D-1}+1)(\theta_1+\cdots+\theta_D)$
    {\em \cite{DH02}},
\item $\G$ is a generalized odd graph, that is, $a_1=\cdots=a_{D-1}=0,\
    a_D\neq 0$ {\em \cite{DHodd, HL99}},
\item $c_1=\cdots=c_{D-1}=1$ {\em \cite{DH02}},
\item $\G$ is the dodecahedron, or the icosahedron {\em
    \cite{HS95}},
\item $\G$ is the coset graph of the extended ternary Golay code {\em
    \cite{DH02}},
\item $\G$ is the Ivanov-Ivanov-Faradjev graph {\em
    \cite{DHKS06}},
\item $\G$ is the Hamming graph $H(3,e)$, with $e \geq 36$ {\em
    \cite{BDK08}}.
\end{enumerate}
\end{prop}

\noindent In fact, more general results hold, because it is actually not in all cases (explicitly) required that the graph is cospectral to a distance-regular graph. Instead, for the graph to be distance-regular, it suffices that a similar spectral condition holds, where the diameter $D$ is replaced by the number of distinct eigenvalues minus one, and the intersection numbers by the so-called {\em preintersection numbers}; for details, we refer to Abiad, Van Dam, and Fiol \cite{AbiadQuasi14}.

Note that the polygons, strongly regular graphs, and bipartite
distance-regular graphs with diameter three are special cases of
(i) and (ii). We also refer to the survey paper by Van Dam and
Haemers \cite{DH03}, where a list of distance-regular graphs that
are known to be determined by the spectrum is included (except that
the antipodal 7-cover of $K_9$ is not mentioned). Van Dam, Haemers,
Koolen, and Spence \cite{DHKS06} give a list of graphs cospectral
with distance-regular graphs on at most 70 vertices (where Hadamard
graphs on $64$ vertices are missing). Note that Van Dam and Haemers
\cite{DH03} conjectured that almost all graphs are determined by
the spectrum. It follows from the prolific constructions of
distance-regular graphs by Fon-Der-Flaass
\cite{FonDerFlaassprolific} (see also Section
\ref{sec:Preparata}) that almost all distance-regular graphs are {\em not}
determined by the spectrum.

For (ix), we refer to Bang, Van Dam, and Koolen \cite{BDK08}, who showed that the Hamming graph $H(3,e)$ with diameter
three is uniquely determined by its spectrum for $e\geq 36$. Moreover, it is shown that for given $D\geq 2$, every
graph cospectral with the Hamming graph $H(D,e)$ is locally the disjoint union of $D$ copies of the complete graph of
size $e-1$, that is, it is geometric, for $e$ large enough. The latter is obtained by bounding the number of common
neighbours of two vertices in terms of the spectrum, and applying Proposition \ref{Metschresult-2.2}. The result on the
Hamming graphs with diameter three then follows from a result by Bang and Koolen \cite{BK08} who showed that if a graph
cospectral with $H(3,e)$ has the same local structure as $H(3,e)$, i.e., if it is geometric, then it is either the
Hamming graph $H(3,e)$ or the dual graph of $H(3,3)$. Furthermore, it is known that for $D\geq e\geq 3$, $(D\geq 4
\mbox{ and }e=2)$, or $(D \geq 2 \mbox{ and }e=4)$, the Hamming graph $H(D,e)$ is not uniquely determined by its
spectrum, whereas for $(2\leq D\leq 3 \mbox{ and }e=2)$ or $(e \geq D=2 \mbox{ and }e\neq 4)$, the Hamming graph
$H(D,e)$ is uniquely determined by its spectrum (cf.~\cite{bcn, DHKS06, HS95,hoffman63}).

Van Dam, Haemers, Koolen, and Spence \cite{DHKS06} showed that the
Ivanov-Ivanov-Faradjev graph is determined by its spectrum, whereas the Johnson
graphs, the Doubled Odd graphs, the Grassmann graphs, the Doubled Grassmann
graphs, the antipodal covers of complete bipartite graphs, and many of the
Taylor graphs are shown to have cospectral mates that are not distance-regular.
These mates are usually obtained by Godsil-McKay switching or by constructing
partial linear spaces that resemble the structure of the distance-regular
graphs in question. Van Dam and Haemers \cite{DH02} also used switching to
construct cospectral mates that are not distance-regular for the Wells graph,
the bipartite double of the Hoffman-Singleton graph, the triple cover of
$GQ(2,2)$, and the Foster graph.

\subsection{The \texorpdfstring{$p$-rank}{p-rank}}\label{sec:prank}

The $p$-ranks of $\G$, that is, the ranks over $GF(p)$ of matrices of the form
$A+\alpha I+\beta J$ with $\alpha,\beta$ integral (and $A$ the adjacency
matrix), can sometimes be used to distinguish cospectral graphs. Peeters
\cite{Peetersprank} studied these $p$-ranks of distance-regular graphs. He
showed among other results that for odd $e$, the Hamming graphs $H(3,e)$ are
determined by the spectrum and the $2$-rank of $A+I$. On the other hand, he
showed that the $p$-ranks of the Doob graphs and the Hamming graphs (with the
same intersection array) are the same.

\subsection{Spectral excess theorem}\label{sec:spectralexcess}

The spectral excess theorem by Fiol and Garriga \cite{FG97} states that a
connected regular graph with $d+1$ distinct eigenvalues is distance-regular
(with diameter $d$) if and only if for every vertex, the number of vertices at
distance $d$ from that vertex (the excess) equals a given expression in terms
of the spectrum (the spectral excess). So a simple `quasi-spectral' property
suffices for a graph to be distance-regular. To specify the result, one should
know that from the spectrum of a regular graph, a system of orthogonal
polynomials $v_i, i=0,1,\dots,d$
--- the so-called {\em predistance polynomials} --- can be constructed. For
distance-regular graphs, this system is well-known, and satisfies $A_i=v_i(A)$,
for $i=0,1,\dots,d$, where $A_i$ is the distance-$i$ adjacency matrix; see
(\ref{distancepolynomials}).

\begin{theorem}{\em (Spectral excess theorem)}\label{spectral excess theorem}
Let $\G$ be a connected k-regular graph on
$n$ vertices with $d+1$ distinct eigenvalues and corresponding orthogonal
polynomials $v_i, i=0,1,\dots,d$, and let $k_d(x)$ be the number of vertices at
distance $d$ from $x$. Then $\G$ is distance-regular if and only if
$k_d(x)=v_d(k)$ for all $x$.
\end{theorem}

\noindent In fact, the theorem can be stated a bit stronger: instead of
requiring that $k_d(x)=v_d(k)$ for all $x$, it is sufficient to require that
the harmonic mean of $n-k_d(x)$ equals $n-v_d(k)$. Another remark is that the
spectral excess $v_d(k)$ can be computed from the spectrum
$\{k=\theta_0^{1},\theta_1^{m_1},\dots,\theta_d^{m_d} \}$ directly as
$$v_d(k)=\frac{n}{\pi_0^2}\left[\sum_{i=0}^d\frac{1}{m_i
\pi_i^2}\right]^{-1},$$ where $\pi_i=\prod_{j \neq
i}|\theta_i-\theta_j|$ for $i=0,1,\dots,d$.

The first result of this kind was obtained by Cvetkovi\'c \cite{C70} and by
Laskar \cite{L69}, who showed that for a Hamming or Doob graph with diameter
three, distance-regularity is determined by the spectrum and having the correct
number of vertices at distance two from each vertex. This result was
generalized to all distance-regular graphs with diameter three by Haemers
\cite{Ha96}, and subsequently by Van Dam and Haemers \cite{DH97}, who proved
the spectral excess theorem for graphs with four distinct eigenvalues (not
assuming that the graph has the spectrum of a distance-regular graph).

At the same time, Fiol, Garriga, and Yebra \cite{FHY96} showed that
a graph with $d+1$ distinct eigenvalues is distance-regular if each
vertex has at least one vertex at distance $d$ and its distance-$d$
adjacency matrix $A_d$ is a polynomial of degree $d$ in the
adjacency matrix $A$, which is the first important step towards the
spectral excess theorem, which was then proved by Fiol and Garriga
in \cite{FG97}. The improvement to considering the above mentioned
harmonic mean was later proved in \cite{Fi02} (see also
\cite{Damexcess}). Fiol also obtained more specific results for
antipodal distance-regular graphs \cite{F97} and for strongly
distance-regular graphs \cite{F00} (a distance-regular graph with
diameter $D$ is {\em strongly distance-regular} if its distance-$D$
graph is strongly regular; examples are the connected strongly
regular graphs, antipodal distance-regular graphs, and
distance-regular graphs with $D=3$ and $\theta_2=-1$). Elementary
proofs of the spectral excess theorem are given by Van Dam
\cite{Damexcess} and Fiol, Gago, and Garriga \cite{Fiolexcess}. The
original proof by Fiol et al.~\cite{FG97, FHY96} has a local
approach and, because of that, it is quite technical.\footnote{In
the language of the Terwilliger algebra, (part of) this local
approach can be interpreted as finding a condition on the thinness
of the primary $\TT$-module; see Footnote
\ref{similar to SET}.} We remark however that by this local
approach, Fiol et al.~manage to prove more related results. Van Dam
and Fiol \cite{DFLaplacian} generalized the spectral excess theorem
by dropping the regularity condition and using the Laplacian
eigenvalues. We refer the interested reader also to surveys by Fiol
\cite{Fi02,Fi05}.

A useful application of the spectral excess theorem is, for example, given by the construction by Van Dam and
Koolen \cite{DK05} of a new family of distance-regular graphs with the same intersection array as certain
Grassmann graphs, see Section \ref{twistedsection}. Distance-regularity of these graphs is proved by showing
that they have the same spectrum as the Grassmann graphs, and then checking the number of vertices at
extremal distance from each vertex. The spectral excess theorem was also used by Van Dam and Haemers
\cite{DHodd} to show that each regular graph with $d+1$ distinct eigenvalues and shortest odd cycle of length
$2d+1$ is a distance-regular generalized odd graph. Lee and Weng \cite{LeeWeng12} generalized this by
dropping the regularity condition, using a version of the spectral excess theorem for nonregular graphs. Van
Dam and Fiol \cite{damfiol12} obtained the same result by an alternative method that avoids the spectral
excess theorem; these results generalize Proposition \ref{drs1} (iv) above.

Kurihara \cite{Kur2011T} obtained a dual version of the spectral excess theorem, in the sense that it
characterizes when a spherical $2$-design generates a cometric association scheme. Kurihara and Nozaki
\cite{KN2011pre} and Nomura and Terwilliger \cite{NT2011LAA} independently derived a spectral
characterization of $P$-polynomial schemes (and hence distance-regular graphs) among symmetric association
schemes that is closely related to the spectral excess theorem.

\subsection{Almost distance-regular graphs}\label{sec:almostdrg}

Motivated by spectral and other algebraic characterizations of distance-regular graphs, Dalf\'{o}, Van Dam,
Fiol, Garriga, and Gorissen \cite{DalfoDamFiol} studied 
\emph{almost} distance-regular graphs.
They used the
spectrum and the predistance polynomials of a graph to discuss concepts such as $m$-walk-regularity and
partial distance-regularity. It was shown by Rowlinson \cite{r97} that a graph is distance-regular if and
only if the number of walks of given length between vertices depends only on the distance between these
vertices. Godsil and McKay \cite{gmk} called a graph walk-regular if the number of closed walks of given
length is constant. The concept of $m$-walk-regularity, as introduced by Dalf\'{o}, Fiol, and Garriga
\cite{DaFiGa09}, generalizes both, and requires the invariance of the number of walks of each given length
between vertices at each given distance at most $m$. Algebraically, this is equivalent to $A_i \circ E_j =
\frac1v Q_{ij} A_i$ for all $i=0,1,\ldots,m$ and $j=0,1,\ldots,d$ (and some $Q_{ij}$), where the notation is
as usual (cf.~Section \ref{sec2:evmult}). An interesting problem raised in \cite{DalfoDamFiol} is to
determine the smallest $m=m(D)$ such that each $m$-walk-regular graph with diameter $D$ is distance-regular.
Informally, the question is till what distance $m$ one needs to check $m$-walk-regularity to assure
distance-regularity. We expect that $m(D)$ is approximately $D/2$.

Dalf\'{o}, Van Dam, and Fiol \cite{dalfoperturbation} showed that
$m$-walk-regular graphs can be characterized through the cospectrality of
certain perturbations of such graphs. As a consequence, some new
characterizations of distance-regularity in terms of certain perturbations are
obtained. C\'{a}mara, Van Dam, Koolen, and Park \cite{CDKP2013} observed a
structural gap between $1$-walk-regularity and $2$-walk-regularity. They showed
among other results that Godsil's bound on the valency in terms of a
multiplicity (in Theorem \ref{thm:godsilbound}), Terwilliger's bounds on the
local eigenvalues \cite[Thm.~4.4.3]{bcn}, and the fundamental bound \eqref{FB}
generalize to $2$-walk-regular graphs. Moreover, they show that there are
finitely many non-geometric $2$-walk-regular graphs with given smallest
eigenvalue and given diameter (in the same spirit as Theorem
\ref{thm:nongeometric}).

Another concept is that of $m$-partial distance-regularity (distance-regularity up to distance $m$). This
means that for $i \le m$, the distance-$i$ matrix can be expressed as a polynomial of degree $i$ in the
adjacency matrix, which is equivalent to saying that the intersection numbers $c_i,a_i,b_i$ are well defined
up to $c_m$. We note that there are $(D-1)$-partially distance-regular graphs with diameter $D$ that are not
distance-regular; for example the direct product of an edge and the folded cube. Lee and Weng
\cite{LeeWeng14} used $2$-partial distance-regularity to characterize the distance-regular graphs among the
bipartite graphs whose halved graphs are distance-regular (cf.~Proposition \ref{prop:halved}).

Related to these concepts are two other generalizations of distance-regular
graphs. Weichsel \cite{w82} called a graph distance-polynomial if each
distance-$i$ matrix can be expressed as a polynomial in the adjacency matrix. A
graph is called distance degree regular if each distance-$i$ graph is regular.
Such graphs were studied by Bloom, Quintas, and Kennedy \cite{bloom}, Hilano
and Nomura \cite {hilanomura}, and also by Weichsel \cite{w82} (as
super-regular graphs). A concept that is dual to partial distance-regularity
was introduced by Dalf\'{o}, Van Dam, Fiol, and Garriga \cite{dalfodual}.

%% file: 11_subgraphs.txt



Let $\G$ be a distance-regular graph.
In this section, a subgraph in $\G$ will always be an induced subgraph.
Recall that a code in $\G$ is simply a non-empty subset of $V_{\G}$.
Therefore, subgraphs, codes, and (vertex) subsets will be virtually the same objects in this section, and we shall adopt one of these names depending on the context.
Completely regular codes will be separately discussed in Section \ref{sec:crc}.

\subsection{Strongly closed subgraphs}\label{sec: strongly closed subgraphs}

Suzuki \cite[Thm.~1.1]{Su295} showed that strongly closed subgraphs of distance-regular graphs are usually
distance-regular.

\begin{theorem}
Let $\Delta$ be a strongly closed subgraph of a distance-regular graph
$\G$. Let $h$ be the head of $\G$ and $k$ be the valency of $\G$. Then one
of the following holds:
\begin{enumerate}[{\em (i)}]
\item $\Delta$ is distance-regular,
\item $2 \leq D_{\Delta} \leq h$,
\item $h$ and $D_{\Delta}$ are even, and $\Delta$ is a distance-biregular graph with $c_{2i-1}
    =c_{2i}$ for all $i=1,2,\dots, \frac{1}{2}D_{\Delta}$,
\item $h =3$, $D_{\Delta}= 5$, and $\Delta$ is isomorphic to the graph obtained by replacing each edge
in a complete graph $K_{\ell+1}$, $\ell \geq 3$, by a path of length $3$,
\item $h =6$, $D_{\Delta} = 8$, and $\Delta$ is isomorphic to the
graph obtained by replacing each edge in a Moore graph with valency $\ell\in \{3,7,57\}$ by a
path of length $3$.
\end{enumerate}
\end{theorem}

\noindent
It follows that (i) holds above precisely when $b_{D_{\Delta}-1}>b_{D_{\Delta}}$, and that (iv) or (v) hold above precisely when $a_1=0$ and $(c_{D_{\Delta}-1},a_{D_{\Delta}-1})=(c_{D_{\Delta}},a_{D_{\Delta}})=(1,1)$.
The Biggs-Smith graph is the only known example of a distance-regular graph with $h\geq 2$ which satisfies $(c_{h+1},a_{h+1})=(c_{h+2},a_{h+2})=(1,1)$.
We note that if $c_2 \geq 2$ then every strongly closed subgraph of $\G$ is distance-regular.

Hiraki \cite{Hi98} introduced the \emph{condition} $\mathrm{(SC)}_m$ as the condition\footnote{We note that $\mathrm{(SC)}_m$ for some $m\in\{1,2,\dots,D_{\G}-1\}$ implies $K_{2,1,1}$-freeness, which in turn implies $h$-boundedness, where $h$ is the head of $\G$; cf.~\cite[p.~129, Remarks]{Hi01}.} that for all vertices $x$ and $y$ at
distance $m$ there exists a strongly closed subgraph $\Delta(x,y)$ with diameter $m$ containing $x$ and $y$.
Hiraki \cite[Thm.~1]{Hi01} showed that $\mathrm{(SC)}_m$ is equivalent to $m$-boundedness for $m=1,2,\dots,D_{\G}-1$.

It is clear that if $\G$ is $m$-bounded then it is $(m+1)$-parallelogram-free.
The converse is not true in general because every bipartite distance-regular graph is parallelogram-free, but it does not even need to be $2$-bounded, as the incidence graph of a $2$-$(11,6,3)$-design shows.
In some cases, however, $(m+1)$-parallelogram-freeness is known to be equivalent to $m$-boundedness.

\begin{prop}\label{mpf}
Let $\G$ be a distance-regular graph with diameter $D$ and let $m\in\{1,2, \ldots,D-1\}$.
Suppose one of the following holds:
\begin{enumerate}[{\em (i)}]
\item $m=1$,
\item $c_2 >1$ and $a_1 >0$,
\item $c_2 =1$ and $a_2 > a_1 >0$,
\item $m=2$ and $a_2 > a_1=0$,
\item $c_{m+1} = 1$ and $a_2 > a_1$.
\end{enumerate}
Then $\G$ is $(m+1)$-parallelogram-free if and only if $\G$ is $m$-bounded.
\end{prop}

\noindent
We remark that (i) is obvious, (ii) was shown by Weng \cite[Thm.~6.4]{W98}, (iii) was shown by Suzuki \cite{S96}, (iv) was
shown by Suzuki \cite{S96} for the case $c_2 =1$ (extending \cite[Lemma 4.3.13]{bcn}) and by Weng \cite[Prop.~6.7]{W98}
for the case $c_2>1$, and (v) was shown by Hiraki \cite{H196}. Hiraki \cite{Hiraki08b} also obtained other sufficient
conditions for a distance-regular graph to be $m$-bounded.

The following proposition summarizes the known results on $\G$ being $m$-bounded for some $m$ that have been obtained by using combinatorial methods.
Some more results are known under the assumption that $\G$ is $Q$-polynomial; cf.~ Section \ref{sec:recentclassical}.
\begin{prop}\label{mbounded}
Let $\G$ be a distance-regular graph with diameter $D \geq 3$ and head
$h\geq 1$. Let $m\in\{1,2,\ldots, D-h\}$. Then $\G$ is $m$-bounded
if one of the following holds:
\begin{enumerate}[{\em (i)}]
\item $c_{m+h} = 1$ and $a_{m-1}<a_m$,
\item $\G$ is $K_{2,1,1}$-free, $a_1>0$, $a_i = c_i a_1$ for $i=1,2, \dots,
    m+h-1$, and $c_{m-1} < c_m$.
\end{enumerate}
\end{prop}

\noindent
Result (i) was obtained by Ivanov and Brouwer (cf.~\cite[Prop.~4.3.11]{bcn}) for $m=2$, and by
Hiraki \cite[Thm.~1.3]{Hi98} for the other cases.
Result (ii) was obtained by Hiraki \cite[Thm.~1.1]{Hi499}, generalizing a result of Brouwer and Wilbrink \cite{BWilbrink} for thick regular near polygons with $h=1$.
We remark that each of the assumptions (i) and (ii) implies $b_{m-1}>b_m$, so that if $x$ and $y$ are at distance $m$ then $\Delta =\Delta(x,y)$ is distance-regular with valency $a_m+c_m$.
In particular, if $c_{2h+1} =1$ and $m = h+1$,
then $\Delta$ is a Moore geometry and it is known that such a graph is either an odd polygon or has diameter at most
$2$; cf.~\cite[Thm.~6.8.1]{bcn}.
This shows the following proposition in the case $a_1 >0$.
The case $a_1 =0$ uses results by Chen, Hiraki, and Koolen \cite{CH99, Hi94,HK02}.

\begin{prop}{\em \cite[Thm.~2]{Hi01}}
Let  $\G$ be a distance-regular graph with head $h \geq 1$ and diameter $D
\geq 2h+3$. Then $h =1$ or $c_{2h+3} \geq 2$.
\end{prop}

\noindent
We note that Wang \cite{W01} did related work. We remark also that if $\G$ is a distance-regular graph with $h=1$ and
$c_4 =1$, then by Proposition \ref{mbounded}(i) and a result from `BCN' \cite[Thm.~5.9.9(i)]{bcn}, $\G$ has a
distance-regular subgraph with diameter 3 and $c_3=1$. No such (latter) graph is known, however. Chen, Hiraki, and
Koolen \cite{CHK98} in fact showed that no such graph with $a_1 \neq 3$ and $a_1 \leq 30$ exists.

Let $\G$ be a distance-regular graph with diameter $D$.
Suppose $\G$ is $D$-bounded and every strongly closed subgraph is regular.
In particular, we have $b_i>b_{i+1}$ for $i=0,1,\dots,D-1$.
Let $\mathscr{S}$ be the poset
consisting of all strongly closed subgraphs of $\G$ with partial order defined by reverse inclusion. Weng
\cite{W97} showed that $\mathscr{S}$ is a ranked meet semilattice and every interval in $\mathscr{S}$ is atomic and
lower semimodular.
He also showed the inequalities
\begin{equation*}
	\frac{b_{D-i-1}-b_{D-i+1}}{b_{D-i-1}-b_{D-i}}\geq \frac{b_{D-i-2}-b_{D-i}}{b_{D-i-2}-b_{D-i-1}} \quad (i=1,2,\dots,D-2),
\end{equation*}
with equality for all $i=1,2,\dots,D-2$ if and only if every interval in $\mathscr{S}$ is a modular atomic lattice.
See also Section \ref{sec: posets}.

For some more work on strongly closed subgraphs in distance-regular graphs, we refer to \cite{Hiraki2012GC} and the references therein.

\subsection{Bipartite closed subgraphs}

Let $\G$ be a distance-regular graph with diameter $D$.
We say the \emph{condition $\mathrm{(BGC)}_j$} holds if for every pair of vertices at
distance $j$ there exists a bipartite closed subgraph with diameter $j$ containing this pair. This condition was introduced by Hiraki \cite{Hi303}, and he showed that $\mathrm{(BGC)}_j$ with $j\in\{1,2,\dots,D-1\}$ implies
$\mathrm{(BGC)}_i$ for all $i=1,2, \ldots, j$.
By combining results of Hiraki \cite{Hi303} and Koolen \cite{Ko192, Ko292}, we have the following.

\begin{prop} {\em \cite[Cor.~4.8]{Hi303}}
Let $\G$ be a distance-regular graph with diameter $D \geq 3$. Let
$t\in\{2,3,\dots, D-1\}$ be such that $c_t = c_{t-1}+1$ and $a_1 = a_2 = \dots =
a_{t-1} = 0$. Then the condition $\mathrm{(BGC)}_t$ holds if and only if one of the
following holds:
\begin{enumerate}[{\em (i)}]
\item $(c_1, c_2, \dots, c_t) = (1, 1, \dots, 1, 2)$ and every bipartite
    closed subgraph with diameter $t$ is the ordinary $2t$-gon,
\item $(c_1, c_2, \dots, c_t) = (1, 2, \dots, t)$ and every bipartite
    closed subgraph with diameter $t$ is the $t$-cube,
\item $t = 2s+1$ is odd, $(c_1, c_2, \dots, c_t) = (1,1, 2,2,\dots, s,s,s+1)$,
    and every bipartite closed subgraph with diameter $t$ is the Doubled Odd
    graph with valency $s+1$,
\item $t=4$, $(c_1, c_2, c_3, c_4) = (1, 1, 2, 3)$, and every bipartite
    closed subgraph with diameter $4$ is the Pappus graph.
\end{enumerate}
\end{prop}

\noindent
Some more general results are obtained by Hiraki \cite{Hi303}.

\subsection{Maximal cliques}

In most cases, it is easy to determine the maximal cliques of classical distance-regular graphs; cf.~\cite{Hemmeter1986EJC}.
However, the structure of the maximal cliques of the quadratic forms graphs turns out to be extremely complicated.
Hemmeter, Woldar, and Brouwer completed the classification of the maximal cliques in this case in a series of papers \cite{Hemmeter1988EJC,HW1990EJCa,HW1990EJCb,BHW1995EJC,HW1999EJC}.
Brouwer and Hemmeter \cite{BH1992EJC} classified the maximal cliques of half dual polar graphs and Ustimenko graphs (which are the distance $1$-or-$2$ graphs of dual polar graphs $\B_m(q)$ and $\C_m(q)$, respectively).
The maximal cliques of twisted Grassmann graphs were described by Van Dam and Koolen \cite{DK05}.

Hemmeter \cite{Hemmeter1986EJC} observed that if $\G$ is a bipartite distance-regular graph with diameter $D\geq 4$, then $\G_1(x)$ is a maximal clique of the halved graph for every $x\in V_{\G}$.
Using this fact, he was able to determine all bipartite distance-regular graphs whose halved graphs belong to one of the known (at the time) infinite families with unbounded diameter; cf.~\cite{Hemmeter1984UM,Hemmeter1988EJC}.
Brouwer, Godsil, Koolen, and Martin \cite[Cor.~2]{BGKM03} showed that if a distance-regular graph $\G$ has a Delsarte clique then it cannot have an antipodal cover of odd diameter.
Van Dam and Koolen \cite{DK05} looked at the structure of the maximal cliques of the twisted Grassmann graphs to show that these graphs are not vertex-transitive.

\subsection{Convex subgraphs}

Lambeck \cite{Lambeck1990D} studied in detail the noncomplete convex subgraphs of classical distance-regular graphs.
He classified such subgraphs in Johnson, Hamming, Grassmann, dual polar, bilinear forms, Hermitian forms, alternating forms graphs, and also quadratic forms graphs $Qua(n,q)$ with $q$ odd.
The noncomplete convex subgraphs of $Qua(n,q)$ with $q$ even were classified by Munemasa, Pasechnik, and Shpectorov \cite{MPS1993JAC}.
It turns out that if $\G$ is one of these graphs then its noncomplete convex subgraphs are distance-regular and belong to the same family as $\G$, with the exception of $Her(D,4)$, which has $K_{2,2}$ as a convex subgraph.

Tanaka \cite{Tanaka2011EJC} used the above results to describe the descendents (cf.~Section \ref{sec: posets}) of these graphs.

\subsection{Designs}
\label{sec: designs - subgraphs}

For recent updates on the study of combinatorial block designs and orthogonal arrays (i.e., $t$-designs in the Johnson and Hamming graphs), we refer the reader to \cite{CD2007B}.
It should be remarked here that Keevash \cite{Keevash2014pre} has recently proved that, given $t$, $k$, and $\lambda$, the natural divisibility conditions for the existence of a block $t$-$(v,k,\lambda)$ design are also sufficient, provided that $v$ is large enough.
This generalizes the result of Wilson \cite{Wilson1975JCTA} for the case $t=2$ and that of Teirlinck \cite{Teirlinck1987DM} which establishes the existence of block $t$-designs for all $t$.

A number of simple $t$-designs over finite fields (i.e., $t$-designs in the Grassmann graphs) with $t$ at most $3$ have been constructed by many researchers; see, e.g., \cite{BKOW2013pre} and the references therein.
Recently, Fazeli, Lovett, and Vardy \cite{FLV2013pre} showed that non-trivial simple $t$-designs over finite fields $\mathbb{F}_q$ exist for all $t$ and $q$.

Delsarte $T$-designs in a distance-regular graph with $|T|=D-1$ (where $D$ is the diameter of the graph) have dual degree $1$.
Such designs are necessarily completely regular with covering radius $1$, and will be briefly discussed in Section \ref{sec: crc in other drg}.

\subsection{The Terwilliger algebra with respect to a code}\label{sec:Terwilliger algebra w.r.t. code}

Let $\G$ be a distance-regular graph with diameter $D\geq 3$, adjacency matrix $A$, and eigenvalues $k=\theta_0>\theta_1>\dots>\theta_D$, and let $C$ be a non-empty subset of $V_{\G}$ (i.e., a code) with covering radius $\rho$.
Let $\{C_0=C,C_1,\dots,C_{\rho}\}$ be the distance partition with respect to $C$, and let $\chi_i$ be the characteristic vector of $C_i$ ($i=0,1,\dots,\rho$).
For each $i=0,1,\dots,\rho$, let $E_i^{\ster}=E_i^{\ster}(C)$ be the diagonal matrix in $M_{v\times v}(\mathbb{C})$ with diagonal entries $(E_i^{\ster})_{yy}=(\chi_i)_y$.
The \emph{Terwilliger algebra} $\TT=\TT(C)$ \emph{with respect to} $C$ is the subalgebra of $M_{v\times v}(\mathbb{C})$ generated by $A,E_0^{\ster},E_1^{\ster},\dots,E_{\rho}^{\ster}$.
The algebra $\TT(C)$ was first introduced and studied by Martin and Taylor \cite{MT1997pre} for binary Hamming graphs.
We shall use the same terminology as in the case of the ordinary Terwilliger algebra (i.e., with respect to a vertex); cf.~Section \ref{sec:Talgebra}.
However, as observed by Martin and Taylor \cite{MT1997pre} and Suzuki \cite{Suzuki2005JAC}, the primary $\TT$-module is thin (and is therefore equal to $\mathrm{span}_{\mathbb{C}}\{\chi_0,\chi_1,\dots,\chi_{\rho}\}$) precisely when $C$ is a completely regular code.

Suzuki \cite{Suzuki2005JAC} studied irreducible $\TT$-modules in detail.
The results in \cite{Suzuki2005JAC} generalize (to some extent) both the theory of tight graphs (cf.~Go and Terwilliger \cite{GT2002EJC}) and the theory of the width of a code (cf.~Brouwer et al. \cite{BGKM03}).
Suzuki \cite{Suzuki2005JAC} showed that a $\TT(C)$-module with endpoint $\nu$ is also a $\TT(C_{\nu})$-module with endpoint $0$, where irreducibility and thinness are also preserved.
This allows us to focus on the irreducible modules with endpoint $0$.

Recall that the width of $C$ is defined by $w=\max\{i:\chi_0^{\mathsf{T}}A_i\chi_0\ne 0\}$, where $A_i$ is the distance-$i$ matrix of $\G$ ($i=0,1,\dots,D$); cf.~Section \ref{sec: posets}.
For $i=0,1,\dots,D$, let $E_i$ be the primitive idempotent associated with $\theta_i$.
Let $\mathbf{v}$ be a nonzero vector in $E_0^{\ster}\mathbb{C}^v$.
Then it is easy to see that there is a polynomial $f$ of degree at most $w$ such that $||E_i\mathbf{v}||^2=f(\theta_i)m(\theta_i)$ ($i=0,1,\dots,D$), where $m(\theta)$ denotes the multiplicity of an eigenvalue $\theta$ of $\G$.
This immediately gives the inequality $w\geq D-r(\mathbf{v})$, where $r(\mathbf{v})=|\{i:E_i\mathbf{v}\ne 0\}|-1$.
(Note that $r(\chi_0)$ is the dual degree of $C$.)
The vector $\mathbf{v}$ is said to be \emph{tight} (with respect to $C$) if $w=D-r(\mathbf{v})$.
Suzuki \cite{Suzuki2005JAC} showed among other results that if $\mathbf{v}$ is tight then $\TT\mathbf{v}$ is a thin irreducible $\TT$-module with endpoint $0$.
This result was previously obtained by Brouwer, Godsil, Koolen, and Martin \cite{BGKM03} for $\mathbf{v}=\chi_0$, and generalizes a theorem of Go and Terwilliger \cite[Thm.~9.8]{GT2002EJC}.
An important consequence is that if $\G$ is $Q$-polynomial then every irreducible module of the ordinary Terwilliger algebra $\TT(x)$ with \emph{displacement} $0$ is thin.\footnote{The \emph{displacement} (\cite{Terwilliger2005GC}) of an irreducible $\TT(x)$-module $W$ is $\eta=e+e^{\ster}+\delta-D$, where $e,e^{\ster},\delta$ are the endpoint, dual endpoint, and the diameter of $W$, respectively. It follows from Caughman's results \cite[Lemmas 5.1, 7.1]{Caughman1999DM} that $0\leq\eta\leq D$.}
This fact was used, e.g., to extend the Assmus-Mattson theorem; cf.~\cite[\S 5]{Tanaka2009EJC}.

Hosoya and Suzuki \cite{HS2007EJC} called $\G$ \emph{tight with respect to} $C$ if the orthogonal complement of $\mathrm{span}_{\mathbb{C}}\{\chi_0\}$ in $E_0^{\ster}\mathbb{C}^v$ is spanned by tight vectors.
By \cite[Thm.~13.6]{GT2002EJC}, $\G$ is tight in the sense of Section \ref{sec:tightDRG} if and only if $\G$ is non-bipartite and tight with respect to $\G_1(x)$ for some (or all) $x\in V_{\G}$.
Hosoya and Suzuki also introduced a homogeneity \emph{with respect to} $C$ in terms of the partition of $V_{\G}$ by the distances from both $C$ and a fixed vertex in $C$, and studied the relation between these two concepts.
They moreover showed that if $\G$ is $Q$-polynomial then the dual eigenmatrix of the association scheme induced on a descendent (cf.~Section \ref{sec: posets}) of $\G$ satisfies a certain system of linear equations, which in particular implies that $\G$ is tight with respect to every descendent.
This system of linear equations turned out to be fundamental to the study of descendents; cf.~\cite{Tanaka2009LAAb,Tanaka2011EJC}.
Lee \cite{Lee2013PhD} studied the above partition for Delsarte cliques (which are descendents of width $1$) in $Q$-polynomial distance-regular graphs that have the most general $q$-\emph{Racah type},\footnote{In the notation of Bannai and Ito \cite[\S III.5]{bi} (cf.~\cite{Talgebra92}), this $Q$-polynomial structure satisfies type I with $s\ne 0$ and $s^*\ne 0$. The polygons are the only known examples of this type.} and showed among other results that there is a natural action of the double affine Hecke algebra of type $(C_1^{\vee},C_1)$ on the subspace of $\mathbb{C}^v$ spanned by the characteristic vectors of the cells of the partition.
See \cite{TTW16+} for detailed information on the Terwilliger algebra with respect to a descendent.

%% file: 12_crc.txt

In this section, we discuss completely regular codes in distance-regular graphs.
Many combinatorial configurations can be viewed as completely regular codes with certain additional properties and/or special parameters in their underlying distance-regular graphs; cf.~Section \ref{sec: crc in other drg}.

As we have seen, a Delsarte clique in a distance-regular graph $\G$ with diameter $D$ is a completely regular code in $\G$ with covering radius $D-1$, and all the geometric distance-regular graphs have plenty of Delsarte cliques.
Martin \cite[Thm.~2.3.3]{Martinthesis} showed that if two distinct vertices $x,y$ in a distance-regular graph $\G$ with diameter $D$ form a completely regular code then either $d(x,y)=1$ and $a_1=a_2=\dots=a_{D-1}=0$, i.e., $\G$ is bipartite or almost bipartite, or $d(x,y)=D$ and $\G$ is antipodal.
C\'{a}mara, Dalf\'{o}, Delorme, Fiol, and Suzuki \cite{CDDFS2013JCTA} showed that all the edges of a connected graph are completely regular codes with the same parameters if and only if the graph is a bipartite or almost bipartite distance-regular graph.
(The assumption on the parameters of the edges was later dropped by Suzuki \cite{Suzuki2014JAC}.)
See \cite{FG1999LAA,FG2001SIAM,CFFG2009EJC,CFFG2010EJC,CFFG2014DAM} and also Section \ref{sec:Terwilliger algebra w.r.t. code} for some algebraic characterizations of complete regularity.

For certain distance-regular graphs, we can show that they come from completely regular partitions in Hamming graphs.

\begin{theorem}\label{thmquotient}
Let $\G$ be a distance-regular graph with diameter $D \geq 3$, valency $k$, intersection numbers $c_i = i$
for $i=1, 2, 3$, and $a_2 = 2 a_1$. Then $a_1 +1$ divides $k$ and there exists a completely regular partition
of the Hamming graph $H(k/(a_1 +1), a_1+2)$ or a Doob graph with valency $k$, with covering radius $D$ and
parameters $\gamma_i = c_i$, $\beta_{i-1} = b_{i-1}$ for $i=1,2, \ldots, D$. The latter case only occurs when
$a_1 = 2$.
\end{theorem}

\noindent
This theorem was shown by Rif\`{a} and Huguet \cite{RiHu} when $a_1 =0$ following ideas of Brouwer \cite{Brouwer1983IEEE} (cf.~\cite[Prop.~4.3.6, Thm.~11.3.2]{bcn}), by Nomura \cite{Nomura90} when  $a_1 \neq 2$, and by Koolen \cite{KoolenDoob} when $a_1 =2$.

\subsection{Parameters}
It is known that the sequence $(c_i)_i$ in a distance-regular graph
$\G$ is non-decreasing, but this is not true in general for the
sequence $(\gamma_i)_i$ of a completely regular code in $\G$.
Koolen \cite{Ko95} gave an infinite family of completely regular
codes in the Doubled Odd graphs with the property that the sequence
$(\gamma_i)_i$ is not necessarily increasing, disproving a
conjecture of Martin \cite{Martinthesis}. Koolen also gave a
sufficient condition for $\G$ that the sequence $(\gamma_i)_i$ is
increasing for every completely regular code in $\G$. Martin
[private communication] showed that the sequence $(\gamma_i)_i$ is
strictly increasing for any completely regular code in a Hamming
graph.

\subsection{Leonard completely regular codes}

Let $\G$ be a distance-regular graph with diameter $D$, valency $k$, and eigenvalues $k=\theta_0,\theta_1,\dots,\theta_D$ (not necessarily in decreasing order).
Let $E_i$ be the primitive idempotent associated with $\theta_i$ for $i=0,1,\dots,D$.
Let $C$ be a completely regular code in $\G$ with covering radius $\rho$.
Let $\{C_0 = C,C_1, \dots, C_{\rho}\}$ be the distance partition with respect to $C$, and let $\mathbf{x}_i$ be the characteristic vector of $C_i$ for $i=0,1,\dots, \rho$.
Let $\mathrm{Spec}(C)=\{\theta_{i_0}=k,\theta_{i_1},\dots,\theta_{i_{\rho}}\}$ be (an ordering of) the spectrum of the quotient matrix of the corresponding distance partition.
We say $C$ is \emph{Leonard} (with respect the above ordering) if
\begin{equation*}
	(E_{i_1} \mathbf{x}_0)^{\circ \ell} \in \mathrm{span}_{\mathbb{C}}\{E_{i_0}\mathbf{x}_0,\dots,E_{i_{\ell}}\mathbf{x}_0\} \, \backslash \, \mathrm{span}_{\mathbb{C}}\{E_{i_0}\mathbf{x}_0,\dots,E_{i_{\ell-1}}\mathbf{x}_0\}
\end{equation*}
for $\ell=1,2,\dots,\rho$.
This definition is due to Koolen, Lee, and Martin \cite{KoLeMa2010}.
Let $A$ be the adjacency matrix of $\G$, and let $A^{\ster}=A^{\ster}(C)$ be the diagonal matrix in $M_{v\times v}(\mathbb{C})$ with diagonal entries $(A^{\ster})_{yy}=\frac{v}{|C|}(E_{i_1}\mathbf{x}_0)_y$.
They showed among other results that $C$ is Leonard if and only if the matrices $A$ and $A^{\ster}$ act on $\mathrm{span}_{\mathbb{C}}\{\mathbf{x}_0,\dots,\mathbf{x}_{\rho}\}=\mathrm{span}_{\mathbb{C}}\{E_{i_0}\mathbf{x}_0,\dots,E_{i_{\rho}}\mathbf{x}_0\}$ as a Leonard pair \cite{Terwilliger2001LAA}.
If $\G$ is a translation distance-regular graph and $C$ is additive, then it follows that $C$ is Leonard if and only if its coset graph is a $Q$-polynomial distance-regular graph.

Next we consider a weaker condition than being Leonard:
\begin{equation*}
	(E_{i_1} \mathbf{x}_0)^{\circ \ell} \in \mathrm{span}_{\mathbb{C}}\{E_{i_0}\mathbf{x}_0,\dots,E_{i_{\ell}}\mathbf{x}_0\} \quad (\ell=1,2,\dots,\rho).
\end{equation*}
As a class of completely regular codes satisfying this condition, Koolen, Lee, and Martin \cite{KoLeMa2010} also introduced \emph{harmonic} completely regular codes as follows.
Suppose $\G$ is $Q$-polynomial with respect to the ordering $\theta_0,\theta_1,\dots,\theta_D$.
We say $C$ is \emph{harmonic} if there is a positive integer $t$ such that $i_{\ell}=t\ell$ for $\ell=0,1,\dots,\rho$.
Descendents (cf.~Section \ref{sec: posets}) are examples of harmonic completely regular codes with $t=1$.
Tanaka \cite[Prop.~4.6]{Tanaka2011EJC} showed that a descendent in $\G$ with width $w$ and dual width $w^*=D-w(=\rho)>1$ is \emph{not} Leonard (with respect to this ordering) \emph{precisely when} $w$ is odd and the $Q$-polynomial structure satisfies type III in the notation of Bannai and Ito \cite[\S III.5]{bi} (cf.~\cite{Talgebra92}).

\subsection{Completely regular codes in the Hamming graphs}

Neumaier \cite{Neu92} conjectured that the only completely regular codes (with at least two words) in the Hamming graphs with minimum distance at least $8$ are the extended binary Golay code and the (binary) repetition codes of length at least $8$.
But he forgot to mention the even subcode of the binary Golay code (i.e., the subcode of the Golay code consisting of the codewords with even weight), as remarked by Borges, Rif\`{a}, and  Zinoviev \cite{BoRiZiDM08}, which was implicitly known to be completely regular.
(The bipartite double of the coset graph of the binary Golay code is distance-regular and has intersection array $\{ 23, 22, 21, 20, 3, 2, 1; 1, 2, 3, 20, 21, 22, 23\}$, and it follows from Theorem \ref{thmquotient} that there is a completely regular partition of the $23$-cube corresponding to this graph.
It is easy to check that this partition corresponds to the cosets of the even subcode of the Golay code, because all distances are even and it has exactly half the number of codewords of the Golay code; see also \cite[p.~362]{bcn}).
So we would like to rephrase Neumaier's conjecture as follows.

\begin{conj}
The only completely regular codes (with at least two words) in the Hamming graphs with minimum distance at least $8$ are the extended binary Golay code, the even subcode of the binary Golay code, and the repetition codes.
\end{conj}

Gillespie \cite{Gillespie2013DM} showed that the only completely regular codes in the binary Hamming graphs $H(D,2)$ with minimum distance greater than $\max\{2,D/2\}$ are the repetition codes and the dual code of the binary $[7,4,3]$-Hamming code.
Meyerowitz \cite{Meyerowitz2003DM} described all the completely regular codes with strength $0$ in the Hamming graphs.

Brouwer \cite{Brouwer1990DM} showed that any truncation of an even and almost even binary completely regular code is again completely regular.
Brouwer \cite{Brouwer1993DM} also gave a necessary and sufficient condition on when the extension of a binary completely regular code is again completely regular.

We say that a binary code $C$ of length $D$ is \emph{self-complementary} if $\mathbf{1} + c\in C$ for all $c \in C$, and \emph{non-self-complementary} otherwise,\footnote{Self-complementary codes and non-self-complementary codes are sometimes called antipodal codes and non-antipodal codes, respectively, in the literature. However, it seems that these are somewhat confusing names.} where $\mathbf{1}=(1,1,\dots,1)$ denotes the all-ones vector in $GF(2)^D$.
Borges, Rif\`{a}, and Zinoviev \cite{BoRiZiDM08} showed among other results that if $C$ is a binary non-self-complementary completely regular code with covering radius $\rho$, minimum distance at least $3$, and distance partition $\{C_0 = C,C_1, \dots, C_{\rho}\}$, then $C_{\rho}=C+\mathbf{1}$, from which it follows that $C \cup C_{\rho}$ is again a completely regular code.
This is a special case of the following construction:
if $C$ is a completely regular code with covering radius $\rho$ and if $\gamma_i = \beta_{\rho-i}$ for $i = 1,2, \ldots, \rho$ with $2i \neq \rho$, then $C \cup C_{\rho}$ is also completely regular.

\subsubsection{Completely transitive codes and generalizations}
\label{sec: completely transitive codes}
Giudici and Praeger \cite{GuPr99} defined the notion of a completely transitive code in a graph.
A code $C$ in a graph $\G$ is called \emph{completely transitive} if there is a group $H$ of automorphisms of $\G$ such that every cell of the distance-partition of $C$ is an orbit of $H$.
It is clear that a completely transitive code is completely regular.
Next suppose that $\G$ is a Cayley graph $\mathrm{Cay}(G, S)$.
Let $C$ be a subgroup of $G$ with covering radius $\rho$ (as a code in $\G$).
We say $C$ is \emph{coset-completely transitive} if the subgroup $\mathfrak{T}$ of the automorphism group of $G$ consisting of the elements that stabilize both $C$ and $S$ has exactly $\rho +1$ orbits on $G/C$.
This is an extension of the notion of coset-completely transitive codes in Hamming graphs $H(D,q)$ defined by Giudici and Praeger \cite{GuPr99},\footnote{To be more precise, they defined the concept for linear codes and considered the stabilizer of $C$ in the group
of weight-preserving sesquilinear automorphisms of $GF(q)^D$.} which in turn generalizes a concept of Sol\'{e} \cite{Sole1990DM}.
It is easy to see that if $C$ is coset-completely transitive and is in the center of $G$,
then $C$ is completely transitive with $H=C\rtimes\mathfrak{T}$.

There are many examples of coset-completely transitive additive codes in Hamming graphs, such as the codes in the Golay family:
the binary Golay code, the extended binary  Golay code, the even subcode of the binary Golay code, the punctured code of the binary Golay code, the even subcode of the punctured Golay code, the twice punctured binary Golay code, the ternary Golay code, and the extended ternary Golay code.
Rif\`{a} and Zinoviev \cite{RiZi11} showed that the lifts of the perfect Hamming codes are coset-completely transitive and that the coset graphs of these codes are the bilinear forms graphs.
Rif\`{a} and Zinoviev \cite{RiZipre08} also constructed an infinite family of binary linear completely transitive codes whose coset graphs are the halved cubes.
Borges, Rif\`{a}, and Zinoviev constructed many more linear completely regular and completely transitive codes;  see, e.g., \cite{BRZ2014DCC,RZ2010IEEE}.
Gillespie and Praeger \cite{GP2013DCC,GP2012-4pre} also considered generalizations of completely transitive codes.

There are only a few completely transitive binary codes known which are not coset-completely transitive, among them are the Hadamard code of length $12$ and its punctured code.
Gillespie and Praeger \cite{GP2013JCTA} showed that these codes are characterized as binary completely regular codes by their lengths and minimum distances.
See also \cite{GP2012-5pre}.
Borges, Rif\`{a}, and Zinoviev \cite{BR2000IEEE,BRZ2001IEEE} showed that the only binary coset-completely transitive codes with minimum distance at least $9$ are the binary repetition codes.
Gillespie, Giudici, and Praeger \cite{GGP2012pre} showed that the only completely transitive codes in the Hamming graphs with minimum distance at least $5$ such that the corresponding groups $H$ of automorphisms are faithful on coordinates are the binary repetition codes.

\subsubsection{Arithmetic completely regular codes}
\label{sec: arithmetic crc}

Harmonic completely regular codes in Hamming graphs
were called \emph{arithmetic} completely regular codes and studied in detail by Koolen, Lee, Martin, and Tanaka \cite{KoLeMapre09}.
Let $C$ be a completely regular code in $H(n,q)$ with covering radius $1$.
Then the cartesian product $C \times C \times \dots \times C$ ($t$ times) is an arithmetic completely regular code with covering radius $t$ in $H(nt, q)$.
Next, let $C$ be a completely regular code in $H(n,q)$ with covering radius $\rho\geq 1$ and parameters $\beta_i,\gamma_i$ ($i=0,1,\dots,\rho$).
Similarly, let $C'$ be a completely regular code in $H(n',q')$ with covering radius $\rho'\geq 1$ and corresponding parameters $\beta_i',\gamma_i'$ ($i=0,1,\dots,\rho'$).
Koolen, Lee, Martin, and Tanaka \cite[Prop.~3.4]{KoLeMapre09} showed that $C \times C'$ is completely regular in $H(n,q)\times H(n',q')$ if and only if there are integers $\beta,\gamma$ such that $\beta_{\rho-i}=\beta i$, $\gamma_i=\gamma i$ for $i=0,1,\dots,\rho$, and $\beta_{\rho'-i}'=\beta i$, $\gamma_i'=\gamma i$ for $i=0,1,\dots,\rho'$.
We note that completely regular codes having parameters of this form are arithmetic.
From this result it follows that if we take $C'$ to be a perfect binary $1$-error correcting code with length $2^t-1$ which is not isomorphic to the binary Hamming code $C$ of the same length, then the cartesian product of $C'$ and $s$ copies of $C$ is completely regular with covering radius $s+1$, but this code is certainly not completely transitive, answering a problem of Gillespie \cite[Problem 11.6]{gillespiethesis}.

Koolen, Lee, Martin, and Tanaka \cite[Thm.~3.16]{KoLeMapre09} also classified all arithmetic completely regular linear codes.
This is a generalization of a result of Bier \cite{Bier1987DM}.
Borges, Rif\`{a}, and Zinoviev \cite{BRZ2010AMC} classified all completely regular linear codes with covering radius $1$ (which are clearly arithmetic) using a different approach.

Fon-Der-Flaass \cite{FonDerFlaass2007SMJ} showed that, for fixed positive integers $\beta_0$ and $\gamma_1$, there is a completely regular code in $H(D,2)$ with covering radius $1$ and parameters $\beta_0$ and $\gamma_1$ for \emph{some} $D$ if and only if $\frac{\beta_0+\gamma_1}{(\beta_0,\gamma_1)}$ is a power of $2$.
(He attributed this result to S.~Avgustinovich and A.~Frid.)
Note that if $C$ is such a code in $H(D,2)$ then so is $C\times GF(2)$ in $H(D+1,2)$.
Fon-Der-Flaass \cite{FonDerFlaass2007SMJ} also obtained lower and upper bounds on the smallest diameter $D=D_0(\beta_0,\gamma_1)$ for which such a code exists.
A code in $H(D,2)$ is called \emph{degenerated} if it is isomorphic to $C\times GF(2)$ for some code $C$ in $H(D-1,2)$, and \emph{non-degenerated} otherwise.
Simon \cite{Simon1983P} showed that for any non-degenerated bipartition of $GF(2)^D(=V_{H(D,2)})$ there is a vertex adjacent to at least $\Omega(\log_2 D)$ vertices in the other cell.
This gives an upper bound on the maximum diameter $D=D_1(\beta_0,\gamma_1)(\geq D_0(\beta_0,\gamma_1))$ for which there is a non-degenerated binary completely regular code with covering radius $1$ and parameters $\beta_0$ and $\gamma_1$.
We remark that the method of Simon works in general for binary completely regular codes, not only for covering radius $1$.

\subsection{Completely regular codes in other distance-regular graphs}
\label{sec: crc in other drg}

The completely regular codes with strength $0$ in the Johnson graphs as well as Hamming graphs were described by Meyerowitz \cite{Meyerowitz1992JCISS,Meyerowitz2003DM}.
Note that descendents (cf.~Section \ref{sec: posets}) in $Q$-polynomial distance-regular graphs are examples of completely regular codes with strength $0$.
Brouwer, Godsil, Koolen, and Martin \cite{BGKM03} used Meyerowitz's results to determine all the descendents in the Johnson and Hamming graphs.
Tanaka \cite{Tanaka2006JCTA,Tanaka2011EJC} extended the classification of descendents to all of the $15$ known infinite families of distance-regular graphs having classical parameters with unbounded diameter.

Martin \cite{Martin1994JAC} determined the completely regular codes with strength $1$ and minimum distance at least $2$ in the Johnson graphs.
Martin \cite{Martin1998JCD} also studied general completely regular $t$-designs in the Johnson graphs in detail.
Sporadic examples include the $5$-$(24,8,1)$, $4$-$(23,7,1)$, and $3$-$(12,6,2)$ designs.
Completely transitive codes (cf.~Section \ref{sec: completely transitive codes}) in the Johnson graphs were studied by Godsil and Praeger \cite{GP1997pre}.
Liebler and Praeger \cite{LP2013pre} also considered generalizations of completely transitive codes in the Johnson graphs.
Completely regular codes in the Odd graphs were studied by Martin \cite{Martin-pre}.
Koolen \cite{Ko95} classified the completely regular codes in the Biggs-Smith graph.

The completely regular codes of a distance-regular graph with covering radius $1$ are exactly the same as (non-trivial) \emph{intriguing sets} studied by De Bruyn and Suzuki \cite{DBS10}.
Note that a code in a distance-regular graph is an intriguing set if and only if it has dual degree $1$.
Tight sets and $m$-ovoids in finite polar spaces are examples of intriguing sets.
Hemisystems (cf.~Section \ref{sec:recentclassical}) are $1$-designs in the dual polar graphs $^2\A_3(\sqrt{q})$ with $q$ odd, and are therefore intriguing sets.
Gavrilyuk and Mogilnykh \cite{GM2012pre} showed among other results the non-existence and uniqueness of certain Cameron-Liebler line classes in $PG(3,q)$, which are intriguing sets with dual width $1$ in $J_q(4,2)$.
See also \cite{Metsch2014JCTA,GM2014JCTA,FMX2014pre,BDMR2014pre}.

Recall that the incidence graph of a symmetric design is a $Q$-polynomial bipartite distance-regular graph with diameter $3$.
Martin \cite{Martin2001JCMCC} observed that several geometric substructures in finite projective spaces are Delsarte $T$-designs with $T\in\{\{1,3\},\{2,3\}\}$ in the corresponding bipartite distance-regular graphs, so that they provide more examples of intriguing sets.

Vanhove \cite{Vanhove2011JCD} showed among other results that partial spreads with maximum size $\sqrt{q^3}+1$ in $^2\A_5(\sqrt{q})$ as well as spreads in $\B_D(q)$ and $\C_D(q)$ with $D\in\{3,5\}$ are completely regular.
See also \cite{Vanhove2011PhD} for more results.

Perfect codes in a distance-regular graph $\G$ are completely regular, but non-trivial ones are very rare.\footnote{Here, `non-trivial' means that the minimum distance is at least $3$, and also at most $D_{\G}-2$ if $\G$ is an antipodal $2$-cover with $D_{\G}$ odd.}
It is well known that the only non-trivial perfect codes in the Hamming graphs $H(D,q)$ with minimum distance $\delta\geq 7$ (or $\delta=5$ and $q$ a prime power) are the binary Golay code and the ternary Golay code; cf.~\cite[\S 11.1D]{bcn}.
See, e.g., \cite{Etzion2007JCD} and the references therein for recent progress towards proving a longstanding conjecture of Delsarte \cite[p.~55]{del} that there are no non-trivial perfect codes in the Johnson graphs.
Chihara \cite{Chihara1987SIAM} showed that there are no non-trivial perfect codes in the Grassmann graphs, dual polar graphs, and the forms graphs, except possibly $\B_D(q)$ and $\C_D(q)$ with $D=2^m-1$ for some positive integer $m$.
Her proof depends only on a detailed analysis of the orthogonal polynomials $(v_i)_{i=0}^D$ associated with these graphs (see \eqref{distancepolynomials} and the remark after \eqref{TTR}), so that we also obtain, e.g., the non-existence for the twisted Grassmann graphs.
Martin and Zhu \cite{MZ1995DCC} gave a simple proof of the non-existence for the Grassmann and bilinear forms graphs using Delsarte's `Anticode Bound' (cf.~\cite[Prop.~2.5.3]{bcn}).
We note that the maximum anticodes in this case are precisely the descendents, in view of the Erd\H{o}s-Ko-Rado theorem for these graphs; cf.~\cite{Tanaka2006JCTA}.
Koolen and Munemasa \cite{KM2000JSPI} constructed perfect codes with minimum distance $3$ in the two Doob graphs with diameter $5$.
Krotov \cite{Krotov2014pre} recently showed among other results the existence of perfect codes with minimum distance $3$ in infinitely many Doob graphs.

%% file: 13_morecomb.txt

\subsection{Distance-regular graphs with a relatively small number of vertices} The Taylor graphs and Hadamard graphs
form infinite families of graphs that have a relative small number of vertices compared to the valency $k$. Indeed,
Taylor graphs have $2k+2$ vertices and Hadamard graphs have $4k$ vertices. The following result, obtained by Koolen and
Park \cite{KoPapre12}, shows that these two families are exceptions.
\begin{theorem}
Let $\alpha > 2$. Then there are finitely many distance-regular graphs with $v$ vertices, valency $k$, diameter $D \geq
3$ satisfying $v \leq \alpha k$, besides imprimitive distance-regular graphs with diameter $3$ and antipodal bipartite
distance-regular graphs with diameter $4$.
\end{theorem}
\noindent As a consequence, they also obtained the following.
\begin{theorem}
Let $0 < \epsilon < 1$. Then there are finitely many distance-regular graphs with valency $k \geq 3$, diameter $D \geq
3$ satisfying $c_2 \geq \epsilon k$, besides imprimitive distance-regular graphs with diameter $3$ and antipodal
bipartite distance-regular graphs with diameter $4$.
\end{theorem}
\noindent For $k \leq 1/ \epsilon$, this result follows from the Bannai-Ito conjecture (see Section
\ref{sec:proofBIconjecture}). If $k
> 1/\epsilon,$ then $c_2 \geq 2$ and one can use the Ivanov bound (Theorem \ref{ivanovbound}) to bound the diameter, and hence  one
can bound the number of vertices by a constant times $k$.

In the case that $\G$ contains a quadrangle, Koolen and Park \cite{KoPapre} obtained the following bound on $c_2$ in
terms of the valency and the diameter.

\begin{prop}\label{c2}
Let $\Gamma$ be a distance-regular graph with valency $k \geq 3$ and diameter
$D \geq 4$. If $\Gamma$ contains an induced quadrangle, then $c_2 \leq
\frac{2}{D} k$ with equality if and only if $D \geq 5$ and $\Gamma$ is a
$D$-cube or $D=4$ and $\Gamma$ is a Hadamard graph.
\end{prop}

\noindent
The assumption of having induced quadrangles is necessary as the
Foster graph and the Biggs-Smith graph have $k =3$, $c_2 =1$ and $D \geq 7$.
We wonder whether the assumption can be removed for $k$ large
enough.

Note also that diameter three is exceptional, because the complete bipartite
graph $K_{k+1,k+1}$ minus a perfect matching has valency $k$ and $c_2=k-1$.
Koolen and Park \cite{KoPa11} showed that if a distance-regular graph has
diameter three, then $c_2 \leq k/2$ or it is bipartite or a Taylor graph. They
also showed that if a distance-regular graph with diameter at least three and
valency $k$ has $a_1 \geq \frac{k-2}{2}$, then it is a Taylor graph, a line
graph, the Johnson graph $J(7,3)$, or the halved 7-cube.

For $4 \leq D \leq 6$, Koolen and Park \cite{KoPapre} strengthened Proposition
\ref{c2} as follows:

\begin{prop}
Let $\Gamma$ be a distance-regular graph with valency $k \geq 3$ and
diameter $D$.
Then the following hold:
\begin{enumerate}[{\em (i)}]
\item If $D \geq 6$ and $c_2 \geq 2$, then $c_2 \leq k/3$,
\item If  $D =4$ and $c_2 > k/3$ then  $\Gamma$ is a Hadamard graph (and hence
$c_2 = k/2$),
\item If $D =5$ and $c_2 > k/3$, then $\Gamma$ is the $5$-cube,
\item If $D=6$ and $c_2 > 2k/7$, then $\Gamma$ is the $6$-cube or the
    generalized dodecagon of order $(1,2)$.
\end{enumerate}
\end{prop}

\subsection{Distance-regular graphs with multiple \texorpdfstring{$P$-polynomial}{P-polynomial} orderings}\label{sec:2Porder}
In the following, we assume that both the diameter and valency of $\G$ are at least three, and we follow `BCN' \cite[\S
4.2.D]{bcn}. If a distance-regular graph $\Gamma$ has a second $P$-polynomial ordering then the corresponding
distance-regular graph $\Delta$ with the second ordering is either the distance-2 graph $\G_2$, the
distance-$(D-1)$ graph $\G_{D-1}$, or the distance-$D$ graph $\G_D$.

The first case ($\Delta=\G_{2}$) is only possible if $\Delta$ is a Taylor graph (with $c_2 < k-1$) or a generalized odd
graph.

The second case ($\Delta=\G_{D-1}$) occurs if and only if $\Gamma$ is an antipodal 2-cover with diameter $D$ such that
$a_i = 0$ for all $i < (D-1)/2$. In this case, the folded graph is either bipartite or a generalized odd graph. If the
folded graph is bipartite, then $\Gamma$ is bipartite and the diameter $D$ is even. For diameter 4 only the Hadamard
graphs occur; for diameter 6 only the 6-cube. For larger diameter only the $D$-cubes are known. If the folded graph is
a generalized odd graph with diameter 3 then $\Gamma$ is a Taylor graph. For larger diameter and bipartite $\Gamma$
only the Doubled Odd graphs and the cubes are known. If $\Gamma$ is not bipartite, then only three graphs are known:
the Wells graph ($D=4)$, the dodecahedron $(D=5)$, and the coset graph of the truncated even subcode of the binary
Golay code (see \cite[p.~365]{bcn}).

For the last case ($\Delta=\G_{D}$) either $\Delta$ is a generalized odd graph or $a_D=0$. In the latter case
($a_D=0$), it holds that $p^D_{2D} \neq 0$ and if $D =4$ then $p^4_{34} = 0$; moreover, Suzuki \cite{Suzuki1996JCTA}
showed that $D \leq 4$.

\subsection{Characterizing antipodality and the height}\label{sec:charantipodal}

Let $\Gamma$ be a distance-regular graph with {valency at least $3$. It is well-known that $\Gamma$ is an antipodal
$2$-cover if and only if $b_{D-i} = c_i$ for all $i =1,2,\dots,D$ \cite{G74}. Araya and Hiraki \cite{AH98} improved
this by showing that $\Gamma$ is an antipodal 2-cover if and only if $b_{D-i} = c_i$ for $i =1,2,\dots,\lceil D/2
\rceil$. This also improved earlier work of Araya, Hiraki, and Juri\v si\'c \cite{AHJ96}, who showed that a
distance-regular graph is an antipodal 2-cover if there is a $j$ with $b_j =1$ and $D \geq 2j$. Araya, Hiraki, and
Juri\v si\'c \cite{AHJ97} also showed that if $b_2=1$, then $\G$ is an antipodal cover, in particular it is either an
antipodal cover of a complete graph ($D=3$), an antipodal $2$-cover of a strongly regular graph with $\lambda =0$ and
$\mu = 2$ ($D=4$), or the dodecahedron ($D=5$). This solved one of the problems in `BCN' \cite[Prob.~(i),~p.~182]{bcn}.

Suzuki \cite{S294} showed that if $k_i = k_j$ for some $i < j$ with $i+j \leq D$, then either $k_D=1$ or $k_i = k_{i+1}
= \cdots = k_j$, thus solving another problem in `BCN' \cite[Prob.~(ii),~p.~168]{bcn}. The only known distance-regular
graphs with $k_i = k_j$ for some $i < j$ with $i+j \leq D$ and $k_D \geq 2$ are the odd polygons, but it is unknown
whether any others could exist. Hiraki, Suzuki, and Wajima \cite{HSW95} showed that if $k_2=k_j$ for $2+j\leq D$ and
$2<j$, then $\Gamma$ is indeed a polygon $(k=2)$ or an antipodal 2-cover ($k_D=1$). In order to show this result, the
height of a distance-regular graph was used; a notion that we will introduce next.

The {\em height} $\text{ht}(\Gamma)$ of a distance-regular graph $\Gamma$ with diameter $D$ is defined as the maximal
$i$ such that the intersection number $p^D_{Di}$ is nonzero. Note that $\Gamma$ is an antipodal 2-cover if and only if
$\text{ht}(\Gamma) = 0$.  The case $\text{ht}(\Gamma) = 1$ occurs exactly when the distance-$D$ graph $\G_D$ is a
generalized odd graph; see \cite[Prop.~4.2.10]{bcn}. Nakano \cite{N01} strengthened some of the results in \cite{HSW95}
by showing that if $k_i=k_j$ for some $i$ and $j$ such that $i<j\leq D-i$ and $\text{ht}(\Gamma)$ is even and at most
$2(D-2i)$, then this height must be zero, that is, $k_D = 1$.

Suzuki \cite{S594} asked whether $\Gamma$ can be characterized by its induced
subgraphs on $\Gamma_D(x)$, for $x \in V$. Some results in this direction were
obtained by Hiraki \cite{Hi299}, who showed that if every induced subgraph on
$\Gamma_{D,\text{ht}(\Gamma)}(x,y)$ is a clique whenever $d(x,y) = D$, then
$\text{ht}(\Gamma) = D$, $D-1$, or $1$.  He also showed that if
$p^D_{D,\text{ht}(\Gamma)} = 1$, then $\text{ht}(\Gamma)$ equals $D$, $1$, or
$0$. Also Tomiyama \cite{To96, To98} gave some results in this direction, that
is, for the case $\text{ht}(\Gamma) = 2$.

\subsection{Bounds on \texorpdfstring{$k_D$}{kD} for primitive distance-regular
graphs}\label{sec:boundskD}

Let $\G$ be distance-regular with diameter $D$ and valency $k$. Brouwer et al.~\cite[Prop.~5.6.1]{bcn} showed that if
$\Gamma$ is not antipodal, then $k \leq k_D(k_D-1)$. Suzuki \cite{Su91, S594} showed that in this case also the
diameter is bounded by a function of $k_D$. This now also follows from the above statement and the validity of the
Bannai-Ito conjecture (see Section \ref{sec:BIconjecture}).

Park \cite{Papre12} showed that if $\G$ has valency and diameter at least $3$ and satisfies $k_{D-1} + k_{D} \leq 2k$,
then $\G$ is an antipodal $2$-cover, $\G$ is bipartite with $D=3$, $\G$ is the Johnson graph $J(7,3)$, or $\G$ is the
halved 7-cube. In the case $D=3$, there are infinitely many bipartite non-antipodal distance-regular graphs with $k_2 +
k_3 \leq 2k$, for example the incidence graphs of the complements of projective planes of order at least $3$. This
result also confirms a conjecture by Bendito, Carmona, Encinas, and Mitjana \cite{BCEM12} that states that no primitive
distance-regular graph with diameter three has the so-called $M$-property.

\subsection{Terwilliger graphs and existence of quadrangles}\label{sec:existencequadranglesT}

Recall that a distance-regular graph without induced quadrangles is called a {\em Terwilliger graph}. In this section
we collect some sufficient conditions for a distance-regular graph with $c_2 \ge 2$ to contain induced quadrangles.
Note that by Proposition \ref{thm:terwilintersectionnos}, the existence of a quadrangle implies that $c_i - b_i  \geq
c_{i-1} - b_{i-1} + a_1 + 2$. On the other hand, it is shown in the proof of \cite[Thm.~5.4.1]{bcn} that a
distance-regular graph with $c_3 < 2c_2$ and $c_2 \geq 2$ has an induced quadrangle. The following are some more such
combinatorial conditions.
\begin{prop}
Let $\Gamma$ be an amply regular  Terwilliger graph
with diameter $D \geq 2$ and with parameters $(v,k, \lambda, \mu)$ such that $\mu \geq 2$.
\begin{enumerate}[{\em (i)}]
\item If  $k \leq (6 + \frac{8}{57} )(\lambda + 1)$, then $\Gamma$ is the icosahedron, the Doro graph (see \cite[\S
    12.1]{bcn}), or the Conway-Smith graph (see \cite[\S 13.2]{bcn}) \cite[Prop.~6]{KoPa11},
\item If $k < 50(\mu-1)$, then $\Gamma$ is the icosahedron, the Doro graph, or the Conway-Smith graph
    \cite[Cor.~1.16.6(ii)]{bcn},
\item If $24\mu > 10(\lambda +1)$, then $\Gamma$ is the icosahedron, the
    Doro graph, or the Conway-Smith graph \cite{kooterwpre}.
\end{enumerate}
\end{prop}

\noindent In \cite[Thm.~4.4.11]{bcn}, the distance-regular graphs with second largest eigenvalue $b_1 -1$ are
classified. For Terwilliger graphs we can go a little further.

\begin{prop} {\em \cite{kooterwpre}}
Let $\Gamma$ be a distance-regular Terwilliger graph with diameter $D \geq 3$
and distinct eigenvalues $\theta_0 > \theta_1 > \cdots > \theta_D$. Then
$\theta_1 \leq b_1/2 -1$ and $\theta_D \geq -b_1/3 -1$, unless $\Gamma$ is the
icosahedron, the Doro graph, or the Conway-Smith graph.
\end{prop}

\noindent Note that if we would know the Terwilliger distance-regular graphs that are locally Hoffman-Singleton, then
we could improve the above results. Note that there are two feasible intersection arrays known which could be locally
Hoffman-Singleton: $\!\{ 50, 42, 9; 1, 2, 42\}$ and $\{ 50, 42, 1; 1, 2, 50\}$, see \cite[p.~36]{bcn}. Gavrilyuk and
Makhnev \cite{GavMak2011} have worked on the classification of these graphs.

\subsection{Connectivity and the second eigenvalue}

\subsubsection{Connectivity and matchings}
Brouwer and Koolen \cite{BrKo09} showed that a (non-complete) distance-regular
graph $\G$ with valency $k>2$ is  $k$-connected and that the only way to
disconnect $\G$ by removing $k$ vertices is to remove the neighborhood of some
vertex. This implies that also the edge-connectivity of $\G$ equals its
valency, and consequently, that every distance-regular graph on an even number
of vertices has a perfect matching. This had been derived before by Brouwer and
Haemers \cite{BH05}, who also showed that the only way to disconnect $\G$ by
removing $k$ edges is to remove the edges through some vertex.

It was noted by Beezer and Farrell \cite{BeFa00} that in general, the number of
perfect matchings does not follow from the intersection array. They showed that
the numbers of matchings consisting of $i$ edges are determined by the
intersection array for $i=1,2,\ldots, 5$; however the Hamming graph $H(2,4)$
and the Shrikhande graph (which have the same intersection array) have
different numbers of matchings with $i$ edges for every $i>5$.

\subsubsection{The second largest eigenvalue}\label{sec:DIAMETEReigenvalues}
Koolen, Park, and Yu \cite{KoPaYu11} showed that for given $\alpha >1$, there
are only finitely many distance-regular graphs with $k \geq 3$ and $D \geq 3$
whose second largest eigenvalue $\theta_1$ satisfies $\alpha \geq \theta_1 >1$.
Note that the (infinite family of) regular complete bipartite graphs minus a
perfect matching are the only distance-regular graphs with $D \geq 3$ and
$\theta_1=1$, and there are no distance-regular graphs with $D \geq 3$ and
$\theta_1<1$. The distance-regular graphs with $D \geq 3$ and $\theta_1 \leq 2$
were also classified.

The distance-regular graphs with $D \geq 3$ and $a_1 \geq 2$ such that all
local graphs have second largest eigenvalue at most one have been classified by
Koolen and Yu \cite{KoYu11}. One may wonder whether, given $\alpha \geq 1$,
there are only finitely many distance-regular graphs with $D \geq 3$ and $a_1 >
\alpha$ such that each local graph has second largest eigenvalue at most
$\alpha$. The condition $a_1 > \alpha$ ensures that the local graphs are
connected, and thus excludes infinite families such as the Hamming graphs $H(D,
\alpha + 2)$.

\subsubsection{The standard sequence}\label{sec:standardsequence}

Let $\Gamma$ be a distance-regular graph with diameter $D$ and let
$$L(i)= \left[
\begin{tabular}{lllllll}
$0$ & $b_0$\\
$c_1$ & $a_1$ & $b_1$\\
  & $c_2$ & $a_2$ & $b_2$\\
  &       & . & . & .\\
  &       &       & . & . & .\\
  &       &       &   & $c_{i-1}$ & $a_{i-1}$ & $b_{i-1}$\\
  &       &       &   &\makebox{\hspace{.324cm}}   & $c_i$  & $a_i$
\end{tabular}
\right],$$

$$M(i)= \left[
\begin{tabular}{llllll}
$a_i$ & $b_i$\\
$c_{i+1}$ & $a_{i+1}$ & $b_{i+1}$\\
  & . & . & .\\
  &       & . & . & .\\
  &       &   & $c_{D-1}$ & $a_{D-1}$ & $b_{D-1}$\\
  &       &   &\makebox{\hspace{.324cm}}   & $c_D$  & $a_D$
\end{tabular}
\right].$$
Cioab\u{a} and Koolen \cite{CioaKopre} studied the eigenvalues
of these matrices in order to answer a question by Brouwer. Note that the
eigenvalues of $L= L(D)=M(0)$ are the $D+1$ distinct eigenvalues $\theta_0>
\theta_1 > \cdots
>\theta_D$ of $\Gamma$. Let $\rho_i$ be the largest eigenvalue of $L(i)$, let
$\sigma_i$ be the largest eigenvalue of $M(i)$, and let $u_0=1, u_1, \ldots,
u_D$ be the standard sequence of the second largest eigenvalue $\theta_1$ of
$\G$. By the theory of orthogonal polynomials, it follows that this sequence
has one sign change. Also, for $i= 2,3, \ldots, D-1$ and $\varepsilon \in \{+1,
-1\}$, if $\varepsilon u_i
>0$, then $\varepsilon \rho_{i-1} < \varepsilon \theta_1 < \varepsilon
\sigma_{i+1}$. If $u_i = 0$, then $\theta_1 = \rho_{i-1} = \sigma_{i+1}$.

For $D=3$, this means that $\theta_1$ lies between $a_3$ and
$\frac{a_1 + \sqrt{a_1^2 + 4k}}{2}$, and if two of these three numbers are
equal, then they are all equal. The latter case defines the class of Shilla
graphs introduced by Koolen and Park \cite{KoPa10}.

Cioab\u{a} and Koolen \cite{CioaKopre} used the above to derive
that the induced subgraph $\Xi(j)$ on $\Gamma_j(x) \cup \Gamma_{j+1}(x) \cup
\ldots \cup \Gamma_D(x)$ is connected if $j \leq D/2$, and that $\Xi(D/2+1)$ is
not connected if and only if  $\Gamma$ is an antipodal $r$-cover with $r \geq
3$. This answers a question by Brouwer. It is not clear when $\Xi((D+1)/2)$ is
disconnected.

A final remark on the standard sequence of the second
eigenvalue is that Park, Koolen, and Markowsky \cite{PaKoMar} showed that
$u_{j}> 0$ if $j < D/2$ and $u_j \geq 0$ if $j = D/2$. Moreover, they showed
that $u_{D/2} = 0$ if and only if $\Gamma$ is an antipodal cover.

We remark that if $\theta_1 < \alpha k$, 
then $c_t + a_t < \alpha k$ and hence $b_t = k- (a_t + c_t) > (1-\alpha)k$ for $2t+2 \leq D$.  This implies
that 
$k_t > k \frac{(1-\alpha)^{t-1}}{\alpha^{t-1}}$ for $2t + 2 \leq D$. So if $\theta_1 < k/2$, then $k_t > k_{t-1}$,
which gives a partial answer to a problem in `BCN' \cite[p.~189]{bcn}. If $\theta_1 < k/3$, then we can improve Pyber's
bound of Section \ref{sec:Pyber}.

\subsection{The \texorpdfstring{distance-$D$}{distance-D} graph}\label{sec:distance-D graph}

The spectral excess theorem (cf.~Theorem \ref{spectral excess theorem}) states that a connected regular graph with $d+1$ distinct eigenvalues is distance-regular precisely when the distance-$d$ graph is regular with the `right' valency determined by the spectrum of the graph.
As mentioned in Section \ref{sec:spectralexcess}, Fiol \cite{F00} specialized this theorem to strongly distance-regular graphs.
Fiol \cite[Conj.~3.6]{Fiol2001CPC} also conjectured that a distance-regular graph with diameter at least $4$ is strongly distance-regular if and only if it is antipodal.\footnote{We note that a distance-regular graph with diameter $3$ is strongly distance-regular precisely when it has $-1$ as an eigenvalue; cf.~\cite[Prop.~4.2.17]{bcn}. Some non-antipodal examples are the Odd graph $O_4$ and the Sylvester graph.}

Fiol \cite{Fiol2001CPC} showed that a distance-regular graph with diameter $D$ and distinct eigenvalues $k=\theta_0>\theta_1>\dots>\theta_D$ is strongly distance-regular if and only if, for $i=1,2,\dots,D$, the multiplicity $m_i$ of $\theta_i$ is expressed as a certain rational function in $\theta_0,\theta_1,\dots,\theta_D$, and the number of vertices, $v$.
Brouwer and Fiol \cite{BF2014pre} among other results strengthened this result for the case $D=4$ as follows:

\begin{prop}
Let $\G$ be a distance-regular graph with diameter $4$.
Then the following are equivalent:
\begin{enumerate}[{\em (i)}]
\item $\G$ is strongly distance-regular, i.e., the distance-$4$ graph $\G_4$ is strongly regular,
\item $b_3 = a_4 +1$ and $b_1 = b_3c_3$,
\item $(\theta_1 +1)(\theta_3 +1) = (\theta_2 +1)(\theta_4 +1) = -b_1$.
\end{enumerate}
\end{prop}

\noindent
See also \cite{Fiol2014pre}.
In \cite{BF2014pre}, Brouwer and Fiol in fact studied the more general situation where the distance-$D$ graph has at most $D$ distinct eigenvalues; or equivalently, where the distance-$D$ matrix $A_D$ generates a proper subalgebra of the adjacency algebra $\AL$.
They showed for example that a distance-regular graph with diameter $D$ belongs to this class provided that $D$ is odd and the distance $1$-or-$2$ graph is distance-regular (e.g., the Odd graph $O_{D+1}$, the folded $(2D+1)$-cube, and the dual polar graphs $\B_D(q)$ and $\C_D(q)$).

%% file: 14_multiplicities.txt

\subsection{Terwilliger's tree bound}\label{sec:tree_bound}

Terwilliger \cite{TerMult82} showed that if a distance-regular graph $\G$, say with valency $k$, contains an
isometric subgraph that is also a tree, then the multiplicity of each eigenvalue $\theta \neq \pm k$ of $\G$ is at
least the number of leaves (i.e., vertices of valency one) of that tree. This has been generalized by Hiraki and Koolen
\cite[Prop.~3.1]{HK02} to the case where the subgraph is a block graph (i.e., a graph whose $2$-connected components
are complete). Their result is too technical to state here, however, we mention its following consequence.

\begin{prop}{\em (cf.~\cite[Prop.~3.3]{HK02})}
Let $\G$ be a distance-regular graph of order $(s,t)$ and with
head $h$.
Let $\theta \neq k$ be an
eigenvalue with multiplicity $m$, and let $n=\lfloor \frac{h+1}{2} \rfloor$. Then the following hold:
\begin{enumerate}[{\em (i)}]
\item If $h$ is odd and $\theta \neq -t-1$, then
$m \geq (t+1)t^{n-1}s^n$,
\item If $h$ is even and $\theta \neq -t-1$, then
$m \geq (s+1)(st)^n$,
\item If $h$ is odd and $\theta = -t-1$, then
$m \geq 1+(t+1)(s-1)\frac{(st)^{n}-1}{st-1}$,
\item If $h$ is even and $\theta = -t-1$, then
$m \geq \frac{(s-1)(s+1)((st)^{n+1}-1)}{(st-1)s} + \frac{1}{s}$.
\end{enumerate}
\end{prop}

\noindent This generalizes a result of Zhu \cite[Prop.~3.5]{Zhu93} who obtained that $m\geq (t+1)(s-1)$ for $n=1$. It
also generalizes a result of Bannai and Ito \cite{BI87} who showed
that if $a_1 \neq 0$, then $m \geq (k/2)^n$.

C\'{a}mara, Van Dam, Koolen, and Park
\cite{CDKP2013} showed that in a $1$-walk-regular graph with valency $k$
and an eigenvalue $\theta\neq k$ with multiplicity $m$, a clique
can have size at most $m+1$. This result is well-known for
distance-regular graphs. Powers \cite{Pow88} already observed
earlier that for distance-regular graphs equality in this clique
bound cannot occur if $\theta$ is the second eigenvalue, except for
the complete graph. We can generalize this as follows.

\begin{prop}
Let $\G$ be a distance-regular graph with valency $k$. If $\G$
contains a clique with $c$ vertices, and $\theta \neq k$ is an
eigenvalue of $\G$ with multiplicity $m$, then $c \leq m+1$, with
equality only if $\theta=\theta_{\min}$ and $\G$ is complete,
complete multipartite, or bipartite.
\end{prop}

\begin{proof} Consider a clique $C$ with $c$ vertices, the
idempotent $E=UU^{\top}$ and standard sequence $(u_i)_{i=0}^D$
corresponding to $\theta$. Recall from the proof of Biggs' formula
(Theorem
\ref{Biggsformula}) that $E =\sum_{i=0}^D
\nu_i A_i$, where $\nu_i = \nu_0 u_i$ for $i=0, 1,
\ldots, D$. The submatrix of $E$ indexed by the vertices of $C$
equals $\nu_0(I+u_1(J-I))$, which has rank at least $c-1$ (recall
that $u_1=\theta/k \neq 1$). Because the rank of $E$ equals $m$,
the bound $c\leq m+1$ follows. If equality holds then
$1+u_1(c-1)=0$, and hence $c=1-k/\theta$, which implies that $C$ is
a Delsarte clique and $\theta=\theta_{\min}$.

Suppose now that $\G$ is not a complete graph. We aim to show first
that $u_2=1$. Consider a representation associated to $\theta$ (see
Section \ref{sec2:evmult}); for simplicity we normalize it so that
the vectors $\hat{x}$ have length one for all $x \in V$, and the
inner products between these vectors are given by the standard
sequence. Because the rank of the above submatrix of $E$ is $m$, it
follows that the vectors $\hat{z}$, with $z
\in C$ span the row space of $U$. In particular, if we consider a vertex $x$ at distance one from $C$,
then $\hat{x}=\sum_{z \in C}\alpha_z \hat{z}$ for certain
$\alpha_z$. By taking inner products with $\hat{z}$, it follows
that $\alpha_z$ depends only on whether $x$ is adjacent to $z$ or
not. Hence, because $\sum_{z \in C} \hat{z} =0$,
see
\eqref{eq:repdelsarte}, we may assume that $\alpha_z=0$ for $z
\nsim x$. Now let $y$ be a vertex in $C$ that is not adjacent to
$x$. We then obtain that $$1-u_2=\langle
\hat{x},\hat{x} \rangle-\langle
\hat{x},\hat{y} \rangle=\sum_{z \in C}\alpha_z(\langle
 \hat{z},\hat{x} \rangle-\langle
 \hat{z},\hat{y} \rangle)=0,$$ and hence indeed $u_2=1$.

From
\eqref{standard_sequence}, it now follows that $a_1=k+\theta$, and
hence the polynomial $v_2(z)=\frac{1}{c_2}(z^2-a_1z-k)$ from
\eqref{distancepolynomials} satisfies $v_2(\theta)=v_2(k)$,
which implies that $\G_2$ is disconnected. If the diameter equals
two, then $G$ is a complete multipartite graph, so we may now
assume that $D>2$.

Because $\G_2$ is disconnected, it follows that $a_2=0$, for
otherwise $p_{22}^1>0$, which would imply that from every path
between two given vertices in $\G$ one can construct a path in
$\G_2$ between these two vertices. Suppose now that $a_1>0$. Let
$x_0
\sim x_1\sim x_2\sim x_3$ be a shortest path between two vertices $x_0$ and
$x_3$ at distance three, and let $y$ be a common neighbor of $x_1$
and $x_2$. Because $a_2$ is zero, it follows that $y$ is also
adjacent to $x_0$ and $x_3$, and so the latter are not at distance
three, which is a contradiction. Thus $a_1=0$, and because
$a_1=k+\theta$, it follows that $\theta=-k$, and hence $\G$ is
bipartite.
\end{proof}

We remark that the above proof, and hence the result, is also valid
for $2$-walk-regular graphs, just like part of
Godsil's bound in the next section (cf.~Section
\ref{sec:almostdrg}).

\subsection{Godsil's bound}

Godsil \cite{God88} obtained the following lower bound on the
multiplicity of an eigenvalue.
\begin{theorem}\label{thm:godsilbound}{\em (Godsil's bound)}
Let $\G$ be a distance-regular graph with valency $k$ and diameter
$D$, and suppose $\G$ is not a complete
multipartite graph. If $\G$ has an eigenvalue with multiplicity $m
\geq 3$, then $D \leq 3m-4$ and $k \leq
\frac{(m+2)(m-1)}{2}$.
\end{theorem}

\noindent Godsil's diameter bound was improved by Hiraki and Koolen \cite{HK02}.

\begin{prop} {\em \cite[Thm.~1.1,~1.2]{HK02}}
Let $\G$ be a distance-regular graph with diameter $D$. If $\G$ has
an eigenvalue with multiplicity $m \geq 3$, then $D \leq m+6$, unless $h=1$ and $c_2 =1$, in which case $D < m+2 + \log_5 m$.
\end{prop}

\noindent Note that the Doubled Odd graph with valency $k$ has diameter $2k-1$ and an
eigenvalue $k-1$ with multiplicity $2k-2$, so this result is close to the
best possible. Yet another lower bound is obtained by Juri\v si\' c, Terwilliger, and \v{Z}itnik \cite{JTZ2010EJC} (cf.~Section \ref{sec: Hadamard products}):
\begin{prop}\label{light tail bound}
Let $\G$ be a distance-regular graph with valency $k$. If $\theta \neq \pm k$ is an eigenvalue of $\G$
with multiplicity $m$, then
\begin{equation*}
	m \geq k - \frac{a_1 k (\theta +1)^2}{(k+ \theta)^2 + a_1 (\theta^2 - k)}.
\end{equation*}
\end{prop}
Koolen, Kim, and Park \cite{KoKiPapre} refined the above valency bound of Godsil. By using the theory developed by
Juri\v si\' c et al.~\cite{JTZ2010EJC}, they were able to show that for $k \geq 3$ and $m \geq 3$, the only possible
distance-regular graphs with diameter at least 3 and $k =\frac{(m+2)(m-1)}{2}$ are Taylor graphs with intersection
array $\{ (2\alpha +1)^2(2\alpha^2 + 2 \alpha -1), 2\alpha^3(2\alpha +3), 1; 1, 2\alpha^3(2\alpha +3), (2\alpha
+1)^2(2\alpha^2 + 2 \alpha -1)\}$, with $m = 4\alpha^2 + 4 \alpha -1$, where $\alpha$ is an integer not equal to $0$
and $-1$ or $\alpha = \frac{-1\pm \sqrt{5}}{2}$ ($m=3$).

Terwilliger (cf.~\cite[Thm.~4.4.4]{bcn}) showed that if a distance-regular graph with valency $k$ has an eigenvalue
$\theta \neq k$ with multiplicity $m <k$, then $\theta$ is either the second largest or the smallest eigenvalue.
Moreover, in this case $-1 - \frac{b_1}{\theta+1}$ is an algebraic integer as it is an eigenvalue of a local graph.
Also, if $m \leq (k-1)/2$, then $\theta$ is an integer such that $\theta +1$ divides $b_1$. Terwilliger's result was
slightly improved by Godsil and Hensel \cite{GodsilHensel92} for antipodal distance-regular graphs, and by Godsil and
Koolen \cite{GoKo95} for the case that $a_D= 0$. As a consequence, Godsil and Koolen showed that a distance-regular
graph with intersection array $\{\mu(2\mu+1), (\mu-1)(2\mu +1), \mu^2, \mu; 1, \mu, \mu(\mu-1), \mu(2\mu +1) \}$ with
$\mu \geq 2$ does not exist.

\subsection{The distance-regular graphs with a small multiplicity}
Let $\G$ be a distance-regular graph with valency $k$. The eigenvalues $k$, and $-k$ in case $\G$ is bipartite,
are the only eigenvalues with multiplicity one. Each eigenvalue of a polygon, besides $\pm k$, has multiplicity two;
and the polygons are the only distance-regular graphs with an eigenvalue having multiplicity two. The five Platonic
solids, i.e., the icosahedron, dodecahedron, cube, octahedron, and tetrahedron are the only distance-regular graphs
with an eigenvalue having multiplicity three. Zhu \cite{Zhu93} (see also \cite{Zhuthesis}) determined the
distance-regular graphs with an eigenvalue having multiplicity four, whereas Martin and Zhu \cite{MarZhupre} (see also
\cite[Ch.~7]{Koothesis}) determined those with an eigenvalue having multiplicity five, six, or seven. Koolen and Martin
\cite{KooMarpre} (see also \cite[Ch.~7]{Koothesis}) determined the distance-regular graphs with an eigenvalue having
multiplicity eight.

\subsection{Integrality of multiplicities}\label{sec:integralmultiplicities}
Biggs' formula (Theorem \ref{Biggsformula}) for the multiplicities of the eigenvalues
and the requirement that these multiplicities are positive integers
is a crucial part of Biggs' definition \cite[Def.~21.5]{biggs} of feasible
intersection arrays of distance-regular graphs.

Godsil and McKay \cite{gmk} generalized Biggs' formula to walk-regular
graphs, thus obtaining feasibility conditions for such graphs. Recall that a
graph is called walk-regular if the number of closed walks of given length from
a vertex to itself is independent of the chosen vertex but depends only on the
length, $\ell$ say; in other words, every power $A^{\ell}$ of the adjacency
matrix has constant diagonal.

Chv\'{a}tal \cite[Thm.~3]{Chvatal} showed that for strongly regular graphs,
Biggs' feasibility condition of integer multiplicities implies the condition
that the number of closed walks of length $p$ is divisible by $p$ for every
prime $p$. The latter condition is essentially a condition on the spectrum
because the number of closed walks of length $p$ equals $\tr A^p =\sum_{i=0}^d
m_i \theta_i^p.$

Here we generalize Chv\'{a}tal's result as follows.

\begin{prop} Let $\{\theta_0^{m_0},\theta_1^{m_1},\dots,\theta_d^{m_d}\}$ be the
multiset of roots of a monic polynomial with coefficients in ${\mathbb Z}$. If
$\sum_{i=0}^d m_i\theta_i=0$, then every prime $p$ divides $\sum_{i=0}^d m_i\theta_i^p$.
\end{prop}

\begin{proof}
By grouping algebraic conjugates, say $\Theta_i$ is the set of algebraic
conjugates of $\theta_i$, and observing that $\sum_{\theta \in \Theta_i}
\theta^p \equiv (\sum_{\theta \in \Theta_i} \theta)^p \equiv (\sum_{\theta \in
\Theta_i} \theta) \mod p$ (the latter equality is by Fermat), it follows that
$\sum_{i=0}^d m_i \theta_i^p \equiv \sum_{i=0}^d m_i \theta_i \equiv 0 \mod p$.
\end{proof}

For a distance-regular graph, both the eigenvalues with multiplicities and the numbers of walks of length $p$
follow from its intersection array. When one wants to generate putative intersection arrays for
distance-regular graphs, testing the integrality of multiplicities is an important but computationally
expensive part. The proposition indicates that testing integrality of the multiplicities is stronger than
testing that the number of closed walks of length $p$ is divisible by $p$. Note that given the intersection
array it is relatively easy to compute the number of closed walks of given length recursively by using the
distance polynomials $v_i$ in \eqref{distancepolynomials} and their sum, the Hoffman polynomial (hence it is
not necessary to first compute the eigenvalues and multiplicities); cf.~\cite[Thm.~2]{Chvatal}. Brouwer,
Cohen, and Neumaier \cite[p.~134]{bcn} give the intersection array $\{26,25,5,1;1,5,25,26\}$ for a
distance-regular graph for which some of the eigenvalues have irrational multiplicities. We found that the
number of closed walks of length $7$ is not divisible by $7$ in this example. An example of an intersection
array that survives the tests on the number of walks is $\{18,15;1,2\}$; this corresponds to a strongly
regular graph with parameters $(v,k,\lambda,\mu)=(154,18,2,2)$. Indeed, the number of closed walks of length
$\ell$ equals $18(18^{\ell-1}-4^{\ell-1})$ for odd $\ell$, and $18^2+153\cdot4^2$ for $\ell=2$. However, the
eigenvalues $4$ and $-4$ have multiplicities $74.25$ and $78.75$, respectively.

%% file: 15_applications.txt


\subsection{Combinatorial optimization}\label{sec:applcombo}

One of the formulations of Lov\'{a}sz's $\vartheta$-\emph{function bound} \cite{Lovasz1979IEEE} on the independence number and the Shannon capacity of a graph is as a semidefinite program (SDP).
McEliece, Rodemich, and Rumsey \cite{MRR1978JCISS} and Schrijver \cite{Schrijver1979IEEE} observed that if the adjacency matrix of the graph belongs to the Bose-Mesner algebra $\AL$ of an association scheme then we may solve the SDP as an ordinary linear program (LP) and the resulting bound `essentially' coincides with Delsarte's \emph{linear programming bound} \cite{del} on which his theory on codes and designs is based.
In fact, the same idea works for any SDP whenever the matrices defining the problem belong to $\AL$;
Goemans and Rendl \cite{GR1999C} applied this to the \emph{max-cut problem},
and Vallentin \cite{Va08} to finding the least \emph{distortion embeddings} of distance-regular graphs.
Delsarte's theory has been most successful when the association scheme is metric and/or cometric (but see also \cite[\S 8]{martintanaka}).
We especially recommend the survey by Delsarte and Levenshtein \cite{DL1998IEEE} on this topic.
Besides codes and designs, Delsarte's LP was also used to prove the \emph{Erd\H{o}s-Ko-Rado theorem} for several families of $Q$-polynomial distance-regular graphs; cf.~Section \ref{sec:EKR}.

In 2005, Schrijver \cite{Schrijver2005IEEE} applied a variant of an extension of the $\vartheta$-function bound based on \emph{matrix cuts} \cite{LS1991SIAM} to get a new upper bound on the sizes of binary codes.
In this case, the matrices defining the SDP belong to the Terwilliger algebra $\TT$ of the hypercube $H(D,2)$.
This method has been applied to codes in Johnson graphs
by Schrijver \cite{Schrijver2005IEEE}, to codes in (nonbinary) Hamming graphs
by Gijswijt, Schrijver, and Tanaka \cite{Gijswijt2005PhD,GST2006JCTA}, and also to the \emph{kissing number} problem in the real sphere $S^{n-1}\subset\mathbb{R}^n$ by Bachoc and Vallentin \cite{BV2008JAMS}; see \cite{BGSV2012B} for more results on this topic.

Generalizing De Klerk and Sotirov's idea of exploiting group symmetry \cite{DKSot10}, De Klerk, De Oliveira
Filho, and Pasechnik \cite{KOP2012B} recently proposed an SDP relaxation\footnote{It seems that this possible
generalization was indicated implicitly in \cite[p.~232]{DKSot10}. It should be remarked that the Bose-Mesner
algebra may also be replaced by a coherent algebra.} of the \emph{quadratic assignment problem} (without
linear term) for which one of the two defining matrices belongs to the Bose-Mesner algebra $\AL$. For
example, the polygons and the complete multipartite graphs correspond to the \emph{traveling salesman
problem} (cf.~\cite{KPS2008SIAM}) and the \emph{maximum $k$-partition problem} (cf.~\cite{Dobre}),
respectively. De Klerk and Sotirov \cite{KS2011MPA} showed that the relaxation can be strengthened further,
provided the association scheme is vertex transitive. Their computational results involve instances related
to the hypercube $H(D,2)$, and we again encounter the Terwilliger algebra $\TT$; see also
\cite{Sotirov2012B}. Van Dam and Sotirov \cite{DS12, DS13} exploited symmetry to obtain bounds for the
bandwidth of among others Hamming and Johnson graphs, and for the graph partition problem for several other
distance-regular graphs.

We finally mention Lee \cite{Leespectral, Leedimension}, who used the symmetry
of the Johnson graph (and other association schemes) to derive results on the
dimension of certain polytopes that are relevant to cutting plane algorithms
for combinatorial optimization problems such as the bisection problem, the
traveling salesman problem, and the perfect matching problem.

\subsection{Random walks, diffusion models, and quantum walks}\label{sec:randomwalks}

Given a graph $\G$, a {\it random walk} on $\G$ is the walk of a particle that
travels at random upon the vertices of a graph. At each stage the particle
moves to a vertex that is adjacent to its current location, and the
probabilities that it moves to each of its neighbors are equal. The particle
has no memory, so it is as likely to return to a vertex it has just been to as
it is to move to a new vertex.

\subsubsection{Diffusion models and stock portfolios}\label{sec:diffusionmodels}

Random walks on certain distance-regular graphs correspond to important models
for the diffusion of particles. The Ehrenfests' urn model that was proposed to
explain the second law of thermodynamics corresponds to random walks on the
hypercube $H(D,2)$. The classical Bernoulli-Laplace diffusion model corresponds
to random walks on the Johnson graph $J(2D,D)$. Diaconis and Shahshahani
\cite{DiaSha87} obtained results on the rate of convergence (i.e., total
variation distance) to the stationary distribution in the latter model, using
the algebraic properties of the Johnson graph. They find a sharp cut-off point
at about $\frac14 D \log D$ steps, in the sense that a few steps earlier the
variation distance is essentially maximal, while a few steps later it tends to
0 exponentially fast. Belsley \cite{Bels98} obtained similar results for the
(non-bipartite) classical families of examples of Section
\ref{sec:clasfamilies}, and Hora \cite{Hora00} obtained results for the halved
cube and the quadratic forms graph (among others). Diaconis and Saloff-Coste
\cite{DiSaCo} studied separation cut-offs for infinite families of Markov
chains and applied their results to families of distance-regular graphs with
unbounded diameter.

Distance-regular graphs have also been used as examples for interacting particle systems such as the so-called
antivoter model; see \cite[Ch.~14]{Aldousbook}.

An application in finance is given by Billio, Cal\`{e}s, and Gu\'{e}gan
\cite{Billio}, who used random walks on the Johnson graph to study the momentum
strategy for stock portfolios.

\subsubsection{Chip-firing and the abelian sandpile model}

Chip-firing on a graph is a solitaire game that is related to random walks. It
is played with a pile of chips at each vertex of the graph. At each step of the
game, a vertex is {\em fired}, in the sense that a chip moves to each of its
adjacent vertices (if the vertex has sufficient chips). Chip-firing is related
to the abelian sandpile model for self-organized criticality from statistical
physics \cite{LevineSandpile}, to avalanche models, and the dollar game.
Related to these games and models is the critical group (sandpile group, Picard
group) of a graph. Biggs \cite{Biggschipfiring} observed that the critical
group of a distance-regular graph is in general not determined by the
intersection array. He also introduced a subgroup of layered configurations
that {\em is} determined by the intersection array. The Shrikhande graph and
Hamming graph $H(2,4)$ were discussed to illustrate these issues.

\subsubsection{Biggs' conjecture on resistance and potential}\label{sec:biggsconjecture}

Let $x$ be a vertex of $\G$, and suppose we start a random walk at $x$. For
every other vertex $y$, we let the {\it hitting time} $H_{xy}$ be the
expected number of steps needed to get to $y$. The {\it cover time} $C_x(\G)$
is the expected number of steps that a random walk started at $x$ requires
before it has visited every vertex of $\G$.

The calculations required to determine these notions for arbitrary graphs are
often quite intensive, even for moderately-sized graphs. It is therefore
desirable to study graphs possessing symmetry properties that make calculations
feasible. Van Slijpe \cite{Slijpe84}, Devroye and Sbihi \cite{DevSb90}, and
Biggs \cite{biggsp} all (independently) derived that in distance-regular
graphs, the hitting times for vertices $x$ and $y$ at distance $j$ are given in
terms of the intersection numbers and valencies by
$$H_{xy}=k \sum_{i=1}^j \frac{1}{k_ic_i} \sum_{h=i}^D k_h.$$
Biggs \cite{biggsp} did not have this result explicitly, as he stated it in
terms of potentials and electric resistance: let us consider a graph to be an
electric circuit with edges corresponding to resistors of unit resistance. The
effective resistance between two vertices can
--- in theory --- be calculated using the familiar rules for resistances in
series and in parallel. This resistance measures how easily electricity may
flow between the vertices, and can likewise be shown (see \cite{doysne} or
\cite{biggs97}) to measure how easily a random walk will move from one vertex
to another. Naturally, the higher the resistance between $x$ and $y$, the more
difficult it is for a random walk to pass from $x$ to $y$, and conversely. For
distance-regular graphs it is possible to give an explicit value for the
resistance between two vertices, as we shall now describe.

Let $\G$ be a distance-regular graph with valency at least 3. Using the
intersection numbers, define the {\it Biggs potentials} $\phi_i$ recursively by
$\phi_0=v-1$ and $\phi_i = (c_i\phi_{i-1}-k)/b_i$, for $i=0,1,\ldots,D-1.$

The resistance $\rho_j$ between vertices at distance $j$ is then
obtained by $\rho_j=2 \sum_{i=0}^{j-1}\phi_i/vk$; see \cite{biggsp}
(and the hitting time is a factor $vk/2$ larger). This shows that
understanding the behavior of the Biggs potentials is crucial for
the study of electric resistance in distance-regular graphs. Biggs
\cite{biggsp} conjectured that $\phi_1+ \phi_2+
\cdots + \phi_{D-1} \leq \frac{94}{101}\phi_0$ and thus $\max_j \rho_j = \rho_D
\leq (1 + \frac{94}{101}) \rho_1$ with equality only in the case of the
Biggs-Smith graph. This conjecture was later proved by Markowsky and Koolen
\cite{MK10}. It implies that the resistance between two vertices is always at
most $4/k$, and turns out to be a characteristic feature of the Biggs
potentials, namely that the sum of the later $\phi_i$ is dominated by the
earlier ones. In particular, Koolen, Markowsky, and Park \cite{phi1} showed\footnote{It was even conjectured that
the factor $3j+3$ can be replaced by a universal constant, but this was disproved in \cite{kmcollection}.}
that $\phi_{j+1} + \phi_{j+2}+ \cdots + \phi_{D-1} < (3j+3)\phi_j$ for each
$j = 0,1,\ldots,D-2$, and that $\phi_{2} + \phi_3 +
\cdots + \phi_{D-1} \leq \phi_1$ with equality only in the case of the
dodecahedron. The latter result can be used to prove Biggs' conjecture, and is
a much stronger statement. It also implies that if $D\geq 3$, then
$\rho_D/\rho_1 \leq 1+6/k$, which shows that for large $k$, all vertices become
nearly equidistant when measured with respect to the resistance metric.

By applying techniques from \cite{comcov}, the above implies the following for
the hitting times and cover times in distance-regular graphs: For all vertices
$x,y$, we have that $H_{xy} \leq (1+\min(\frac{6}{k},\frac{94}{101}))(v-1)$ and
$C_x(\G) \leq (4+o(1))(v-1)\ln{v}.$ In fact, Feige \cite{fue} showed that for
arbitrary graphs, we have $C_x(\G) \geq (1+o(1))v\ln{v}$, so that the upper
bound for distance-regular graphs is the best possible for large $v$, up to the
multiplicative constant.

We note that the resistances between vertices at distance at
most 3 in distance-regular graphs on at most 70 vertices have been calculated
explicitly by Jafarizadeh, Sufiani, and Jafarizadeh \cite{JaSuJa07}, whereas
those in some other families of distance-regular graphs, such as Hadamard
graphs, were calculated in \cite{JaSuJa09}.

\subsubsection{Quantum walks}

 The above classical random walks can be considered as
 Markov chains on the set of vertices of a graph, with the stochastic transition matrix $\frac1k A$. These
 have applications as described, but also in classical randomized algorithms.
 Likewise, in quantum information theory and quantum physics, there are applications of {\em quantum walks} in quantum
 computing. In the case of quantum walks, the state space is the set of
 directed edges (where each edge in the graph is replaced by two oppositely directed
 edges), and the transition matrix $U$ is unitary, that is, $UU^*=I$. For
 details, we refer to the introductory overview by
 Kempe \cite{Kempe03} and the more graph-theoretical description by Emms, Severini, Wilson, and
 Hancock \cite{Emms}.

The first results on quantum walks on distance-regular graphs
concerned the hypercubes, and were obtained by Moore and Russell \cite{MR2002P}
and Kempe \cite{Kempe}. Jafarizadeh and Salimi \cite{JaSa06} studied quantum
walks on, among others, Hamming graphs and Johnson graphs, using the quantum
decomposition $A=L+F+R$ (see \eqref{quantum decomposition}) of the adjacency
matrix.
See also Section \ref{sec:spectral analysis}.
Similarly, Salimi \cite{Salimi} considered the Odd graphs.

An important feature of quantum networks is {\em perfect state transfer}. This occurs for example between the
antipodes in the hypercubes. Godsil \cite{Godsilperiodic} obtained that if a distance-regular graph $\G$ has
perfect state transfer, then $\G$ is an antipodal double cover. He also constructed a family of Taylor graphs
with perfect state transfer coming from certain Hadamard matrices. Jafarizadeh and Sufiani \cite{JaSu08}
discussed perfect state transfer in several other distance-regular antipodal double covers. Coutinho, Godsil,
Guo, and Vanhove \cite{CGGV14} determined for many more distance-regular antipodal double covers whether they
have perfect state transfer; among others for all such graphs in the tables of `BCN' \cite{bcn}.
Chan \cite{Chan2013pre} showed among other results that for arbitrary $\tau>0$ there exist graphs having perfect state transfer at time less than $\tau$ by taking unions of some of the distance-$i$ graphs of the hypercubes.

See \cite{MR2002P,BKMT2008IJQI,Chan2013pre} for some work on \emph{instantaneous uniform mixing} of continuous-time quantum walks on the Hamming graphs, folded cubes, halved cubes, and the folded halved cubes.

\subsection{Miscellaneous applications}

New classes of error-correcting pooling designs were constructed by Bai, Huang,
and Wang \cite{BHW09} from Johnson graphs, Grassmann graphs, antipodal
distance-regular graphs, and distance-regular graphs of order $(s,t)$. Other
classes were constructed by Zhang, Guo, and Gao \cite{ZGG09}, who used
$D$-bounded distance-regular graphs. Gao, Guo, Zhang, and Fu \cite{GGZF08} used
subspaces in $D$-bounded distance-regular graphs to construct authentication
codes.

Some applications of distance-transitive graphs referred to by Cohen
\cite{cohen04} also apply to distance-regular graphs: Driscoll, Healy Jr., and
Rockmore \cite{DrHeRo} apply the discrete polynomial transform to obtain fast
algorithms for data analysis on distance-transitive graphs, mainly just by
using its three-term recurrence relation; Jwo and Tuan \cite{JwoTu} determined
the transmitting delay in networks that can be modeled as a distance-transitive
antipodal double cover (such as the hypercube).

Distance-regularity is one of the symmetry properties that are studied in a
survey paper by Lakshmivarahan, Jwo, and Dhall \cite{LaJwoDh} on
interconnection networks. Among others, shortest routing algorithms in such
networks are discussed.

Distance-regular graphs, in particular strongly regular graphs,
occur in constructions of energy minimizing spherical codes. As shown by Cohn,
Elkies, Kumar, and Sch\"urmann \cite{CEKS10}, spherical codes obtained from a
spectral embedding of a strongly regular graph are {\em balanced}, that is,
they are in equilibrium under all force laws acting between pairs of points
with strength given by a fixed function of distance. As a consequence, these
spherical codes appear frequently in the study of universally optimal spherical
codes; see \cite{CK07, BBCGKS09}.

%% file: 16_miscellaneous.txt

\subsection{Distance-transitive graphs}\label{sec:dtgmisc} As
already expressed in `BCN' \cite[Ch.~7]{bcn}, it seems to be feasible to classify all distance-transitive graphs,
starting with the primitive ones, given the classification of finite simple groups. For more information on this, we
refer to the historical essay
--- and then state of the art survey --- by Ivanov \cite{IvanovDTG}, the
introduction to the field --- and survey --- by Cohen \cite{cohen04}, and the
(currently) most recent survey by Van Bon \cite{vanbon07}. Concerning the
classification of imprimitive distance-transitive graphs, we mention the
classification of antipodal distance-transitive covers of complete graphs by
Godsil, Liebler, and Praeger \cite{GLP98} and the classification of
distance-transitive covers of complete bipartite graphs by Ivanov, Liebler,
Penttila, and Praeger \cite{ILPP97}. Moreover, Alfuraidan and Hall \cite{AH09}
`finished' the classification of distance-regular graphs whose so-called
primitive core (i.e., the primitive graph obtained after halving and/or
quotienting) is a known distance-transitive graph with diameter at least three.

\subsection{The metric dimension}\label{sec:metricdim}

Given a graph, a {\em resolving set} $W$ is a set of vertices such
that every vertex in the graph is uniquely determined by the
distances to the vertices in $W$. The {\em metric dimension} of a
graph is the size of a smallest resolving set. Babai
\cite{Babaimetric} studied the metric dimension of graphs motivated
by the graph isomorphism problem (he actually studied the problem
more generally in coherent configurations). His results imply an
upper bound on the metric dimension for primitive distance-regular
graphs in terms of the number of vertices, the diameter, and the
valencies. Chv\'{a}tal \cite{Chvatalmetric} obtained an asymptotic
result on the metric dimension of the Hamming graphs, as a result
of his work on strategies for the game {\em Mastermind}. For
details, we refer to the survey paper by Bailey and Cameron
\cite{Metricdimension}. For recent results on the metric dimension
of Johnson graphs, Grassmann graphs, bilinear forms graphs, and
symplectic dual polar graphs, we refer to \cite{Johnsonmetric,
GWL13}, \cite{BaileyMeagher}, \cite{FengWang}, and \cite{GWL13DP},
respectively. Guo, Wang, and Li \cite{GWL13} also obtained results
on the Doubled Odd graphs, Doubled Grassmann graphs, and twisted
Grassmann graphs. The metric dimension of all `small'
distance-regular graphs was determined by Bailey
\cite{Baileysmall}. Bailey \cite{Baileyimprimitive} also related
the metric dimension of several families of imprimitive
distance-regular graphs to the metric dimension of corresponding
primitive distance-regular graphs. The {\em fractional metric
dimension} of vertex-transitive distance-regular graphs, in
particular Hamming and Johnson graphs, was studied by Feng, Lv, and
Wang \cite{FLWfractionalmetric}.

\subsection{The chromatic number}\label{sec:chromatic}

The {\em chromatic number} of a graph is the smallest number of
colors needed to color the vertices such that adjacent vertices
have different colors. Bipartite graphs clearly have chromatic
number $2$. It is not hard to see that the chromatic number of the
Hamming graph $H(d,q)$ equals $q$. Blokhuis, Brouwer, and Haemers
\cite{BlokhuisBH07} studied distance-regular graphs with chromatic
number $3$. They showed that besides the complete tripartite
graphs, the intersection number $a_1$ is at most $1$ in such
graphs, and they obtained several results for the case $a_1=1$; it
seems that the triangle-free case is much more difficult. All
graphs with chromatic number $3$ among the known distance-regular
graphs were classified by Blokhuis et al.~\cite{BlokhuisBH07}:
these are the complete tripartite graphs, the odd cycles, the Odd
graphs, the Hamming graphs $H(D,3)$, and nine exceptional graphs.
It was also shown that the folded cubes have chromatic number $4$.
Koolen and Qiao \cite{KQ15} classified the non-bipartite distance-regular
graphs with diameter three, valency $k$, and smallest eigenvalue at
most $-k/2$. Using these results, they obtained a complete
classification of the distance-regular graphs with diameter three
and chromatic number $3$. Hahn, Kratochv\'{\i}l,
\v{S}ir\'{a}\v{n}, and Sotteau
\cite{injectivecolor} obtained results on the chromatic number of the halved cubes; see also \cite{webchromatic}.
Etzion and Bitan \cite{colorJohnson} and Brouwer and Etzion \cite{BrEt11} provide a summary of results on the chromatic
number of the Johnson graphs.

\subsection{Cores}\label{sec:cores}

A \emph{core} is a graph having no endomorphisms other than
automorphisms. Every graph is homomorphically equivalent (i.e.,
there are homomorphisms in both directions) to a unique core,
called the \emph{core of} the graph. We say that a graph is
\emph{core-complete} if it is either a core or has a complete core.
Cameron and Kazanidis \cite{CK2008JAMS} showed among other results
that rank $3$ graphs are core-complete. Godsil and Royle
\cite{GR2011AC} showed among other results the core-completeness of
many infinite families of geometric strongly regular graphs. In
particular, by virtue of a result of Neumaier \cite{Neu80}, it
follows that, for given $m\geq 2$, all but finitely many strongly
regular graphs with smallest eigenvalue at least $-m$ are
core-complete. Roberson \cite{Rob16} finally showed
that all strongly regular graphs are core-complete. Concerning
general distance-regular graphs, Godsil and Royle
\cite{GR2011AC} showed that distance-transitive graphs are
core-complete, and that triangle-free non-bipartite
distance-regular graphs are cores. Huang, Lv, and Wang
\cite{HLW2014pre} studied cores and endomorphisms of the Grassmann
graphs.

\subsection{Modular representations}

Some work has been done on the adjacency algebra (denoted $\AL_K\subset M_{v\times v}(K)$) of a
distance-regular graph $\G$ over a field $K$ of characteristic $p>0$. Arad, Fisman, and Muzychuk
\cite[Thm.~1.1]{AFM1999IJM} showed among other results that $\AL_K$ is semisimple if and only if the
\emph{Frame number} $v^{D+1}\prod_{i=1}^D(k_i/m_i)$ (which is an integer) is not divisible by $p$. See also
\cite[Thm.~4.2]{Hanaki2000JA}. Hanaki \cite{Hanaki2002AM} showed that $\AL_K$ is a local algebra if $v$ is a
power of $p$.
For strongly regular graphs, Hanaki and Yoshikawa \cite{HY2005JAC} determined the structure of $\AL_K$ and studied the modular standard module $K^v$.
In this case, the $p$-rank of $M\in\AL_K$ (cf.~Section \ref{sec:prank}) can be interpreted as the dimension of the submodule $MK^v$, and their results provide us information as to which elements of $\AL_K$ we should look at.
Yoshikawa \cite{Yoshikawa2004JAC} determined the structure of $\AL_K$ for Hamming graphs. The
structure of $\AL_K$ for Johnson graphs was studied by Shimabukuro
\cite{Shimabukuro2005AC,Shimabukuro2011DM}. Shimabukuro \cite{Shimabukuro2007EJC} also computed the number of
irreducible representations of $\AL_K$ for the classical families of distance-regular graphs.
Shimabukuro and Yoshikawa \cite{SY2014pre} recently studied the structure of $\AL_K$ for Grassmann graphs.
For more information on modular representations of general (non-commutative) association schemes (i.e., homogeneous coherent configurations), we refer to the recent survey by Hanaki \cite{Hanaki2009EJC}.

\subsection{Asymptotic spectral analysis}\label{sec:spectral analysis}

Let $\G$ be a distance-regular graph with adjacency matrix $A$.
Observe that the complex Bose-Mesner algebra $\AL=\AL_{\mathbb{C}}$ of $\G$ is a commutative $*$-algebra, and that $\frac{1}{v}\tr$ is a state on $\AL$, i.e., a unital positive linear $*$-functional on $\AL$.
Thus, $(\AL,\frac{1}{v}\tr)$ is a classical algebraic probability space, and we may view $A$ as an algebraic random variable.
From this point of view, Hora \cite{Hora1998IDAQPRT} obtained, as variations of the central limit theorem, asymptotic spectral distributions for the families of Hamming graphs, Johnson graphs, halved cubes, and Grassmann graphs.
He used the information on the spectra of these graphs directly,
but then the method of quantum decomposition, first introduced in this context by Hashimoto \cite{Hashimoto2001IDAQPRT}, was applied to Hamming graphs by Hashimoto, Obata, and Tabei \cite{HOT2001P} and to Johnson graphs (among others) by Hashimoto, Hora, and Obata \cite{HHO2003JMP}, which provided a more conceptual and succinct (and `fully-quantum') approach to the results of Hora \cite{Hora1998IDAQPRT}.
See also \cite{HO2003P}.
Their theory, in the final form given in \cite{HO2007B,HO2008TAMS}, turns out to be closely related to the Terwilliger algebra (in the case of distance-regular graphs; though they did not use the language of the Terwilliger algebra).
Let $\TT$ be the Terwilliger algebra of $\G$ with respect to $x\in V$.
Let $L$, $F$, and $R$ be the lowering, flat, and raising matrices, respectively; cf.~\eqref{quantum decomposition}.
The \emph{quantum decomposition}\footnote{The quantum decomposition for the complete graph $K_2$ is sometimes referred to as a \emph{quantum coin-tossing}.} of $A$ is the expression $A=L+F+R$.
The primary $\TT$-module, together with $L$ and $R$, naturally has the structure of a one-mode interacting Fock space,\footnote{In this context, $L$ and $R$ are called the \emph{annihilation} and the \emph{creation operators}, respectively.} and they took the limit of the coefficients of the three-term recurrence relation of the associated orthogonal polynomials (which are certain normalizations of the $v_i$ from \eqref{distancepolynomials}) to get the \emph{quantum central limit theorem}.
See \cite{HO2007B,HO2008TAMS} for more details.
For the above four families of distance-regular graphs, these orthogonal polynomials belong to the Askey scheme by virtue of Leonard's theorem (see also the comments after \eqref{TTR}), and their results agree with the limit relations of the polynomials in the Askey scheme as described in \cite{KS1998R,KLS2010B}.
We note that they also considered some other states as well; see \cite{Hora2000PTRF,HO2007B}.
The case of the Odd graphs was discussed in detail by Igarashi and Obata \cite{IO2006P}.
Associated to the Odd graphs are the Bannai/Ito polynomials (which are a $q\rightarrow -1$ limit of the most general $q$-Racah polynomials), and the generalized Hermite polynomials arise as the orthogonal polynomials corresponding to the limit distribution.
See also \cite{GVZ2014SIGMA}.

\subsection{Spin models}\label{sec:spinmodels}

A (symmetric) \emph{spin model} is a nowhere-zero symmetric matrix $W\in M_{v\times v}(\mathbb{C})$ which
satisfies certain `invariance equations', and was introduced by Jones \cite{Jones1989PJM} as a tool for
creating invariants of knots and links. Nomura \cite{Nomura1997JAC} showed that every spin model $W$ belongs
to its \emph{Nomura algebra} $\mathcal{N}_W\subset M_{v\times v}(\mathbb{C})$, which is the Bose-Mesner
algebra (over $\mathbb{C}$) of a self-dual association scheme; see also
\cite{Jaeger1996P,JMN1998JAC,CGM2003TAMS}. We say that a distance-regular graph $\G$ with diameter $D$
\emph{supports} a spin model $W$ if its adjacency algebra $\AL$ (over $\mathbb{C}$) satisfies
$W\in\AL\subset\mathcal{N}_W$. In this case, $\AL$ inherits the duality of $\mathcal{N}_W$. In particular,
$\G$ is $Q$-polynomial. Many examples of spin models have been constructed in this situation; cf.~\cite[\S
9]{CN1999JCTB}. Write $W=\sum_{i=0}^Dt_iA_i$, where $A_i$ is the distance-$i$ matrix of $\G$ for
$i=0,1,\dots,D$. Curtin and Nomura \cite{CN1999JCTB} showed among other results that if $t_1\ne\pm t_0$ then
the intersection array of $\G$ is described by $t_1/t_0$, $t_0t_2/t_1^2$, and $D$. Curtin \cite{Curtin1999DM}
showed that $\G$ is thin if $t_i\ne\pm t_0$ for $i=1,2,\dots,D$, and Caughman and Wolff \cite{CW2005JAC}
determined the structure of the Terwilliger algebra $\TT$. Moreover, Curtin \cite[Thm.~1.6]{Curtin2007RJ}
showed that, in view of \cite[Thm.~5.3]{CW2005JAC}, every irreducible $\TT$-module affords not just a Leonard
system (cf.~Section \ref{sec: TD systems}) but a \emph{Leonard triple system} \cite{Curtin2007LAA}. See
\cite{Huang2012pre,Brown2013pre,TZ2013pre} and the references therein for related results.
Nomura \cite{Nomura1995JCTB,Nomura1996P} and Curtin and Nomura \cite{CN2004JAC} studied the homogeneity of
$\G$. It is known (cf.~\cite[\S 4.4]{Jaeger1996P}) that the diagonal matrix $T$ of size $D+1$ defined by
$T_{ii}=t_i$ ($i=0,1,\dots,D$) satisfies the \emph{modular invariance property}, i.e., $(PT)^3$ is a scalar
matrix, where $P=Q$ is the eigenmatrix of $\G$. (Recall $P^2=vI$.) Chihara and Stanton \cite{CS1995GC} showed
that $\G$ has at most $12$ (diagonal) solutions $T$ to $(PT)^3=I$ in general, and classified the solutions
when $\G$ is a forms graph or a Hamming graph. See also \cite{Nomura2002KJM}.

\subsection{Cometric association schemes}\label{sec:cometricschemes}

Distance-regular graphs form the class of metric (or $P$-polynomial) association schemes. Cometric (or
$Q$-polynomial) association schemes are the `dual version' of distance-regular graphs, but the systematic
study of cometric (but not necessarily metric) association schemes has begun rather recently. One of the
pioneers in this area is Suzuki \cite{Suzuki1998JACa,Suzuki1998JACb}, who studied imprimitive cometric
association schemes and association schemes with multiple $Q$-polynomial orderings, using a method of Dickie
\cite{Dickie1995D} based on matrix identities; cf.~\eqref{3-tensor}. In particular, he showed that an
imprimitive cometric association scheme with $D\geq 7$ and with first multiplicity $m_1>2$ is $Q$-bipartite
and/or $Q$-antipodal; cf.~Theorem \ref{prop:imprimitive}. In his classification, there remained two cases of
open parameter sets with $D\in\{4,6\}$. These were recently ruled out by Cerzo and Suzuki \cite{CS2009EJC}
for $D=4$, and by Tanaka and Tanaka \cite{TT2011EJC} for $D=6$. Similarly, there was an open case in the
classification of association schemes with multiple $Q$-polynomial orderings, which was recently ruled out by
Ma and Wang \cite{MaWang5Q}. Thus, the situation is dual to that of Section \ref{sec:2Porder}. Van Dam,
Martin, and Muzychuk \cite{DMM2010pre} studied cometric $Q$-antipodal association schemes and showed that
these are uniform and related to linked systems. Three-class cometric $Q$-antipodal association schemes for
example are equivalent to linked systems of symmetric designs.
LeCompte, Martin, and Owens \cite{LMO2010} showed that  four-class cometric $Q$-antipodal and
$Q$-bipartite association schemes are equivalent to real mutually unbiased bases.
See also \cite{Suda2009pre}. See \cite{MMW07,DMM2010pre} for a comprehensive study on imprimitive cometric
association schemes. Martin and Williford \cite{MW2009EJC} proved the dual of the Bannai-Ito conjecture
discussed in Section \ref{sec:BIconjecture}: There are finitely many cometric association schemes with fixed
first multiplicity at least three. Kurihara \cite{Kur2011T} obtained a dual version of the spectral excess
theorem discussed in Section \ref{sec:spectralexcess}. See also \cite{KN2012JCTA,Nozaki2013pre}.

Concerning constructions of cometric (but not metric) association schemes, the main sources are block
designs, spherical designs, real mutually unbiased bases, and hemisystems and other strongly regular
decompositions of a strongly regular graph; see a survey by Bannai and Bannai \cite{BB2009EJC}, and also the
online table by Martin \cite{Martin2010www}. The `bipartite doubles' of the association schemes of the
Hermitian dual polar graphs $^2\A_{2D-1}(\sqrt{q})$ provide an infinite family of cometric but not metric
association schemes with unbounded diameter; cf.~\cite[pp.~313--315]{bi}. The `extended $Q$-bipartite double'
construction was introduced and worked out in detail by Martin, Muzychuk, and Williford \cite{MMW07}.
Penttila and Williford \cite{PenWil2011} constructed the first known infinite family of primitive cometric
association schemes that are not metric. Hollmann and Xiang \cite{HX2006JAC} earlier constructed a family of
$3$-class association schemes that have the same parameters as those found by Penttila and Williford, without
realizing it was cometric. See also \cite{Cossidente2013JAC}.
Recently, Moorhouse and Williford \cite{Williford2014AGT} constructed an infinite family of cometric $Q$-bipartite association schemes as certain `double covers' of the association schemes of the symplectic dual polar graphs $\C_D(q)$ with $q\equiv 1\pmod{4}$.
These association schemes have two $Q$-polynomial orderings, and are not metric; cf.~Section \ref{sec:Qmultipleordering}.
Such a `double cover' has a quadratic splitting field when $q$ is a non-square, and Moorhouse and Williford asked whether or not it is in general the `extended $Q$-bipartite double' of a primitive cometric but not metric association scheme when $q$ is a square.
If this is indeed the case, then this family would provide counterexamples to the conjecture of Bannai and Ito \cite[p.~312]{bi} mentioned at the end of Section \ref{sec:2Qpol}.

We refer the reader to \cite{BB2009EJC,martintanaka,DMM2010pre} for more
information and recent updates on cometric association schemes.

%% file: 17_tables.txt
In this section, we report progress on the `feasibility' and `uniqueness' of
the intersection arrays that were listed in the tables of `BCN'
\cite{bcn,BCNcoradd}, and some additional (larger) ones not in the tables. We
note that these tables are also available online in machine readable form
(although they are not exactly the same) \cite{tables}.

\subsection{Diameter \texorpdfstring{$3$}{3} and primitive}

\subsubsection{Uniqueness}

For the following intersection array, there is a unique distance-regular graph with that array:

\bigskip

\noindent $\{ 6,5,2; 1, 1,3\}$ ($v=57$): Perkel graph; Coolsaet and Degraer
\cite{CooDe05}.

\subsubsection{Existence}

For the following intersection arrays, there is a distance-regular graph with
that array:

\bigskip

\noindent $\{20,18,6; 1,1,15\}$ ($v=525$): unitary non-isotropics graph ($q=5$); \cite[Thm.~12.4.1]{bcn}.\\
$\{26,24,19; 1,3,8\}$ ($v=729$): Brouwer graph ($q=3$); Brouwer and Pasechnik \cite{BP2011}, see Section \ref{sec:Kasami}.\\
$\{31, 30, 17; 1, 2, 15\}$ ($v=1024$): Kasami graph ($q=2,j=2$); \cite[Thm.~11.2.1~(13)]{bcn}.\\
$\{110, 81, 12; 1, 18, 90\}$ ($v=672$): Moscow-Soicher graph (see Section \ref{sec: KoolenRiebeek}).

\subsubsection{Nonexistence}

The following intersection arrays are not feasible:

\bigskip

\noindent $\{5, 4, 3; 1, 1, 2\}$ ($v=56$): Fon-Der-Flaass \cite{FDF193}.\\
$\{13, 10, 7; 1, 2, 7\}$ ($v=144$): Coolsaet \cite{Co95}.\\
$\{19, 12, 5; 1, 4, 15\}$ ($v=96$): Coolsaet and Juri\v{s}i\'{c} \cite{CoJu08}, Neumaier \cite[p.~15]{BCNcoradd}.\\
$\{21, 16, 8; 1, 4, 14\}$ ($v=154$): Coolsaet \cite{Co05}.\\
$\{22, 16, 5; 1, 2, 20\}$ ($v=243$): Sumalroj and Worawannotai \cite{WS16}.\\
$\{35, 24, 8; 1, 6, 28\}$ ($v=216$): Juri\v{s}i\'{c} and Vidali \cite{JurVidpre}.\\
$\{36, 25, 8; 1, 4, 20\}$ ($v=352$): $\theta_1 = 14$ with multiplicity $32$ \cite[Thm.~4.4.4]{bcn}.\\
$\{40, 33, 8; 1, 8, 30\}$ ($v=250$): Juri\v{s}i\'{c} and Vidali \cite{JurVidpre}.\\
$\{44, 30, 5; 1, 3, 40\}$ ($v=540$): Koolen and Park \cite{KoPa10}, see Section \ref{sec:claws}.\\
$\{45, 30, 7; 1, 2, 27\}$ ($v=896$): Gavrilyuk and Makhnev \cite{GavMak2011T}.\\
$\{52, 35, 16; 1, 4, 28\}$ ($v=768$): Gavrilyuk and Makhnev \cite{GavMakpre}.\\
$\{55, 36, 11; 1, 4, 45\}$ ($v=672$): Bang \cite{Bapre} and Gavrilyuk \cite{Gav11}.\\
$\{56, 36, 9; 1, 3, 48\}$ ($v=855$): Bang \cite{Bapre} and Gavrilyuk \cite{Gav11}.\\
$\{65, 44,11; 1, 4, 55\}$ ($v=924$): Koolen and Park \cite{KoPa10}, see Section \ref{sec:claws}.\\
$\{69, 48,24; 1, 4, 46\}$ ($v=1330$): Gavrilyuk and Makhnev \cite{GavMakpre}.\\
$\{72, 45, 16; 1, 8, 54\}$ ($v=598$): $\theta_1 = 26$ with multiplicity $45$ \cite[Thm.~4.4.4]{bcn}.\\
$\{74, 54, 15; 1, 9,60\}$ ($v=630$): Coolsaet and Juri\v{s}i\'{c} \cite{CoJu08}.\\
$\{77, 60, 13; 1, 12, 65\}$ ($v=540$):  Coolsaet and Juri\v{s}i\'{c} \cite{CoJu08}.\\
$\{85, 54, 25; 1, 10, 45\}$ ($v=800$): $\theta_1 = 35$ with multiplicity $34$ \cite[Thm.~4.4.4]{bcn}.\\
$\{90,60, 12; 1, 12, 72\}$ ($v=616$): $\theta_1 = 27$ with multiplicity $48$ \cite[Thm.~4.4.4]{bcn}.\\
$\{104, 66, 8; 1, 12, 88\}$ ($v=729$): Urlep \cite{urlep12}.\\
$\{105, 102, 99; 1, 2, 35\}$ ($v=20608$): De Bruyn and Vanhove \cite{DV2012pre}.\\
$\{112,  77, 16; 1, 16, 88\}$ ($v=750$): $\theta_1 = 32$ with multiplicity $49$ \cite[Thm.~4.4.4]{bcn}.\\
$\{119, 96, 18; 1, 16, 102\}$ ($v=960$): Juri\v{s}i\'{c} and Vidali \cite{JurVidpre}.\\
$\{145, 84, 25; 1, 20, 105\}$ ($v=900$): $\theta_1 = 55$ with multiplicity $29$
\cite[Thm. 4.4.4]{bcn}.

%


\subsection{Diameter \texorpdfstring{$4$}{4} and primitive}

\subsubsection{Uniqueness}

For the following intersection array, there is a unique distance-regular graph with that array:

\bigskip

\noindent $\{ 280, 243, 144, 10; 1, 8, 90, 280\}$ ($v=22880$): Patterson graph
(see \cite[\S 13.7]{bcn}); Brouwer, Juri\v{s}i\'{c}, and Koolen
\cite{BrJuKo08}.

\subsubsection{Nonexistence}\label{sec:tablesD4prnon}
The following intersection arrays are not feasible:

\bigskip

\noindent $\{5,4,3,3;1,1,1,2\}$ ($v=176$): Fon-Der-Flaass \cite{FDF293}.\\
$\{39,32,20,2;1,4,16,30\}$ ($v=768$): Lambeck \cite{Lamb93}.\\
$\{ \mu(2\mu +1), (\mu-1)(2\mu +1), \mu^2, \mu; 1, \mu, \mu(\mu-1),
\mu(\mu+1)\}$ ($v=8\mu^2(\mu+1)$), $\mu \geq 2$: \linebreak Godsil and Koolen
\cite{GoKo95} (besides this family, in the tables also those with $\mu=4,5,6,7$
are explicitly mentioned: $\{ 36, 27, 16, 4; 1,4,12,36\}$,
$\{55,44,25,5; 1, 5,20, 55\}$, $\{ 78, 65, 36, 6;\linebreak 1, 6,
30, 78\}$, $\{105, 90, 49, 7; 1, 7, 42, 105\}$).\\ $\{50,48 ,48,
32; 1, 1, 9, 25\}$ ($v=31635$): De Bruyn \cite{DeB2010ElJC}.


\subsection{Diameter \texorpdfstring{$4$}{4} and bipartite}

\subsubsection{Existence}

For the following intersection array, there is a distance-regular graph with that array:

\bigskip

\noindent $\{45, 44, 36, 5; 1, 9, 40, 45\}$ ($v=486$): Koolen-Riebeek graph (see Section \ref{sec:
KoolenRiebeek}).

\subsubsection{Nonexistence}

The following intersection arrays are not feasible:

\bigskip

\noindent $\{36,35,27,6; 1,9,30,36\}$ ($v=324$): Galazidis (see \cite{tables})\footnote{$\G_1\cup \G_4$ would be a strongly regular graph with parameters $(324,57,0,12)$, which does not exist \cite{GavMak2005}.}.\\
 $\{36,35,33,3; 1,3,33,36\}$ ($v=912$): Huang (see \cite{tables})\footnote{The halved graphs would be the complement of a strongly regular graph with parameters $(456,35,10,2)$, which does not exist \cite{BN88} by Proposition \ref{shillaterw}.}.\\
  $\{88,87,77,4; 1,11,84,88\}$ ($v=1452$): Huang (see \cite{tables})\footnote{The halved graphs would be the complement of a strongly regular graph with parameters $(726,29,4,1)$, which does not exist by Proposition \ref{shillaterw}.}.

%
%


\subsection{Diameter \texorpdfstring{$4$}{4} and antipodal}

\subsubsection{Uniqueness}

For the following intersection arrays, there is a unique distance-regular graph with that array:

\bigskip

\noindent $\{32, 27, 8, 1; 1, 4, 27, 32\}$ ($v=315$): Soicher graph (see Section \ref{sec: Soicher and Meixner}); Soicher \cite{Soicher15}.\\
$\{ 45, 32, 12, 1,; 1, 6, 32, 45\}$ ($v=378$): $3.O_6^-(3)$-graph
(see \cite[\S
13.2C]{bcn}); Juri\v{s}i\'{c} and Koolen \cite{JuKo11}.\\
$\{56, 45, 16, 1; 1, 8, 45, 56\}$ ($v=486$): Soicher graph (see Section \ref{sec: Soicher and Meixner});
Brouwer \cite{AEBSoicher} (see also \cite[Thm.~11.4.6]{BCNcoradd}).\\
$\{117, 80, 24, 1; 1, 12, 80, 117\}$ ($v=1134$): $3.O_7(3)$-graph (see \cite[\S
13.2D]{bcn}); Juri\v{s}i\'{c} and Koolen \cite{JuKopre}.\\
$\{176, 135, 36, 1; 1, 12, 135, 176\}$ ($v=2688$): Meixner 4-cover (see Section
\ref{sec: Soicher and Meixner} and \cite[\S 12.4A]{BCNcoradd}); Juri\v{s}i\'{c}
and Koolen \cite{JuKopre}.

\subsubsection{Existence}

For the following intersection arrays, there is a distance-regular graph with
that array:

\bigskip

\noindent
$\{176, 135, 24, 1; 1, 24, 135, 176\}$ ($v=1344$): Meixner 2-cover (see Section \ref{sec: Soicher and
Meixner} and \cite[\S 12.4A]{BCNcoradd}).\\
$\{416,315,64,1;1,32,315,416\}$ ($v=5346$): Soicher graph (see Section \ref{sec: Soicher and Meixner}).

\subsubsection{Nonexistence}

The following intersection arrays are not feasible:

\bigskip

\noindent $\{32, 27, 6, 1; 1, 6, 27, 32\}$ ($v=210$): Soicher \cite{Soicher15}.\\
$\{32, 27, 9, 1; 1, 3, 27, 32\}$ ($v=420$): Soicher \cite{Soicher15}.\\
$\{45, 32, 9, 1; 1, 9, 32, 45\}$ ($v=252$): Juri\v{s}i\'{c}
and Koolen \cite{JuKo00EuJC}.\\
$\{45, 32, 15, 1; 1, 3, 32, 45\}$ ($v=756$):  Juri\v{s}i\'{c}
and Koolen \cite{JuKo00EuJC}.\\
$\{45, 40, 11, 1; 1, 1, 40, 45\}$ ($v=2352$): $\overline{\G}$ is of order $(5,8)$ with $c_2=12$, \cite[Thm.~4.2.7]{bcn}.\\
$\{ 56, 45, 12, 1; 1, 12, 45, 56\}$ ($v=324$): Brouwer \cite[Thm.~11.4.6]{BCNcoradd}.\\
$\{ 56, 45, 18, 1; 1, 6, 45, 56\}$ ($v=648$): Brouwer \cite[Thm.~11.4.6]{BCNcoradd}.\\
$\{ 56, 45, 20, 1; 1, 4, 45, 56\}$ ($v=972$): Brouwer \cite[Thm.~11.4.6]{BCNcoradd}.\\
$\{ 56, 45, 21, 1; 1, 3, 45, 56\}$ ($v=1296$): Brouwer \cite[Thm.~11.4.6]{BCNcoradd}.\\
$\{81, 56, 18, 1; 1, 9, 56, 81\}$ ($v=750$): Juri\v{s}i\'{c}
and Koolen \cite{JuKo11}.\\
$\{81, 56, 24, 1; 1, 3, 56, 81\}$ ($v=2250$): Juri\v{s}i\'{c}
and Koolen \cite{JuKo00EuJC}.\\
$\{ 96, 75, 24, 1; 1, 8, 75, 96\}$ ($v=1288$): Juri\v{s}i\'{c}
and Koolen \cite{JuKo11}.\\
$\{ 96, 75, 28, 1; 1, 4, 75, 96\}$ ($v=2576$): Juri\v{s}i\'{c}
and Koolen \cite{JuKo00EuJC}.\\
$\{115, 96, 32, 1; 1, 8, 96, 115\}$ ($v=1960$): Juri\v{s}i\'{c}
and Koolen \cite{JuKo00EuJC}.\\
$\{115, 96, 35, 1; 1, 5, 96, 115\}$ ($v=3136$): Juri\v{s}i\'{c}
and Koolen \cite{JuKo00EuJC}.\\
$\{115, 96, 36, 1; 1, 4, 96, 115\}$ ($v=3920$): Juri\v{s}i\'{c}
and Koolen \cite{JuKo00EuJC}.\\
$\{117, 80, 27, 1; 1, 9, 80, 117\}$ ($v=1512$): Juri\v{s}i\'{c}
and Koolen \cite{JuKo00EuJC}.\\
$\{117, 80, 30, 1; 1, 6, 80, 117\}$ ($v=2268$): Juri\v{s}i\'{c}
and Koolen \cite{JuKo00EuJC}.\\
$\{117, 80, 32, 1; 1, 4, 80, 117\}$ ($v=3402$): Juri\v{s}i\'{c}
and Koolen \cite{JuKo00EuJC}.\\
$\{175, 144, 25, 1; 1, 25, 144, 175\}$ ($v=1360$): Juri\v{s}i\'{c}
and Koolen \cite{JuKo00EuJC}.\\
$\{175, 144, 40, 1; 1, 10, 144, 175\}$ ($v=3400$): Juri\v{s}i\'{c}
and Koolen \cite{JuKo11}.\\
$\{176, 135, 40, 1; 1,8,135, 176\}$ ($v=4032$): Juri\v{s}i\'{c}
and Koolen \cite{JuKo00EuJC}.\\
$\{189, 128, 27, 1; 1, 27, 128, 189\}$ ($v=1276$): Juri\v{s}i\'{c}
and Koolen \cite{JuKo00EuJC}.\\
$\{189, 128, 36, 1; 1, 18, 128, 189\}$ ($v=1914$): Juri\v{s}i\'{c}
and Koolen \cite{JuKo11}.\\
$\{189, 128, 45, 1; 1, 9, 128, 189\}$ ($v=3828$): Juri\v{s}i\'{c}
and Koolen \cite{JuKo00EuJC}.\\
$\{204, 175, 40, 1; 1, 20, 175, 204\}$ ($v=2400$): Juri\v{s}i\'{c}
and Koolen \cite{JuKo00EuJC}.\\
$\{204, 175, 45, 1; 1, 15, 175, 204\}$ ($v=3200$): Juri\v{s}i\'{c}
and Koolen \cite{JuKo00EuJC}.\\
$\{261, 176, 54, 1; 1, 18, 176, 261\}$ ($v=3600$): Juri\v{s}i\'{c}
and Koolen \cite{JuKo00EuJC}.\\
$\{414, 350, 45, 1; 1, 45, 350, 414\}$ ($v=4050$): Juri\v{s}i\'{c} and Koolen
\cite{JuKo00EuJC}.

\subsection{Diameter \texorpdfstring{$5$}{5} and antipodal} 

\subsubsection{Uniqueness}

For the following intersection array, there is a unique distance-regular graph with that array:

\bigskip

\noindent $\{ 22, 20, 18, 2,1 ; 1,2, 9, 20, 22\}$ ($v=729$): coset graph of the
dual of the ternary Golay code; Blokhuis, Brouwer, and Haemers
\cite{BlokhuisBH07}.

\subsubsection{Nonexistence}

The following intersection arrays are not feasible:

\bigskip

\noindent $\{ 2 \mu^2 + \mu, 2\mu^2 + \mu -1, \mu^2, \mu, 1; 1, \mu, \mu^2,
2\mu^2 + \mu -1, 2\mu^2 + \mu\}$ ($v=4\mu^2(2\mu+3)$), $\mu \geq
2$: Coolsaet, Juri\v{s}i\'{c}, and Koolen \cite{CoJuKo08EuJC}
(besides this family, in the tables also those with $\mu=3,4,5,6,7$
are explicitly mentioned: $\{21, 20, 9,3,1; 1,3,9,20,21\}$,
$\{36,35,16,4,1;1,4,16,35,36\}$,\linebreak$\{55, 54, 25, 5, 1;
1,5,25,54,55\}$, $\{78, 77, 36, 6, 1; 1, 6, 36, 77, 78\}$, $\{105,
104, 49, 7, 1; 1, 7, 49, \linebreak104, 105\}$).\\ $\{105, 90, 49,
7, 1; 1, 7, 49, 90, 105\}$ ($v=2912$): $\theta_1 = 35$ with
multiplicity $78$ \cite[Thm.~4.4.4]{bcn}.

\subsection{Diameter \texorpdfstring{$5$}{5} and bipartite} 

\subsubsection{Uniqueness}

For the following intersection arrays, there is a unique distance-regular graph with that array:

\bigskip

\noindent $\{ 7, 6,6,4,4; 1,1, 3, 3, 7\}$ ($v=310$): Doubled Grassmann ($q=2$); Cuypers \cite{Cuypers92}, see Section \ref{sec:otherinfinite}.\\
$\{13, 12, 12, 9, 9; 1, 1, 4, 4, 13\}$ ($v=2420$): Doubled Grassmann ($q=3$);
Cuypers \cite{Cuypers92}, see Section \ref{sec:otherinfinite}.

\subsubsection{Nonexistence}

The following intersection array is not feasible:

\bigskip

\noindent $\{55, 54, 50, 35, 10; 1, 5, 20, 45, 55\}$ ($v=3500$): Vidali \cite{Vidali}.

\subsection{Diameter \texorpdfstring{$6$}{6} and imprimitive}

\subsubsection{Nonexistence}

The following intersection arrays are not feasible:

\bigskip

\noindent $\{15, 14, 12, 6, 1, 1; 1, 1, 3, 12, 14, 15\}$ ($v=1518$, antipodal):
Ivanov and Shpectorov \cite{IvaShp90}\\
$\{7, 6, 6, 5, 4, 3; 1, 1, 2, 3, 4, 7\}$ ($v=686$, bipartite): Koolen \cite{Ko292}.

%% file: 18_openproblems.txt
In this section we mention some important open problems. The most
important one is to classify all distance-regular graphs of large
enough diameter.

\subsection{The classification of distance-regular graphs of large diameter}

We first restrict the classification of distance-regular graphs to
the following three problems.

\begin{problem} Classify the $Q$-polynomial distance-regular graphs with large enough
diameter. \end{problem}

\begin{problem}
Classify the geometric distance-regular graphs with large
enough diameter.\end{problem}

\begin{problem}\label{problem:3}
Prove or disprove the following conjecture of Bannai and Ito
\cite[p.~312]{bi}: A primitive distance-regular graph with large
enough diameter is $Q$-polynomial.\footnote{This conjecture is true
in the case of distance-transitive graphs and also for thick
regular near polygons with $c_2 \geq 3$ and diameter at least $4$.
On the other hand, the recent construction of an infinite family of
imprimitive cometric but not metric association schemes by
Moorhouse and Williford \cite{Williford2014AGT} may be relevant to
disproving the dual version of this conjecture on cometric
association schemes (which was also raised by Bannai and Ito,
see Section \ref{sec:2Qpol}); cf. Section
\ref{sec:cometricschemes}.}
\end{problem}

\subsection{General problems}

\begin{problem} Generalize results on distance-regular graphs to larger classes of
graphs; for example, the Delsarte clique bound.
\end{problem}

\begin{problem} Which results on distance-regular graphs can be dualized to cometric association
schemes? See Section \ref{sec:cometricschemes}.
\end{problem}

Below, we will give a list of more specific (and typically smaller)
problems related to the classification of distance-regular graphs.

\subsection{\texorpdfstring{$Q$-polynomial}{Q-polynomial} distance-regular graphs}

The $Q$-polynomial distance-regular graphs fall into types I, IA,
II, IIA, IIB, IIC, and III from
\cite{bi}.
In Section \ref{Q-classification}, we showed that type IA cannot occur and that the distance-regular graphs
of types IIA, IIB, IIC, and III are completely determined.
For type II, the classification is known for $D\geq 14$.

\begin{problem} \begin{enumerate}[(i)]
\item Determine the graphs of type II with $D \leq 13$.
\item Determine the $Q$-polynomial distance-regular graphs of type I. A
subproblem is to determine the distance-regular graphs with classical
parameters with $b \neq 1$.
\end{enumerate}
\end{problem}

The classification of imprimitive $Q$-polynomial distance-regular
graphs is complete for $D\geq 12$, except for the classification of
the distance-regular graphs with the same intersection array as the
bipartite dual polar graphs. See Sections
\ref{sec:bipartiteDRG} and \ref{sec:antipodalDRG}.

\begin{problem} Classify the
graphs that have the same intersection array as the bipartite dual
polar graphs and the Hemmeter graphs for $D \geq 12$. Also, improve
the condition $D\geq 12$ for the bipartite case.
\end{problem}

\begin{problem}
Show that a $Q$-polynomial distance-regular graph with diameter $D$ is imprimitive if and only if $a_D = 0$.
\end{problem}

Lang and Terwilliger \cite{LT2007EJC} almost classified the $Q$-polynomial generalized odd graphs with diameter at
least three, leaving open one set of intersection arrays for $D=3$. See Section \ref{sec:Qalmostbipartite}.

\begin{problem} Classify the $Q$-polynomial generalized odd graphs with diameter three.
\end{problem}

\begin{problem}\label{prob:2P+Q}
Classify the primitive $Q$-polynomial distance-regular graphs with two
$P$-polynomial orderings and diameter three or four. See Section
\ref{sec:2Porder}. \end{problem}

\begin{problem} Classify the primitive distance-regular graphs with two
$Q$-polynomial orderings and diameter three. 
See Section \ref{sec:Qmultipleordering}.\end{problem}

\begin{problem}
Let $\eta_0, \eta_1, \ldots, \eta_D$ be a $Q$-polynomial ordering of the eigenvalues (that is, of the
corresponding idempotents), and let $\theta_0 > \theta_1 > \cdots > \theta_D$ be the natural ordering of the eigenvalues.
What can be said of the relation between the $\eta_i$ and the $\theta_j$?
For example, determine whether $\eta_1 \in \{ \theta_1, \theta_D, \theta_{D-1}\}$,
or whether $\{\theta_1, \theta_D\} \cap \{\eta_1, \eta_D\} \neq \emptyset $.
(For bipartite graphs and antipodal graphs, see \cite{Caughman1998GC,Pascasio1999JAC} and also Section \ref{sec:antipodalDRG}.)
\end{problem}

\begin{problem}
Classify the tight $Q$-polynomial distance-regular graphs with $D=4$.
\end{problem}

\subsection{Vanishing Krein parameters}

Bannai and Ito \cite[p.~312]{bi} conjectured that primitive
distance-regular graphs with large enough diameter are
$Q$-polynomial (see Problem \ref{problem:3}), and
so for such graphs most Krein parameters vanish.

\begin{problem}\label{problem:primalof}
Show that there exists a constant $C$ such that for every primitive distance-regular graph there exists a primitive idempotent, say, $E_1$, such that $|\{j:q_{1j}^i\ne 0\}|\leq C$ for all $i$.
\end{problem}

Note that this problem is dual to Problem \ref{problem:dualof}.

\begin{problem} Let $\theta_i$ be a tail (see Section \ref{sec:Qpolcharacterizations}). \begin{enumerate}[(i)]
\item  Determine whether $\theta_i \in \{\theta_1, \theta_D, \theta_{D-1}\}$. This last case ($\theta_i=\theta_{D-1}$)
    should only occur for bipartite distance-regular graphs.
\item Determine $j\ne 0,i$ such that $q^j_{ii} \neq 0$.
\item Is it possible to classify the distance-regular graphs with a light tail? Besides the antipodal $Q$-polynomial
    distance-regular graphs, there seem to be only the halved cubes and the Hermitian dual polar graphs
    $^2\A_{2D-1}(\sqrt{q})$.
\end{enumerate}
\end{problem}

\begin{problem} Sometimes, one can use the absolute bound to show that some Krein parameters vanish,
if one of the non-trivial eigenvalues has a small multiplicity.
\begin{enumerate}[(i)]
\item Find more conditions that imply that some Krein parameters vanish.
\item Study distance-regular graphs with no vanishing (non-trivial) Krein
parameters. Among the primitive distance-regular graphs with diameter three that are not
$Q$-polynomial, there are few that have a vanishing (non-trivial) Krein
parameter (we checked that there is only one on at most 100 vertices: the
Sylvester graph).
\end{enumerate}
\end{problem}

\begin{problem}
Juri\v{s}i\'{c}, Coolsaet, and others have used vanishing Krein parameters to show that certain families of intersection arrays are not feasible, and also to show the uniqueness of some distance-regular graphs by their intersection arrays.
See Section \ref{sec:triple intersection numbers}.
But it is not known when the method of vanishing Krein parameters gives enough extra information in order to decide the non-existence of certain intersection arrays.
Explore this.
\end{problem}

\subsection{Classical parameters}
\begin{problem} Characterize the classical distance-regular
graphs by their intersection arrays. \end{problem}

\begin{problem}
Show that $Q$-polynomial geometric distance-regular graphs which are not polygons have classical parameters.
\end{problem}

\begin{problem} Decide whether the Grassmann graphs $J_q(2D,D)$ are determined by their intersection
arrays. See Section \ref{sec:clasfamilies}. Decide whether there are other
distance-regular graphs than the twisted Grassmann graphs with the same
intersection arrays as the Grassmann graphs $J_q(2D+1,D)$.
\end{problem}

The following problem was posed by Vanhove in his thesis \cite[Pr.~8]{Vanhove2011PhD}.

\begin{problem}
Determine whether all distance-regular graphs with classical parameters
$(D, b, \alpha, \beta) = (D,-q,-(q + 1)/2,-((-q)^D + 1)/2)$, $q$ odd, are
subgraphs of the Hermitian dual polar graph $^2\A_{2D-1}(q)$ (for sufficiently large $D$).
See Theorem \ref{thm:b<-1} and the paragraph that follows it.
\end{problem}

\subsection{Geometric distance-regular graphs}

\begin{problem} Determine whether for a given integer $m \geq 2$, there are only finitely many
geometric distance-regular graphs with $D \geq 3$, $c_2 \geq 2$, and smallest eigenvalue $-m$, besides the Grassmann
graphs, Johnson graphs, bilinear forms graphs, and Hamming graphs. See Section \ref{sec:fixedsmallestev}. Note that the
generalized $2D$-gons of order $(q,1)$ for $D=3,4,6$ (which exist for all prime powers $q$) are geometric with smallest
eigenvalue $-2$, but they have $c_2 =1$; see also \cite[Thm.~4.2.16]{bcn}.\end{problem}

\begin{problem}
Classify the geometric distance-regular
graphs with $a_1 \geq 1$ and $c_2 \geq 2$.
\end{problem}

\begin{problem}
Classify the (non-bipartite) geometric distance-regular graphs that are also $Q$-polynomial.
\end{problem}

Let $\G$ be a geometric distance-regular graph with respect to a set of cliques $\cal C$. We call an induced subgraph $\Delta$ of $\G$ {\em a subspace}  if $\Delta$ is closed and for each edge $xy$ contained in
$\Delta$,  all the vertices of the clique $C \in {\cal C}$ containing $x$ and $y$ are in $\Delta$.

\begin{problem} Find sufficient and necessary conditions for the existence of subspaces, in a similar fashion as the $m$-boundedness condition. See Section \ref{sec: subgraphs}.
\end{problem}

\begin{problem} Classify the geometric distance-regular graphs having the property that for each pair of distinct vertices $x$ and $y$ there exists a (unique) subspace $\Phi(x,y)$ of diameter $d(x,y)$. This would be an
extension of the classification of thick regular near polygons with $c_2 \geq 2$.\end{problem}

\begin{problem} Complete the classification of thick regular near polygons with diameter at least $4$ and $c_2 \geq 2$. Only the case $c_2 =2$
and $c_3 >3$ needs to be considered. The case $c_2=1$ seems to be too difficult at the moment. See also Theorem
\ref{thm:RNPthick}.
\end{problem}

\begin{problem}
Let $\G$ be a geometric distance-regular graph with respect to $\cal C$. Define the dual graph on
vertex set $\cal C$, where two cliques are adjacent if they intersect. Determine when this dual graph is
distance-regular (this happens for the Johnson graphs and the Grassmann graphs).
\end{problem}

\subsection{The Bannai-Ito conjecture}

The Bannai-Ito conjecture can be interpreted as a diameter bound in terms of the valency, but the current proof (cf.~Section \ref{sec:proofBIconjecture}) gives a very bad bound.
On the other hand, all the known distance-regular graphs with valency $k$ at least three have $D \leq 2k+2$,
with equality only for the Foster graph.

\begin{problem}  Find a good diameter bound in terms of the valency.
\end{problem}

\begin{problem} Let $\G$ be a distance-regular graph with diameter $D$, head $h$, and valency $k$ at least three.
\begin{enumerate}[(i)]
\item Prove Ivanov's conjecture that $\ell(c_i, a_i, b_i) \leq h+1$ \cite[p.~191]{bcn},
\item Show that $D \leq (2k-3)h + 1$ except if $\G$ is the dodecahedron (in which case $D =5$, $h =1$, and $k=3$),
\item Show that if $c_i \geq 2$ for some $i$ then $|\{i : c_i = c\}| \leq \min \{i : c_i \geq 2\} - 1$ for $c=2,3,\dots, k-1$. See
    also Proposition \ref{BHKprop}.
\end{enumerate}
\end{problem}

All the
known distance-regular graphs except for the polygons have $h \leq 5$, with equality for the generalized
dodecagons.

\begin{problem}
Prove the conjecture of Suzuki \cite[Conj.~1.5.2]{Su99} that claims that there exists a constant $H$ such that all
distance-regular graphs with valency at least three have head $h \leq H$. \end{problem}

\begin{problem}
Suzuki's conjecture in the above problem would imply that the girth of a distance-regular graph is bounded. Prove the more specific (unpublished) conjecture by Koolen and Suzuki that the girth of a distance-regular graph with valency at least three is at most $12$.
\end{problem}

\begin{problem}
Show that every distance-regular graph with valency and diameter at least three has an
integral eigenvalue besides the valency. This was posed as a question by `BCN' \cite[p.~130]{bcn}. Clearly this is the case for bipartite distance-regular graphs and more
generally for geometric distance-regular graphs.
\end{problem}

\begin{problem}
Define the degree of an algebraic integer as the degree of its minimal polynomial. All the
eigenvalues of the known distance-regular graphs have degree at most three; with the
Biggs-Smith graph as the only example having an eigenvalue with degree equal to three. In this light we propose the following conjecture: {\em Every eigenvalue of
a distance-regular graph with valency at least three has degree at most three.} This conjecture could be a first step to
show the above conjecture of Suzuki.\end{problem}

\begin{problem} Develop theory for distance-regular graphs
with only integral eigenvalues. It is easy to show that for such graphs the diameter $D$ is bounded by $2k$, where $k$ is the valency.
If possible, improve this bound.
Also obtain a good bound for the head.
\end{problem}

\subsection{Combinatorics}

\begin{problem} Classify the $1$-homogeneous distance-regular graphs that are not bipartite nor a generalized odd graph. See Section
\ref{sec:homogeneity}.\end{problem}

\begin{problem}
Study distance-regular graphs that are locally strongly regular.
\end{problem}

\begin{problem} Determine whether $k_i = k_j$ for some distinct $i$ and $j$ with
$i+j \leq D$ and $k_D \geq 2$ implies that $k=2$. See Section
\ref{sec:charantipodal}.\end{problem}

\begin{problem} Given an integer $\alpha \geq 1$, determine whether there are only finitely many
distance-regular graphs with diameter at least three and $a_1 > \alpha$ such
that each local graph has second largest eigenvalue at most $\alpha$. See
Section \ref{sec:DIAMETEReigenvalues}.\end{problem}

It is known that if $c_2 \geq 2$ then $c_3 > c_2$ \cite[Thm.~5.4.1]{bcn}.

\begin{problem} Show that if $c_2 \geq 2$, then the $c_i$
are strictly increasing. \end{problem}

\begin{problem} Determine whether one needs to remove at least $2k - 2 - a_1$ vertices
in order to disconnect a distance-regular graph with diameter at least three
such that each resulting component has at least two vertices. Note that if this
is the case, then this is best possible because $2k - 2 - a_1$ is the size of the neighborhood
of an edge. Cioab\u{a}, Kim, and Koolen \cite{CKK12} showed that it is not
true for strongly regular graphs, but it is believed it may be true for strongly regular graphs
with $k \geq 2a_1 +3$.\end{problem}

In Section \ref{sec:2Porder} we discussed distance-regular graphs $\G$ with multiple $P$-polynomial orderings.

\begin{problem}\begin{enumerate}[(i)]
\item Classify the generalized odd graphs.
\item Determine new putative intersection arrays for generalized odd graphs.
\item Show that if $\G$ is a bipartite antipodal $2$-cover with diameter $2e$ and $e \geq 3$, then $\G$ is a
    $2e$-cube.
\item Classify the non-bipartite antipodal $2$-covers with diameter $D \geq 4$ that have a generalized odd graph as
    folded graph.
\item Show that if $\Delta=\G_{D}$ is also distance-regular with diameter $D$, then $\Delta$ is a generalized odd graph or a Taylor graph.
\end{enumerate}
\noindent
See also Problem \ref{prob:2P+Q}.
\end{problem}

The following problem is due to Fiol \cite[Conj.~3.6]{Fiol2001CPC}.
\begin{problem}
Show that a distance-regular graph with diameter at least $4$ is strongly distance-regular (cf.~Section \ref{sec:spectralexcess}) if and only if it is antipodal.
See Section \ref{sec:distance-D graph}.
\end{problem}

The following problem is due to Pyber \cite[Conj.~1, 3.1]{PyberHam} who showed that all but finitely many strongly regular graphs are Hamiltonian.
Recall also the well-known Lov\'{a}sz conjecture that all but finitely many connected vertex-transitive graphs are Hamiltonian.
\begin{problem}
\begin{enumerate}[(i)]
\item Show that all but finitely many distance-regular graphs are Hamiltonian.
\item In particular, show that all but finitely many distance-regular graphs with a fixed diameter $D$ are Hamiltonian.
\end{enumerate}
\noindent
The case $D=3$ seems to be the way to attack this problem.
\end{problem}

\begin{problem}
Determine which distance-regular graphs are core-complete.
Currently no distance-regular graphs are known that are not core-complete.
See Section \ref{sec:cores}.
\end{problem}

\begin{problem}\begin{enumerate}[(i)]
\item Show or disprove the conjecture of Neumaier, i.e., that all completely regular codes in the Hamming graphs with minimum distance at least $8$ are known.
\item Neumaier \cite{ Neu92} challenged his readers to classify the completely regular codes in
    the Hamming graphs with $nq \leq 48$. But there are many feasible intersection arrays with small covering radius, say $2$ and $3$.
We therefore would like to modify the challenge to classify the completely regular  codes in the Hamming graphs with $nq \leq 48$ with covering radius at least $4$.
\item Give more results on the intersection array of a completely regular code in a distance-regular graph.
\end{enumerate}
\end{problem}

\subsection{Uniqueness and non-existence}

\begin{problem} Decide whether the Livingstone graph is determined by its
intersection array $\{11,10,6,1;1,1,5,11\}$. See \cite[\S 13.5]{bcn}.\end{problem}

\begin{problem}
Construct a distance-regular graph with intersection array $\{7,6,6; 1,1,2\}$, or show that none exists. See
\cite[p.~148]{bcn}.\end{problem}

\begin{problem} Classify the non-bipartite distance-regular graphs with diameter at least four with the same intersection array as a
regular near polygon. Currently, the only known ones that are not regular near polygons are the Doob graphs. See
Section \ref{sec:rnp}.
\end{problem}

\begin{problem} Classify the distance-regular graphs that are locally Hoffman-Singleton, i.e., those with
intersection arrays $\{ 50, 42, 9; 1, 2, 42\}$ and $\{ 50, 42, 1; 1, 2, 50\}$.
See Section \ref{sec:existencequadranglesT}.\end{problem}

The following problem was raised by Bannai [private communication].
\begin{problem}
Determine whether the following is true:
if a distance-regular graph $\G$ with diameter $D\geq 4$ has intersection numbers $c_i = i^2$ and $b_i = (n-d-i)(d-i)$  for $i \leq D-1$ for some positive integers $n$ and $d$, then $\G$ is the folded Johnson graph with diameter $D$.
Similar problems can be formulated for other classical families of distance-regular graphs.
See, e.g., Theorem \ref{thmquotient} for the case of Hamming and Doob graphs.
\end{problem}

\subsection{The Terwilliger algebra}

\begin{problem}
Determine the structure of the Terwilliger algebra for the four families of forms graphs and also for the twisted Grassmann graphs.
\end{problem}

\begin{problem} Develop theory for $1$-thin distance-regular graphs with exactly three irreducible $\TT$-modules with endpoint $1$ up to isomorphism.
\end{problem}

\begin{problem}
The vectors $\mathbf{f}_t$ (cf.~\eqref{MacLean's vector}) can be defined for any (i.e., not necessarily bipartite) distance-regular graph.
Give more results using the positive semidefiniteness of the Gram matrix of some of the $\mathbf{f}_t$.
For example, is it possible to prove the Terwilliger tree bound (cf.~Section \ref{sec:tree_bound}) in this way?
\end{problem}

\begin{problem} An irreducible $\TT$-module $W$ is called {\em sharp} if $\dim E_t^{\ster} W = 1$, where $t$ is its endpoint.
Give sufficient and necessary conditions such that all the irreducible $\TT$-modules of a distance-regular graph are
sharp.\end{problem}

The following problem was raised by Terwilliger [private communication].

\begin{problem}
Find all the
$2$-thin bipartite distance-regular graphs with diameter $D\geq 4$ with at most two irreducible $\TT$-modules with endpoint $2$ up to isomorphism.
The
$Q$-polynomial bipartite distance-regular graphs are included in this class, and so are the
taut graphs; cf.~Section \ref{sec: Hadamard products}. By a recursion obtained by Curtin \cite{Curtin1999GC}, for these distance-regular graphs
the intersection array is determined by at most four parameters (besides $D$); cf.~Section \ref{sec:thinness}.
A related problem is to find a closed form for the intersection numbers.
\end{problem}

The following three problems were also posed by Terwilliger \cite{Talgebra92,Terwilliger1993N}.

\begin{problem}
Suppose $\G$ is a thin distance-regular graph with diameter $D$ and is not $Q$-polynomial.
Show that if $D$ is sufficiently large then one of the following holds.
\begin{enumerate}[{\textup (i)}]
\item $\G$ is bipartite, and the halved graph is thin and $Q$-polynomial.
\item $\G$ is antipodal, and the folded graph is thin and $Q$-polynomial.
\end{enumerate}
\end{problem}

\begin{problem}
Let $\G$ be a thin non-bipartite $Q$-polynomial distance-regular graph.
Take any two distinct vertices $x,y$.
Show that the minimal convex subgraph containing $x$ and $y$ is a thin $Q$-polynomial distance-regular graph with diameter $d(x,y)$.
If this claim turns out to be false, then find a simple additional assumption on $\G$ under which it is true.
\footnote{That $\G$ is non-bipartite was not assumed in \cite{Terwilliger1993N}. Terwilliger [private communication] pointed out that the Hemmeter graphs (and all distance-regular graphs with the same intersection array as (but not isomorphic to) the bipartite dual polar graphs) provide counterexamples.}
\end{problem}

\begin{problem}
Classify the thin $Q$-polynomial distance-regular graphs.
\end{problem}

The following problem was raised by Ito [private communication].
\begin{problem}
Study the structures of the irreducible $\TT$-modules of $Q$-polynomial distance-regular graphs from the point of view of the theory of tridiagonal systems, in particular as `tensor products' of Leonard systems.
See Section \ref{sec: TD systems}.
Is it true that each of the corresponding tridiagonal systems is a `tensor product' of at most two Leonard systems?
We note that this is indeed true for the irreducible $\TT$-modules with endpoint $1$; see Section \ref{sec:thinness-Q}.
\end{problem}
\begin{problem}
Is the isomorphism class of an irreducible $\TT$-module for a $Q$-polynomial distance-regular graph with $c_2 \geq 2$ and $a_1 \neq 0$ determined by its local eigenvalue and endpoint?
\end{problem}

\subsection{Other classification problems}

\begin{problem} Classify the distance-regular Terwilliger graphs. See Section \ref{sec:existencequadranglesT}.\end{problem}

There are infinitely many putative parameter sets $(v,k,\lambda,\mu)$ for strongly regular graphs with $\mu =1$ and
$\lambda =2$.

\begin{problem} Show that there are only finitely many strongly regular graphs with $\mu =1$. By Proposition \ref{mbounded} ($m=2,h=1$), a consequence of this would be that
there are finitely many distance-regular graphs with $c_3 =1$ and $a_1\neq a_2$. It would also contribute to the
classification of Terwilliger graphs with $\mu \geq 2$ by considering its local graphs. \end{problem}

\begin{problem} Generalize results for strongly regular graphs that do not yet have analogues for distance-regular graphs.\end{problem}

\begin{problem}
Fuglister \cite{Fuglister} uses $\mathrm{mod} \ p$ calculations for the multiplicities to show that if a distance-regular graph
has $D = h+1$, where $h$ is the head, then $h \leq 12$. Suzuki \cite[p.~87]{Su99} claims this can be generalized to the
case $D \leq h+3$. Brouwer and Koolen \cite{BK99} use a similar argument for the generation of feasible arrays for the
distance-regular graphs with valency $4$. Find more instances where this method works.
\end{problem}

\begin{problem}
For a coconnected distance-regular graph $\G$, show that the intersection number $c_2$ is bounded above by a function of $\frac{b_1}{\theta_1 +1}$. For strongly regular graphs, this is true as in that case
$\frac{b_1}{\theta_1 +1}$ is equal to $-\theta_2-1$ and hence it follows by Neumaier's \cite{Neu80} $\mu$-bound. It is also true for distance-regular graphs with $\theta_1 = b_1 -1$ by the classification of such graphs (see \cite[Thm.~4.4.11]{bcn}), as they all have $c_2 \leq 10$.
But even for $Q$-polynomial distance-regular graphs, a bound for $c_2$ in terms of $\frac{b_1}{\theta_1 +1}$ is not known to exist.
\end{problem}

\begin{problem} Classify the distance-regular graphs of order $(s,2)$. For $s=1$ these are precisely the distance-regular graphs with valency three.
For $s=2$, they were classified by Hiraki, Nomura, and Suzuki
\cite{HNS}. Yamazaki \cite{Y95} obtained some results for $s \geq
3$.
\end{problem}

\begin{problem} The fact that the multiplicities of the eigenvalues of a distance-regular graph are positive integers seems to be
one of the strongest known conditions for its intersection array. They are however expensive to compute.
Find necessary but easy to compute properties of the intersection numbers of distance-regular graphs that follow from the integrality of the multiplicities, such as the fact that for all prime numbers $p$ the number of
closed walks of length $p$ is divisible by $p$.
See Section \ref{sec:integralmultiplicities}.
\end{problem}

\begin{problem}
Bang \cite{Bang2014} showed that for $g\equiv 3 ~(\text{mod}~ 4)$ and $g=5$, there exists a positive $\epsilon_g$ such that if $\G$ is a triangle-free distance-regular graph with girth $g$ and large enough valency $k$,
then the smallest eigenvalue of $\Gamma$ is at least $(\epsilon_g -1)k$.
Show the same result for distance-regular graphs with odd girth. For non-bipartite distance-regular graphs with even
girth $g$  we can only expect that the smallest eigenvalue is at least $-k + C_g$ for a positive constant $C_g$, as the
folded $(2m+1)$-cube has smallest eigenvalue $-2m+1=-k+2$ and the Odd graph with valency $k$ has smallest eigenvalue $-k+1$.
\end{problem}

\begin{problem}
Show that for large enough $k$, the second largest eigenvalue of a
distance-regular graph with valency $k$ is at most $k-1$, as
conjectured by Koolen (unpublished). This would be best possible as
the Doubled Odd graphs have second largest eigenvalue $k-1$. For
distance-regular graphs with girth $6$, it was shown by Bang,
Koolen, and Park \cite{BKP2014pre}.
\end{problem}

\begin{problem}
\begin{enumerate}[(i)]
\item Determine the vertex-transitive distance-regular graphs.
\item Determine the distance-regular Cayley graphs.
\item Determine the arc-transitive distance-regular graphs.
\end{enumerate}
\noindent
See, e.g., \cite{ADJ15, Ma1994DCC,MP2003EJC,MP2007JCTB,MS2014JCTB}
for some results on distance-regular Cayley graphs.
\end{problem}

\begin{problem} Classify the distance-regular graphs with chromatic number $3$
and $a_1=1$. See Section \ref{sec:chromatic}.\end{problem}

We finish with a problem that is dual to Problem
\ref{problem:primalof} and that is relevant for the dual of Problem
\ref{problem:3}.

\begin{problem}\label{problem:dualof}
Show that there exists a constant $C$ such that for every primitive
cometric association scheme there exists an adjacency matrix, say,
$A_1$, such that $|\{j:p_{1j}^i\ne 0\}|\leq C$ for all $i$.
\end{problem}

%% file: bibliography.txt

%% file: drgrevision.arxiv.bbl
\begin{thebibliography}{999}

\bibitem{ADJ15}
Abdollahi, A., Van Dam, E.R., Jazaeri, M., Distance-regular Cayley
graphs with least eigenvalue $-2$, {\sl Des. Codes Cryptogr.}, to
appear;
\arxiv{1512.06019}.


\bibitem{AbiadQuasi14}
Abiad, A., Van Dam, E.R., Fiol, M.A.,
Some spectral and quasi-spectral characterizations of distance-regular graphs, preprint (2014);
\arxiv{1404.3973}.


\bibitem{Aldousbook} Aldous, D., Fill, J.A., Reversible Markov
    chains and random walks on graphs, manuscript (2001); \url{http://www.stat.berkeley.edu/~aldous/RWG/book}.


\bibitem{AlHaSmith} Alfuraidan, M.R., Hall, J.I., Smith's theorem and a
    characterization of the 6-cube as distance-transitive graph, {\sl  J.
    Algebraic Combin.} 24 (2006), 195--207.

\bibitem{AH09} Alfuraidan, M.R., Hall, J.I., Imprimitive distance-transitive
    graphs with primitive core of diameter at least 3, {\sl Michigan Math. J.} 58 (2009), 31--77.

\bibitem{AFM1999IJM}
Arad, Z., Fisman, E., Muzychuk, M.,
Generalized table algebras,
{\sl Israel J. Math.} 114 (1999), 29--60.

\bibitem{AH98} 
Araya, M., Hiraki, A., Distance-regular graphs with $c\sb i=b\sb {d-i}$ and
antipodal double covers, {\sl J. Algebraic Combin.} 8 (1998), 127--138.

\bibitem{AHJ96} 
Araya, M., Hiraki, A., Juri\v si\'c, A., Distance-regular graphs with $b\sb
t=1$ and antipodal double-covers, {\sl J. Combin. Theory Ser. B} 67 (1996), 278--283.

\bibitem{AHJ97} 
Araya, M., Hiraki, A., Juri\v si\'c, A., Distance-regular graphs with $b\sb
2=1$ and antipodal covers, {\sl European J. Combin.} 18 (1997),
243--248.

\bibitem{Babaimetric} Babai, L., On the order of uniprimitive permutation
    groups, {\sl Ann. of Math.} 113 (1981), 553--568.



\bibitem{BGSV2012B} Bachoc, C., Gijswijt, D.C., Schrijver, A., Vallentin, F.,
    Invariant semidefinite programs, {\sl Handbook on Semidefinite, Conic and
    Polynomial Optimization} (M.F. Anjos, J.B. Lasserre, eds.), Springer, New York, 2012, pp.~219--269;
    \arxiv{1007.2905}.

\bibitem{BV2008JAMS}
Bachoc, C., Vallentin, F.,
New upper bounds for kissing numbers from semidefinite programming,
{\sl J. Amer. Math. Soc.} 21 (2008), 909--924;
\arxiv{math/0608426}.

\bibitem{BHW09} Bai, Y., Huang, T., Wang, K.,
    Error-correcting pooling designs associated with some distance-regular graphs,
    {\sl Discrete Appl. Math.} 157 (2009), 3038--3045.

\bibitem{Baileyimprimitive} Bailey, R.F., On the metric dimension of imprimitive distance-regular
    graphs, preprint (2013);
    \arxiv{1312.4971}.

\bibitem{Baileysmall}
Bailey, R.F.,
The metric dimension of small distance-regular and strongly regular graphs,
{\sl Australas. J. Combin.} 62 (2015), 18--34;
\arxiv{1312.4973}.



\bibitem{Johnsonmetric} Bailey, R.F., C\'{a}ceres, J., Garijo, D.,
    Gonz\'{a}lez, A., M\'{a}rquez, A., Meagher, K., Puertas, M.L., Resolving sets for Johnson and Kneser graphs, {\sl European J. Combin.} 34 (2013), 736--751; \arxiv{1203.2660}.

\bibitem{Metricdimension} Bailey, R.F., Cameron, P.J., Base size, metric
    dimension and other invariants of groups and graphs, {\sl Bull. London Math. Soc.} 43 (2011), 209--242.

\bibitem{BaileyMeagher} Bailey, R.F., Meagher, K., On the metric dimension
    of Grassmann graphs, {\sl Discrete Math. Theor. Comput. Sci.} 13 (2011), 97--104;
    \arxiv{1010.4495}.


\bibitem{BBCGKS09} Ballinger, B., Blekherman, G., Cohn, H., Giansiracusa, N.,
    Kelly, E., Sch\"urmann, A., Experimental study of energy-minimizing point
    configurations on spheres, {\sl Exp. Math.} 18 (2009), 257--283;
\arxiv{math/0611451}.




\bibitem{BGR2010BLMS} Bamberg, J., Giudici, M., Royle, G.F., Every flock generalized quadrangle has a hemisystem,
{\sl Bull. London Math. Soc.} 42 (2010), 795--810;
    \arxiv{0912.2574}.


\bibitem{Bapre}  Bang, S., Geometric distance-regular graphs without 4-claws,
    {\sl Linear Algebra Appl.} 438 (2013), 37--46;
\arxiv{1101.0440}.

\bibitem{Bang2014} Bang, S., Distance-regular graphs with an eigenvalue $-k<\theta \leq 2-k$, {\sl Electron.
    J. Combin.} 21 (2014), P1.4.


\bibitem{BDK08} Bang, S., Van Dam, E.R., Koolen, J.H.,
    Spectral characterization of the Hamming graphs, {\sl Linear Algebra
    Appl.} 429 (2008), 2678--2686.

\bibitem{BaDuKoMo09}
Bang, S., Dubickas, A., Koolen, J.H., Moulton, V.,
There are only finitely many distance-regular graphs of fixed valency greater than two,
{\sl Adv. Math.} 269 (2015), 1--55;
\arxiv{0909.5253}.

\bibitem{BFK09} Bang, S., Fujisaki, T., Koolen, J.H., The
    spectra of the local graphs of the twisted Grassmann graphs,
    {\sl European J. Combin.} 30 (2009), 638--654.

\bibitem{BGK4claws} Bang, S., Gavrilyuk, A.L., Koolen, J.H., The distance-regular graphs without 4-claws, in
    preparation.

\bibitem{BHK06} 
Bang, S., Hiraki, A., Koolen, J.H., Improving diameter bounds for
distance-regular graphs, {\sl European J. Combin.} 27 (2006), 79--89.

\bibitem{BaHiKo07} Bang, S., Hiraki, A., Koolen, J.H., Delsarte
    clique graphs, {\sl European J. Combin.} 28 (2007),
    501--516.

\bibitem{BaHiKo10} Bang, S., Hiraki, A., Koolen, J.H.,
    Delsarte set graphs with small $c_2$, {\sl
    Graphs Combin.} 26 (2010), 147--162.




\bibitem{BK08} Bang, S., Koolen, J.H., Graphs cospectral with
    $H(3,q)$ which are locally disjoint union of at most three
    complete graphs, {\sl Asian-Eur. J. Math.} 1 (2008), 147--156.

\bibitem{BaKo14} Bang, S., Koolen, J.H., On geometric distance-regular graphs with
diameter three, {\sl European J. Combin.} 36 (2014), 331--341.

\bibitem{BaKo??} Bang, S., Koolen, J.H., A sufficient condition for
    distance-regular graphs to be geometric, in preparation.



\bibitem{BKM03} 
Bang, S., Koolen, J.H., Moulton, V., A bound for the number of columns $l\sb
{(c,a,b)}$ in the intersection array of a distance-regular graph, {\sl European
J. Combin.} 24 (2003), 785--795.

\bibitem{BaKoMo07} Bang, S., Koolen, J.H., Moulton, V., Two
    theorems concerning the Bannai-Ito conjecture, {\sl European J. Combin.} 28 (2007), 2026--2052.

\bibitem{BKP2014pre}
Bang, S., Koolen, J.H., Park, J., Some results on the eigenvalues
of distance-regular graphs, {\sl Graphs Combin.} 31 (2015),
1841--1853.

\bibitem{BB2005EJC} Bannai, E., Bannai, E., A note on the spherical embeddings
    of strongly regular graphs, {\sl European J. Combin.} 26 (2005),
    1177--1179.

\bibitem{BB2009EJC}
Bannai, E., Bannai, E.,
A survey on spherical designs and algebraic combinatorics on spheres,
{\sl European J. Combin.} 30 (2009), 1392--1425.

\bibitem{bi} Bannai, E., Ito, T., {\sl Algebraic
    Combinatorics I: Association Schemes},
    Ben\-ja\-min-Cummings, Menlo Park, 1984.

\bibitem{BI87} Bannai, E., Ito, T., On distance-regular
    graphs with fixed valency, {\sl Graphs Combin.} 3 (1987),
    95--109.

\bibitem{BI89} Bannai, E., Ito, T., On distance-regular
    graphs with fixed valency, IV, {\sl European J. Combin.} 10
    (1989), 137--148.



\bibitem{BeFa00} Beezer, R.A., Farrell, E.J., The matching polynomial of a
    distance-regular graph, {\sl Int. J. Math. Math. Sci.} 
    23 (2000), 89--97.


\bibitem{Bels98} Belsley, E.D., Rates of convergence of random walk on distance
    regular graphs, {\sl Probab. Theory Related Fields} 112 (1998), 493--533.


\bibitem{BF98} 
Bending, T.D., Fon-Der-Flaass, D., Crooked functions, bent functions, and
distance regular graphs, {\sl Electron. J. Combin.} 5 (1998), R34.

\bibitem{BCEM12} Bendito, E., Carmona, A., Encinas, A.M., Mitjana, M.,
    Distance-regular graphs having the $M$-property, {\sl Linear Multilinear
    Algebra} 60 (2012), 225--240.

\bibitem{BKMT2008IJQI}
Best, A., Kliegl, M., Mead-Gluchacki, S., Tamon, C.,
Mixing of quantum walks on generalized hypercubes,
{\sl Int. J. Quantum Inf.} 6 (2008), 1135--1148;
\arxiv{0808.2382}.

\bibitem{Bier1987DM}
Bier, T.,
A family of nonbinary linear codes,
{\sl Discrete Math.} 65 (1987), 47--51.

\bibitem{Bierbrauer} Bierbrauer, J., A family of crooked
    functions, {\sl Des. Codes Cryptogr.} 50 (2009),
    235--241.

\bibitem{biggs} Biggs, N., {\sl Algebraic Graph Theory},
    Cambridge University Press, Cambridge, 1974, second
    edition, 1993.



\bibitem{biggsp} Biggs, N.L., Potential theory on distance-regular graphs, {\sl
    Comb. Probab. Comput.} 2 (1993), 243--255.
    Also: {\sl Combinatorics, Geometry and Probability: A Tribute to Paul Erd\H{o}s}
    (B. Bollob\'{a}s, A. Thomason, eds.), Cambridge University Press, Cambridge, 1997, pp.~107--119.

\bibitem{Biggschipfiring} Biggs, N., Chip firing on distance-regular graphs,
    {\sl CDAM Research Report Series}, LSE-CDAM-96-11, 1996.

\bibitem{biggs97} Biggs, N., Algebraic potential theory on graphs, {\sl Bull.
    London Math. Soc.} 29 (1997), 641--682.

\bibitem{BBS} Biggs, N.L., Boshier, A.G., Shawe-Taylor, J., Cubic
    distance-regular graphs, {\sl J. London Math. Soc.} (2) 33 (1986),
    385--394.

\bibitem{Billio} Billio, M., Cal\`{e}s, L., Gu\'{e}gan, D., Portfolio
    symmetry and momentum, {\sl European J. Oper. Res.} 214 (2011),
    759--767.

\bibitem{BlBr97} Blokhuis, A., Brouwer, A.E., Determination of the distance-regular graphs without $3$-claws,
  {\sl Discrete Math.} 163 (1997), 225--227.

\bibitem{BlokhuisBH07} Blokhuis, A., Brouwer, A.E.,
    Haemers, W.H., On 3-chromatic distance-regular graphs,
    {\sl Des. Codes Cryptography} 44 (2007), 293--305.



\bibitem{bloom} Bloom, G.S., Quintas, L.W., Kennedy, J.W.,
    Distance degree regular graphs, {\sl The Theory and Applications of
    Graphs, Proceedings of the 4th international conference on the theory and applications of graphs, Kallamazoo,
    1980} (G. Chartrand et al., eds.), Wiley, New York, 1981, pp.~95--108.

\bibitem{Bockting-Conrad2012LAA}
Bockting-Conrad, S.,
Two commuting operators associated with a tridiagonal pair,
{\sl Linear Algebra Appl.} 437 (2012), 242--270;
\arxiv{1110.3434}.

\bibitem{vanbon07} Van Bon, J., Finite primitive
    distance-transitive graphs, {\sl European J. Combin.} 28 (2007), 517--532.

\bibitem{BB1988P} Van Bon, J.T.M., Brouwer, A.E., The distance-regular
    antipodal covers of classical distance-regular graphs, {\sl Combinatorics}
    (A. Hajnal, L. Lov\'{a}sz, V.T. S\'{o}s, eds.),
    Colloquia Mathematica Societatis J\'{a}nos Bolyai, vol.~52,
    North-Holland, Amsterdam, 1988, pp.~141--166.

\bibitem{BondyMurty} Bondy, J.A., Murty, U.S.R., {\sl Graph Theory},
Graduate Texts in Mathematics, vol.~244,
    second edition, Springer, New York, 2008.

\bibitem{BR2000IEEE}
Borges, J., Rif\`{a}, J.,
On the nonexistence of completely transitive codes,
{\sl IEEE Trans. Inform. Theory} 46 (2000), 279--280.

 \bibitem{BRZ2001IEEE}
Borges, J., Rif\`{a}, J., Zinoviev, V.A.,
Nonexistence of completely transitive codes with error-correcting capability $e>3$,
{\sl IEEE Trans. Inform. Theory} 47 (2001), 1619--1621.

\bibitem{BoRiZiDM08}
Borges, J., Rif\`{a}, J., Zinoviev, V.A.,
On non-antipodal binary completely regular codes,
{\sl Discrete Math.} 308 (2008), 3508--3525.

\bibitem{BRZ2010AMC}
Borges, J., Rif\`{a}, J., Zinoviev, V.A.,
On $q$-ary linear completely regular codes with $\rho=2$ and antipodal dual,
{\sl Adv. Math. Commun.} 4 (2010), 567--578;
\arxiv{1002.4510}.

\bibitem{BRZ2014DCC} Borges, J., Rif\`{a}, J., Zinoviev, V.A., New families of completely regular codes and
    their corresponding distance regular coset graphs, {\sl Des. Codes Cryptogr.} 70 (2014), 139--148.

\bibitem{Bose} Bose, R.C., Strongly regular graphs, partial geometries and
    partially balanced designs, {\sl Pacific J. Math.} 13 (1963), 389--419.

\bibitem{BBMM08} Bracken, C., Byrne, E., Markin, N., McGuire,
    G., A few more quadratic APN functions, {\sl Cryptogr. Commun.}
    3 (2011), 43--53;
    \arxiv{0804.4799}.


\bibitem{BKOW2013pre}
Braun, M., Kohnert, A.,  \"{O}sterg\r{a}rd, P.R.J., Wassermann, A.,
Large sets of $t$-designs over finite fields,
{\sl J. Combin. Theory Ser. A} 124 (2014), 195--202;
\arxiv{1305.1455}.

\bibitem{Brouwer1981} Brouwer, A.E., The nonexistence of a regular near hexagon on $1408$ points,
{\sl Math. Centr. Report} ZW163, Amsterdam (1981).

\bibitem{Brouwer1983IEEE}
Brouwer, A.E.,
On the uniqueness of a certain thin near octagon (or partial $2$-geometry, or parallelism) derived from the binary Golay code,
{\sl IEEE Trans. Inform. Theory} 29 (1983), 370--371.

\bibitem{AEBSoicher} Brouwer, A.E., The Soicher graph - an
    antipodal $3$-cover of the second subconstituent of the McLaughlin graph,
    manuscript (1990).

\bibitem{Brouwer1990DM}
Brouwer, A.E.,
A note on completely regular codes,
{\sl Discrete Math.} 83 (1990), 115--117.

\bibitem{Brouwer1993DM}
Brouwer, A.E.,
On complete regularity of extended codes,
{\sl Discrete Math.} 117 (1993), 271--273.



\bibitem{BCNcoradd} Brouwer, A.E., Corrections and
    additions to the book `Distance-regular Graphs',
    \url{http://www.win.tue.nl/~aeb/drg/BCN-ac.ps.gz}
    (October 2008).



\bibitem{tables} Brouwer, A.E., Parameters of distance-regular
    graphs, \url{http://www.win.tue.nl/~aeb/drg/drgtables.html} (June 2011).


\bibitem{webchromatic} Brouwer, A.E., Cube-like graphs,
    \url{http://www.win.tue.nl/~aeb/graphs/cubelike.html} (January 2012).





\bibitem{bcohen} Brouwer, A.E., Cohen, A.M., Local recognition of Tits geometries of classical type,
{\sl Geom. Dedicata} 20 (1986), 181--199.

\bibitem{bcn} Brouwer, A.E., Cohen, A.M., Neumaier, A.,
    {\sl Distance-Regular Graphs}, Springer-Verlag, Berlin,
    1989.

\bibitem{BrEt11} Brouwer, A.E., Etzion, T., Some new distance-$4$ constant
    weight codes, {\sl Adv. Math. Commun.} 5 (2011),
    417--424.

\bibitem{BF2014pre}
Brouwer, A.E., Fiol, M.A.,
Distance-regular graphs where the distance-$d$ graph has fewer distinct eigenvalues,
{\sl Linear Algebra Appl.} 480 (2015), 115--126;
\arxiv{1409.0389}.

\bibitem{BGKM03} Brouwer, A.E., Godsil, C.D., Koolen, J.H.,
    Martin, W.J., Width and dual width of subsets in polynomial
    association schemes,  {\sl J. Combin. Theory Ser. A} 102 (2003), 255--271.

\bibitem{BH93} Brouwer, A.E., Haemers, W.H., The Gewirtz
    graph: An exercise in the theory of graph spectra, {\sl European
    J. Combin.} 14 (1993), 397--407.

\bibitem{bh95} Brouwer, A.E., Haemers, W.H.,
    Association schemes, {\sl Handbook of Combinatorics}
    Vol.~1, 2 (R.L.~Graham, M.~Gr\"{o}tschel, L.~Lov\'{a}sz, eds.), Elsevier, Amsterdam, 1995, pp.~747--771.

\bibitem{BH05} Brouwer, A.E., Haemers, W.H., Eigenvalues and perfect matchings,
    {\sl Linear Algebra Appl.} 395 (2005), 155--162.

\bibitem{BrHa} Brouwer, A.E., Haemers, W.H., {\sl Spectra of Graphs}, Springer, New York, 2012;
    \url{http://homepages.cwi.nl/~aeb/math/ipm/}.

\bibitem{BH1992EJC}
Brouwer, A.E., Hemmeter, J.,
A new family of distance-regular graphs and the $\{0,1,2\}$-cliques in dual polar graphs,
{\sl European J. Combin.} 13 (1992), 71--79.

\bibitem{BHW1995EJC} Brouwer, A.E., Hemmeter, J., Woldar, A.,
The complete list of maximal cliques of
    $\mathrm{Quad}(n,q)$, $q$ odd, {\sl European J. Combin.} 16 (1995), 107--110.

\bibitem{BrJuKo08} Brouwer, A.E., Juri\v si\'c, A., Koolen,  J.H.,
    Characterization of the Patterson graph, {\sl J. Algebra} 320 (2008),
    1878--1886.

\bibitem{BK99} Brouwer, A.E., Koolen, J.H., The distance-regular graphs of
    valency four, {\sl J. Algebraic Combin.} 10 (1999), 5--24.

\bibitem{BrKo09} Brouwer, A.E., Koolen, J.H., The vertex-connectivity
    of a distance-regular graph, {\sl European J. Combin.} 30 (2009), 668--673.

\bibitem{BKR98} 
Brouwer, A.E., Koolen, J.H., Riebeek, R.J., A new distance-regular graph
associated to the Mathieu group $M\sb {10}$, {\sl J. Algebraic Combin.} 8
(1998), 153--156.

\bibitem{BN88} Brouwer, A.E., Neumaier, A.,
A remark on partial linear spaces of girth 5 with an application to strongly regular graphs, {\sl Combinatorica} 8 (1988), 57--61.

\bibitem{BP2011} Brouwer, A.E., Pasechnik, D.V., Two distance-regular graphs,
    {\sl J. Algebraic Combin.} 36 (2012), 403--407;
    \arxiv{1107.0475}.

\bibitem{BWilbrink} Brouwer, A.E., Wilbrink, H.A., The structure of near polygons with quads,
{\sl Geom. Dedicata} 14 (1983), 145--176.

\bibitem{Brown2013pre}
Brown, G.M.F.,
Hypercubes, Leonard triples and the anticommutator spin algebra,
preprint (2013);
\arxiv{1301.0652}.

\bibitem{BCL08} Budaghyan, L., Carlet, C., Leander, G., Two classes of
    quadratic APN binomials inequivalent to power functions, {\sl IEEE Trans. Inform. Theory} 54 (2008), 4218--4229.

\bibitem{BCL09} Budaghyan, L., Carlet, C., Leander, G.,
    Constructing new APN functions from known ones, {\sl Finite
    Fields Appl.} 15 (2009), 150--159.

\bibitem{BN1992MC} Bussemaker, F.C., Neumaier, A., Exceptional graphs with
    smallest eigenvalue $-2$ and related problems, {\sl Math. Comp.} 59 (1992),
    583--608.

\bibitem{DeCaenFonDerFlaass} De Caen, D.,
    Fon-Der-Flaass, D., Distance regular covers of complete graphs
    from Latin squares, {\sl Des. Codes Cryptogr.} 34 (2005), 149--153.

\bibitem{CMM95} 
De Caen, D., Mathon, R., Moorhouse, G.E., A family of antipodal
distance-regular graphs related to the classical Preparata codes, {\sl J.
Algebraic Combin.} 4 (1995), 317--327.

\bibitem{CDDFS2013JCTA}
C\'{a}mara, M., Dalf\'{o}, C., Delorme, C., Fiol, M.A., Suzuki, H.,
Edge-distance-regular graphs are distance-regular,
{\sl J. Combin. Theory Ser. A} 120 (2013), 1057--1067;
\arxiv{1210.5649}.

\bibitem{CDKP2013} C\'{a}mara, M., Van Dam, E.R., Koolen, J.H., Park, J.,
Geometric aspects of $2$-walk-regular graphs, {\sl Linear Algebra Appl.} 439 (2013), 2692--2710;
\arxiv{1304.2905}.

\bibitem{CFFG2009EJC}
C\'{a}mara, M., F\`{a}brega, J., Fiol, M.A., Garriga, E.,
Some families of orthogonal polynomials of a discrete variable and their applications to graphs and codes,
{\sl Electron. J. Combin.} 16 (2009), R83.

\bibitem{CFFG2010EJC}
C\'{a}mara, M., F\`{a}brega, J., Fiol, M.A., Garriga, E.,
Combinatorial vs. algebraic characterizations of completely pseudo-regular codes,
{\sl Electron. J. Combin.} 17 (2010), R37.

\bibitem{CFFG2014DAM}
C\'{a}mara, M., F\`{a}brega, J., Fiol, M.A., Garriga, E.,
On the local spectra of the subconstituents of a vertex set and completely pseudo-regular codes,
{\sl Discrete Appl. Math.} 176 (2014), 12--18;
\arxiv{1212.3815}.

\bibitem{CGS1978NAWIM}
Cameron, P.J., Goethals, J.-M., Seidel, J.J.,
The Krein condition, spherical designs, Norton algebras and permutation groups,
{\sl Nederl. Akad. Wetensch. Indag. Math.} 40 (1978), 196--206.

\bibitem{CGS1978JA}
Cameron, P.J., Goethals, J.-M., Seidel, J.J.,
Strongly regular graphs having strongly regular subconstituents,
{\sl J. Algebra} 55 (1978), 257--280.

\bibitem{CK2008JAMS}
Cameron, P.J., Kazanidis, P.A.,
Cores of symmetric graphs,
{\sl J. Aust. Math. Soc.} 85 (2008), 145--154.

\bibitem{Cau97}  Caughman IV, J.S., Intersection numbers of bipartite
    distance-regular graphs, {\sl Discrete Math.} 163 (1997), 235--241.

\bibitem{Caughman1998GC}
Caughman IV, J.S.,
Spectra of bipartite $P$- and $Q$-polynomial association schemes,
{\sl Graphs Combin.} 14 (1998), 321--343.

\bibitem{Caughman1999DM} Caughman IV, J.S., The Terwilliger algebras of
    bipartite $P$- and $Q$-polynomial schemes, {\sl Discrete Math.} 196 (1999),
    65--95.

\bibitem{Caughman1999JCTB}
Caughman IV, J.S.,
Bipartite $Q$-polynomial quotients of antipodal distance-regular graphs,
{\sl J. Combin. Theory Ser. B} 76 (1999), 291--296.

\bibitem{Caughman2003EJC}
Caughman IV, J.S.,
The last subconstituent of a bipartite $Q$-polynomial distance-regular graph,
{\sl European J. Combin.} 24 (2003), 459--470.

\bibitem{Caughman2004GC} Caughman IV, J.S., Bipartite $Q$-polynomial
    distance-regular graphs, {\sl Graphs Combin.} 20 (2004), 47--57.

\bibitem{CMT2005DM} Caughman IV, J.S., MacLean, M.S., Terwilliger, P., The
    Terwilliger algebra of an almost-bipartite $P$- and $Q$-polynomial
    association scheme, {\sl Discrete Math.} 292 (2005), 17--44;
    \arxiv{math/0508401}.

\bibitem{CW2005JAC}
Caughman IV, J.S., Wolff, N.,
The Terwilliger algebra of a distance-regular graph that supports a spin model,
{\sl J. Algebraic Combin.} 21 (2005), 289--310.

\bibitem{Cerzo2010LAA}
Cerzo, D.R.,
Structure of thin irreducible modules of a $Q$-polynomial distance-regular graph,
{\sl Linear Algebra Appl.} 433 (2010), 1573--1613;
\arxiv{1003.5368}.

\bibitem{CS2009EJC}
Cerzo, D.R., Suzuki, H.,
Non-existence of imprimitive $Q$-polynomial schemes of exceptional type with $d=4$,
{\sl European J. Combin.} 30 (2009), 674--681.

\bibitem{Chan2013pre}
Chan, A.,
Complex Hadamard matrices, instantaneous uniform mixing and cubes,
preprint (2013);
\arxiv{1305.5811}.

\bibitem{CGM2003TAMS}
Chan, A., Godsil, C., Munemasa, A.,
Four-weight spin models and Jones pairs,
{\sl Trans. Amer. Math. Soc.} 355 (2003), 2305--2325;
\arxiv{math/0111035}.

\bibitem{comcov} Chandra, A.K., Raghavan, P., Ruzzo, W.L., Smolensky, R.,
    Tiwari, P., The electrical resistance of a graph captures its commute and
    cover times, {\sl Comput. Complexity} 6 (1996), 312--340.


\bibitem{CP1991CMP}
Chari, V., Pressley, A.,
Quantum affine algebras,
{\sl Comm. Math. Phys.} 142 (1991), 261--283.

\bibitem{CH99} 
Chen, Y., Hiraki, A., On the non-existence of certain distance-regular graphs,
{\sl Kyushu J. Math.} 53 (1999), 1--15.

\bibitem{CHK98} 
Chen, Y., Hiraki, A., Koolen, J.H., On distance-regular graphs with $c\sb 4=1$
and $a\sb 1\ne a\sb 2$, {\sl Kyushu J. Math.} 52 (1998), 265--277.


\bibitem{Chihara1987SIAM}
Chihara, L.,
On the zeros of the Askey-Wilson polynomials, with applications to coding theory,
{\sl SIAM J. Math. Anal.} 18 (1987), 191--207.



\bibitem{CS1995GC}
Chihara, L., Stanton, D.,
A matrix equation for association schemes,
{\sl Graphs Combin.} 11 (1995), 103--108.

\bibitem{Chvatalmetric} Chv\'{a}tal, V., Mastermind, {\sl Combinatorica} 3
    (1983), 325--329.

\bibitem{Chvatal} Chv\'{a}tal, V., Comparison of two techniques for proving
    nonexistence of strongly regular graphs, {\sl Graphs Combin.} 27 (2011), 171--175;
\arxiv{0906.5389}.


\bibitem{CKK12} Cioab\u{a}, S.M., Kim, K., Koolen, J.H., On a conjecture of
    Brouwer involving the connectivity of strongly regular graphs, {\sl J.
    Combin. Theory Ser. A} 119 (2012), 904--922;
\arxiv{1105.0796}.


\bibitem{CioaKopre} Cioab\u{a}, S.M., Koolen, J.H., On the connectedness of the complement of a ball in distance-regular graphs, {\sl J. Algebraic Combin.} 38 (2013), 191--195.


\bibitem{cohen04} Cohen, A.M., Distance-transitive graphs,
    {\sl Topics in Algebraic Graph Theory} (L.W. Beineke et al., eds.), Cambridge University Press, Cambridge, 2004,
    pp.~222--249.


\bibitem{CPsrg}
Cohen, N., Pasechnik, D.V., Implementing Brouwer's database of
strongly regular graphs, preprint (2016);
\arxiv{1601.00181}.




\bibitem{CEKS10} Cohn, H., Elkies, N.D., Kumar, A., Sch\"urmann, A., Point
    configurations that are asymmetric yet balanced, {\sl Proc. Amer. Math.
    Soc.} 138 (2010), 2863--2872;
\arxiv{0812.2579}.

\bibitem{CK07} Cohn, H., Kumar, A., Universally optimal distribution of points
    on spheres,  {\sl J. Amer. Math. Soc.} 20 (2007), 99--148;
\arxiv{math/0607446}.

\bibitem{CD2007B}
Colbourn, C.J., Dinitz, J.H. (Eds.),
{\sl Handbook of Combinatorial Designs}, second edition,
Chapman \& Hall/CRC, Boca Raton, FL, 2007.

\bibitem{Collins1997GC}
Collins, B.V.C.,
The girth of a thin distance-regular graph,
{\sl Graphs Combin.} 13 (1997), 21--30.

\bibitem{Collins2000DM}
Collins, B.V.C.,
The Terwilliger algebra of an almost-bipartite distance-regular graph and its antipodal $2$-cover,
{\sl Discrete Math.} 216 (2000), 35--69.

\bibitem{Co95} Coolsaet, K., Local structure of graphs with $\lambda=\mu=2,$ $a_2=4$, {\sl Combinatorica} 15 (1995), 481--487.

\bibitem{Co05} 
Coolsaet, K., A distance regular graph with intersection array $(21,
16, 8; 1, 4, 14)$ does not exist, {\sl European J. Combin.} 26
(2005), 709--716.

\bibitem{CooDe05} Coolsaet, K., Degraer, J., A computer-assisted proof of the
    uniqueness of the Perkel graph, {\sl Des. Codes Cryptogr.} 34 (2005), 155--171.


\bibitem{CoJu08} Coolsaet, K., Juri\v{s}i\'{c}, A., Using equality in the Krein
    conditions to prove the nonexistence of certain distance-regular graphs,
    {\sl J. Combin. Theory Ser. A} 115 (2008), 1086--1095.

\bibitem{CoJuKo08EuJC} Coolsaet, K., Juri\v{s}i\'{c}, A., Koolen, J.H.,  On
    triangle-free distance-regular graphs with an eigenvalue multiplicity equal
    to the valency, {\sl European J. Combin.} 29 (2008), 1186--1199.

\bibitem{Cossidente2013JAC}
Cossidente, A.,
Relative hemisystems on the Hermitian surface,
{\sl J. Algebraic Combin.} 38 (2013), 275--284.

\bibitem{CP2005JLMS}
Cossidente, A., Penttila, T.,
Hemisystems on the Hermitian surface,
{\sl J. London Math. Soc.} 72 (2005), 731--741.

\bibitem{CGGV14}
Coutinho, G., Godsil, C., Guo, K., Vanhove, F.,
Perfect state transfer on distance-regular graphs and association schemes,
{\sl Linear Algebra Appl.} 478 (2015), 108--130;
\arxiv{1401.1745}.



\bibitem{Curtin1998DM} Curtin, B., $2$-Homogeneous bipartite distance-regular
    graphs, {\sl Discrete Math.} 187 (1998), 39--70.

\bibitem{Curtin1999GC}
Curtin, B.,
Bipartite distance-regular graphs, I,
{\sl Graphs Combin.} 15 (1999), 143--158;
II,
{\sl Graphs Combin.} 15 (1999), 377--391.

\bibitem{Curtin1999EJC}
Curtin, B.,
The local structure of a bipartite distance-regular graph,
{\sl European J. Combin.} 20 (1999), 739--758.

\bibitem{Curtin1999DM}
Curtin, B.,
Distance-regular graphs which support a spin model are thin,
{\sl Discrete Math.} 197/198 (1999), 205--216.

\bibitem{Curtin2000EJC}
Curtin, B.,
Almost $2$-homogeneous bipartite distance-regular graphs,
{\sl European J. Combin.} 21 (2000), 865--876.

\bibitem{Curtin2001JCTB}
Curtin, B.,
The Terwilliger algebra of a $2$-homogeneous bipartite distance-regular graph,
{\sl J. Combin. Theory Ser. B} 81 (2001), 125--141.

\bibitem{Curtin2007RJ}
Curtin, B.,
Spin Leonard pairs,
{\sl Ramanujan J.} 13 (2007), 319--332.

\bibitem{Curtin2007LAA}
Curtin, B.,
Modular Leonard triples,
{\sl Linear Algebra Appl.} 424 (2007), 510--539.

\bibitem{CN1999JCTB}
Curtin, B., Nomura, K.,
Some formulas for spin models on distance-regular graphs,
{\sl J. Combin. Theory Ser. B} 75 (1999), 206--236.

\bibitem{CN2000JAC}
Curtin, B., Nomura, K.,
Distance-regular graphs related to the quantum enveloping algebra of $sl(2)$,
{\sl J. Algebraic Combin.} 12 (2000), 25--36.

\bibitem{CN2004JAC}
Curtin, B., Nomura, K.,
Homogeneity of a distance-regular graph which supports a spin model,
{\sl J. Algebraic Combin.} 19 (2004), 257--272.

\bibitem{CN2005JCTB}
Curtin, B., Nomura, K.,
$1$-Homogeneous, pseudo-$1$-homogeneous, and $1$-thin distance-regular graphs,
{\sl J. Combin. Theory Ser. B} 93 (2005), 279--302.

\bibitem{Cuypers92} Cuypers, H., The dual of Pasch's axiom, {\sl European J.
    Combin.} 13 (1992), 15--31.



\bibitem{C70} Cvetkovi\'c, D.M., New characterizations of the
    cubic lattice graphs, {\sl Publ. Inst. Math. (Beograd)} 10 (1970),
    195--198.

\bibitem{CDS} Cvetkovi\'c, D., Doob, M., Sachs, H., {\sl Spectra of Graphs},
    Academic Press, New York, 1980.


\bibitem{dalfoperturbation} Dalf\'{o}, C., Van
    Dam, E.R., Fiol, M.A., On perturbations of almost distance-regular
    graphs, {\sl Linear Algebra Appl.} 435 (2011), 2626--2638;
    \arxiv{1202.3313}.


\bibitem{dalfodual} Dalf\'{o}, C., Van
    Dam, E.R., Fiol, M.A., Garriga, E., Dual concepts of almost distance-regularity
    and the spectral excess theorem, {\sl Discrete Math.} 312 (2012), 2730--2734;
    \arxiv{1207.3606}.



\bibitem{DalfoDamFiol} Dalf\'{o}, C., Van Dam, E.R.,
    Fiol, M.A.,
    Garriga, E., Gorissen, B.L., On almost distance-regular graphs,
    {\sl J. Combin. Theory Ser. A} 118 (2011), 1094--1113;
    \arxiv{1202.3265}.


\bibitem{DaFiGa09} Dalf\'{o}, C., Fiol, M.A., Garriga, E.,
    On $k$-walk-regular graphs, {\sl Electron. J. Combin.}
    16:1 (2009), R47.





\bibitem{Damexcess} Van Dam, E.R., The spectral excess theorem
    for distance-regular graphs: a global (over)view, {\sl
    Electron. J. Combin.} 15 (2008), R129.


\bibitem{damfiol12} Van Dam, E.R., Fiol, M.A., A short proof of the odd-girth
    theorem, {\sl
    Electron. J. Combin.} 19:3 (2012), P12;
    \arxiv{1205.0153}.

\bibitem{DFLaplacian}
Van Dam, E.R., Fiol, M.A.,
The Laplacian spectral excess theorem for distance-regular graphs, {\sl Linear Algebra Appl.} 458 (2014), 245--250;
\arxiv{1405.0169}.


\bibitem{EF00} Van Dam, E.R., Fon-Der-Flaass, D., Uniformly packed codes and
    more distance regular graphs from crooked functions, {\sl J. Algebraic
    Combin.} 12 (2000), 115--121.

\bibitem{EF03} 
Van Dam, E.R., Fon-Der-Flaass, D., Codes, graphs, and schemes from nonlinear
functions, {\sl European J. Combin.} 24 (2003), 85--98.

\bibitem{DH97} 
Van Dam, E.R., Haemers, W.H., A characterization of distance-regular graphs
with diameter three, {\sl J. Algebraic Combin.} 6 (1997), 299--303.


\bibitem{DH02} 
Van Dam, E.R., Haemers, W.H., Spectral characterizations of some
distance-regular graphs, {\sl J. Algebraic Combin.} 15 (2002), 189--202.


\bibitem{DH03} Van Dam, E.R., Haemers, W.H., Which graphs are
    determined by their spectrum?, {\sl Linear Algebra Appl.}
    373 (2003), 241--272.

\bibitem{DH09} Van Dam, E.R., Haemers, W.H., Developments on
    spectral characterizations of graphs, {\sl Discrete Math.}
    309 (2009), 576--586.

\bibitem{DHodd} Van Dam, E.R., Haemers, W.H., An odd
    characterization of the generalized odd graphs,
    {\sl J. Combin. Theory Ser. B} 101 (2011), 486--489;
    \arxiv{1202.2300}.

\bibitem{DHKS06} Van Dam, E.R., Haemers, W.H., Koolen, J.H.,
    Spence, E., Characterizing distance-regularity of graphs by the
    spectrum, {\sl J. Combin. Theory Ser. A} 113 (2006),
    1805--1820.


\bibitem{DK05} 
Van Dam, E.R., Koolen, J.H., A new family of distance-regular graphs with
unbounded diameter, {\sl Invent. Math.} 162 (2005), 189--193.

\bibitem{DMM2010pre} Van Dam, E.R., Martin, W.J., Muzychuk, M.E., Uniformity in
    association schemes and coherent configurations: cometric $Q$-antipodal
    schemes and linked systems, {\sl J. Combin. Theory Ser. A} 120 (2013), 1401--1439;
    \arxiv{1001.4928}.



\bibitem{DS12}
Van Dam, E.R., Sotirov, R.,
On bounding the bandwidth of graphs with symmetry,
{\sl INFORMS J. Comput.} 27 (2015), 75--88;
\arxiv{1212.0694}.

\bibitem{DS13}
Van Dam, E.R., Sotirov, R.,
Semidefinite programming and eigenvalue bounds for the graph partition problem,
{\sl Math. Program. Ser. B} 151 (2015), 379--404;
\arxiv{1312.0332}.

\bibitem{BDMR2014pre}
De Beule, J., Demeyer, J., Metsch, K., Rodgers, M.,
A new family of tight sets in $\mathcal{Q}^+(5,q)$,
{\sl Des. Codes Cryptogr.} 78 (2016), 655--678;
\arxiv{1409.5634}.

\bibitem{DBr06a} De Bruyn, B., The completion of the
    classification of the regular near octagons with thick
    quads,
{\sl J. Algebraic Combin.} 24 (2006), 23--29.

\bibitem{DeB2010ElJC}  De Bruyn, B., The nonexistence of regular near octagons
    with parameters $(s,t,t_2,t_3)=(2,24,0,8)$, {\sl Electron. J. Combin.} 17
    (2010), R149.

\bibitem{DBS10} De Bruyn, B., Suzuki, H., Intriguing sets of vertices of regular graphs, {\sl Graphs Combin.} 26 (2010), 629--646.

\bibitem{DeBVhpre} De Bruyn, B., Vanhove, F., Inequalities for regular near polygons, with applications to $m$-ovoids,
    {\sl European J. Combin.} 34 (2013), 522--538.

\bibitem{DV2012pre}
De Bruyn, B., Vanhove, F.,
On $Q$-polynomial regular near $2d$-gons,
{\sl Combinatorica} 35 (2015), 181--208.

\bibitem{DeDeKuTo06} De Clerck, F., De Winter, S., Kuijken, E.,
    Tonesi, C.,
    Distance-regular $(0,\alpha)$-reguli, {\sl Des. Codes
    Cryptogr.} 38 (2006), 179--194.

\bibitem{Deg07} Degraer, J., {\sl Isomorph-free exhaustive
    generation algorithms for association schemes}, thesis,
    Ghent University, 2007.

\bibitem{DegraerCoolsaet} Degraer, J., Coolsaet, K.,
    Classification of three-class association schemes using
    backtracking with dynamical variable ordering, {\sl Discrete
    Math.} 300 (2005), 71--81.


\bibitem{del} Delsarte, P., An algebraic approach to the association schemes of
    coding theory, {\sl Philips Res.\ Reports Suppl.} 10 (1973).

\bibitem{Delsarte1976JCTA}
Delsarte, P.,
Association schemes and $t$-designs in regular semilattices,
{\sl J. Combin. Theory Ser. A} 20 (1976), 230--243.


\bibitem{D} Delsarte, P., Bilinear forms over a finite field
    with applications to coding theory, {\sl J. Combin. Theory Ser. A}
    25 (1978), 226--241.

\bibitem{DG} Delsarte, P., Goethals, J.-M., Alternating
    bilinear forms over $GF(q)$, {\sl J. Combin. Theory Ser. A} 19
    (1975), 26--50.

\bibitem{DL1998IEEE} Delsarte, P., Levenshtein, V.I., Association schemes and
    coding theory, {\sl IEEE Trans. Inform. Theory} 44 (1998), 2477--2504.

\bibitem{DevSb90} Devroye, L., Sbihi, A., Random walks on highly symmetric
    graphs, {\sl J. Theoret. Probab.} 3 (1990), 497--514.


\bibitem{DiSaCo} Diaconis, P., Saloff-Coste, L., Separation cut-offs
    for birth and death chains, {\sl Ann. Appl. Probab.} 16 (2006), 2098--2122;
    \arxiv{math/0702411}.


\bibitem{DiaSha87} Diaconis, P., Shahshahani, M., Time to reach stationarity in
    the Bernoulli-Laplace diffusion model, {\sl SIAM J. Math. Anal.} 18 (1987),
    208--218.

\bibitem{Dickie1995D} Dickie, G.A., {\sl $Q$-polynomial structures for
    association
    schemes and distance-regular graphs}, thesis, University of Wisconsin,
    1995.

\bibitem{DT1996EJC} Dickie, G.A., Terwilliger, P., Dual bipartite
    $Q$-polynomial distance-regular graphs, {\sl European J. Combin.} 17
    (1996), 613--623.

\bibitem{DT1998JAC} Dickie, G.A., Terwilliger, P., A note on thin
    $P$-polynomial and dual-thin $Q$-polynomial symmetric association schemes,
    {\sl J. Algebraic Combin.} 7 (1998), 5--15.


\bibitem{doysne}
Doyle, P.G., Snell, J.L.,
{\sl Random Walks and Electric Networks},
Mathematical Association of America, Washington, DC, 1984;
    \url{http://www.stanford.edu/class/msande337/notes/walks.pdf}.

\bibitem{DrHeRo} Driscoll, J.R., Healy Jr., D.M., Rockmore, D.N., Fast discrete
    polynomial transforms with applications to data analysis for distance
    transitive graphs, {\sl SIAM J. Comput.} 26 (1997), 1066--1099.

\bibitem{Dunkl1976IUMJ}
Dunkl, C.F.,
A Krawtchouk polynomial addition theorem and wreath products of symmetric groups,
{\sl Indiana Univ. Math. J.} 25 (1976), 335--358.

\bibitem{Dunkl1978SIAM}
Dunkl, C.F.,
An addition theorem for Hahn polynomials: the spherical functions,
{\sl SIAM J. Math. Anal.} 9 (1978), 627--637.

\bibitem{Dunkl1978MM}
Dunkl, C.F.,
An addition theorem for some $q$-Hahn polynomials,
{\sl Monatsh. Math.} 85 (1978), 5--7.

\bibitem{Dunkl1979P}
Dunkl, C.F.,
Orthogonal functions on some permutation groups,
{\sl Relations Between Combinatorics and Other Parts of Mathematics} (D.K. Ray-Chaudhuri, ed.),
Proceedings of Symposia in Pure Mathematics, vol.~XXXIV,
American Mathematical Society, Providence, RI, 1979, pp.~129--147.

\bibitem{Edel} Edel, Y., On some representations of quadratic APN functions and
    dimensional dual hyperovals, {\sl RIMS K\^{o}ky\^{u}roku} 1687 (2010),
    118--130.

\bibitem{E-b1} Egawa, Y., Characterization of $H(n,q)$ by the
    parameters, {\sl J. Combin. Theory Ser. A} 31 (1981), 108--125.

\bibitem{E} Egawa, Y., Association schemes of quadratic
    forms,
    {\sl J. Combin. Theory Ser. A} 38 (1985), 1--14.

\bibitem{Emms} Emms, D., Severini, S., Wilson, R.C., Hancock, E.R.,
    Coined quantum walks lift the cospectrality of graphs and trees, {\sl Pattern
    Recognition} 42 (2009), 1988--2002.

\bibitem{Etzion2007JCD}
Etzion, T.,
Configuration distribution and designs of codes in the Johnson scheme,
{\sl J. Combin. Des.} 15 (2007), 15--34.


\bibitem{colorJohnson} Etzion, T., Bitan, S., On the chromatic number,
    colorings, and codes of the Johnson graph, {\sl Discrete Appl. Math.} 70
    (1996), 163--175.


\bibitem{FII} Faradjev, I.A., Ivanov, A.A., Ivanov, A.V., Distance-transitive
    graphs of valency 5, 6 and 7, {\sl European J. Combin.} 7 (1986),
    303--319.



\bibitem{FKM94} Farad\v{z}ev, I.A., Klin, M.H.,
    Muzichuk, M.E.,
    Cellular rings and groups of automorphisms of graphs,
    {\sl Investigations in Algebraic Theory of Combinatorial
    Objects} (I.A. Farad\v{z}ev, A.A. Ivanov, M.H. Klin,
    A.J. Woldar, eds.), Kluwer Academic Publishers, Dordrecht, 1994,
    pp.~1--153.

\bibitem{FLV2013pre}
Fazeli, A., Lovett, S., Vardy, A.,
Nontrivial $t$-designs over finite fields exist for all $t$,
{\sl J. Combin. Theory Ser. A} 127 (2014), 149--160;
\arxiv{1306.2088}.

\bibitem{fue} Feige, U., A tight lower bound on the cover time for random walks
    on graphs, {\sl Random Struct. Algorithms} 6 (1995), 433--438.

\bibitem{FLWfractionalmetric} Feng, M., Lv, B., Wang, K., On the fractional metric dimension of graphs,  {\sl
    Discrete Appl. Math.} 170 (2014), 55--63;
\arxiv{1112.2106}.

\bibitem{FengWang} Feng, M., Wang, K., On the metric dimension of bilinear
    forms graphs, {\sl Discrete Math.} 312 (2012), 1266--1268;
    \arxiv{1104.4089}.

\bibitem{FMX2014pre}
Feng, T., Momihara, K., Xiang, Q.,
Cameron-Liebler line classes with parameter $x=\frac{q^2-1}{2}$,
{\sl J. Combin. Theory Ser. A} 133 (2015), 307--338;
\arxiv{1406.6526}.

\bibitem{F97} 
Fiol, M.A., An eigenvalue characterization of antipodal distance-regular
graphs, {\sl Electron. J. Combin.} 4:1 (1997), R30.

\bibitem{F00}
Fiol, M.A.,
A quasi-spectral characterization of strongly distance-regular graphs,
{\sl Electron. J. Combin.} 7 (2000), R51.

\bibitem{Fiol2001CPC}
Fiol, M.A.,
Some spectral characterizations of strongly distance-regular graphs,
{\sl Combin. Probab. Comput.} 10 (2001), 127--135.

\bibitem{Fi02} 
Fiol, M.A., Algebraic characterizations of distance-regular graphs,
{\sl Discrete Math.} 246 (2002), 111--129.

\bibitem{Fi05} 
Fiol, M.A., Spectral bounds and distance-regularity, {\sl Linear Algebra Appl.}
397 (2005), 17--33.

\bibitem{Fiol2014pre}
Fiol, M.A.,
A spectral characterization of strongly distance-regular graphs with diameter four,
preprint (2014);
\arxiv{1407.1392}.

\bibitem{Fiolexcess} Fiol, M.A., Gago, S., Garriga, E.,
    A simple proof of the spectral excess theorem for
    distance-regular graphs, {\sl Linear Algebra Appl.}
    432 (2010), 2418--2422.

\bibitem{FG97} 
Fiol, M.A., Garriga, E., From local adjacency polynomials to locally
pseudo-distance-regular graphs, {\sl J. Combin. Theory Ser. B} 71 (1997),
162--183.

\bibitem{FG1999LAA}
Fiol, M.A., Garriga, E.,
On the algebraic theory of pseudo-distance-regularity around a set,
{\sl Linear Algebra Appl.} 298 (1999), 115--141.

\bibitem{FG2001SIAM}
Fiol, M.A., Garriga, E.,
An algebraic characterization of completely regular codes in distance-regular graphs,
{\sl SIAM J. Discrete Math.} 15 (2001/02), 1--13.

\bibitem{FHY96} 
Fiol, M.A., Garriga, E., Yebra, J.L.A., Locally pseudo-distance-regular graphs,
{\sl J. Combin. Theory Ser. B} 68 (1996), 179--205.

\bibitem{FDF193} Fon-Der-Flaass, D.G., There exists no
    distance-regular graph with intersection array
    $(5,4,3;1,1,2)$, {\sl European J. Combin.} 14 (1993), 409--412.

\bibitem{FDF293} Fon-Der-Flaass, D.G., A distance-regular
    graph with intersection array $(5,4,3,3; \break 1,1,1,2)$ does not
    exist, {\sl J. Algebraic Combin.} 2 (1993), 49--56.


\bibitem{FonDerFlaassprolific} Fon-Der-Flaass, D.G., New
    prolific
    constructions of strongly regular graphs, {\sl Adv. Geom.}
    2 (2002), 301--306.

\bibitem{FonDerFlaass2007SMJ}
Fon-Der-Flaass, D.G.,
Perfect $2$-colorings of a hypercube,
{\sl Siberian Math. J.} 48 (2007), 740--745.

\bibitem{Fuglister} Fuglister, F.J., On generalized Moore geometries I, II, {\sl Discrete Math.} 67 (1987), 249--258, 259--269.

\bibitem{FKT07}
Fujisaki, T., Koolen, J.H., Tagami, M., Some properties of the twisted
Grassmann graphs, {\sl Innov. Incidence Geom.} 3 (2006), 81--87.

\bibitem{Funk-Neubauer2009LAA}
Funk-Neubauer, D.,
Tridiagonal pairs and the $q$-tetrahedron algebra,
{\sl Linear Algebra Appl.} 431 (2009), 903--925;
\arxiv{0806.0901}.

\bibitem{GGZF08} Gao, S., Guo, J., Zhang, B., Fu, L.,
    Subspaces in $d$-bounded distance-regular graphs and their
    applications, {\sl European J. Combin.} 29 (2008), 592--600.

\bibitem{G74} Gardiner, A., Antipodal covering graphs, {\sl J. Combin. Theory
   Ser. B} 16 (1974), 255--273.


\bibitem{GavElec10} Gavrilyuk, A.L., On the Koolen-Park inequality and
    Terwilliger graphs, {\sl Electron. J. Combin.} 17 (2010), R125;
    \arxiv{1007.3339}.


\bibitem{Gav11} Gavrilyuk, A.L., Distance-regular graphs with intersection
    arrays $\{55, 36, 11; \break 1, 4, 45\}$ and $\{56, 36, 9; 1, 3, 48\}$ do not exist, {\sl
    Dokl. Math.} 84 (2011), 444--446.

\bibitem{GK2012pre}
Gavrilyuk, A.L., Koolen, J.H., The Terwilliger polynomial of a
$Q$-polynomial distance-regular graph and its application to
pseudo-partition graphs, {\sl Linear Algebra Appl.} 466 (2015),
117--140;
\arxiv{1403.4027}.


\bibitem{GKbilinear15}
Gavrilyuk, A.L., Koolen, J.H., The characterization of the graphs
of bilinear $(d \times d)$-forms over $\mathbb{F}_2$, preprint
(2015);
\arxiv{1511.09435}.



\bibitem{GavMak2005} Gavrilyuk, A.L., Makhnev, A.A., On Krein graphs without triangles, {\sl Dokl. Math.} 72 (2005), 591--594.

\bibitem{GavMak2011} Gavrilyuk, A.L., Makhnev, A.A., On Terwilliger graphs in
    which the neighborhood of each vertex is isomorphic to the
    Hoffman--Singleton graph, {\sl Mathematical Notes} 89 (2011),
    633--644.

\bibitem{GavMakpre} Gavrilyuk, A.L., Makhnev, A.A., Distance-regular graphs
    with intersection arrays $\{52,35,16;1,4,28\}$ and $\{69,48,24;1,4,46\}$ do not exist, {\sl Des. Codes
    Cryptogr.} 65 (2012), 49--54.

\bibitem{GavMak2011T} Gavrilyuk, A.L., Makhnev, A.A., Distance-regular graph
    with the intersection array $\{45, 30, 7; 1, 2, 27\}$ does not exist, {\sl Discrete Math. Appl.} 23 (2013), 225--244.

\bibitem{GM2014JCTA}
Gavrilyuk, A.L., Metsch, K.,
A modular equality for Cameron--Liebler line classes,
{\sl J. Combin. Theory Ser. A} 127 (2014), 224--242.

\bibitem{GM2012pre}
Gavrilyuk, A.L., Mogilnykh, I.Y.,
Cameron-Liebler line classes in $PG(n,4)$,
{\sl Des. Codes Cryptogr.} 73 (2014), 969--982;
\arxiv{1205.2351}.

\bibitem{GVZ2014SIGMA}
Genest, V.X., Vinet, L., Zhedanov, A.,
A ``continuous'' limit of the complementary Bannai--Ito polynomials: Chihara polynomials,
{\sl SIGMA Symmetry Integrability Geom. Methods Appl.} 10 (2014), Paper 038;
\arxiv{1309.7235}.

\bibitem{Gijswijt2005PhD}
Gijswijt, D.,
{\sl Matrix algebras and semidefinite programming techniques for codes},
thesis,
University of Amsterdam, 2005;
\arxiv{1007.0906}.

\bibitem{GST2006JCTA} Gijswijt, D., Schrijver, A., Tanaka, H., New upper bounds
    for nonbinary codes based on the Terwilliger algebra and semidefinite
    programming, {\sl J. Combin. Theory Ser. A} 113 (2006), 1719--1731.

\bibitem{gillespiethesis}
Gillespie, N.I.,
{\sl Neighbour transitivity on codes in Hamming graphs},
thesis,
The University of Western Australia, 2011.

\bibitem{Gillespie2013DM}
Gillespie, N.I.,
A note on binary completely regular codes with large minimum distance,
{\sl Discrete Math.} 313 (2013), 1532--1534;
\arxiv{1210.6459}.

\bibitem{GGP2012pre}
Gillespie, N.I., Giudici, M., Praeger, C.E.,
Classification of a family of completely transitive codes, preprint (2012);
\arxiv{1208.0393}.

\bibitem{GP2013JCTA}
Gillespie, N.I., Praeger, C.E.,
Uniqueness of certain completely regular Hadamard codes,
{\sl J. Combin. Theory Ser. A} 120 (2013), 1394--1400;
\arxiv{1112.1247}.

\bibitem{GP2013DCC}
Gillespie, N.I., Praeger, C.E.,
Neighbour transitivity on codes in Hamming graphs,
{\sl Des. Codes Cryptogr.} 67 (2013), 385--393;
\arxiv{1112.1244}.

\bibitem{GP2012-4pre}
Gillespie, N.I., Praeger, C.E.,
Diagonally neighbour transitive codes and frequency
    permutation arrays, {\sl J. Algebraic Combin.} 39 (2014), 733--747;
\arxiv{1204.2900}.

\bibitem{GP2012-5pre}
Gillespie, N.I., Praeger, C.E.,
Complete transitivity of the Nordstrom-Robinson codes,
preprint (2012);
\arxiv{1205.3878}.

\bibitem{GuPr99}
Giudici, M., Praeger, C.E.,
Completely transitive codes in Hamming graphs,
{\sl European J. Combin.} 20 (1999), 647--662.

\bibitem{Go2002EJC}
Go, J.T.,
The Terwilliger algebra of the hypercube,
{\sl European J. Combin.} 23 (2002), 399--429.

\bibitem{GT2002EJC} Go, J.T., Terwilliger, P., Tight distance-regular graphs
    and the subconstituent algebra, {\sl European J. Combin.} 23 (2002),
    793--816.

\bibitem{God88} Godsil, C.D., Bounding the diameter of distance-regular
    graphs, {\sl Combinatorica} 8 (1988), 333--343.

\bibitem{Godsil1992AJC}
Godsil, C.D.,
Krein covers of complete graphs,
{\sl Australas. J. Combin.} 6 (1992), 245--255.

\bibitem{Godsilac} Godsil, C.D., {\sl Algebraic Combinatorics},
Chapman \& Hall, New York, 1993.


\bibitem{Godsil93} Godsil, C.D., Geometric distance-regular
    covers, {\sl New Zealand J. Math.} 22 (1993),
    31--38.

\bibitem{Godsil96} Godsil, C.D., Covers of complete graphs,
{\sl Progress in Algebraic Combinatorics} (E. Bannai, A. Munemasa, eds.),
Advanced Studies in Pure Mathematics, vol.~24,
Mathematical Society of Japan, Tokyo, 1996, pp.~137--163.

\bibitem{G98} 
Godsil, C.D., Eigenpolytopes of distance regular graphs, {\sl Canad. J. Math.}
50 (1998), 739--755.

\bibitem{Godsilperiodic} Godsil, C., Periodic graphs, {\sl Electron. J.
    Combin.} 18:1 (2011), P23;
    \arxiv{0806.2074}.

\bibitem{GodsilHensel92} Godsil, C.D., Hensel, A.D.,
    Distance regular covers of the complete graph, {\sl J. Combin. Theory Ser. B} 56 (1992), 205--238.

\bibitem{GoKo95} Godsil, C.D., Koolen, J.H., On the
    multiplicity of eigenvalues of distance-regular graphs,
    {\sl Linear Algebra Appl.} 226-228 (1995), 273--275.

\bibitem{GLP98} 
Godsil, C.D., Liebler, R.A., Praeger, C.E., Antipodal distance transitive
covers of complete graphs, {\sl European J. Combin.} 19 (1998),
455--478.

\bibitem{gmk} Godsil, C.D., McKay, B.D., Feasibility
    conditions for the existence of walk-regular graphs, {\sl
    Linear Algebra Appl.} 30 (1980), 51--61.

\bibitem{GM2016B}
Godsil, C., Meagher, K.,
{\sl Erd\H{o}s--Ko--Rado Theorems: Algebraic Approaches},
Cambridge University Press, Cambridge, 2016.

\bibitem{GP1997pre}
Godsil, C.D., Praeger, C.E.,
Completely transitive designs,
manuscript (1997);
\arxiv{1405.2176}.

\bibitem{GodsilRoy08} Godsil, C., Roy, A., Two
    characterizations of crooked functions, {\sl IEEE Trans. Inform. Theory} 54 (2008), 864--866;
    \arxiv{0704.1293}.

\bibitem{GR2011AC}
Godsil, C., Royle, G.F.,
Cores of geometric graphs,
{\sl Ann. Comb.} 15 (2011), 267--276;
\arxiv{0806.1300}.

\bibitem{GSbiregular} Godsil, C.D., Shawe-Taylor, J., Distance-reguralised
    graphs are distance-regular or distance-biregular, {\sl J. Combin. Theory Ser. B} 43
    (1987), 14--24.

\bibitem{GR1999C}
Goemans, M.X., Rendl, F.,
Semidefinite programs and association schemes,
{\sl Computing} 63 (1999), 331--340.

\bibitem{GWL13} Guo, J., Wang, K., Li, F., Metric dimension of some distance-regular graphs, {\sl J. Comb. Optim.} 26 (2013), 190--197;
\arxiv{1105.1847}.


\bibitem{GWL13DP} Guo, J., Wang, K., Li, F., Metric dimension of symplectic dual polar graphs and symmetric bilinear
    forms graphs, {\sl Discrete Math.} 313 (2013), 186--188.

\bibitem{GM2013} Guo, W., Makhnev, A.A., On distance-regular graphs without 4-claws, {\sl
    Dokl. Math.} 88 (2013), 625--629.


\bibitem{Haemersthesis} Haemers, W.H., {\sl Eigenvalue techniques in design and
    graph theory}, thesis,
Eindhoven University of Technology, 1979;
\url{http://alexandria.tue.nl/extra3/proefschrift/PRF3A/7909413.pdf}.


\bibitem{HaeInterlacing} Haemers, W.H., Interlacing eigenvalues and graphs,
    {\sl Linear Algebra
    Appl.} 226-228 (1995), 593--616.




\bibitem{Ha96} 
Haemers, W.H., Distance-regularity and the spectrum of graphs, {\sl Linear
Algebra Appl.} 236 (1996), 265--278.

\bibitem{HS95} 
Haemers, W.H., Spence, E., Graphs cospectral with distance-regular graphs, {\sl
Linear Multilinear Algebra} 39 (1995), 91--107.

\bibitem{injectivecolor} Hahn, G., Kratochv\'{\i}l, J., \v{S}ir\'{a}\v{n}, J.,
    Sotteau, D., On the injective chromatic number of graphs, {\sl Discrete Math.} 256 (2002),
    179--192.

\bibitem{Hanaki2000JA}
Hanaki, A.,
Semisimplicity of adjacency algebras of association schemes,
{\sl J. Algebra} 225 (2000), 124--129.

\bibitem{Hanaki2002AM}
Hanaki, A.,
Locality of a modular adjacency algebra of an association scheme of prime power order,
{\sl Arch. Math. (Basel)} 79 (2002), 167--170.

\bibitem{Hanaki2009EJC}
Hanaki, A.,
Representations of finite association schemes,
{\sl European J. Combin.} 30 (2009), 1477--1496.

\bibitem{HY2005JAC}
Hanaki, A., Yoshikawa, M.,
On modular standard modules of association schemes,
{\sl J. Algebraic Combin.} 21 (2005), 269--279.

\bibitem{Hartwig2007LAA}
Hartwig, B.,
The tetrahedron algebra and its finite-dimensional irreducible modules,
{\sl Linear Algebra Appl.} 422 (2007), 219--235;
\arxiv{math/0606197}.

\bibitem{Hashimoto2001IDAQPRT}
Hashimoto, Y.,
Quantum decomposition in discrete groups and interacting Fock spaces,
{\sl Infin. Dimens. Anal. Quantum Probab. Relat. Top.} 4 (2001), 277--287.

\bibitem{HHO2003JMP}
Hashimoto, Y., Hora, A., Obata, N.,
Central limit theorems for large graphs: method of quantum decomposition,
{\sl J. Math. Phys.} 44 (2003), 71--88.

\bibitem{HOT2001P}
Hashimoto, Y., Obata, N., Tabei, N.,
A quantum aspect of asymptotic spectral analysis of large Hamming graphs,
{\sl Quantum Information III} (T. Hida, K. Sait\^{o}, eds.),
World Scientific Publishing, River Edge, NJ, 2001, pp.~45--57.

\bibitem{HI2012pre}
Hattai, T., Ito, T.,
On a certain subalgebra of $U_q(\widehat{\mathfrak{sl}}_2)$ related to the degenerate $q$-Onsager algebra,
{\sl SIGMA Symmetry Integrability Geom. Methods Appl.} 11 (2015), Paper 007;
\arxiv{1205.2946}.

\bibitem{Hemmeter1984UM}
Hemmeter, J.,
Halved graphs, Johnson and Hamming graphs,
{\sl Utilitas Math.} 25 (1984), 115--118.

\bibitem{Hemmeter1986EJC}
Hemmeter, J.,
Distance-regular graphs and halved graphs,
{\sl European J. Combin.} 7 (1986), 119--129.

\bibitem{Hemmeter1988EJC}
Hemmeter, J.,
The large cliques in the graph of quadratic forms,
{\sl European J. Combin.} 9 (1988), 395--410.

\bibitem{HW1990EJCa}
Hemmeter, J., Woldar, A.,
On the maximal cliques of the quadratic forms graph in even characteristic,
{\sl European J. Combin.} 11 (1990), 119--126.

\bibitem{HW1990EJCb}
Hemmeter, J., Woldar, A.,
Classification of the maximal cliques of size $\geq q+4$ in the quadratic forms graph in odd characteristic,
{\sl European J. Combin.} 11 (1990), 433--449.

\bibitem{HW1999EJC}
Hemmeter, J., Woldar, A.,
The complete list of maximal cliques of $\mathrm{Quad}(n,q)$, $q$ even,
{\sl European J. Combin.} 20 (1999), 81--85.

\bibitem{hilanomura} Hilano, T., Nomura, K., Distance
    degree
    regular graphs, {\sl J. Combin. Theory Ser. B} 37
    (1984), 96--100.

\bibitem{Hi94} 
Hiraki, A., An improvement of the Boshier-Nomura bound, {\sl J. Combin. Theory
Ser. B} 61 (1994), 1--4.

\bibitem{Hi95} 
Hiraki, A., Distance-regular subgraphs in a distance-regular graph, I, II, {\sl
European J. Combin.} 16 (1995), 589--602, 603--615.

\bibitem{H196} 
Hiraki, A., Distance-regular subgraphs in a distance-regular graph, III, {\sl
European J. Combin.} 17 (1996), 629--636.

\bibitem{H97} 
Hiraki, A., Distance-regular subgraphs in a distance-regular graph, IV, {\sl
European J. Combin.} 18 (1997), 635--645.

\bibitem{Hi198} Hiraki, A., Distance-regular subgraphs in a
    distance-regular graph, V, {\sl European J. Combin.} 19
    (1998), 141--150.

\bibitem{Hi98} 
Hiraki, A., Distance-regular subgraphs in a distance-regular graph, VI, {\sl
European J. Combin.} 19 (1998), 953--965.

\bibitem{Hi299} 
Hiraki, A., Distance-regular graphs of the height $h$, {\sl Graphs Combin.} 15 (1999), 417--428.

\bibitem{Hi499} Hiraki, A., Strongly closed subgraphs in a
    regular thick near polygon, {\sl European J. Combin.} 20
    (1999), 789--796.

\bibitem{Hi01} 
Hiraki, A., A distance-regular graph with strongly closed subgraphs, {\sl J.
Algebraic Combin.} 14 (2001), 127--131.

\bibitem{Hi03D} Hiraki, A., A characterization of the doubled Grassmann graphs,
    the doubled Odd graphs, and the Odd graphs by strongly closed subgraphs,
    {\sl European J. Combin.} 24 (2003), 161--171.

\bibitem{Hi303} 
Hiraki, A., A distance-regular graph with bipartite geodetically closed
subgraphs, {\sl European J. Combin.} 24 (2003), 349--363.

\bibitem{Hi103} 
Hiraki, A., The number of columns $(1,k-2,1)$ in the intersection array of a
distance-regular graph, {\sl Graphs Combin.} 19 (2003), 371--387.

\bibitem{Hi07} 
Hiraki, A., A characterization of the odd graphs and the doubled odd graphs
with a few of their intersection numbers, {\sl European J. Combin.} 28 (2007), 246--257.

\bibitem{Hiraki08b} Hiraki, A., Strongly closed subgraphs in a
    distance-regular graph with $c_2>1$, {\sl Graphs Combin.} 24
    (2008), 537--550.

\bibitem{Hiraki09a} Hiraki, A., Distance-regular graph with
    $c_2>1$ and $a_1=0<a_2$, {\sl Graphs Combin.}
    25 (2009), 65--79.

\bibitem{Hiraki2012GC}
Hiraki, A.,
A characterization of the Hamming graphs and the dual polar graphs by completely regular subgraphs,
{\sl Graphs Combin.} 28 (2012), 449--467.

\bibitem{HK98} 
Hiraki, A., Koolen, J., An improvement of the Ivanov bound, {\sl Ann. Comb.} 2
(1998), 131--135.

\bibitem{HK02} 
Hiraki, A., Koolen, J., An improvement of the Godsil bound, {\sl Ann. Comb.} 6
(2002), 33--44.

\bibitem{HiKo04a} Hiraki, A., Koolen, J., A note on regular
    near polygons, {\sl Graphs Combin.} 20 (2004), 485--497.

\bibitem{HiKo04c} Hiraki, A., Koolen, J., A Higman-Haemers
    inequality for thick regular near polygons, {\sl J. Algebraic
    Combin.} 20 (2004), 213--218.

\bibitem{HiKo04b} Hiraki, A., Koolen, J., The regular near
    polygons of order $(s,2)$, {\sl J. Algebraic Combin.} 20 (2004), 219--235.

\bibitem{HiKo06} Hiraki, A., Koolen, J., A generalization
    of an inequality of Brouwer-Wilbrink, {\sl J. Combin. Theory Ser. A} 109 (2005), 181--188.

\bibitem{HNS} Hiraki, A., Nomura, K., Suzuki, H., Distance-regular graphs of
    valency 6 and $a_1 =1$, {\sl J. Algebraic Combin.} 11 (2000),
    101--134.

\bibitem{HSW95} 
Hiraki, A., Suzuki, H., Wajima, M., On distance-regular graphs with $k\sb
i=k\sb j$, II, {\sl Graphs Combin.} 11 (1995), 305--317.

\bibitem{HI1998JAC} Hobart, S., Ito, T., The structure of nonthin irreducible
    $T$-modules of endpoint $1$: ladder bases and classical parameters, {\sl J.
    Algebraic Combin.} 7 (1998), 53--75.

\bibitem{hoffman63} Hoffman, A.J., On the polynomial of a
    graph, {\sl Amer. Math. Monthly} 70 (1963), 30--36.

\bibitem{HX2006JAC}
Hollmann, H.D.L., Xiang, Q.,
Association schemes from the action of $\mathrm{PGL}(2,q)$ fixing a nonsingular conic in $\mathrm{PG}(2,q)$,
{\sl J. Algebraic Combin.} 24 (2006), 157--193;
\arxiv{math/0503573}.

\bibitem{Hong} Hong, Y., On the nonexistence of nontrivial perfect $e$-codes
    and tight $2e$-designs in Hamming schemes $H(n,q)$ with $e\geq 3$ and
    $q\geq 3$, {\sl Graphs Combin.} 2 (1986), 145--164.

\bibitem{Hora1998IDAQPRT}
Hora, A.,
Central limit theorems and asymptotic spectral analysis on large graphs,
{\sl Infin. Dimens. Anal. Quantum Probab. Relat. Top.} 1 (1998), 221--246.

\bibitem{Hora2000PTRF}
Hora, A.,
Gibbs state on a distance-regular graph and its application to a scaling limit of the spectral distributions of discrete Laplacians,
{\sl Probab. Theory Related Fields} 118 (2000), 115--130.

\bibitem{Hora00} Hora, A., An axiomatic approach to the cut-off phenomenon
    for random walks on large distance-regular graphs, {\sl Hiroshima Math. J.}
30 (2000), 271--299.

\bibitem{HO2003P}
Hora, A., Obata, N.,
Quantum decomposition and quantum central limit theorem,
{\sl Fundamental Aspects of Quantum Physics} (L. Accardi, S. Tasaki, eds.),
QP--PQ: Quantum Probability and White Noise Analysis, vol.~17,
World Scientific Publishing, River Edge, NJ, 2003, pp.~284--305.

\bibitem{HO2007B}
Hora, A., Obata, N.,
{\sl Quantum Probability and Spectral Analysis of Graphs},
Springer, Berlin, 2007.

\bibitem{HO2008TAMS}
Hora, A., Obata, N.,
Asymptotic spectral analysis of growing regular graphs,
{\sl Trans. Amer. Math. Soc.} 360 (2008), 899--923.

\bibitem{HS2007EJC}
Hosoya, R., Suzuki, H.,
Tight distance-regular graphs with respect to subsets,
{\sl European J. Combin.} 28 (2007), 61--74.

\bibitem{Hua1945} Hua, L., Geometries of matrices. I. Generalizations of Von Staudt's theorem, {\sl Trans. Amer. Math. Soc.} 57 (1945), 441--481.

\bibitem{Huang2012pre}
Huang, H.,
Finite-dimensional irreducible modules of the universal Askey-Wilson algebra,
{\sl Comm. Math. Phys.} 340 (2015), 959--984;
\arxiv{1210.1740}.

\bibitem{HLW2014pre}
Huang, L.-P., Lv, B., Wang, K.,
The endomorphisms of Grassmann graphs,
{\sl Ars Math. Contemp.} 10 (2016), 383--392;
\arxiv{1404.7578}.

\bibitem{Huang87} Huang, T., A characterization of the
    association schemes of bilinear forms, {\sl European J. Combin.} 8
    (1987), 159--173.

\bibitem{HL99} 
Huang, T., Liu, C., Spectral characterization of some generalized odd graphs,
{\sl Graphs Combin.} 15 (1999), 195--209.

\bibitem{HPW2009pre}
Huang, Y., Pan, Y., Weng, C.,
Nonexistence of a class of distance-regular graphs,
{\sl Electron. J. Combin.} 22:2 (2015), P2.37.

\bibitem{IO2006P}
Igarashi, D., Obata, N.,
Asymptotic spectral analysis of growing graphs: odd graphs and spidernets,
{\sl Quantum Probability} (M. Bo\.{z}ejko, W. M{\l}otkowski, J. Wysocza\'{n}ski, eds.),
Banach Center Publications, vol.~73,
Polish Academy of Sciences, Institute of Mathematics, Warsaw, 2006,
pp.~245--265.

\bibitem{IM2013DCC}
Ihringer, F., Metsch, K.,
On the maximum size of Erd\H{o}s--Ko--Rado sets in $H(2d+1,q^2)$,
{\sl Des. Codes Cryptogr.} 72 (2014), 311--316.

\bibitem{symmetricdesigns} Ionin, Y.J., Shrikhande, M.S., {\sl Combinatorics
    of Symmetric Designs}, Cambridge University Press, Cambridge, 2006.

\bibitem{INT2011LAA} Ito, T., Nomura, K., Terwilliger, P., A classification of
    sharp tridiagonal pairs, {\sl Linear Algebra Appl.} 435 (2011), 1857--1884;
    \arxiv{1001.1812}.

\bibitem{IS2014LAA}
Ito, T., Sato, J.,
TD-pairs of type II with shape $1,2,\dots,2,1$,
{\sl Linear Algebra Appl.} 461 (2014), 51--91.

\bibitem{ITT2001P} Ito, T., Tanabe, K., Terwilliger, P., Some algebra related
    to $P$- and $Q$-polynomial association schemes, {\sl Codes and Association Schemes}
    (A. Barg, S. Litsyn, eds.), American Mathematical Society,
    Providence, RI, 2001, pp.~167--192;
    \arxiv{math/0406556}.

\bibitem{IT2004JPAA}
Ito, T., Terwilliger, P.,
The shape of a tridiagonal pair,
{\sl J. Pure Appl. Algebra} 188 (2004), 145--160;
\arxiv{math/0304244}.

\bibitem{IT2007JAA}
Ito, T., Terwilliger, P.,
Two non-nilpotent linear transformations that satisfy the cubic $q$-Serre relations,
{\sl J. Algebra Appl.} 6 (2007), 477--503;
\arxiv{math/0508398}.

\bibitem{IT2007CA}
Ito, T., Terwilliger, P.,
The $q$-tetrahedron algebra and its finite dimensional irreducible modules,
{\sl Comm. Algebra} 35 (2007), 3415--3439;
\arxiv{math/0602199}.

\bibitem{IT2007LAA}
Ito, T., Terwilliger, P.,
Tridiagonal pairs of Krawtchouk type,
{\sl Linear Algebra Appl.} 427 (2007), 218--233;
\arxiv{0706.1065}.

\bibitem{IT2009EJC}
Ito, T., Terwilliger, P.,
Distance-regular graphs and the $q$-tetrahedron algebra,
{\sl European J. Combin.} 30 (2009), 682--697;
\arxiv{math/0608694}.

\bibitem{IT2009MMJ}
Ito, T., Terwilliger, P.,
Distance-regular graphs of $q$-Racah type and the $q$-tetrahedron algebra,
{\sl Michigan Math. J.} 58 (2009), 241--254;
\arxiv{0708.1992}.

\bibitem{IT2009JCISS}
Ito, T., Terwilliger, P.,
The Drinfel'd polynomial of a tridiagonal pair,
{\sl J. Comb. Inf. Syst. Sci.} 34 (2009), 255--292;
\arxiv{0805.1465}.

\bibitem{IT2009JA}
Ito, T., Terwilliger, P.,
Tridiagonal pairs of $q$-Racah type,
{\sl J. Algebra} 322 (2009), 68--93;
\arxiv{0807.0271}.

\bibitem{IT2010KJM} Ito, T., Terwilliger, P., The augmented tridiagonal
    algebra, {\sl Kyushu J. Math.} 64 (2010), 81--144;
    \arxiv{0904.2889}.

\bibitem{Iv83}
Ivanov, A.A.,
Bounding the diameter of a distance-regular graph,
{\sl Dokl. Akad. Nauk SSSR} 271 (1983), 789--792.

\bibitem{IvanovDTG} Ivanov, A.A., Distance-transitive graphs and their
    classification, {\sl Investigations in Algebraic Theory of Combinatorial
    Objects} (I.A. Farad\v{z}ev, A.A. Ivanov, M.H. Klin,
    A.J. Woldar, eds.), Kluwer Academic Publishers, Dordrecht, 1994,
    pp.~283--378.

\bibitem{ILPP97} Ivanov, A.A., Liebler, R.A., Penttila, T., Praeger, C.E.,
    Antipodal distance-transitive covers of complete bipartite graphs, {\sl European J. Combin.} 18 (1997), 11--33.

\bibitem{IMU1989EJC}
Ivanov, A.A., Muzichuk, M.E., Ustimenko, V.A.,
On a new family of ($P$ and $Q$)-polynomial schemes,
{\sl European J. Combin.} 10 (1989), 337--345.

\bibitem{IvaShp90} Ivanov, A.A., Shpectorov, S.V., The $P$-geometry for
    $M_{23}$ has no non-trivial 2-coverings, {\sl European J. Combin.} 11 (1990), 373--379.

\bibitem{IS1991JMSJ} Ivanov, A.A., Shpectorov, S.V., A characterization of the
    association schemes of Hermitian forms, {\sl J. Math. Soc. Japan} 43
    (1991), 25--48.

\bibitem{Jaeger1996P}
Jaeger, F.,
Towards a classification of spin models in terms of association schemes,
{\sl Progress in Algebraic Combinatorics} (E. Bannai, A. Munemasa, eds.),
Advanced Studies in Pure Mathematics, vol.~24,
Mathematical Society of Japan, Tokyo, 1996, pp.~197--225.

\bibitem{JMN1998JAC}
Jaeger, F., Matsumoto, M., Nomura, K.,
Bose-Mesner algebras related to type II matrices and spin models,
{\sl J. Algebraic Combin.} 8 (1998), 39--72.

\bibitem{JaSa06} Jafarizadeh, M.A., Salimi, S., Investigation
    of continuous-time quantum walk via modules of Bose-Mesner and
    Terwilliger algebras, {\sl J. Phys. A} 39 (2006),
    13295--13323;
\arxiv{quant-ph/0603139}.

\bibitem{JaSu08} Jafarizadeh, M.A., Sufiani, R., Perfect state
    transfer over distance-regular spin networks, {\sl Phys. Rev. A} 77 (2008), 022315;
\arxiv{0709.0755}.



\bibitem{JaSuJa07} Jafarizadeh, M.A., Sufiani, R., Jafarizadeh,
    S., Calculating two-point resistances in distance-regular
    resistor networks, {\sl J. Phys. A} 40 (2007), 4949--4972;
\arxiv{cond-mat/0611683}.

\bibitem{JaSuJa09} Jafarizadeh, M.A., Sufiani, R., Jafarizadeh,
    S., Recursive calculation of effective resistances in distance-regular networks
based on Bose-Mesner algebra and Christoffel-Darboux identity, {\sl J. Math.
Phys.} 50 (2009), 023302;
\arxiv{0705.2480}.

\bibitem{Jones1989PJM}
Jones, V.F.R.,
On knot invariants related to some statistical mechanical models,
{\sl Pacific J. Math.} 137 (1989), 311--334.

\bibitem{JuTo09} Jungnickel, D., Tonchev, V.D., Polarities,
    quasi-symmetric designs and Hamada's conjecture, {\sl Des.
    Codes Cryptogr.} 51 (2009), 131--140.


\bibitem{Jurisic2003DM}
Juri\v{s}i\'{c}, A.,
AT4 family and $2$-homogeneous graphs,
{\sl Discrete Math.} 264 (2003), 127--148.

\bibitem{JuKo00DCC} Juri\v{s}i\'{c}, A., Koolen, J.H.,  A local approach to
    $1$-homogeneous graphs, {\sl Des. Codes
    Cryptogr.} 21 (2000), 127--147.

\bibitem{JuKo00EuJC} Juri\v{s}i\'{c}, A., Koolen, J.H., Nonexistence of some
    antipodal distance-regular graphs of diameter four, {\sl European J.
    Combin.} 21 (2000), 1039--1046.

\bibitem{JK2002DM} Juri\v{s}i\'{c}, A., Koolen, J.H., Krein parameters and
    antipodal tight graphs with diameter $3$ and $4$, {\sl Discrete Math.} 244
    (2002), 181--202.

\bibitem{JK2003JAC}
Juri\v{s}i\'{c}, A., Koolen, J.H.,
$1$-Homogeneous graphs with cocktail party $\mu$-graphs,
{\sl J. Algebraic Combin.} 18 (2003), 79--98.

\bibitem{JuKo07JAC} Juri\v{s}i\'{c}, A., Koolen, J.H., Distance-regular graphs
    with complete multipartite $\mu$-graphs and AT4 family, {\sl J. Algebraic
    Combin.} 25 (2007), 459--471.

\bibitem{JuKo11} Juri\v si\'c, A., Koolen,  J.H., Classification of the family
    $\mathrm{AT}4(qs,q,q)$ of antipodal tight graphs, {\sl J. Combin. Theory Ser. A}
    118 (2011), 842--852.

\bibitem{JuKopre} Juri\v{s}i\'{c}, A., Koolen, J.H., The uniqueness of two
    antipodal covers of diameter $4$, in preparation.

\bibitem{JKM2005JCTB}
Juri\v{s}i\'{c}, A., Koolen, J.H., Miklavi\v{c}, \v{S}.,
Triangle- and pentagon-free distance-regular graphs with an eigenvalue multiplicity equal to the valency,
{\sl J. Combin. Theory Ser. B} 94 (2005), 245--258.

\bibitem{JKT2000JAC} Juri\v{s}i\'{c}, A., Koolen, J.H., Terwilliger, P., Tight
    distance-regular graphs, {\sl J. Algebraic Combin.} 12 (2000), 163--197;
    \arxiv{math/0108196}.

\bibitem{JKZ2008EJC}
Juri\v{s}i\'{c}, A., Koolen, J.H., \v{Z}itnik, A.,
Triangle-free distance-regular graphs with an eigenvalue multiplicity equal to their valency and diameter $3$,
{\sl European J. Combin.} 29 (2008), 193--207.

\bibitem{JuMuTa10} Juri\v{s}i\'{c}, A., Munemasa, A., Tagami, Y.,  On graphs
    with complete multipartite $\mu$-graphs, {\sl Discrete Math.} 310 (2010), 1812--1819.

\bibitem{JT2008JAC}
Juri\v{s}i\'{c}, A., Terwilliger, P.,
Pseudo $1$-homogeneous distance-regular graphs,
{\sl J. Algebraic Combin.} 28 (2008), 509--529.

\bibitem{JTZ2010JCTB} Juri\v{s}i\'{c}, A., Terwilliger, P., \v{Z}itnik, A., The
    $Q$-polynomial idempotents of a distance-regular graph, {\sl J. Combin.
    Theory Ser. B} 100 (2010), 683--690;
    \arxiv{0908.4098}.

\bibitem{JTZ2010EJC} Juri\v{s}i\'{c}, A., Terwilliger, P., \v{Z}itnik, A.,
    Distance-regular graphs with light tails, {\sl European J. Combin.} 31
    (2010), 1539--1552.

\bibitem{JurVidpre} Juri\v{s}i\'{c}, A., Vidali, J., Extremal $1$-codes in
    distance-regular graphs of diameter $3$, {\sl Des. Codes
    Cryptogr.} 65 (2012), 29--47.

\bibitem{JwoTu} Jwo, J.-S., Tuan, T.-C., On transmitting delay in
    a distance-transitive strongly antipodal graph, {\sl Inform. Process. Lett.} 51 (1994), 233--235.

\bibitem{Keevash2014pre}
Keevash, P.,
The existence of designs,
preprint (2014);
\arxiv{1401.3665}.

\bibitem{Kempe03} Kempe, J., Quantum random walks: An introductory
    overview, {\sl Contemporary Physics} 44 (2003), 307--327;
    \arxiv{quant-ph/0303081}.

\bibitem{Kempe} Kempe, J., Discrete quantum walks hit exponentially
    faster, {\sl Probab. Theory Related Fields} 133 (2005), 215--235;
    \arxiv{quant-ph/0205083}.

\bibitem{Kim2009EJC}
Kim, J.,
Some matrices associated with the split decomposition for a $Q$-polynomial distance-regular graph,
{\sl European J. Combin.} 30 (2009), 96--113;
\arxiv{0710.4383}.

\bibitem{Kim2010DM}
Kim, J.,
A duality between pairs of split decompositions for a $Q$-polynomial distance-regular graph,
{\sl Discrete Math.} 310 (2010), 1828--1834;
\arxiv{0705.0167}.

\bibitem{kooterwpre} Kim, J., Koolen, J.H., Yu, H., Some notes on Terwilliger graphs, manuscript
    (2012).

\bibitem{KOP2012B} De Klerk, E., De Oliveira Filho, F.M., Pasechnik, D.V.,
    Relaxations of combinatorial problems via association schemes, {\sl
    Handbook on Semidefinite, Conic and Polynomial Optimization} (M.F. Anjos,
    J.B. Lasserre, eds.), Springer, New York, 2012, pp.~171--199.

\bibitem{KPS2008SIAM} De Klerk, E., Pasechnik, D.V., Sotirov, R., On
    semidefinite programming relaxations of the traveling salesman problem,
    {\sl SIAM J. Optim.} 19 (2008), 1559--1573.


\bibitem{Dobre} De Klerk, E., Pasechnik, D., Sotirov, R., Dobre, C., On
    semidefinite programming relaxations of maximum $k$-section, {\sl Math.
    Program. Ser. B} 136 (2012), 253--278.


\bibitem{DKSot10} De Klerk, E., Sotirov, R., Exploiting group symmetry in
    semidefinite programming relaxations of the quadratic assignment problem,
    {\sl Math. Program. Ser. A} 122 (2010), 225--246.

\bibitem{KS2011MPA} De Klerk, E., Sotirov, R., Improved semidefinite
    programming bounds for quadratic assignment problems with suitable
    symmetry, {\sl Math. Program. Ser. A} 133 (2012), 75--91.

\bibitem{KlinPech}
Klin, M., Pech, C.,
A new construction of antipodal distance regular covers of complete graphs through the use of Godsil-Hensel matrices,
{\sl Ars Math. Contemp.} 4 (2011), 205--243.

\bibitem{KLS2010B} Koekoek, R., Lesky, P.A., Swarttouw, R.F., {\sl
    Hypergeometric Orthogonal Polynomials and Their $q$-Analogues},
    Springer-Verlag, Berlin, 2010.

\bibitem{KS1998R} Koekoek, R., Swarttouw, R.F., The Askey scheme of
    hypergeometric orthogonal polynomials and its $q$-analog, report 98-17,
    Delft University of Technology, 1998;
    \url{http://aw.twi.tudelft.nl/~koekoek/askey.html}.

\bibitem{Ko192} Koolen, J.H., On subgraphs in distance-regular graphs,
{\sl J. Algebraic Combin.} 1 (1992), 353--362.

\bibitem{Ko292} Koolen, J.H., A new condition for distance-regular graphs,
{\sl European J. Combin.} 13 (1992), 63--64.

\bibitem{Ko93} Koolen, J.H., On uniformly geodetic graphs, {\sl
    Graphs Combin.} 9 (1993), 325--333.


\bibitem{Koothesis} Koolen, J.H., {\sl Euclidean representations and
    substructures of distance-regular graphs}, thesis, Eindhoven University of Technology, 1994;
    \url{http://alexandria.tue.nl/extra3/proefschrift/PRF10A/9412042.pdf}.

\bibitem{KoolenDoob} Koolen, J.H., A characterization of the Doob graphs, {\sl J. Combin. Theory Ser. B} 65 (1995), 125--138.

\bibitem{Ko95} Koolen, J.H., On a conjecture of Martin on the parameters of
    completely regular codes and the classification of the completely regular
    codes in the Biggs-Smith graph, {\sl Linear Multilinear Algebra} 39 (1995), 3--17.

\bibitem{KoBa10} Koolen, J.H., Bang, S., On distance-regular
    graphs with smallest eigenvalue at least $-m$,  {\sl J. Combin. Theory Ser. B} 100 (2010), 573--584;
    \arxiv{0908.2017}.

\bibitem{KoKiPapre} Koolen, J.H., Kim, J., Park, J., Distance-regular graphs with a relatively small eigenvalue
    multiplicity, {\sl Electron. J. Combin.} 20:1 (2013), P1.



\bibitem{KoLeMa2010}
Koolen, J.H., Lee, W., Martin, W.J.,
Characterizing completely regular codes from an algebraic viewpoint,
{\sl Combinatorics and Graphs} (R.A. Brualdi et al., eds.),
Contemporary Mathematics, vol.~531,
American Mathematical Society, Providence, RI, 2010, pp.~223--242;
\arxiv{0911.1828}.

\bibitem{KoLeMapre09}
Koolen, J.H., Lee, W., Martin, W.J., Tanaka, H.,
Arithmetic completely regular codes,
{\sl Discrete Math. Theor. Comput. Sci.} 17 (2016), 59--76;
\arxiv{0911.1826}.

\bibitem{kmcollection} Koolen, J.H., Markowsky, G.,
A collection of results concerning electric resistance and simple random walk on distance-regular graphs,
{\sl Discrete Math.} 339 (2016), 737--744.

\bibitem{phi1} Koolen, J.H., Markowsky, G., Park, J., On electric resistances
    for distance-regular graphs, {\sl European J. Combin.} 34 (2013), 770--786;
    \arxiv{1103.2810}.



\bibitem{KooMarpre} Koolen, J.H., Martin, W.J., Distance-regular graphs with
    an eigenvalue with multiplicity 8, manuscript (1994).

\bibitem{KM2000JSPI}
Koolen, J.H., Munemasa, A.,
Tight $2$-designs and perfect $1$-codes in Doob graphs,
{\sl J. Statist. Plann. Inference} 86 (2000), 505--513.

\bibitem{KoPa10} Koolen, J.H., Park, J., Shilla
    distance-regular graphs, {\sl European J. Combin.} 31
    (2010), 2064--2073;
    \arxiv{0902.3860}.

\bibitem{KoPa11} Koolen, J.H., Park, J., Distance-regular graphs with
    $a_1$ or $c_2$ at least half the valency, {\sl J. Combin. Theory Ser. A} 119 (2012), 546--555;
    \arxiv{1008.1209}.

\bibitem{KoPapre} Koolen, J.H., Park, J., A relationship between the diameter
    and the intersection number
$c_2$ for a distance-regular graph, {\sl Des. Codes Cryptogr.} 65 (2012),
55--63;
\arxiv{1109.2195}.

\bibitem{KoPapre12} Koolen, J.H., Park, J., A note on distance-regular graphs with a small number of vertices compared to the valency, {\sl European J. Combin.} 34 (2013), 935--940.

\bibitem{KoPaYu11} Koolen, J.H., Park, J., Yu, H., An inequality involving the
    second largest and smallest eigenvalue  of a distance-regular graph, {\sl
    Linear Algebra Appl.} 434 (2011), 2404--2412.

\bibitem{KQ15} Koolen, J.H., Qiao, Z., Distance-regular graphs with valency $k$ having smallest eigenvalue at most
$-k/2$, preprint (2015);
\arxiv{1507.04839}.


\bibitem{KS94} 
Koolen, J.H., Shpectorov, S.V., Distance-regular graphs the distance matrix of
which has only one positive eigenvalue, {\sl European J. Combin.} 15 (1994), 269--275.

\bibitem{KoYu11} Koolen, J.H., Yu, H., The distance-regular graphs such that
    all of its second largest local eigenvalues are at most one, {\sl Linear
    Algebra Appl.} 435 (2011), 2507--2519.

\bibitem{Krotov2014pre}
Krotov, D.S.,
Perfect codes in Doob graphs,
Des. Codes Cryptogr., to appear;
\arxiv{1407.6329}.

\bibitem{KuiTon05} Kuijken, E., Tonesi, C.,
    Distance-regular graphs and $(\alpha,\beta)$-geometries,
    {\sl J. Geom.} 82 (2005), 135--145.


\bibitem{Kur2011T} Kurihara, H., An excess theorem for spherical $2$-designs,
    {\sl Des. Codes
    Cryptogr.} 65 (2012), 89--98;
\arxiv{1203.3257}.

\bibitem{KN2012JCTA} Kurihara, H., Nozaki, H., A characterization of
    $Q$-polynomial association schemes, {\sl J. Combin. Theory Ser. A} 119
    (2012), 57--62;
    \arxiv{1007.0473}.

\bibitem{KN2011pre}
Kurihara, H., Nozaki, H.,
A spectral equivalent condition of the $P$-polynomial property for association schemes,
{\sl Electron. J. Combin.} 21:3 (2014), P3.1;
\arxiv{1110.4975}.

\bibitem{LaJwoDh} Lakshmivarahan, S., Jwo, J.-S., Dhall, S.K., Symmetry in
    interconnection networks based on Cayley graphs of permutation groups: A
    survey, {\sl Parallel Comput.} 19 (1993), 361--407.

\bibitem{Lambeck1990D} Lambeck, E.W., {\sl Contributions to the theory of distance regular graphs}, thesis, Eindhoven
    University of Technology, 1990; \url{http://alexandria.tue.nl/extra3/proefschrift/PRF7A/9006771.pdf}.

\bibitem{Lamb93} Lambeck, E., Some elementary inequalities for distance-regular graphs, {\sl European J. Combin.} 14
    (1993), 53.

\bibitem{Lang2002EJC} Lang, M.S., Tails of bipartite distance-regular graphs,
    {\sl European J. Combin.} 23 (2002), 1015--1023.

\bibitem{Lang2003JAC}
Lang, M.S.,
Leaves in representation diagrams of bipartite distance-regular graphs,
{\sl J. Algebraic Combin.} 18 (2003), 245--254.

\bibitem{Lang2004JCTB} Lang, M.S., A new inequality for bipartite
    distance-regular graphs, {\sl J. Combin. Theory Ser. B} 90 (2004), 55--91.

\bibitem{Lang2008EJC}
Lang, M.S.,
Pseudo primitive idempotents and almost $2$-homogeneous bipartite distance-regular graphs,
{\sl European J. Combin.} 29 (2008), 35--44.

\bibitem{LT2007EJC} Lang, M.S., Terwilliger, P.M., Almost-bipartite
    distance-regular graphs with the $Q$-polynomial property, {\sl European J.
    Combin.} 28 (2007), 258--265;
    \arxiv{math/0508435}.

\bibitem{L69} Laskar, R., Eigenvalues of the adjacency matrix
    of the cubic lattice graph, {\sl Pacific J. Math.} 29 (1969),
    623--629.

\bibitem{LMO2010} LeCompte, N., Martin, W.J., Owens, W., On the equivalence
    between real mutually unbiased bases and a certain class of association
    schemes, {\sl European J. Combin.} 31 (2010), 1499--1512.

\bibitem{LeeWeng12} Lee, G.-S., Weng, C.-W., A spectral excess theorem for
    nonregular graphs, {\sl J. Combin. Theory Ser. A} 119 (2012), 1427--1431.

\bibitem{LeeWeng14} Lee, G.-S., Weng, C.-W., A characterization of bipartite distance-regular graphs, {\sl
    Linear Algebra Appl.} 446 (2014), 91--103.

\bibitem{Leespectral} Lee, J., A spectral approach to polyhedral dimension,
    {\sl Math. Program.} 47 (1990), 441--459.

\bibitem{Leedimension} Lee, J., Characterizations of the dimension for classes
    of concordant polytopes, {\sl Math. Oper. Res.} 15 (1990), 139--154.

\bibitem{Lee2013PhD}
Lee, J.-H.,
$Q$-polynomial distance-regular graphs and a double affine Hecke algebra of rank one,
{\sl Linear Algebra Appl.} 439 (2013), 3184--3240;
\arxiv{1307.5297}.

\bibitem{LevineSandpile} Levine, L., Propp, J., What is . . . a sandpile?, {\sl
    Notices Amer. Math. Soc.} 57 (2010), 976--979.


\bibitem{Lewis2000DM} Lewis, H.A., Homotopy in $Q$-polynomial distance-regular
    graphs, {\sl Discrete Math.} 223 (2000), 189--206.

\bibitem{LW97} 
Liang, Y., Weng, C., Parallelogram-free distance-regular graphs, {\sl J.
Combin. Theory Ser. B} 71 (1997), 231--243.

\bibitem{LP2013pre}
Liebler, R.A., Praeger, C.E.,
Neighbour-transitive codes in Johnson graphs,
{\sl Des. Codes Cryptogr.} 73 (2014), 1--25;
\arxiv{1311.0113}.

\bibitem{Lovasz1979IEEE}
Lov\'{a}sz, L.,
On the Shannon capacity of a graph,
{\sl IEEE Trans. Inform. Theory} 25 (1979), 1--7.

\bibitem{LS1991SIAM}
Lov\'{a}sz, L., Schrijver, A.,
Cones of matrices and set-functions and $0$-$1$ optimization,
{\sl SIAM J. Optim.} 1 (1991), 166--190.

\bibitem{KM2014pre}
Ma, J., Koolen, J.H., Twice $Q$-polynomial distance-regular graphs
of diameter $4$, {\sl Sci. China Math.} 58 (2015), 2683--2690;
\arxiv{1405.2546}.

\bibitem{MaWang5Q} Ma, J., Wang, K., Nonexistence of exceptional $5$-class association schemes with two $Q$-polynomial structures,
{\sl Linear Algebra Appl.} 440 (2014), 278--285;
    \arxiv{1311.5942}.

\bibitem{Ma1994DCC}
Ma, S.L.,
A survey of partial difference sets,
{\sl Des. Codes Cryptogr.} 4 (1994), 221--261.

\bibitem{MacLean2000DM}
MacLean, M.S.,
An inequality involving two eigenvalues of a bipartite distance-regular graph,
{\sl Discrete Math.} 225 (2000), 193--216.

\bibitem{MacLean2003JAC}
MacLean, M.S.,
Taut distance-regular graphs of odd diameter,
{\sl J. Algebraic Combin.} 17 (2003), 125--147.

\bibitem{MacLean2004JCTB}
MacLean, M.S.,
Taut distance-regular graphs of even diameter,
{\sl J. Combin. Theory Ser. B} 91 (2004), 127--142.

\bibitem{MacLean2012DM}
MacLean, M.S.,
A new approach to the Bipartite Fundamental Bound,
{\sl Discrete Math.} 312 (2012), 3195--3202.

\bibitem{MT2006DM}
MacLean, M.S., Terwilliger, P.,
Taut distance-regular graphs and the subconstituent algebra,
{\sl Discrete Math.} 306 (2006), 1694--1721;
\arxiv{math/0508399}.

\bibitem{MT2008DM}
MacLean, M.S., Terwilliger, P.,
The subconstituent algebra of a bipartite distance-regular graph; thin modules with endpoint two,
{\sl Discrete Math.} 308 (2008), 1230--1259;
\arxiv{math/0604351}.

\bibitem{MK10}  Markowsky, G.,  Koolen, J.H., A conjecture of
    Biggs concerning the resistance of a distance-regular
    graph, {\sl Electron. J. Combin.} 17 (2010), R78;
    \arxiv{1006.2687}.

\bibitem{Martinthesis} Martin, W.J., {\sl Completely regular subsets}, thesis,
    University of Waterloo, 1992; \url{http://users.wpi.edu/~martin/RESEARCH/THESIS/}.

\bibitem{Martin1994JAC}
Martin, W.J.,
Completely regular designs of strength one,
{\sl J. Algebraic Combin.} 3 (1994), 177--185.

\bibitem{Martin1998JCD}
Martin, W.J.,
Completely regular designs,
{\sl J. Combin. Des.} 6 (1998), 261--273.

\bibitem{Martin2001JCMCC}
Martin, W.J.,
Symmetric designs, sets with two intersection numbers and Krein parameters of incidence graphs,
{\sl J. Combin. Math. Combin. Comput.} 38 (2001), 185--196.

\bibitem{MartinCRC} Martin, W.J., Completely regular codes: a viewpoint and
    some problems, {\sl Proceedings of 2004 Com${}^2$MaC Workshop on
    Distance-Regular Graphs and Finite Geometry}, Pusan, Korea, 2004, pp.~43--56;
    \url{http://users.wpi.edu/~martin/RESEARCH/busan.pdf}.

\bibitem{Martin2010www} Martin, W.J., Cometric association
    schemes, \url{http://users.wpi.edu/~martin/RESEARCH/QPOL/} (April 2010).

\bibitem{Martin-pre}
Martin, W.J.,
Completely regular codes in the Odd graphs,
unpublished manuscript.

\bibitem{MMW07} Martin, W.J., Muzychuk, M., Williford, J.,
    Imprimitive
    cometric association schemes: constructions and analysis,
    {\sl J. Algebraic Combin.} 25 (2007), 399--415.

\bibitem{martintanaka} Martin, W.J., Tanaka, H., Commutative
    association schemes, {\sl European J. Combin.} 30
    (2009), 1497--1525;
    \arxiv{0811.2475}.

\bibitem{MT1997pre}
Martin, W.J., Taylor, T.D.,
On the subconstituent algebra of a completely regular code,
manuscript (1997).

\bibitem{MW2009EJC}
Martin, W.J., Williford, J.S.,
There are finitely many $Q$-polynomial association schemes with given first multiplicity at least three,
{\sl European J. Combin.} 30 (2009), 698--704.


\bibitem{MarZhupre} Martin, W.J., Zhu, R.R., Distance-regular graphs having
    an eigenvalue of small multiplicity, {\sl Univ. Waterloo Res. Rep.} CORR 92-06, 1996;
    \url{http://users.wpi.edu/~martin/RESEARCH/multmain.ps}.


\bibitem{MZ1995DCC}
Martin, W.J., Zhu, X.J.,
Anticodes for the Grassmann and bilinear forms graphs,
{\sl Des. Codes Cryptogr.} 6 (1995), 73--79.


\bibitem{MRR1978JCISS}
McEliece, R.J., Rodemich, E.R., Rumsey, H.C., Jr.,
The Lov\'{a}sz bound and some generalizations,
{\sl J. Combin. Inform. System Sci.} 3 (1978), 134--152.


\bibitem{Meixner91} Meixner, T., Some polar towers, {\sl
    European J. Combin.} 12 (1991), 397--415.

\bibitem{metsch91} Metsch, K., Improvement of Bruck's
    completion theorem, {\sl Des. Codes Cryptogr.} 1
    (1991), 99--116.



\bibitem{Me95} 
Metsch, K., A characterization of Grassmann graphs, {\sl European J. Combin.}
16 (1995), 639--644.


\bibitem{M297} 
Metsch, K., On the characterization of the folded Johnson graphs and the folded
halved cubes by their intersection arrays, {\sl European J. Combin.} 18 (1997), 65--74.

\bibitem{M197} 
Metsch, K., Characterization of the folded Johnson graphs of small diameter by
their intersection arrays, {\sl European J. Combin.} 18 (1997),
901--913.

\bibitem{Me99} 
Metsch, K., On a characterization of bilinear forms graphs, {\sl European J.
Combin.} 20 (1999), 293--306.

\bibitem{Me03} 
Metsch, K., On the characterization of the folded halved cubes by their
intersection arrays,
{\sl Des. Codes Cryptogr.} 29 (2003), 215--225.

\bibitem{Metsch2014JCTA}
Metsch, K.,
An improved bound on the existence of Cameron--Liebler line classes,
{\sl J. Combin. Theory Ser. A} 121 (2014), 89--93.

\bibitem{Meyerowitz1992JCISS}
Meyerowitz, A.,
Cycle-balanced partitions in distance-regular graphs,
{\sl J. Combin. Inform. System Sci.} 17 (1992), 39--42.

\bibitem{Meyerowitz2003DM}
Meyerowitz, A.,
Cycle-balance conditions for distance-regular graphs,
{\sl Discrete Math.} 264 (2003), 149--165.

\bibitem{Miklavic2004EJC} Miklavi\v{c}, \v{S}., $Q$-polynomial distance-regular
    graphs with $a_1=0$, {\sl European J. Combin.} 25 (2004), 911--920.

\bibitem{Miklavic2005JCTB}
Miklavi\v{c}, \v{S}.,
An equitable partition for a distance-regular graph of negative type,
{\sl J. Combin. Theory Ser. B} 95 (2005), 175--188.

\bibitem{Miklavic2007DM} Miklavi\v{c}, \v{S}., On bipartite $Q$-polynomial
    distance-regular graphs with $c_2=1$, {\sl Discrete Math.} 307 (2007),
    544--553.

\bibitem{Miklavic2007EJC}
Miklavi\v{c}, \v{S}.,
On bipartite $Q$-polynomial distance-regular graphs,
{\sl European J. Combin.} 28 (2007), 94--110.

\bibitem{Mik09} Miklavi\v c, \v S., The Terwilliger
    algebra of a distance-regular graph of negative type,
    {\sl Linear Algebra
    Appl.} 430 (2009), 251--270;
    \arxiv{0804.1650}.

\bibitem{Miklavic2009EJC}
Miklavi\v{c}, \v{S}.,
$Q$-polynomial distance-regular graphs with $a_1=0$ and $a_2\ne 0$,
{\sl European J. Combin.} 30 (2009), 192--207.

\bibitem{Miklavic2013GC} Miklavi\v{c}, \v{S}., On bipartite distance-regular
    graphs with intersection numbers $c_i=(q^i-1)/(q-1)$, {\sl Graphs Combin.}
    29 (2013), 121--130.

\bibitem{MP2003EJC}
Miklavi\v{c}, \v{S}., Poto\v{c}nik, P.,
Distance-regular circulants,
{\sl European J. Combin.} 24 (2003), 777--784.

\bibitem{MP2007JCTB}
Miklavi\v{c}, \v{S}., Poto\v{c}nik, P.,
Distance-regular Cayley graphs on dihedral groups,
{\sl J. Combin. Theory Ser. B} 97 (2007), 14--33.

\bibitem{MS2014JCTB}
Miklavi\v{c}, \v{S}., \v{S}parl, P.,
On distance-regular Cayley graphs on abelian groups,
{\sl J. Combin. Theory Ser. B} 108 (2014), 102--122;
\arxiv{1205.6948}.

\bibitem{MT2011pre} Miklavi\v{c}, \v{S}., Terwilliger, P., Bipartite
    $Q$-polynomial distance-regular graphs and uniform posets, {\sl J. Algebraic Combin.} 38 (2013), 225--242;
    \arxiv{1108.2484}.


\bibitem{Mdrline} Mohar, B., Shawe-Taylor, J.,
    Distance-biregular graphs with 2-valent vertices and
    distance-regular line graphs, {\sl J. Combin. Theory Ser. B} 38 (1985),
    193--203.

\bibitem{Moon} Moon, A., Characterization of the odd graphs $O_k$ by
    parameters, {\sl Discrete Math.} 42 (1982), 91--97.

\bibitem{MR2002P}
Moore, C., Russell, A.,
Quantum walks on the hypercube,
{\sl Randomization and Approximation Techniques in Computer Science} (J.D.P. Rolim, S. Vadhan, eds.),
Lecture Notes in Computer Science, vol.~2483,
Springer-Verlag, Berlin, 2002, pp.~164--178;
\arxiv{quant-ph/0104137}.

\bibitem{Williford2014AGT}
Moorhouse, G.E., Williford, J.,
Double covers of symplectic dual polar graphs,
{\sl Discrete Math.} 339 (2016), 571--588;
\arxiv{1504.01067}.

\bibitem{MP2012pre}
Morales, J.V.S., Pascasio, A.A.,
An action of the tetrahedron algebra on the standard module for the Hamming graphs and Doob graphs,
{\sl Graphs Combin.} 30 (2014), 1513--1527.

\bibitem{Munemasa1986GC}
Munemasa, A.,
An analogue of $t$-designs in the association schemes of alternating bilinear forms,
{\sl Graphs Combin.} 2 (1986), 259--267.

\bibitem{Munemasa2004EJC}
Munemasa, A.,
Spherical $5$-designs obtained from finite unitary groups,
{\sl European J. Combin.} 25 (2004), 261--267.

\bibitem{Munemasa2014SHCC}
Munemasa, A., Godsil-McKay switching and twisted Grassmann graphs,
preprint (2015);
\arxiv{1512.09232}.

\bibitem{MPS1993JAC}
Munemasa, A., Pasechnik, D.V., Shpectorov, S.V.,
The automorphism group and the convex subgraphs of the quadratic forms graph in characteristic $2$,
{\sl J. Algebraic Combin.} 2 (1993), 411--419.

\bibitem{MuTo09} Munemasa, A., Tonchev, V.D., The
    twisted Grassmann graph is the block graph of a design,
    Innov. Incidence Geom. 12 (2011), 1--6;
    \arxiv{0906.4509}.

\bibitem{Muzy07} Muzychuk, M., A generalization of
    Wallis--Fon-Der-Flaass construction of strongly regular
    graphs, {\sl J. Algebraic Combin.} 25 (2007), 169--187.



\bibitem{N01} 
Nakano, H., On a distance-regular graph of even height with $k\sb e=k\sb f$,
{\sl Graphs Combin.} 17 (2001), 707--716.


\bibitem{Neu80} Neumaier, A., Strongly regular graphs with
    smallest eigenvalue $-m$, {\sl Arch. Math.} 33
    (1980), 392--400.

\bibitem{Neu-b1} Neumaier, A., Characterization of a class
    of distance regular graphs, {\sl J. Reine Angew. Math.} 357 (1985), 182--192.

\bibitem{Neu1990JCTA}  Neumaier, A., Krein conditions and near polygons,
    {\sl J. Combin. Theory Ser. A} 54 (1990), 201--209.

\bibitem{Neu92} Neumaier, A., Completely regular codes, {\sl Discrete Math.}
    106-107 (1992), 353--360.

\bibitem{N04} 
Neumaier, A., Dual polar spaces as extremal distance-regular graphs, {\sl
European J. Combin.} 25 (2004), 269--274.

\bibitem{N}
Newman, M.,
{\sl Integral Matrices},
Academic Press, New York-London, 1972.

\bibitem{Nomura90} Nomura, K., Distance-regular graphs of Hamming type, {\sl J.
    Combin. Theory Ser. B} 50 (1990), 160--167.

\bibitem{N294} Nomura, K., Homogeneous graphs and regular near polygons, {\sl
    J. Combin. Theory Ser. B} 60 (1994), 63--71.

\bibitem{Nomura1995JCTB} Nomura, K., Spin models on bipartite distance-regular
    graphs, {\sl J. Combin. Theory Ser. B} 64 (1995), 300--313.

\bibitem{Nomura1996P}
Nomura, K.,
Spin models and almost bipartite $2$-homogeneous graphs,
{\sl Progress in Algebraic Combinatorics} (E. Bannai, A. Munemasa, eds.),
Advanced Studies in Pure Mathematics, vol.~24,
Mathematical Society of Japan, Tokyo, 1996, pp.~285--308.

\bibitem{Nomura1997JAC}
Nomura, K.,
An algebra associated with a spin model,
{\sl J. Algebraic Combin.} 6 (1997), 53--58.

\bibitem{Nomura2002KJM}
Nomura, K.,
A property of solutions of modular invariance equations for distance-regular graphs,
{\sl Kyushu J. Math.} 56 (2002), 53--57.

\bibitem{NT2008LAAd}
Nomura, K., Terwilliger, P.,
The structure of a tridiagonal pair,
{\sl Linear Algebra Appl.} 429 (2008), 1647--1662;
\arxiv{0802.1096}.

\bibitem{NT2009LAAa}
Nomura, K., Terwilliger, P.,
Tridiagonal pairs and the $\mu$-conjecture,
{\sl Linear Algebra Appl.} 430 (2009), 455--482;
\arxiv{0908.2604}.

\bibitem{NT2010LAAa}
Nomura, K., Terwilliger, P.,
On the shape of a tridiagonal pair,
{\sl Linear Algebra Appl.} 432 (2010), 615--636;
\arxiv{0906.3838}.

\bibitem{NT2010LAAb}
Nomura, K., Terwilliger, P.,
Tridiagonal pairs of $q$-Racah type and the $\mu$-conjecture,
{\sl Linear Algebra Appl.} 432 (2010), 3201--3209;
\arxiv{0908.3151}.

\bibitem{NT2011LAA} Nomura, K., Terwilliger, P., Tridiagonal matrices with
    nonnegative entries,
    {\sl Linear Algebra Appl.} 434 (2011), 2527--2538;
    \arxiv{1010.1305}.

\bibitem{Nozaki2013pre}
Nozaki, H.,
Polynomial properties on large symmetric association schemes,
{\sl Ann. Comb.}, to appear;
\arxiv{1305.2539}.

\bibitem{PanLuWeng08} Pan, Y., Lu, M., Weng, C.,
    Triangle-free distance-regular graphs, {\sl J. Algebraic Combin.} 27 (2008),
    23--34.

\bibitem{PanWeng08} Pan, Y., Weng, C., $3$-Bounded property in a
    triangle-free distance-regular graph, {\sl European J.
    Combin.} 29 (2008), 1634--1642;
    \arxiv{0709.0564}.

\bibitem{PanWeng09} Pan, Y., Weng, C., A note on triangle-free
    distance-regular graphs with $a_2 \neq 0$, {\sl J. Combin. Theory Ser. B} 99
    (2009), 266--270.

\bibitem{Papre12} Park, J., The distance-regular graphs with
    valency $k$ and number of vertices at most $3k+1$, preprint (2012).

\bibitem{PaKoMar} Park, J., Koolen, J.H., Markowsky, G., There are only finitely
    many distance-regular graphs with valency $k$ at least three, fixed ratio
    $\frac{k_2}{k}$ and large diameter, {\sl J. Combin. Theory Ser. B} 103
    (2013), 733--741;
    \arxiv{1012.2632}.

\bibitem{Pascasio1999JAC}
Pascasio, A.A.,
Tight graphs and their primitive idempotents,
{\sl J. Algebraic Combin.} 10 (1999), 47--59.

\bibitem{Pascasio2001GC} Pascasio, A.A., Tight distance-regular graphs and the
    $Q$-polynomial property, {\sl Graphs Combin.} 17 (2001), 149--169.

\bibitem{Pascasio2002EJC} Pascasio, A.A., On the multiplicities of the
    primitive idempotents of a $Q$-polynomial distance-regular graph, {\sl
    European J. Combin.} 23 (2002), 1073--1078.

\bibitem{Pascasio2003DM} Pascasio, A.A., An inequality in character algebras,
    {\sl Discrete Math.} 264 (2003), 201--209.

\bibitem{Pascasio2008DM} Pascasio, A.A., A characterization of $Q$-polynomial
    distance-regular graphs, {\sl Discrete Math.} 308 (2008), 3090--3096.

\bibitem{PT2006LAA}
Pascasio, A.A., Terwilliger, P.,
The pseudo-cosine sequences of a distance-regular graph,
{\sl Linear Algebra Appl.} 419 (2006), 532--555;
\arxiv{math/0312150}.

\bibitem{PY01} 
Pasini, A., Yoshiara, S., New distance regular graphs arising from dimensional
dual hyperovals, {\sl European J. Combin.} 22 (2001), 547--560.

\bibitem{Peetersprank} Peeters, R., On the $p$-ranks of the adjacency matrices
    of distance-regular graphs, {\sl J. Algebraic Combin.} 15 (2002), 127--149.

\bibitem{PenWil2011} Penttila, T., Williford, J., New families of
    $Q$-polynomial association schemes, {\sl J. Combin.
    Theory Ser. A} 118 (2011), 502--509.

\bibitem{PSV2011JCTA}
Pepe, V., Storme, L., Vanhove, F.,
Theorems of Erd\H{o}s--Ko--Rado type in polar spaces,
{\sl J. Combin. Theory Ser. A} 118 (2011), 1291--1312.

\bibitem{Pow88} Powers, D.L., Eigenvectors of distance-regular graphs, {\sl
    SIAM J. Matrix Anal. Appl.} 9 (1988), 399--407.


\bibitem{Powerssemiregular} Powers, D.L., Semiregular graphs and their algebra,
    {\sl Linear Multilinear Algebra} 24 (1988), 27--37.



\bibitem{Py99}
Pyber, L., A bound for the diameter of distance-regular graphs, {\sl
Combinatorica} 19 (1999), 549--553.

\bibitem{PyberHam}
%
Pyber, L.,
Large connected strongly regular graphs are Hamiltonian,
preprint (2014);
\arxiv{1409.3041}.


\bibitem{RCSprague} Ray-Chaudhuri, D.K., Sprague, A.P., Characterization of
    projective incidence structures, {\sl Geom. Dedicata} 5 (1976), 361--376.


\bibitem{RiHu} Rif\`{a}, J., Huguet, L., Classification of a class of distance-regular graphs via completely regular codes, {\sl Discrete Appl. Math.} 26 (1990), 289--300.

\bibitem{RiZipre08}
Rif\`{a}, J., Zinoviev, V.A.,
On a class of binary linear completely transitive codes with arbitrary covering radius,
{\sl Discrete Math.} 309 (2009), 5011--5016.

\bibitem{RZ2010IEEE}
Rif\`{a}, J., Zinoviev, V.A.,
New completely regular $q$-ary codes based on Kronecker products,
{\sl IEEE Trans. Inform. Theory} 56 (2010), 266--272;
\arxiv{0810.4993}.

\bibitem{RiZi11}
Rif\`{a}, J., Zinoviev, V.A.,
On lifting perfect codes,
{\sl IEEE Trans. Inform. Theory} 57 (2011), 5918--5925;
\arxiv{1002.0295}.

\bibitem{Rob16}
Roberson, D.E., Homomorphisms of strongly regular graphs, preprint
(2016);
\arxiv{1601.00969}.

\bibitem{r97} Rowlinson, P., Linear algebra, {\sl Graph
    Connections} (L.W. Beineke, R.J. Wilson, eds.),
    Oxford Lecture Series in Mathematics and its Applications, vol.~5,
    Oxford University Press, New York, 1997, pp.~86--99.

\bibitem{Salimi} Salimi, S., Quantum central limit theorem for
    continuous-time quantum walks on Odd graphs in quantum probability theory,
    {\sl Int. J. Theor. Phys.} 47 (2008), 3298--3309;
    \arxiv{0710.3043}.

\bibitem{Schrijver1979IEEE}
Schrijver, A.,
A comparison of the Delsarte and Lov\'{a}sz bounds,
{\sl IEEE Trans. Inform. Theory} 25 (1979), 425--429.

\bibitem{Schrijver2005IEEE}
Schrijver, A.,
New code upper bounds from the Terwilliger algebra and semidefinite programming,
{\sl IEEE Trans. Inform. Theory} 51 (2005), 2859--2866.


\bibitem{Seidel-2} Seidel, J.J., Strongly regular graphs with
    $(-1,1,0)$
    adjacency matrix having eigenvalue 3, {\sl Linear Algebra
    Appl.}
    1 (1968), 281--298.

\bibitem{Shimabukuro2005AC}
Shimabukuro, O.,
An analogue of Nakayama's conjecture for Johnson schemes,
{\sl Ann. Comb.} 9 (2005), 101--115.

\bibitem{Shimabukuro2007EJC}
Shimabukuro, O.,
On the number of irreducible modular representations of a $P$ and $Q$ polynomial scheme,
{\sl European J. Combin.} 28 (2007), 145--151.

\bibitem{Shimabukuro2011DM}
Shimabukuro, O.,
On structures of modular adjacency algebras of Johnson schemes,
{\sl Discrete Math.} 311 (2011), 978--983.

\bibitem{SY2014pre}
Shimabukuro, O., Yoshikawa, M.,
Modular adjacency algebras of Grassmann graphs,
{\sl Linear Algebra Appl.} 466 (2015), 208--217.

\bibitem{Shp98} 
Shpectorov, S.V., Distance-regular isometric subgraphs of the halved cubes,
{\sl European J. Combin.} 19 (1998), 119--136.

\bibitem{Simon1983P}
Simon, H.-U.,
A tight $\Omega(\log\log n)$-bound on the time for parallel RAMs to compute nondegenerated Boolean functions,
{\sl Foundations of Computation Theory} (M. Karpinski, ed.),
Lecture Notes in Computer Science, vol.~158, Springer-Verlag, Berlin, 1983, pp.~439--444.

\bibitem{Slijpe84} Van Slijpe, A.R.D., Random walks on regular polyhedra and
    other distance-regular graphs, {\sl Stat. Neerl.} 38 (1984), 273--292.


\bibitem{So93} Soicher, L.H., Three new distance-regular
    graphs,
    {\sl European J. Combin.} 14 (1993), 501--505.

\bibitem{So95} 
Soicher, L.H., Yet another distance-regular graph related to a Golay code,
{\sl Electron. J. Combin.} 2 (1995), N1.

\bibitem{Soicher15}
Soicher, L.H., The uniqueness of a distance-regular graph with
intersection array $\{32,27,8,1;1,4,27,32\}$ and related results,
preprint (2015);
\arxiv{1512.05976}.


\bibitem{Sole1990DM}
Sol\'{e}, P.,
Completely regular codes and completely transitive codes,
{\sl Discrete Math.} 81 (1990), 193--201.

\bibitem{Sotirov2012B} Sotirov, R., SDP relaxations for some combinatorial
    optimization problems, {\sl Handbook on Semidefinite, Conic and Polynomial
    Optimization} (M.F. Anjos, J.B. Lasserre, eds.), Springer, New York, 2012, pp.~795--819.

\bibitem{Sprague} Sprague, A.P., Incidence structures whose
    planes are nets, {\sl European J. Combin.} 2
    (1981), 193--204.

\bibitem{Stanton1981GD}
Stanton, D.,
Three addition theorems for some $q$-Krawtchouk polynomials,
{\sl Geom. Dedicata} 10 (1981), 403--425.

\bibitem{Stanton1985JCTA}
Stanton, D.,
Harmonics on posets,
{\sl J. Combin. Theory Ser. A} 40 (1985), 136--149.

\bibitem{Stanton1986GC}
Stanton, D.,
$t$-Designs in classical association schemes,
{\sl Graphs Combin.} 2 (1986), 283--286.

\bibitem{Suda2011JCD}
Suda, S.,
On spherical designs obtained from $Q$-polynomial association schemes,
{\sl J. Combin. Des.} 19 (2011), 167--177;
\arxiv{0910.4628}.

\bibitem{Suda2012JCTA}
Suda, S.,
New parameters of subsets in polynomial association schemes,
{\sl J. Combin. Theory Ser. A} 119 (2012), 117--134;
\arxiv{1008.0189}.

\bibitem{Suda2012EJC} Suda, S., On $Q$-polynomial association schemes of small
    class, {\sl Electron. J. Combin.} 19:1 (2012), P68;
    \arxiv{1202.5627}.

\bibitem{Suda2009pre}
Suda, S.,
Characterizations of regularity for certain $Q$-polynomial association schemes,
{\sl Electron. J. Combin.} 22:1 (2015), P1.12;
\arxiv{0910.4629}.

\bibitem{WS16}
Sumalroj, S., Worawannotai, C., The nonexistence of a
distance-regular graph with intersection array
$\{22,16,5;1,2,20\}$, {\sl Electron. J. Combin.} 23:1 (2016), P1.32.

\bibitem{Su91} 
Suzuki, H., Bounding the diameter of a distance regular graph by a function of
$k\sb {d}$, {\sl Graphs Combin.} 7 (1991), 363--375.

\bibitem{S294} 
Suzuki, H., On distance-regular graphs with $k\sb i=k\sb j$,
{\sl J. Combin. Theory Ser. B} 61 (1994), 103--110.

\bibitem{S594} 
Suzuki, H., Bounding the diameter of a distance regular graph by a function of
$k\sb {d}$, II, {\sl J. Algebra} 169 (1994), 713--750.

\bibitem{Su95} 
Suzuki, H., Distance-semiregular graphs, {\sl Algebra Colloq.} 2 (1995),
315--328.

\bibitem{Su295} 
Suzuki, H., On strongly closed subgraphs of highly regular graphs, {\sl
European J. Combin.} 16 (1995), 197--220.

\bibitem{S96} 
Suzuki, H., Strongly closed subgraphs of a distance-regular graph with
geometric girth five, {\sl Kyushu J. Math.} 50 (1996), 371--384.

\bibitem{Suzuki1996JCTA}
Suzuki, H., A note on association schemes with two $P$-polynomial structures of Type III, {\sl J. Combin. Theory Ser.  A} 74 (1996), 158--168.

\bibitem{Suzuki1998JACa}
Suzuki, H.,
Imprimitive $Q$-polynomial association schemes,
{\sl J. Algebraic Combin.} 7 (1998), 165--180.

\bibitem{Suzuki1998JACb}
Suzuki, H.,
Association schemes with multiple $Q$-polynomial structures,
{\sl J. Algebraic Combin.} 7 (1998), 181--196.

\bibitem{Su99} Suzuki, H., An introduction to distance-regular
    graphs, {\sl Three Lectures in Algebra} (K. Shinoda, ed.), Sophia
    Kokyuroku in Mathematics, vol.~41, Sophia University, Tokyo, 1999, pp.~57--132.

\bibitem{Suzuki2005JAC}
Suzuki, H.,
The Terwilliger algebra associated with a set of vertices in a distance-regular graph,
{\sl J. Algebraic Combin.} 22 (2005), 5--38.

\bibitem{Suzuki2006EJC}
Suzuki, H.,
The geometric girth of a distance-regular graph having certain thin irreducible modules for the Terwilliger algebra,
{\sl European J. Combin.} 27 (2006), 235--254.

\bibitem{Suzuki2007EJC}
Suzuki, H.,
On strongly closed subgraphs with diameter two and the $Q$-polynomial property,
{\sl European J. Combin.} 28 (2007), 167--185.

\bibitem{Suzuki2008GC}
Suzuki, H.,
Almost $2$-homogeneous graphs and completely regular quadrangles,
{\sl Graphs Combin.} 24 (2008), 571--585.

\bibitem{Suzuki2014JAC}
Suzuki, H.,
Completely regular clique graphs,
{\sl J. Algebraic Combin.} 40 (2014), 233--244.

\bibitem{Tanabe1997JAC} Tanabe, K., The irreducible modules of the Terwilliger
    algebras of Doob schemes, {\sl J. Algebraic Combin.} 6 (1997), 173--195.

\bibitem{Tanaka2006JCTA} Tanaka, H., Classification of subsets with minimal
    width and dual width in Grassmann, bilinear forms and dual polar graphs,
    {\sl J. Combin. Theory Ser. A} 113 (2006), 903--910.

\bibitem{Tanaka2009EJC}
Tanaka, H.,
New proofs of the Assmus--Mattson theorem based on the Terwilliger algebra,
{\sl European J. Combin.} 30 (2009), 736--746;
\arxiv{math/0612740}.

\bibitem{Tanaka2009LAAa}
Tanaka, H.,
A note on the span of Hadamard products of vectors,
{\sl Linear Algebra Appl.} 430 (2009), 865--867;
\arxiv{0806.2075}.

\bibitem{Tanaka2009LAAb}
Tanaka, H.,
A bilinear form relating two Leonard systems,
{\sl Linear Algebra Appl.} 431 (2009), 1726--1739;
\arxiv{0807.0385}.

\bibitem{Tanaka2011EJC} Tanaka, H., Vertex subsets with minimal width and dual
    width in $Q$-polynomial distance-regular graphs,
    {\sl Electron. J. Combin.} 18:1 (2011), P167;
    \arxiv{1011.2000}.

\bibitem{Tanaka2010pre} Tanaka, H., The Erd\H{o}s--Ko--Rado theorem for twisted
    Grassmann graphs, {\sl Combinatorica} 32 (2012), 735--740;
    \arxiv{1012.5692}.

\bibitem{Tanaka2012pre}
Tanaka, H.,
The Erd\H{o}s--Ko--Rado basis for a Leonard system,
{\sl Contrib. Discrete Math.} 8 (2013), 41--59;
\arxiv{1208.4050}.

\bibitem{TT2011EJC}
Tanaka, H., Tanaka, R.,
Nonexistence of exceptional imprimitive $Q$-polynomial association schemes with six classes,
{\sl European J. Combin.} 32 (2011), 155--161;
\arxiv{1005.3598}.

\bibitem{TTW16+}
Tanaka, H., Tanaka, R., Watanabe, Y.,
The Terwilliger algebra of a $Q$-polynomial distance-regular graph with respect to a set of vertices,
in preparation.

\bibitem{Teirlinck1987DM}
Teirlinck, L.,
Nontrivial $t$-designs without repeated blocks exist for all $t$,
{\sl Discrete Math.} 65 (1987), 301--311.

\bibitem{TerMult82} Terwilliger, P., Eigenvalue multiplicities of highly
    symmetric graphs, {\sl Discrete Math.} 41 (1982), 295--302.

\bibitem{Terwscak} Terwilliger, P., Distance-regular graphs and
    $(s,c,a,k)$-graphs, {\sl J. Combin. Theory Ser. B} 34 (1983), 151--164.

\bibitem{Ter-quad} Terwilliger, P., Distance-regular graphs with girth $3$ or $4$:
    I, {\sl J. Combin. Theory Ser. B} 39 (1985), 265--281.

\bibitem{Terwilliger1985AGG}
Terwilliger, P.,
Counting $4$-vertex configurations in $P$- and $Q$-polynomial association schemes,
{\sl Algebras Groups Geom.} 2 (1985), 541--554.

\bibitem{Ter-b1} Terwilliger, P., The Johnson graph $J(d,r)$
    is unique if $(d,r) \neq (2,8)$, {\sl Discrete Math.} 58 (1986), 175--189.

\bibitem{Terwilliger1986JCTB} Terwilliger, P., A class of distance-regular
    graphs that are $Q$-polynomial, {\sl J. Combin. Theory Ser. B} 40 (1986),
    213--223.

\bibitem{Ter-root}
Terwilliger, P.,
Root systems and the Johnson and Hamming graphs,
{\sl European J. Combin.} 8 (1987), 73--102.

\bibitem{Terwilliger1987JCTA} Terwilliger, P., A characterization of $P$- and
    $Q$-polynomial association schemes, {\sl J. Combin. Theory Ser. A} 45
    (1987), 8--26.

\bibitem{Terwilliger1987JCTB} Terwilliger, P., $P$ and $Q$ polynomial schemes
    with $q=-1$, {\sl J. Combin. Theory Ser. B} 42 (1987), 64--67.

\bibitem{Terwilliger1988C} Terwilliger, P., The classification of
    distance-regular graphs of type IIB, {\sl Combinatorica} 8 (1988),
    125--132.

\bibitem{Terwilliger1988GC} Terwilliger, P., Balanced sets and $Q$-polynomial
    association schemes, {\sl Graphs Combin.} 4 (1988), 87--94.

\bibitem{Terwilliger1990P} Terwilliger, P., The incidence algebra of a uniform
    poset, {\sl Coding Theory and Design Theory, Part I} (D. Ray-Chaudhuri,
    ed.), The IMA Volumes in Mathematics and its Applications, vol.~20,
    Springer-Verlag, New York, 1990, pp.~193--212.

\bibitem{Talgebra92} Terwilliger, P., The subconstituent
    algebra of an association scheme, I, {\sl J. Algebraic Combin.} 1
    (1992), 363--388; II, {\sl J. Algebraic Combin.}
    2 (1993), 73--103; III, {\sl J. Algebraic Combin.} 2
    (1993), 177--210.

\bibitem{Terwilliger1993EJC}
Terwilliger, P.,
$P$- and $Q$-polynomial association schemes and their antipodal $P$-polynomial covers,
{\sl European J. Combin.} 14 (1993), 355--358.

\bibitem{Terwilliger1993N} Terwilliger, P., The subconstituent algebra of a
    graph, the thin condition, and the $Q$-polynomial property, unpublished
    lecture notes (1993).

\bibitem{Terwilliger1995EJC} Terwilliger, P., Kite-free distance-regular
    graphs, {\sl European J. Combin.} 16 (1995), 405--414.

\bibitem{Terwilliger1995DM} Terwilliger, P., A new inequality for
    distance-regular graphs, {\sl Discrete Math.} 137 (1995), 319--332.

\bibitem{Terwilliger1996P} Terwilliger, P., Quantum matroids, {\sl Progress in
    Algebraic Combinatorics} (E. Bannai, A. Munemasa, eds.), Advanced Studies
    in Pure Mathematics, vol.~24, Mathematical Society of Japan, Tokyo,
    1996, pp.~323--441.

\bibitem{Terwilliger2001P} Terwilliger, P., Two relations that generalize the
    $q$-Serre relations and the Dolan--Grady relations, {\sl Physics and
    Combinatorics 1999} (A.N. Kirillov, A. Tsuchiya, H. Umemura, eds.), World
    Scientific Publishing, River Edge, NJ, 2001, pp.~377--398;
    \arxiv{math/0307016}.

\bibitem{Terwilliger2001LAA}
Terwilliger, P.,
Two linear transformations each tridiagonal with respect to an eigenbasis of the other,
{\sl Linear Algebra Appl.} 330 (2001), 149--203;
\arxiv{math/0406555}.

\bibitem{Terwilliger2002LAA} Terwilliger, P., The subconstituent algebra of a
    distance-regular graph; thin modules with endpoint one, {\sl Linear Algebra
    Appl.} 356 (2002), 157--187.

\bibitem{Terwilliger2004JAC} Terwilliger, P., An inequality involving the local
    eigenvalues of a distance-regular graph, {\sl J. Algebraic Combin.} 19
    (2004), 143--172.

\bibitem{Terwilliger2005GC}
Terwilliger, P.,
The displacement and split decompositions for a $Q$-polynomial distance-regular graph,
{\sl Graphs Combin.} 21 (2005), 263--276;
\arxiv{math/0306142}.

\bibitem{Terwilliger2006N}
Terwilliger, P.,
An algebraic approach to the Askey scheme of orthogonal polynomials,
{\sl Orthogonal Polynomials and Special Functions, Computation and Applications} (F. Marcell\'{a}n, W. Van Assche, eds.),
Lecture Notes in Mathematics, vol.~1883,
Springer-Verlag, Berlin, 2006, pp.~255--330;
\arxiv{math/0408390}.

\bibitem{TW04}
Terwilliger, P., Weng, C.,
Distance-regular graphs, pseudo primitive idempotents, and the Terwilliger algebra,
{\sl European J. Combin.} 25 (2004), 287--298;
\arxiv{math/0307269}.

\bibitem{TW05} Terwilliger, P., Weng, C., An inequality
    for regular near polygons, {\sl European J. Combin.} 26
    (2005), 227--235;
    \arxiv{math/0312149}.

\bibitem{TW2013AMC}
Terwilliger, P., Worawannotai, C.,
Augmented down-up algebras and uniform posets,
{\sl Ars Math. Contemp.} 6 (2013), 409--417;
\arxiv{1206.0455}.

\bibitem{TZ2013pre}
Terwilliger, P., \v{Z}itnik, A.,
Distance-regular graphs of $q$-Racah type and the universal Askey-Wilson algebra,
{\sl J. Combin. Theory Ser. A} 125 (2014), 98--112;
\arxiv{1307.7968}.

\bibitem{To96} 
Tomiyama, M., On distance-regular graphs with height two, {\sl J. Algebraic
Combin.} 5 (1996), 57--76.

\bibitem{To98} 
Tomiyama, M., On distance-regular graphs with height two, II, {\sl J. Algebraic
Combin.} 7 (1998), 197--220.

\bibitem{Tomiyama2001DM}
Tomiyama, M.,
On the primitive idempotents of distance-regular graphs,
{\sl Discrete Math.} 240 (2001), 281--294.

\bibitem{TY1994KJM}
Tomiyama, M., Yamazaki, N.,
The subconstituent algebra of a strongly regular graph,
{\sl Kyushu J. Math.} 48 (1994), 323--334.

\bibitem{JurTonpre} Tonejc, J., More balancing for distance-regular graphs, {\sl European J. Combin.} 34 (2013),
    195--206.

\bibitem{urlep12}
Urlep, M.,
Triple intersection numbers of $Q$-polynomial distance-regular graphs,
{\sl European J. Combin.} 33 (2012), 1246--1252.

\bibitem{Va08} Vallentin, F., Optimal distortion embeddings of
    distance regular graphs into Euclidean spaces, {\sl J. Combin. Theory Ser. B} 98 (2008), 95--104;
\arxiv{math/0509716}.

\bibitem{Vanhove2011PhD}
Vanhove, F.,
{\sl Incidence geometry from an algebraic graph theory point of view},
thesis,
Ghent University, 2011;
\url{https://cage.ugent.be/geometry/Theses/50/VanhovePhd.pdf}.

\bibitem{Vanhove2011JCD}
Vanhove, F.,
Antidesigns and regularity of partial spreads in dual polar graphs,
{\sl J. Combin. Des.} 19 (2011), 202--216.

\bibitem{Vanhove2012JAC} Vanhove, F., A Higman inequality for regular near
    polygons, {\sl J. Algebraic Combin.} 34 (2011), 357--373.

\bibitem{Vidali} Vidali, J., There is no distance-regular graph with intersection array $\{55, 54, 50, \break 35, 10; 1, 5, 20, 45, 55\}$, preprint (2013).

\bibitem{VJ2013pre}
Vidali, J., Juri\v{s}i\'{c}, A.,
Nonexistence of a family of tight distance-regular graphs with classical parameters,
preprint (2013).

\bibitem{Waji94} Wajima, M., A remark on distance-regular graphs with a circuit
    of diameter $t+1$, {\sl Math. Japon.} 40 (1994),
    433--437.


\bibitem{W01} 
Wang, K., The existence of strongly closed subgraphs in highly regular graphs,
{\sl Algebra Colloq.} 8 (2001), 257--266.

\bibitem{w82} Weichsel, P.M., On distance-regularity in
    graphs,
    {\sl J. Combin. Theory Ser. B} 32 (1982), 156--161.

\bibitem{Weng1995GC} Weng, C.-W., Kite-free $P$- and $Q$-polynomial schemes,
    {\sl Graphs Combin.} 11 (1995), 201--207.

\bibitem{W97} 
Weng, C., $D$-bounded distance-regular graphs, {\sl European J. Combin.} 18
(1997), 211--229.

\bibitem{W98} 
Weng, C., Weak-geodetically closed subgraphs in distance-regular graphs, {\sl
Graphs Combin.} 14 (1998), 275--304.

\bibitem{We99} 
Weng, C., Classical distance-regular graphs of negative type, {\sl J. Combin.
Theory Ser. B} 76 (1999), 93--116.

\bibitem{Wilson1975JCTA}
Wilson, R.M.,
An existence theory for pairwise balanced designs, III, Proof of the existence conjectures,
{\sl J. Combin. Theory Ser. A} 18 (1975), 71--79.

\bibitem{Worawannotai2012PhD}
Worawannotai, C.,
Dual polar graphs, the quantum algebra $U_q(\mathfrak{sl}_2)$, and Leonard systems of dual $q$-Krawtchouk type,
{\sl Linear Algebra Appl.} 438 (2013), 443--497;
\arxiv{1205.2144}.

\bibitem{Y95} 
Yamazaki, N., Distance-regular graphs with $\Gamma(x)\simeq 3*K\sb {a+1}$, {\sl
European J. Combin.} 16 (1995), 525--536.

\bibitem{Yamazaki1996JCTB}
Yamazaki, N.,
Bipartite distance-regular graphs with an eigenvalue of multiplicity $k$,
{\sl J. Combin. Theory Ser. B} 66 (1996), 34--37.

\bibitem{Yoshikawa2004JAC}
Yoshikawa, M.,
Modular adjacency algebras of Hamming schemes,
{\sl J. Algebraic Combin.} 20 (2004), 331--340.

\bibitem{ZGG09} Zhang, X., Guo, J., Gao, S., Two new
    error-correcting pooling designs from $d$-bounded
    distance-regular graphs, {\sl J. Comb. Optim.} 17 (2009), 339--345.

\bibitem{Zhuthesis} Zhu, R.R., {\sl Distance-regular graphs and eigenvalue
    multiplicities}, thesis, Simon Fraser University, 1989;
    \url{http://ir.lib.sfu.ca/bitstream/1892/6153/1/b14460099.pdf}.

\bibitem{Zhu93} Zhu, R.R., Distance-regular graphs with an eigenvalue of
    multiplicity four, {\sl J. Combin. Theory Ser. B} 57 (1993),
    157--182.




\end{thebibliography}
